\documentclass[11pt]{amsart}
\usepackage{latexsym}
\usepackage{amsthm,amsmath,amssymb}
\usepackage[english]{babel}
\usepackage{enumerate}
\usepackage{cite}
\usepackage{graphicx}
\usepackage{caption}
\usepackage{subcaption}

\usepackage{color}
\usepackage[dvipsnames]{xcolor}
\usepackage{fullpage}
\usepackage{array}
\usepackage[colorlinks=true]{hyperref}

\def\R{\mathbb R}

\newtheorem{theo}{\textbf{Theorem}}[section]

\newtheorem{lem}[theo]{\textbf{Lemma}}

\newtheorem{cor}[theo]{\textbf{Corollary}}

\newtheorem{rem}[theo]{\textbf{Remark}}

\newtheorem{ass}{\textbf{Assumption}}

\DeclareMathOperator\supp{supp}

\title{Exponential convergence to a steady state for a population genetics model with sexual reproduction and selection}

\author{Gaël Raoul}

\address{CMAP, CNRS, \'Ecole polytechnique, I.P. Paris, 91128 Palaiseau, France.}
\email{gael.raoul@polytechnique.edu}



\begin{document}

\maketitle


\begin{abstract}
 We are interested in the dynamics of a population structured by a phenotypic trait. Individuals reproduce sexually, which is represented by a non-linear integral operator. This operator is combined to a multiplicative operator representing selection. When the strength of selection is small, we show that the dynamics of the population is governed by a simple macroscopic differential equation, and that solutions converge exponentially to steady states that are locally unique. The analysis is based on Wasserstein distance inequalities using a uniform lower bound on distributions. These inequalities are coupled to tail estimates to show the stability of the steady states.
\end{abstract}

\section{Introduction}

In this paper, we are interested in the dynamics of the following kinetic equation that describes a population that is structured by a phenotypic trait and reproduces sexually:
\begin{align}
\partial_t n(t,x)&=\iint_{\mathbb R^2}\Gamma_{\sigma^2}\left(x-\frac{y_*+y_*'}2\right)(1+\alpha a(y_*))n(t,y_*)(1+ \alpha a(y_*'))n(t,y_*')\,dy_*\,dy_*'\nonumber\\
&\quad-\left(1+\alpha \int_{\mathbb R} a(y)n(t,y)\,dy\right)^2n(t,x).\label{eq:model}
\end{align}
In this model, $t\in[0,\infty)$ is the time variable, $x\in\mathbb R$ is a phenotypic trait. We assume that $\int_{\mathbb R} n(0,x)\,dx=1$, and notice that it implies $\int_{\mathbb R} n(t,x)\,dx=1$ for $t\geq 0$, so that $n(t,\cdot)$ is the density of individuals along the phenotypic trait $x$. We refer to section~\ref{sec:biologicalsetting} for a discussion of this property that is related to a classical simplification used in population genetics. The function $a=a(x)$ represents selection: the trait $x$ of an individual affects the number of reproductions it takes part in, and this is modulated by a parameter $\alpha>0$ (the \emph{strength of selection}) that we will assume small in this study. Throughout this manuscript, $\Gamma_{\sigma^2}$ is a Gaussian distribution:
\begin{equation}\label{def:Gaussian}
\Gamma_{\sigma^2}(x):= \frac 1{\sigma\sqrt{2\pi}}e^{-\frac{x^2}{2\sigma^2}}.
\end{equation}
The first term on the right hand side of \eqref{eq:model} is the birth term: each individual of trait $y$ contributes to the production of gametes at a rate $(1+\alpha a(y))$. Two parents, with traits $y_*$ and $y_*'$, are necessary for a reproduction to occur, and the offspring's trait is then drawn from a Gaussian distribution centred in $\frac{y_*+y_*'}2$, following Fisher's \emph{Infinitesimal Model} \cite{fisher1919xv,barton2017infinitesimal}. The last term  of \eqref{eq:model} is a death term, with a death rate that is independent of $x$, and that keeps the population size constant. For a discussion of the biological aspects of this model, we refer to Section~\ref{sec:biologicalsetting}.

\medskip

Selection models for asexual populations have been thoroughly studied in the last decades (see e.g. \cite{diekmann2005dynamics,desvillettes2008selection}). A fruitful approach has been to use a Hopf-Cole transform to relate the asexual population models to \emph{constrained Hamilton-Jacobi equations} when the mutation rate is small \cite{diekmann2005dynamics,lorz2011dirac}. In \cite{calvez2019asymptotic}, this idea was extended to a sexual selection model (very close to \eqref{eq:model}), through $U(t,x):=-\sigma^2 \log n(t,x)$ and the asymptotics of $\sigma>0$ is small. The authors are then able to show the local uniqueness of steady states (the existence of which is proven in \cite{bourgeron2017existence}). The idea is then to take advantage of the contractive effect of the reproduction operator on derivatives of $U$, that is $\partial_x^i U(t,x)$ for $i\in\{1,2,3\}$. A combination of these derivatives is employed to build a functional space where the reproduction operator is contractive close to a steady state of the equation. 
In this manuscript, we will try to describe the dynamics of \eqref{eq:model}, that is closely related to the model considered in \cite{calvez2019asymptotic}. We will however use a different analysis approach, based on Wasserstein distances rather than regularity estimates. 

\medskip

The $2-$Wasserstein distance is defined on $\mathcal P_2(\mathbb{R})$, the set of probability measures on $\mathbb R$ with a bounded second moment: for $n,m\in\mathcal P_2(\mathbb R)$, 
\[W_2(n,m)= \left(\inf_{\pi\in \Pi(n,m)} \iint_{\mathbb R^2}|x-y|^2\,d\pi(x,y)\right)^{1/2},\]
where $\Pi(n,m)$ is the set of measures on $\mathbb R^2$ with marginals $n$ and $m$. Equivalently, the distance $W_2$ can be defined as $W_2(n,m)=\inf\left\{\left[\mathbb E\left(|X-Y|^2\right)\right]^{1/2}; \,X\sim n,\,Y\sim m\right\}$, which enhances its relations to coupling methods in probability theory. Wasserstein distances have  proven useful in probability and in analysis and we review below some of existing studies that are close to the analysis done in this manuscript.

Equation \eqref{eq:model} is closely related to models that describe the alignment of either biological individuals (fish, birds, etc.) or physical rods \cite{bertin2006boltzmann,vicsek1995novel,degond2014local}. More broadly, it is related to inelastic Boltzmann equations \cite{villani2006mathematics}. In this context, the Wasserstein distances have been used to describe the long time dynamics of solutions \cite{pulvirenti2004asymptotic,bolley2007tanaka}, following the seminal work of Tanaka \cite{tanaka1978probabilistic}. Note that a related set of works uses Fourier-based distances that have related properties \cite{bobylev1976fourier, carrillo2007contractive}.

Wasserstein distance methods have also been used to analyse Fokker-Planck equations augmented with non-local interaction potentials. They have proven very useful to show the existence and stability of steady states for these models, when the potentials satisfy some convexity assumptions.  We refer to \cite{carrillo2007contractive} for dimension~$1$ and to \cite{malrieu2003convergence} for the general case. An interesting development was proposed by {\sc J.A. Carrillo}, {\sc R.J. McCann}, and {\sc C. Villani} in  \cite{carrillo2003kinetic,carrillo2006contractions}, where the authors considered interaction potentials between particles that are (slightly)  non-convex. The non-convexity of the interaction potentials is detrimental to the contraction of solutions towards the steady state: the equation having a gradient-flow structure, the stability of steady states is related to the convexity of an energy along some geodesics. It can however be shown that these detrimental effects can be balanced out by the contraction effects from other operators present in the equation, provided the non-convexity only occurs in parts of the domain where numerous particles are present. We refer to \cite{cattiaux2008probabilistic,bolley2012convergence,bolley2013uniform} for the development of this idea into more quantitative results.

In this manuscript, we are interested in the stability of the steady states of \eqref{eq:model} when $\alpha>0$ is small.  This hypothesis on $\alpha$ is similar (although it is not exactly the same) to the asymptotics considered in \cite{patout2020cauchy,calvez2019asymptotic}: the authors assume a fixed selection function affecting the death rate of individuals, and an asymptotic parameter $0<\sigma\ll 1$ (designated by $\varepsilon>0$ in \cite{patout2020cauchy,calvez2019asymptotic}). A rescaling of the model they introduce along the phenotypic variable ($\tilde x=x/\sigma$) and the time variable ($\tilde t =\sigma^2t$) shifts the asymptotic parameter in the model considered in \cite{patout2020cauchy,calvez2019asymptotic} to the mortality rate, that becomes asymptotically flat, which is coherent with the weak selection asymptotics we consider here - although the model is considered on a longer time scale. We will however develop a different idea to describe the dynamics of solutions: we will rely on the Wasserstein contraction implied by the reproduction operator (see \eqref{eq:contraction} and \cite{raoul2017macroscopic}). The main difficulty is then to balance the detrimental effects of the selection term. We restrict our analysis to a situation where the selection function $(1+\alpha a(\cdot))$ is constant equal to $1$ outside of a compact set, so that $n(t,\cdot)$ can be bounded from below on the set where $(1+\alpha a(x))\neq 1$. This allows us to use an argument inspired by the developments made on Fokker-Planck equations with non-convex interaction potentials and that have been discussed above.



In \cite{mirrahimi2013dynamics}, a model close to \eqref{eq:model} with an additional spatial variable was considered. Heuristic arguments and numerical simulations were used to describe the macroscopic quantities of that model and show that their dynamics can be approximated by the Kirkpatrick-Barton model \cite{kirkpatrick1997evolution}. In \cite{raoul2017macroscopic}, this idea was made rigorous thanks to Wasserstein estimates combined with parabolic estimates. Recently, in \cite{guerand2023moment}, the authors introduce a different idea to derive a macroscopic limit, for a model related to \eqref{eq:model}. They succeeded to obtain a closed estimate on some higher moments of the distribution $n(t,\cdot)$, which provides an upper bound uniform on the small parameter used for the macroscopic limit. The convergence of the full distribution to the distribution described by the macroscopic model ensues.

Note that in the present study, we try to go beyond this macroscopic asymptotics: we want to prove that the distribution $n(t,\cdot)$ of the population converges to a unique steady state. To do so, one needs to show that the contraction effect of the reproduction operator $T$ can dominate the influence of the multiplicative operator $n\mapsto (1+\alpha a(\cdot))n$.

Since the first preprint of this study was released, very interesting results have been obtained on the dynamics of models similar to \eqref{eq:model}, with non-overlapping generations. In \cite{PoyatoCalvez}, a case with a quadratic death rate  has been considered, and the authors have been able to use this particular death rate to derive explicit estimates on the binary tree of ancestors of individuals alive at a given time, and obtain contraction estimates from that analysis, obtaining the convergence of solutions to a unique steady state. In \cite{calvez2023uniform}, {\sc V. Calvez}, {\sc D. Poyato} and {\sc F. Santambrogio} have considered general convex death rates which, thanks to $L^\infty$-Wasserstein estimates, take advantage of the log-concavity of the steady state of the model. Note that this result does not assume an exact form of the selection function at infinity (like in our study where $(1+\alpha a(x))=1$ for $x$ large, or in \cite{PoyatoCalvez} where the selection is assumed quadratic), and it provides a powerful contraction result on the reproduction operator. These results were obtained in \cite{calvez2023uniform} in the case where $x\in\mathbb R$, and the recent manuscript \cite{khudiakova2024infty} generalizes these estimates to $x\in\mathbb R^N$, testifying the flexibility of this approach. The estimates obtained  in \cite{calvez2023uniform,khudiakova2024infty} open exciting research perspectives, one of them being the extension of the framework to non-convex selection terms:  the present manuscript suggests that solutions could remain log-concave provided the non-convexity is not too pronounced.

\subsection{Main results and structure of the paper}

Our goal is to prove the existence and uniqueness of a steady state of \eqref{eq:model} in a neighbourhood of $\Gamma_{2\sigma^2}(\cdot-\bar Z)$, provided $\alpha>0$ is small enough. Furthermore, we show that this steady state is stable: under some conditions on the initial condition, solutions converge exponentially to that steady state. Given $\bar Z>0$ and $R>0$, that respectively represent a stable phenotypic trait for a macroscopic model (see \eqref{EDOpopgen}) and a bound on the support of the selection function $a(\cdot)$, we make the following assumptions, on the coefficients of the model (Assumption~\ref{ass1}) and the initial condition (Assumption~\ref{ass2}):
\begin{ass}\label{ass1}
Let $\bar Z\in\mathbb R$ and $a\in W^{2,\infty}(\mathbb R,\mathbb R_+)$. We assume that $a$ is compactly supported, and
\[F(\bar Z)=0,\quad F'(\bar Z)<0,\]
where
\begin{equation}\label{def:F}
F(Y):= \int_{\mathbb R} x a(x)\Gamma_{2\sigma^2}(x-Y)\,dx-Y\int_{\mathbb R} a(x)\Gamma_{2\sigma^2}(x-Y)\,dx.
\end{equation}
\end{ass}
From a biological point of view, this first assumption can be understood as follows: the trait $\bar Z$ is a local fitness maximum, under the assumption that the population is distributed as a Gaussian distribution around its mean phenotypic trait. We will show that the mean phenotypic trait of the population described by \eqref{eq:model} will stabilize close to $\bar Z$. Assumption~\ref{ass2} below provides conditions on the initial population. The \emph{stability of the phenotypic trait $\bar Z$} states that the trait $\bar Z$ only needs to be locally stable: we will show that $\bar Z$ attracts solutions with an initial mean phenotypic trait that belongs to the basin of attraction of the vector field $F$. In particular, the model can have several traits $\bar Z_1,\,\bar Z_2,\dots$ that satisfy Assumption~\ref{ass1}, as illustrated in Figure~\ref{fig:1}. The other conditions considered in Assumption~\ref{ass2} are technical and we may hope to relax some of them in future work; the main purpose of these technical assumptions is to propagate the exponential decay of solutions in $x$. Assumption~\ref{ass2} involves a parameter $R$. Note that if Assumption~\ref{ass2} holds for $R>0$, it also holds for any $R'\geq R$. 
\begin{ass} \label{ass2} Let $\bar Z\in\mathbb R$ and $a\in W^{2,\infty}(\mathbb R,\mathbb R_+)$ satisfying Assumption~\ref{ass1} Let $R>0$ such that 
\[\textrm{supp }a\subset [-R,R].\] 
We say that $n^0$ satisfies Assumption~\ref{ass2} if $n^0\in\mathcal P_2(\mathbb R)\cap C^1(\mathbb R)$, and if the following conditions are met, where $Z^0:=\int_{\mathbb R} xn^0(x)\,dx$: 
\begin{itemize}
\item Exponential decay: 
\[\forall x\in[R,+\infty),\quad \partial_x n^0(-x)>n^0(-x),\quad \partial_x n^0(x)<-n^0(x).\]
\item Lower bound:  $n^0\geq \frac 1R \Gamma_{1/R}$. 
\item $Z^0$ belongs to the basin of attraction of $\bar Z$ for the vector field $F$, in the sense that $Z^0\in [Z_-,Z_+]$, where: 
\[Z_-=\inf\left\{z\leq \bar Z-\frac{|F'(\bar Z)|}{4\|F''\|_\infty}; \min_{x\in[z,\bar Z-|F'(\bar Z)|/4\|F''\|_\infty]}F(x)\geq \frac 1R\right\},\]
\[Z_+=\inf\left\{z\geq \bar Z+\frac{|F'(\bar Z)|}{4\|F''\|_\infty}; \max_{x\in[\bar Z+|F'(\bar Z)|/4\|F''\|_\infty,z]}F(x)\leq -\frac 1R\right\}.\]
\end{itemize}
where $Z^0:=\int_{\mathbb R} x\,n^0(x)\,dx$ is the mean phenotypic trait of $n^0$ and $F$ is defined by \eqref{def:F}.
\end{ass}


Our first result is that Assumption~\ref{ass1} implies the existence of a {steady state} of \eqref{eq:model} close to $\Gamma_{2\sigma^2}(\cdot-\bar Z)$, provided $\alpha>0$ is small enough:
\begin{theo}\label{Thm:existence-steady}
Let $\bar Z\in\mathbb R$, $a\in W^{2,\infty}(\mathbb R)$ satisfy Assumption~\ref{ass1}.  

There are $\bar \alpha>0$, $C>0$ and $R>0$ such that for $\alpha\in(0,\bar \alpha)$, there is a steady state $\bar n\in\mathcal P_2(\mathbb R)$ of~\eqref{eq:model} that satisfies Assumption~\ref{ass2}, and
\[W_2(\bar n,\Gamma_{2\sigma^2}(\cdot-\bar Z))\leq C\sqrt \alpha.\]
\end{theo}
This theorem is proven in Section~\ref{subsec:proofsteady}.

\medskip

Before stating the main result of this manuscript, we define solutions of model \eqref{eq:model}, in the following theorem:
\begin{theo}\label{cor:existence}
Let $a\in W^{2,\infty}(\mathbb R)$, $\alpha,\sigma>0$ and $T$ defined by \eqref{def:T}. For $n^0\in\mathcal P_2(\mathbb R)\cap L^\infty(\mathbb R)$, there is a unique solution $n$ of \eqref{eq:model} in the following sense: $n\in L^\infty([0,\infty)\times\mathbb R)$, such that $\int_{\mathbb R} n(t,x)\,dx= 1$ for $t\in[0,\infty)$, and  
\begin{align*}
&n(t,x)=n^0(x)e^{-\int_0^t\left(1+\alpha  I_{n}(s)\right)^2\,ds}\nonumber\\
&\quad +\int_0^t\left(1+\alpha  I_n(s)\right)^2T\left(\frac{(1+\alpha a)n(s,\cdot)}{1+\alpha  I_n(s)}\right)(x)e^{-\int_s^t\left(1+\alpha  I_n(\tau)\right)^2\,d\tau}\,ds,
\end{align*}
for $(t,x)\in[0,\infty)\times \mathbb R$, with $I_n(t)=\int_{\mathbb R}a(x)n(t,x)\,dx$.  This solution can be written as $n(t,x)=n^0(x)g(t,x)+h(t,x)$ with $g,h\in C^\infty([0,\infty)\times \mathbb R)$.
\end{theo}
This theorem is proven in the appendix (see Section~\ref{App:existence} and Remark~\ref{rem:eximodel}).

\medskip

The existence and uniqueness of solutions of \eqref{eq:model} is proven in Theorem~\ref{cor:existence}. In the next result, we show that if $\alpha>0$ is small, we can describe the dynamics of the solution $n=n(t,x)$ of \eqref{eq:model} as follows:
\begin{itemize}
\item Rapidly, that is for $t\geq -C\ln \alpha/\alpha$, the $n(t,\cdot)$ is well approximated by a Gaussian distribution of variance $2\sigma^2$ centered around its mean phenotypic trait $Z_n(t)=\int_{\mathbb R} xn(t,x)\,dx$.
\item The mean phenotypic trait $Z_n(t)$ is close to the solution of the ordinary differential equation $Y'(\tau)=F(Y(\tau))$, with the time scale $\tau=\alpha t$.
\item When $t\to\infty$, the $n(t,\cdot)$ converges to $\bar n\in\mathcal P_2(\mathbb R)$, the steady state of \eqref{eq:model} provided by Theorem~\ref{Thm:existence-steady}.
\end{itemize}
More precisely, we prove the following theorem, which is the main result of this manuscript:
\begin{theo}\label{Thm:main}
Let $\bar Z\in\mathbb R$, $R>0$ and $a\in W^{2,\infty}(\mathbb R)$ satisfy Assumption~\ref{ass1}.  

There are $\bar \alpha>0$, $\rho\in(0,1]$, $C>0$ such that for $\alpha\in(0,\bar \alpha)$, if we denote by $\bar n\in\mathcal P_2(\mathbb R)$ the steady state of \eqref{eq:model} given by Theorem~\ref{Thm:existence-steady}, and if $n^0\in\mathcal P_2(\mathbb R)\cap C^1(\mathbb R)$ satisfies Assumption~\ref{ass2} as well as $W_2\left(n^0,\Gamma_{2\sigma^2}\left(\cdot-Z^0\right)\right)\leq \rho$, and if $n\in L^\infty(\mathbb R_+\times\mathbb R)$ is the solution of \eqref{eq:model} with initial data $n^0$ provided by Theorem~\ref{cor:existence}, then for $t\in[0,\infty)$,
\begin{equation}\label{est:diffZY}
\left|\int_{\mathbb R} xn(t,x)\,dx-Y(\alpha t)\right|\leq \frac C{-\ln\alpha},\quad 
W_2\big(n(t,\cdot),\Gamma_{2\sigma^2}(Z_n(t)-\cdot)\big)\leq C\sqrt{\alpha}+C\rho e^{-t/8},
\end{equation}
where $Z_n(t)=\int_{\mathbb R}xn(t,x)\,dx$ and $Y\in C^1([0,\infty))$ is the solution of 
\begin{equation}\label{eq:EDOintro}
\left\{\begin{array}{l}
Y'(\tau)=F(Y(\tau)),\quad \tau\in(0,\infty),\vspace{0.25cm}\\
Y(0)=\displaystyle\int_{\mathbb R}x\,n^0(x)\,dx.
\end{array}\right.
\end{equation}
Moreover, 
\begin{equation}\label{estW2thm}
W_2(n(t,\cdot),\bar n)\leq C e^{\left(F'(\bar Z)+C\sqrt\alpha\right)\alpha t}.
\end{equation}
\end{theo}
This theorem is proven in Section~\ref{subsec:proofmainThm}.

\medskip

Note that $Z_n(t)$ changes at a slow time scale $\tau=\alpha t$ if $\alpha>0$ is small. This is coherent with the Fundamental Theorem of natural selection introduced by Fisher \cite{fisher1958genetical}, described in Section~\ref{sec:biologicalsetting}. Note that if $\alpha=0$, then $Z_n(t)$ becomes invariant in \eqref{eq:model}, and the convergence \eqref{estW2thm} does not hold. The asymptotic result \eqref{estW2thm} then holds for $\alpha>0$ small, but it is not true if $\alpha=0$.

\medskip

Theorem~\ref{Thm:main} also provides a local uniqueness result for steady states of the model \eqref{eq:model}: If $\alpha>0$ is small enough, if $\tilde n\in\mathcal P_2(\mathbb R)$ is a steady state with a mean phenotypic trait in the basin of attraction of $\bar Z$ for the vector field $F$, we show that $\tilde n=\bar n$, where $\bar n\in\mathcal P_2(\mathbb R)$ is the steady state provided by Theorem~\ref{Thm:existence-steady}, provided $\tilde n$ satisfies an exponential tails condition:

\begin{cor}\label{cor:uniqueness}
Let $\bar Z\in\mathbb R$, $\tilde R>0$ and $a\in W^{2,\infty}(\mathbb R)$ satisfy Assumption~\ref{ass1}. Let $\tilde Z_-<\bar Z$, $\tilde Z_+>\bar Z$ such that
\[\forall x\in[\tilde Z_-,\bar Z),\, F(x)>0,\quad \forall x\in(\bar Z,\tilde Z_+],\, F(x)<0.\]
There is $\bar \alpha>0$ such that if $\tilde n\in\mathcal P_2(\mathbb R)\cap C^1(\mathbb R)$ is a steady state of \eqref{eq:model} for $\alpha\in(0,\bar \alpha)$ with $\int_{\mathbb R}x\tilde n(x)\,dx\in [\tilde Z_-,\tilde Z_+]$, and if
\begin{equation}\label{eq:tanx}
\forall x\in [\tilde R,+\infty),\quad \partial_x \tilde n(-x)>\tilde n(-x),\quad \partial_x \tilde n(x)<-\tilde n(x),
\end{equation}
then $\tilde n=\bar n$, where $\bar n\in\mathcal P_2(\mathbb R)$ is the steady state provided by Theorem~\ref{Thm:existence-steady}.
\end{cor}
This Corollary is proven in Section~\ref{sec:proofCor}.


\medskip

The driving idea of this manuscript is contained in Section~\ref{sec:wassest}: inspired by \cite{bolley2012convergence}, we show that if we only consider distributions bounded below away from $0$ on the support of $a$, we can derive improved Wasserstein estimates on the effect of a multiplicative operator. Using explicit examples, we show in Section~\ref{subsec:lowerboundremarks} that a lower bound on solutions is indeed necessary to obtain the estimates we introduce. The Wasserstein estimates requiring a lower bound on distributions can only be used on a compact set, since the distributions cannot be bounded away from $0$ uniformly on $\mathbb R$. They should then be combined to tail estimates. Obtaining such estimates is unfortunately technically difficult, and it is mainly to obtain such estimates that we had to assume that $a$ is compactly supported. Note that combining Wasserstein estimates to tail estimates is also present in the Fokker-Planck manuscripts that have inspired our approach (see \cite{bolley2012convergence,bolley2013uniform}): the tail estimates are then based on the convexity of potentials outside of a compact set. The situation here is more complex because of the highly non-local effect of selection: the fact that the population size is constant implies that each death corresponds to a birth, but these two events can occur for very different traits $x\in\mathbb R$. These birth/death events can thus be seen as "jumps" in the trait space, and such events were not present in aforementioned works. In this manuscript, we only consider the one-dimensional case ($x\in\mathbb R$). One benefit is the use of pseudo-inverse distributions (see \eqref{def:pseudo-inverse}) which provide convenient tools to develop Wasserstein estimates. It would be interesting to try to generalize these Wasserstein estimates to the case of a multi-dimensional phenotypic trait, but obtaining tail estimates would probably be even more difficult.

\medskip

In Section~\ref{subsec:defnot}, we introduce notations and basic properties, before discussing the biological aspects of this study in Section~\ref{sec:biologicalsetting}, and illustrating the result with a few  numerical simulations. 
In Section 2, we show that the macroscopic quantity $Z(t)=\int_{\mathbb R} x\,n(t,x)\,dx$ (the mean phenotypic trait of the population) approximately satisfies an ODE. We take advantage of this first result to obtain uniform tail estimates on $n(t,\cdot)$ and to prove the existence of a steady state. 
In Section 3, we develop Wasserstein distance estimates that will enable us to control the effects of the multiplicative selection operator $n\mapsto \frac{1+\alpha a(\cdot)}{\int_{\mathbb R} (1+\alpha a(y))n(y)\,dy}$. We discuss the necessity of a lower bound assumption to obtain these estimates in Section~\ref{subsec:lowerboundremarks}. In Section 4, we combine the results obtained so far to show that the flow of \eqref{eq:model} contracts solutions around the steady state, concluding the proof of Theorem~\ref{Thm:main}. Finally, in Section~\ref{sec:proofCor}, we prove Corollary~\ref{cor:uniqueness}. In the appendix, we prove the existence and uniqueness of solutions for \eqref{eq:model} and for the truncated model \eqref{eq:modifiedpb} (see Section~\ref{App:existence}), prove properties of solutions of the truncated model (see Sections~\ref{App:unifestim}, \ref{App:cv}), $W_1$ estimates on the reproduction operator (Section~\ref{Appendix:batTrough}), estimates on tails of steady sates of \eqref{eq:model} (see Section~\ref{App:tails}) and some estimates on macroscopic quantities (see Section~\ref{App:macro}).

\subsection{Definitions and notations}\label{subsec:defnot}

In this section we introduce notations that will be useful throughout the manuscript.

\medskip

\textbf{Wasserstein distance and pseudo-inverse functions}

Since we consider measures on $\mathbb R$, we can take advantage of the pseudo-inverse of the distributions (see \cite{carrillo2007contractive} for a review on the one-dimensional Wasserstein distance): if $u,v:[0,1]\to\mathbb R$ are defined for $z\in[0,1]$ by 
\begin{equation}\label{def:pseudo-inverse}
z=\int_{-\infty}^{u(z)}n(x)\,dx=\int_{-\infty}^{v(z)}m(x)\,dx,
\end{equation}
then $W_2(n,m)^2=\int_0^1 |u(z)-v(z)|^2\,dz$. Note that this definition of pseudo-inverse distribution applies to measures continuous with respect to the Lebesgue measure, but this definition can be extended to singular measures as well, see \cite{carrillo2007contractive}.  If $n$ is continuous, the derivation of \eqref{def:pseudo-inverse} provides the following useful property:
\begin{equation}\label{prpoertypseudoinverse}
u'(z)=\frac 1{n(u(z))}.
\end{equation}
In our analysis, we will use the following quantity to describe the distance between two distributions:
\begin{equation}\label{eq:defw}
w(n,m)=\min_{\zeta}W_2\left(n,m(\cdot-\zeta)\right). 
\end{equation}
Using the pseudo-inverse of $n$ and $m$, the quantity $w(\cdot,\cdot)$ can be written
\[w(n,m)^2=\min_{\zeta\in\mathbb R}\int_0^1 |u(z)-v(z)+\zeta|^2\,dz.\]
If we differentiate this expression, we get that $\zeta=\int_0^1 v(z)-u(z)\,dz$. This provides a more explicit expression of $w(\cdot,\cdot)$: 
\begin{equation}\label{eq:sigmaexplicit}
w(n,m)=W_2\left(n\left(\cdot-Z_n\right),m\left(\cdot-Z_m\right)\right),
\end{equation}
where 
\begin{equation}\label{def:Z}
Z_n=\int_{\mathbb R} x n(x)\,dx,\quad Z_m=\int_{\mathbb R} x m(x)\,dx.
\end{equation}

We will use the convexity of $W_2^2$ (convexity of the squared $W_2$ distance, see Theorem~4.8 in \cite{villani2008optimal}): for $n_1,n_2,m_1,m_2\in\mathcal P_2(\mathbb R)$ and $\theta\in(0,1)$,
\begin{equation}\label{eq:W2convex}
W_2(\theta n_1+(1-\theta)n_2,\theta m_1+(1-\theta)m_2)^2\leq \theta W_2(n_1,m_1)^2+(1-\theta)W_2(n_2,m_2)^2,
\end{equation}
as well as a continuous version of this convexity inequality (see Theorem~4.8 in \cite{villani2008optimal}): for $\theta\in [0,1]$, let $n_\theta,m_\theta \in L^\infty(\mathbb R)\cap \mathcal P_2(\mathbb R)$ such that $\left(\theta\mapsto n_\theta(x)\right)\in C^0([0,1])$ and $\left(\theta\mapsto m_\theta(x)\right)\in C^0([0,1])$. Then, we have
\begin{equation}\label{eq:W2convexcontinu}
W_2\left(x\mapsto \int_0^1 n_\theta(x)\,d\theta ,x\mapsto \int_0^1 m_\theta(x)\,d\theta\right)^2\leq \int_0^1 W_2(n_\theta,m_\theta)^2\,d\theta,
\end{equation}
where the integral on the left hand side is defined as a Riemann integral for any fixed $x\in\mathbb R$.

\medskip

\noindent\textbf{Reproduction operator}

We define the reproduction operator $T(n,m)$, for $n,m\in\mathcal P_2(\mathbb R)$ as follows:
\begin{equation}\label{def:T}
T(n,m)(x)=\iint_{\mathbb R^2} \Gamma_{\sigma^2}\left(x-\frac{y_*+y_*'}2\right)n(y_*)m(y_*')\,dy_*\,dy_*',
\end{equation}
and use the notation $T(n):=T(n,n)$. This operator has a number of properties that will be used in this manuscript:
\begin{itemize}
\item Conservation of the center of mass: if $n\in\mathcal P_2(\mathbb R)$, then
\begin{equation}\label{eq:centremass}
\int_{\mathbb R} xT(n)(x)\,dx=\int_{\mathbb R} xn(x)\,dx.
\end{equation}
\item Second moment formula: if $n\in\mathcal P_2(\mathbb R)$, then
\begin{equation}\label{eq:Tmoment1}
\int_{\mathbb R} x^2T(n)(x)\,dx=\sigma^2+\frac 12\int_{\mathbb R} x^2n(x)\,dx+\frac 12\left(\int_{\mathbb R} xn(x)\,dx\right)^2\leq \sigma^2+\int_{\mathbb R} x^2n(x)\,dx,
\end{equation}
and 
\begin{equation}\label{eq:Tmoment2}
\int_{\mathbb R} (x-Z)^2T(n)(x)\,dx=\sigma^2+\frac 12\int_{\mathbb R} (x-Z)^2n(x)\,dx,
\end{equation}
where $Z=\int_{\mathbb R} xT(n)(x)\,dx=\int_{\mathbb R} xn(x)\,dx$.
\item Invariance by translation: if $n\in\mathcal P_2(\mathbb R)$ and $x,h\in\mathbb R$, then
\begin{equation}\label{eq:translationinvar}
T(n)(x-h)=T(n(\cdot-h))(x)
\end{equation}
\item Invariant measures: for any $Z\in\mathbb R$,
\begin{equation}\label{eq:Maxwellienne}
\Gamma_{2\sigma^2}(x-Z)=\iint_{\mathbb R^2}\Gamma_{\sigma^2}\left(x-\frac{y_*+y_*'}2\right)\Gamma_{2\sigma^2}(y_*-Z)\Gamma_{2\sigma^2}(y_*'-Z)\,dy_*\,dy_*',
\end{equation}
that is $T\big(\Gamma_{2\sigma^2}(\cdot-Z)\big)=\Gamma_{2\sigma^2}(\cdot-Z)$.
\item Contraction property: If $n,m\in\mathcal P_2(\mathbb R)$ have the same center of mass, that is if $\int_{\mathbb R}xn(x)\,dx=\int_{\mathbb R}xm(x)\,dx$, then,
\begin{align}\label{eq:contraction}
W_2\left(T(n),T(m)\right)&\leq \frac 1{\sqrt 2} W_2(n,m).
\end{align}
\end{itemize}
The two first properties can be found in \cite{turelli1994genetic}. The third property can be checked by a direct calculation, while the two last properties are detailed in \cite{raoul2017macroscopic}.

\medskip

\noindent\textbf{Integral representation of solutions}

It will be useful to write solutions of \eqref{eq:model} in the following integral form
\begin{align}
n(t,x)&=n(0,x)e^{-\int_0^t\left(1+\alpha I_n(s)\right)^2\,ds}\nonumber\\
&\quad +\int_0^t\left(1+\alpha I_n(s)\right)^2T\left(\frac{(1+\alpha a)n(s,\cdot)}{1+\alpha I_n(s)}\right)(x)e^{-\int_s^t\left(1+\alpha I_n(\tau)\right)^2\,d\tau}\,ds,\label{eq:integralformulation}
\end{align}
where we have used the notation
\begin{equation}\label{def:In}
I_n(t):=\int_{\mathbb R} a(x)n(t,x)\,dx.
\end{equation}

\subsection{Biological setting and numerical simulations}\label{sec:biologicalsetting}

In natural populations, phenotypic traits can be measured on individuals of a given species: size of a bacteria, number of eggs laid by a swan, etc. From these observations, biologists have noticed that the phenotypic traits of offspring are often normally distributed around a mean. A theoretical framework has been developed for such populations by Fisher in 1919: the \emph{Infinitesimal Model} (see \cite{fisher1919xv,turelli1994genetic,barton2017infinitesimal}). This model considers a situation where the phenotype is obtained as a sum of the allelic effects of a large number of loci. Under a \emph{weak selection} assumption (see \cite{hartl1997principles,burger2000mathematical}), that is if the phenotype has a limited impact on the survival of the species (this is similar to the assumption that $\alpha>0$ is small in our study), the allelic values on different loci decorrelate. This decorrelation property is called \emph {Linkage Equilibrium} and it leads to the Gaussian kernel appearing in \eqref{eq:model}. It also implies that the phenotypic traits of the population are normally distributed (this property is related to the relation \eqref{eq:Maxwellienne}), and the variance of that Gaussian is the  \emph{Variance at Linkage Equilibrium} (denoted by $V_{LE}$). This quantity is related to our parameter $\sigma^2$ ($\sigma^2$ which is sometimes called \emph{segregational variance}): $V_{LE}=2\sigma^2$.

\medskip

The fact that the phenotypic traits of a population are normally distributed is a central assumption in population genetics	 (see e.g. \cite{haldane1990causes,wright1935analysis}). Note that normal distribution can also be observed in asexual populations for very different reasons (see \cite{kimura1965stochastic}), so that the assumption that a population is normally distributed is not directly linked to the infinitesimal model. Typically, a biological problem as depicted by \eqref{eq:model} would be analysed as follows: provided the selection is weak (that is $\alpha>0$ is small), one may assume that the population is a Gaussian distribution around its mean, that is $n(t,x)\underset{\alpha\to 0}{\sim}\Gamma_{2\sigma^2}(x-Z(t))$. The dynamics of $Z$ can then be approximated by
\begin{align}
Z'(t)&=\frac d{dt}\int_{\mathbb R} x n(t,x)\,dx\nonumber\\
&=\alpha \left(1+\alpha I_n(t)\right)\left(\int_{\mathbb R} x a(x)n(t,x)\,dx-I_n(t)\int_{\mathbb R} x n(t,x)\,dx\right)\nonumber\\
&\underset{\alpha\to 0}{\sim}\alpha F(Z(t)),\label{EDOpopgen}
\end{align}
where $F$ is defined by \eqref{def:F} and $I_n(t)$ is defined by \eqref{def:In}. This heuristic approximation is based on the approximation $\left(1+\alpha I_n(t)\right)\underset{\alpha\to 0}{\sim}1$. The heuristic approximation \eqref{EDOpopgen} is directly related to the so-called Fundamental Theorem of natural selection introduced by Fisher \cite{fisher1958genetical}. This asymptotic analysis idea introduced by Fisher plays a central role in population genetics, and more broadly in evolutionary biology: it provides a convenient description of the evolutionary dynamics of populations. One outcome of this heuristic argument that would typically be used in theoretical biology is that the mean phenotypic trait of the population, when $t>0$ is large, becomes close to $\bar Z$, and the population can then be approximated by a Gaussian distribution centered in $\bar Z$, with a variance $V_{LE}=2\sigma^2$. 

The heuristic approximation~\eqref{EDOpopgen} is coherent with the result of Theorem~1.2, which states that $n(t,x)\sim \Gamma_{2\sigma^2}\big(x-Y(\alpha t)\big)$ (see \eqref{estW2thm}), with $\frac{d}{dt}Y(\alpha t)=\alpha F\big(Y(\alpha t)\big)$ (see \eqref{eq:EDOintro}), and $Y(t)\underset{t\to\infty}{\rightarrow} \bar Z$. An interesting aspect of our analysis is that it is connected to this heuristic argument: The proof of our main result relies on quantitative estimates around  \eqref{EDOpopgen} and \eqref{eq:Maxwellienne}. We show that the Wasserstein distance is a convenient quantity to obtain these quantitative estimates. 

\medskip

The model \eqref{eq:model} describes the dynamics of a constant size population: $\int_{\mathbb R}n(t,x)\,dx\equiv 1$. This assumption is classical in population genetics. Our analysis can however be used for models that do not satisfy this assumption. We may for instance consider a more ecological model, with a logistic regulation of the population size:
\begin{align}
\partial_t f(t,x)&=\iint_{\mathbb R^2} \Gamma_{\sigma^2}\left(x-\frac{y_*+y_*'}2\right)(1+\alpha a(y_*))f(t,y_*)(1+ \alpha a(y_*'))\frac{f(t,y_*')}{\int_{\mathbb R} f(t,y)\,dy}\,dy_*\,dy_*'\nonumber\\
&\quad-\left(\int_{\mathbb R} K(y)f(t,y)\,dy\right)f(t,x).\label{eq:originalmodel}
\end{align}
Then $n(t,x):=\frac{f(t,x)}{\int_{\mathbb R} f(t,y)\,dy}$ satisfies \eqref{eq:model}, and the convergence of $f$ to a steady state $\bar f$ can be deduced from Theorem~\ref{Thm:main}. Also, if $\bar f$ is a steady state of \eqref{eq:originalmodel}, either $\bar f\equiv 0$, or $\frac{\bar f}{\int_{\mathbb R} \bar f(y)\,dy}$ is a steady state of~\eqref{eq:model}. In \eqref{eq:originalmodel}, the trait $x$ has an effect on the ability of an individual to reproduce, which corresponds to fecundity selection \cite{feldman1983some,calsina2004small}. Typically, $1+\alpha a(x)$ could represent the number of flowers produced by a plant, which affects the number of seeds the plant produces \cite{delph2012sexual}, or the number of eggs laid by a bird \cite{charmantier2006evolutionary,charmantier2008adaptive}. Such quantities can be measured in real populations  \cite{charmantier2006evolutionary,charmantier2008adaptive,wadgymar2017integrating}. We have not specified the phenotypic trait $x\in\R$ in the model, but we have in mind, in particular, phenological traits \cite{boucenna2026model}, i.e. the timing of life-cycle events, such as flowering onset in plants \cite{cleland2007shifting} or chick hatching date in birds \cite{visser2006shifts}. Such traits, together with measurements of reproductive output, could provide a natural connection between the model and field data on wild populations facing climate change.

\medskip

A possible extension of the model is to include viability selection (see \cite{orr2009fitness} for a discussion of the different types of selection), through an additional  death term $-\left(d+\alpha m(x)\right)f(t,x)$ in \eqref{eq:originalmodel}. We refer to \cite{mirrahimi2013dynamics,calvez2023uniform,guerand2023moment,aguilee2016pollen} for models combining an infinitesimal operator and viability selection, while other types of selection were considered in  \cite{karlin1979models,karlin1979models2,gimelfarb1982quantitative,schneider2005competitive}.  When $m(\cdot)$ is compactly supported, it may be possible to generalize the results of this manuscript to include  viability selection, but in more general situations, new ideas would be necessary. An interesting direction for future work would be to investigate whether the approach developed here can be combined with ideas from \cite{calvez2023uniform}, where {\sc V. Calvez}, {\sc D. Poyato} and {\sc F. Santambrogio} consider viability selection with an unbounded function, typically $m(x)=x^2$, in a model with non-overlapping generations.

\medskip

To illustrate the results of this manuscript, we represent numerical simulations of \eqref{eq:model} and \eqref{eq:EDOintro} in figure~1. For \eqref{eq:model}, we use an explicit Euler scheme in time. To compute efficiently the birth term, we notice that it can be written as a double convolution (see \cite{turelli1994genetic,mirrahimi2013dynamics}):
\[T(n,m)(x)=2 \Gamma_{4\sigma^2}\ast n\ast m(2x).\]
Thanks to this relation, it is possible to use a spectral method, leading to rapid simulations. We have chosen the coefficients $\alpha=1$ and $a(x):=2e^{-\frac {(x-5)^2}{4}}+e^{-\frac {(x+5)^2}{4}}$. In Figure~\ref{fig:coeff}, we represent $x\mapsto 1+\alpha a(x)$, that characterizes the number of offspring produced by an individual of trait $x$. In Figure~\ref{fig:F}, we plot $x\mapsto F(x)$ (see \eqref{def:F}), that is the velocity field for the  ODE \eqref{eq:EDOintro} that corresponds to the macroscopic model \eqref{EDOpopgen}. We notice that two stable steady states ($\bar Z_1\sim -5$ and $\bar Z_2\sim5$) and one unstable steady state ($\bar Z_u\sim 0$) exist for this macroscopic model. 

In Figure~\ref{fig:nd}, we represent the function $(t,x)\mapsto n(t,x)$ and the mean phenotypic trait $Z(t)=\int_{\mathbb R}xn(t,x)\,dx$ (continuous black line). The black dashed line is $Y(\alpha t)$, that is the approximation of $Z(t)$ that is provided by \eqref{eq:EDOintro}. Numerically, $n(t,\cdot)$ appears close to a Gaussian distribution centered in $Z(t)$ with variance $2\sigma^2$, which is coherent with estimate \eqref{eq:Ztpetit} stated in Theorem~\ref{Thm:main}.

In Figure~\ref{fig:Z}, we represent $Z(t)=\int_{\mathbb R} xn(t,x)\,dx$ (continuous lines) and $Y(\alpha t)$ (dashed line), for different initial conditions $Z^0=-10,\,-7.5,\,-5,\,-2.5,\,-0.25,\,0,\,0.25,\,2.5,\,5,\,7.5,\,10$. We see that the functions $Z(t)$, just as $Y(\alpha t)$, converge to one of two steady states that correspond to $\bar Z_1\sim -5$ and $\bar Z_2\sim 5$, the two stable steady states of \eqref{eq:EDOintro}, that are also the two traits that satisfy Assumption~\ref{ass1}. Notice also that for the initial values chosen here, $Z(t)$ seems well approximated by $Y(\alpha t)$, in accordance with Theorem~\ref{Thm:main} (see also Lemma~\ref{lem:roughmacro}), even if we have chosen a selection strength $\alpha=1$ that is not very small. The relevance of the approximation $Y(\alpha t)$, that corresponds to the approximation described in \eqref{EDOpopgen}, pleads for the use of this macroscopic approximation. The theoretical result stated in Theorem~\ref{Thm:main} indicates that this accuracy of the macroscopic model extends to the asymptotic distributions: $n(t,\cdot)$ should converge to either of two steady states close to $\Gamma_{2\sigma^2}(\cdot-\bar Z_1)$ and $\Gamma_{2\sigma^2}(\cdot+\bar Z_2)$. In future works, it would be interesting to develop an asymptotic preserving scheme adapted to the limit $\alpha\to 0$, to investigate the numerical aspects of this limit.

\begin{figure}
\begin{center}
\begin{subfigure}[t]{12cm}
\includegraphics[width=12cm]{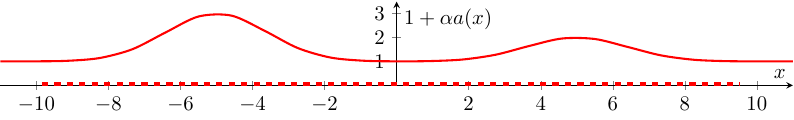}
\caption{Function $x\mapsto 1+\alpha a(x)$ used for the simulations. It is the fitness of trait $x$, and we use the values $\alpha=1$, $a(x):=\max\big(e^{-\frac {(x-5)^2}{4}}+2e^{-\frac {(x+5)^2}{4}}-3e^{-25/4},0\big)$. The dashed red line represents the support of $a(\cdot)$.}\label{fig:coeff}
\end{subfigure}

\begin{subfigure}[t]{12cm}
\includegraphics[width=12cm]{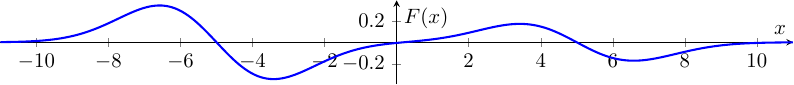}
\caption{Function $F$ as defined by \eqref{def:F}. }\label{fig:F}
\end{subfigure}
\begin{subfigure}[t]{7cm}
\includegraphics[width=7cm]{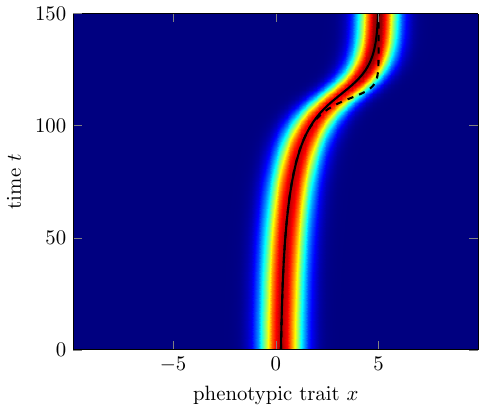}
\caption{The colors correspond to the density of the solution $(t,x)\mapsto n(t,x)$ of \eqref{eq:model}: the red color indicates a high density of individuals. The  continuous black line represents the mean phenotypic trait $t\mapsto Z(t)=\int_{\mathbb R}xn(t,x)\,dx$, while the dashed line is $t\mapsto Y(\alpha t)$ defined by \eqref{eq:EDOintro}.}\label{fig:nd}
\end{subfigure}
\quad 
\begin{subfigure}[t]{7cm}
\includegraphics[width=7cm]{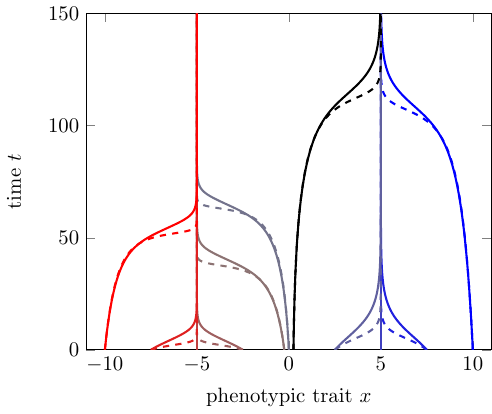}
\caption{Each color corresponds to a different initial condition, with $Z^0=-10,\,-7.5,\allowbreak-5,\allowbreak-2.5,\allowbreak-0.25,\allowbreak0,\allowbreak0.25,\allowbreak2.5,\allowbreak5,\allowbreak7.5,\,10$ respectively. For each initial condition, we represent $t\mapsto Z(t)=\int_{\mathbb R}xn(t,x)\,dx$ (continuous line) and $t\mapsto Y(\alpha t)$ (dashed line), where $n$ is the solution of \eqref{eq:model} and $Y$ is the solution of \eqref{eq:EDOintro}.}\label{fig:Z}
\end{subfigure}

\end{center}

\caption{Numerical simulations of \eqref{eq:model} and comparison with the dynamics of \eqref{eq:EDOintro}.}\label{fig:1}
\end{figure}

\section{Preliminary results}

\subsection{Wasserstein distance to a Gaussian distribution}

In this section, we show that the solutions of \eqref{eq:model} are close to Gaussian distributions in $x$. 

\begin{lem}\label{lem:roughmacro}
Let $R>0$ and $a\in W^{2,\infty}(\mathbb R)$, $\bar Z\in\mathbb R$ satisfying Assumption~\ref{ass1}. There are $\bar \alpha=\bar\alpha(R,a,\bar Z)>0$ and  $C=C(R,a,\bar Z)>0$ such that for $\alpha\in(0,\bar \alpha)$  and $n^0\in\mathcal P_2(\mathbb R)\cap C^1(\mathbb R)$ satisfying Assumption~\ref{ass2}, the solution $n\in L^\infty(\mathbb R_+\times\mathbb R)$ of \eqref{eq:model} with initial data $n^0$ provided by Theorem~\ref{cor:existence},  satisfies
\begin{equation}\label{est:Zbasic}
\forall t\in\left[-C\ln \alpha/\alpha,\infty\right),\quad W_2\left(n(t,\cdot),\Gamma_{2\sigma^2}\left(\cdot-\bar Z\right)\right)\leq C \sqrt\alpha,
\end{equation}
and
\begin{equation}\label{est:Zlinfty}
\forall t\in[0,\infty),\quad |Z_n(t)-Y(\alpha t)|\leq \frac C{-\ln \alpha},
\end{equation}
where $Y$ is the solution of \eqref{eq:EDOintro} with initial data $Z^0=\int_{\mathbb R} x n^0(x)\,dx$, and $Z_n$ is defined by \eqref{def:Z}. Moreover, 
\begin{equation}\label{eq:moment2}
\forall t\in[0,\infty),\quad \int_{\mathbb R} x^2n(t,x)\,dx\leq C.
\end{equation}
\begin{equation}\label{eq:Ztpetitbis}
\forall t\in[0,\infty),\quad |Z_n(t)|\leq R
\end{equation}
\begin{equation}\label{eq:Ztpetit}
\forall t\in[0,\infty),\quad W_2(n(t,\cdot),\Gamma_{2\sigma^2}(Z_n(t)-\cdot))\leq C\sqrt\alpha+C e^{-t/8}W_2(n^0,\Gamma_{2\sigma^2}(\cdot-Z^0)).
\end{equation}
\begin{equation}\label{eq:Ztpetitter}
\forall t\in[-C\ln\alpha/\alpha,\infty),\quad \big|Z_n(t)-\bar Z\big|\leq C\sqrt\alpha.
\end{equation}

\end{lem}

Before proving this lemma, we prove the following technical lemma, that the tail condition in Assumption~\ref{ass2} implies a bound on the second moment of $n^0$:

\begin{lem}\label{lem:m2}
Let $\tilde R>0$. There is $C_{\tilde R}>0$ such that for any $n^0\in \mathcal P_2(\mathbb R)\cap C^1(\mathbb R)$ satisfying 
\begin{equation*}
\forall x\in [\tilde R,+\infty),\quad \partial_x n^0(-x)>n^0(-x),\quad \partial_x n^0(x)<-n^0(x),
\end{equation*}
we have, with $Z^0:=\int_{\mathbb R}xn^0(x)\,dx$,
\begin{equation}\label{est:secmomentn0}
\int_{\mathbb R} x^2\,n^0(x)\,dx\leq C_{\tilde R},\quad \left|Z^0\right|\leq C_{\tilde R},\quad  W_2(n^0,\Gamma_{2\sigma^2}(\cdot-Z^0))\leq C_{\tilde R}.
\end{equation}
\end{lem}
\begin{proof}[Proof of Lemma~\ref{lem:m2}]
Since $n^0$ is decreasing on $[\tilde R,\tilde R+1]$, we have $n^0(\tilde R+1)\leq \int_{\tilde R}^{\tilde R+1} n^0(x)\,dx\leq 1$. Since $\partial_x n^0<-n^0$ on $[\tilde R+1,\infty)$, we have $\int_{[\tilde R+1,\infty)} x^2\,n^0(x)\,dx\leq \int_{[\tilde R+1,\infty)} x^2\,n^0(\tilde R+1)e^{(\tilde R+1)-x}\,dx\leq C$ for some constant $C>0$. The same argument can be repeated for $x\leq -(\tilde R+1)$, and we then get an upper bound on $\int_{\mathbb R} x^2\,n^0(x)\,dx$ that depends on $\tilde R$ only, as stated in the first inequality \eqref{est:secmomentn0}. This bound also implies an estimate on $|Z^0|$: $ |Z^0|\leq \int_{\mathbb R} \frac{1+x^2}2n^0(x)\,dx\leq C_{\tilde R}$, concluding the proof of Lemma~\ref{lem:m2}.
\end{proof}

\begin{proof}[Proof of Lemma~\ref{lem:roughmacro}]

\noindent\textbf{Step 1: We derive uniform estimates on moments of $n(t,\cdot)$}


Notice that Lemma~\ref{lem:m2} provides estimates on the first and second moment of $n^0$ that only depend on $R>0$. We consider $Z_n(t)=\int_{\mathbb R} x n(t,x)\,dx$. If we multiply \eqref{eq:model} by $x$, we get
\begin{align*}
&\left|\partial_t \left(xn(t,x)\right)\right|\\
&\quad =\iint_{\mathbb R^2}\left(|x|\Gamma_{\sigma^2}\left(x-\frac{y_*+y_*'}2\right)\right)(1+\alpha a(y_*))n(t,y_*)(1+ \alpha a(y_*'))n(t,y_*')\,dy_*\,dy_*'\nonumber\\
&\qquad+\left(1+\alpha \int_{\mathbb R} a(y)n(t,y)\,dy\right)^2 |x|n(t,x),
\end{align*}
and for $x\in\R$, we have 
\begin{align*}
&|x|\Gamma_{\sigma^2}\left(x-\frac{y_*+y_*'}2\right)\leq |x|\Gamma_{\sigma^2}\left(x-\frac{y_*+y_*'}2\right)1_{|x|\leq |y_*|+|y_*'|}+|x|\Gamma_{\sigma^2}\left(x-\frac{y_*+y_*'}2\right)1_{|x|\geq |y_*|+|y_*'|}\\
&\quad \leq C(|y_*|+|y_*'|)+|x|\Gamma_{\sigma^2}\left(|x|/2\right)\leq C\left(e^{-|x|}+|y_*|+|y_*'|\right).
\end{align*}
Then,
\begin{align*}
&\left|\partial_t \left(xn(t,x)\right)\right|\\
&\quad \leq \iint_{\mathbb R^2}C\left(e^{-|x|}+|y_*|+|y_*'|\right)(1+\alpha a(y_*))n(t,y_*)(1+ \alpha a(y_*'))n(t,y_*')\,dy_*\,dy_*'\nonumber\\
&\qquad+\left(1+\alpha \int_{\mathbb R} a(y)n(t,y)\,dy\right)^2 |x|n(t,x),
\end{align*}
and these integrals are finite since $n(t,\cdot)\in \mathcal P_2(\R)$. We can then compute $Z_n'(t)$ from \eqref{eq:model} thanks to Leibniz integral rule:
\begin{align}
Z_n'(t)&=(1+\alpha I_n(t))\int_{\mathbb R} y(1+\alpha a(y))n(t,y)\,dy-(1+\alpha I_n(t))^2Z_n(t)\nonumber\\
&=\alpha I_n(t)(1+\alpha I_n(t))\left[\int_{\mathbb R} y\frac{a(y)n(t,y)}{I_n(t)}\,dy-Z_n(t)\right],\label{eq:estZ}
\end{align}
and we recall the definition \eqref{def:In} of $I_n(t)$. Since $a(\cdot)\geq 0$, the quantity $\int_{\mathbb R} y\frac{a(y)}{I_n(t)}n(t,y)\,dy$ belongs to the convex envelop of the support of $a$, and \eqref{eq:estZ} then implies the uniform bound \eqref{eq:Ztpetitbis}. If we multiply \eqref{eq:model} by $x^2$ and integrate (using  the Leibniz integral rule, with an argument similar to the one used above), we obtain:
\begin{align}
\frac d{dt}\int_{\mathbb R} x^2n(t,x)\,dx&\leq(1+\alpha \|a\|_\infty)^2\int_{\mathbb R}x^2T(n(t,\cdot))(x)\,dx-\int_{\mathbb R} x^2n(t,x)\,dx,\nonumber
\end{align}
and thanks to \eqref{eq:Tmoment1}, we have
\begin{align}
&\frac d{dt}\int_{\mathbb R} x^2n(t,x)\,dx\nonumber\\
&\quad \leq\frac {(1+\alpha \|a\|_\infty)^2}{2}\left(2\sigma^2\int_{\mathbb R} x^2n(t,x)\,dx+\left(\int_{\mathbb R} xn(t,x)\,dx\right)^2\right) -\int_{\mathbb R} x^2n(t,x)\,dx\nonumber\\
&\quad \leq \frac {(1+\alpha \|a\|_\infty)^2}{2}\|Z_n\|_\infty^2-\frac {2-(1+\alpha \|a\|_\infty)^2}{2}\int_{\mathbb R} x^2n(t,x)\,dx\leq C-\frac {1}{4}\int_{\mathbb R} x^2n(t,x)\,dx,\label{eq:estV0}
\end{align}
provided $\alpha>0$ is small enough, thanks to \eqref{eq:estZ}. This estimate combined to the upper bound \eqref{est:secmomentn0} on  $\int_{\mathbb R} x^2 n^0(x)\,dx$, implies \eqref{eq:moment2}.

\medskip

\noindent\textbf{Step 2: We show that $n(t,\cdot)$ is close to $\Gamma_{2\sigma^2}(\cdot-Z_n(t))$ when $\alpha>0$ is small and $t\geq -\ln\alpha$.}

We use  \eqref{eq:integralformulation}, which, thanks to \eqref{eq:W2convex}, \eqref{eq:W2convexcontinu}, implies
\begin{align*}
&W_2(n(t,\cdot),\Gamma_{2\sigma^2}(\cdot-Z_n(t)))^2\leq e^{-\int_{0}^t(1+\alpha I_n(s))^2\,ds}W_2(n(0,\cdot),\Gamma_{2\sigma^2}(\cdot-Z_n(t)))^2\nonumber\\
&\qquad+\int_{0}^t e^{-\int_s^t(1+\alpha I_n(\tau))^2\,d\tau}(1+\alpha I_n(s))^2W_2\left(T\left(\frac{(1+\alpha a)n(s,\cdot)}{1+\alpha I_n(s)}\right), \Gamma_{2\sigma^2}(\cdot-Z_n(t))\right)^2\,ds.
\end{align*}
Thanks to \eqref{eq:Maxwellienne} and the positivity of $a$,
\begin{align}
&W_2(n(t,\cdot),\Gamma_{2\sigma^2}(\cdot-Z_n(t)))^2\leq e^{-t}W_2(n(0,\cdot),\Gamma_{2\sigma^2}(\cdot-Z_n(t)))^2\nonumber\\
&\qquad+\int_{0}^t e^{-(t-s)}(1+\alpha I_n(s))^2W_2\left(T\left(\frac{(1+\alpha a)n(s,\cdot)}{1+\alpha I_n(s)}\right), T\Big(\Gamma_{2\sigma^2}(\cdot-Z_n(t))\Big)\right)^2\,ds\nonumber\\
&\quad \leq W_2(n(0,\cdot),\Gamma_{2\sigma^2}(\cdot-Z_n(t)))^2e^{-t}\nonumber\\
&\qquad +\left(1+\alpha\|a\|_\infty\right)^2\int_0^t e^{-(t-s)}\bigg[W_2\left(T\left(\frac{(1+\alpha a)n(s,\cdot)}{1+\alpha I_n(s)}\right),T\left(\Gamma_{2\sigma^2}(\cdot-\hat Z_n(s))\right)\right)\nonumber\\
&\qquad\phantom{dsgsrfghfghjgfdhfgfgegz}+W_2\left(T\left(\Gamma_{2\sigma^2}(\cdot-\hat Z_n(s))\right),T\left(\Gamma_{2\sigma^2}(\cdot-Z_n(t))\right)\right)
\bigg]^2\,ds\nonumber\\
&\quad \leq W_2(n(0,\cdot),\Gamma_{2\sigma^2}(\cdot-Z_n(t)))^2e^{-t}+\left(1+\alpha\|a\|_\infty\right)^2\nonumber\\
&\qquad \int_0^t e^{-(t-s)}\left[\frac 1{\sqrt 2}W_2\left(\frac{(1+\alpha a)n(s,\cdot)}{1+\alpha I_n(s)},\Gamma_{2\sigma^2}\left(\cdot-\hat Z_n(s)\right)\right)+\left|Z_n(t)-\hat Z_n(s)\right|\right]^2\,ds,\nonumber
\end{align}
where we have used the notation $\hat Z_n(s):=\int_{\mathbb R} x\frac{(1+\alpha a(x))n(s,x)}{1+\alpha I_n(s)}\,dx$ and the $W_2-$contraction implied by $T$ on probability measures with the same center of mass (see \eqref{eq:contraction}). We can then use a Young inequality to obtain
\begin{align}
&W_2(n(t,\cdot),\Gamma_{2\sigma^2}(\cdot-Z_n(t)))^2\leq  W_2(n(0,\cdot),\Gamma_{2\sigma^2}(\cdot-Z_n(t)))^2e^{-t}+\left(1+\alpha\|a\|_\infty\right)^2\nonumber\\
&\qquad \int_0^t e^{-(t-s)}\left[\frac 58 W_2\left(\frac{(1+\alpha a)n(s,\cdot)}{1+\alpha I_n(s)},\Gamma_{2\sigma^2}\left(\cdot-\hat Z_n(s)\right)\right)^2+C\left|Z_n(t)-\hat Z_n(s)\right|^2\right]\,ds,\label{est:hatZarg}
\end{align}
The definition of $\hat Z_n$ implies:
\begin{align*}
\big|\hat Z_n(s)-Z_n(s)\big|&=\left|\int_{\mathbb R} x\left(\frac{(1+\alpha a(x))n(s,x)}{1+\alpha I_n(s)}-n(s,x)\right)\,dx\right|\leq\alpha \int_{\mathbb R} |x|\frac{|a(x)-I_n|}{1+\alpha I_n(s)}n(s,x)\,dx\\
&\leq\alpha \int_{\mathbb R} \frac{1+x^2}2\frac{\|a\|_\infty}{1+\alpha I_n(s)}n(s,x)\,dx\leq C\alpha,
\end{align*}
thanks to \eqref{eq:moment2}. Moreover $\|Z_n'\|_\infty=\mathcal O(\alpha)$ (see \eqref{eq:estZ}), which implies
\[\hat Z_n(s)=Z_n(s)+\mathcal O(\alpha)=Z_n(t)+\mathcal O(\alpha)(1+|t-s|),\]
which we use to estimate the term $\big|Z_n(t)-\hat Z_n(s)\big|^2$. 
We can moreover use the following convex decomposition:
\begin{align}\label{eq:convdecomp}
\frac{(1+\alpha a)n(s,\cdot)}{1+\alpha I_n(s)}&=\frac{1}{1+\alpha I_n(s)}n(s,\cdot)+\frac{\alpha I_n(s)}{1+\alpha I_n(s)}\frac{a\,n(s,\cdot)}{I_n(s)},
\end{align}
and the convexity of $W_2^2$ (see \eqref{eq:W2convex}) to show that, for $\alpha>0$ small enough,
\begin{align}
&W_2\left(\frac{(1+\alpha a)n(s,\cdot)}{1+\alpha I_n(s)},\Gamma_{2\sigma^2}\left(\cdot-\hat Z_n(s)\right)\right)^2\nonumber\\
&\quad\leq \frac 1{1+\alpha I_n(s)}W_2\left(n(s,\cdot),\Gamma_{2\sigma^2}\left(\cdot-\hat Z_n(s)\right)\right)^2+\frac{\alpha I_n(s)}{1+\alpha I_n(s)}W_2\left(\frac{a n(s,\cdot)}{I_n(s)},\Gamma_{2\sigma^2}\left(\cdot-\hat Z_n(s)\right)\right)^2\nonumber\\
&\quad\leq \frac 1{1+\alpha I_n(s)}\left(W_2\left(n(s,\cdot),\Gamma_{2\sigma^2}\left(\cdot-Z_n(s)\right)\right)+\left|Z_n(s)-\hat Z_n(s)\right|\right)^2+C\alpha\nonumber\\
&\quad\leq \big(W_2\left(n(s,\cdot),\Gamma_{2\sigma^2}\left(\cdot-Z_n(s)\right)\right)+C\alpha\big)^2+C\alpha,\label{eq:deconn}
\end{align}
where we have used the following estimate, which is a consequence of \eqref{eq:moment2} (which has been proven in Step~1):
\begin{align*}
&I_n(s) W_2^2\left(\frac {an(s,\cdot)}{I_n(s)},\Gamma_{2\sigma^2}(\cdot-\hat Z_n(s))\right)\leq I_n(s)\left(\int_{\R}x^2\frac {a(x)n(s,x)}{I_n(s)}\,dx+\int_{\mathbb R}x^2\Gamma_{2\sigma^2}(x-\hat Z_n(s))\,dx\right)\\
&\quad \leq \|a\|_\infty \int_{\R}x^2n(s,x)\,dx+I_n(s)\left(2\sigma^2+|\hat Z_n(s)|^2\right)\\
&\quad \leq \|a\|_\infty \int_{\R}x^2n(s,x)\,dx+\|a\|_\infty\left(2\sigma^2+\int_{\mathbb R} (1+x^2)n(s,x)\,dx\right)\leq C,
\end{align*}
where we have used the $|\hat Z_n(s)|\leq \int_{\mathbb R} x(1+\alpha \|a\|_\infty)n(s,x)\,dx\leq \int_{\mathbb R} (1+x^2)n(s,x)\,dx$. Then, thanks to a Young inequality,
\begin{align}
&W_2\left(\frac{(1+\alpha a)n(s,\cdot)}{1+\alpha I_n(s)},\Gamma_{2\sigma^2}\left(\cdot-\hat Z_n(s)\right)\right)^2\nonumber\\
&\quad \leq \left(W_2\left(n(s,\cdot),\Gamma_{2\sigma^2}\left(\cdot-Z_n(s)\right)\right)^2+C\alpha^2+2C\alpha W_2\left(n(s,\cdot),\Gamma_{2\sigma^2}\left(\cdot-Z_n(s)\right)\right)\right)+C\alpha\nonumber\\
&\quad \leq \left(\left(1+\lambda\right)W_2\left(n(s,\cdot),\Gamma_{2\sigma^2}\left(\cdot-Z_n(s)\right)\right)^2+\left(C+\frac {C^2}{\lambda}\right)\alpha^2\right)+C\alpha\nonumber\\
&\quad \leq \frac 65 \,W_2\left(n(s,\cdot),\Gamma_{2\sigma^2}\left(\cdot-Z_n(s)\right)\right)^2+C\alpha,\label{est:W2conde}
\end{align}
provided we select a small constant $\lambda>0$ for the Young inequality. The estimate \eqref{est:hatZarg} then becomes, provided $\alpha>0$ is small enough,
\begin{align*}
&W_2\big(n(t,\cdot),\Gamma_{2\sigma^2}(\cdot-Z_n(t))\big)^2\leq W_2\big(n(0,\cdot),\Gamma_{2\sigma^2}(\cdot-Z_n(t))\big)^2e^{-t}\\
&\qquad +\int_0^t e^{-(t-s)}\left[\frac {6}{5}\,\frac 58W_2(n(s,\cdot),\Gamma_{2\sigma^2}(\cdot-Z_n(s)))^2+C\alpha+\mathcal O\left(\alpha^2\right)(1+|t-s|^2)\right]\,ds\\
&\quad \leq W_2\big(n(0,\cdot),\Gamma_{2\sigma^2}(\cdot-Z_n(t))\big)^2e^{-t}\\
&\qquad +\frac{3}{4} \int_0^t e^{-(t-s)}W_2(n(s,\cdot),\Gamma_{2\sigma^2}(\cdot-Z_n(s)))^2\,ds+C\alpha.
\end{align*}
The quantity $\phi(t)=e^{t}W_2(n(t,\cdot),\Gamma_{2\sigma^2}(\cdot-Z_n(t)))^2$ then satisfies
\begin{align}\label{eq:introphi}
\phi(t)\leq W_2(n(0,\cdot),\Gamma_{2\sigma^2}(\cdot-Z_n(t)))^2+C\alpha e^{t}+\frac 34\int_0^t \phi(s)\,ds.
\end{align}
Thanks to a Gr\"onwall's inequality,
\begin{align*}
\phi(t)&\leq W_2(n(0,\cdot),\Gamma_{2\sigma^2}(\cdot-Z_n(t)))^2+C\alpha e^{t}\\
&\quad +\int_0^t\left(W_2(n(0,\cdot),\Gamma_{2\sigma^2}(\cdot-Z_n(t)))^2+C\alpha e^{s}\right)\frac 34 e^{\frac 34(t-s)}\,ds\\
&\leq CW_2(n(0,\cdot),\Gamma_{2\sigma^2}(\cdot-Z_n(t)))^2e^{\frac 34 t}+C\alpha e^{t}.
\end{align*}
Then 
\begin{align*}
W_2(n(t,\cdot),\Gamma_{2\sigma^2}(Z_n(t)-\cdot))^2&\leq CW_2(n(0,\cdot),\Gamma_{2\sigma^2}(\cdot-Z_n(t)))^2e^{-\frac {1}{4}t}+C\alpha.\\
&\leq C\left(W_2(n(0,\cdot),\Gamma_{2\sigma^2}(\cdot-\bar Z)+|\bar Z|+\|Z_n\|_{L^\infty}\right)^2e^{-\frac {1}{4}t}+C\alpha.\\
&\leq C e^{-\frac {1}{4}t}+C\alpha.
\end{align*}
and thus
\begin{equation}\label{est:W2Z00}
W_2(n(t,\cdot),\Gamma_{2\sigma^2}(Z_n(t)-\cdot))\leq CW_2(n(0,\cdot),\Gamma_{2\sigma^2}(\cdot-Z_n(t)))e^{-\frac 1{8}t}+C\sqrt\alpha.
\end{equation}
Next, we can use \eqref{eq:estZ} to show: 
\begin{equation}\label{est:W2Z0}
W_2(n(0,\cdot),\Gamma_{2\sigma^2}(\cdot-Z_n(t)))\leq W_2(n(0,\cdot),\Gamma_{2\sigma^2}(\cdot-Z_n(0)))+C\alpha t,
\end{equation}
and then
\begin{align*}
W_2(n(t,\cdot),\Gamma_{2\sigma^2}(Z_n(t)-\cdot))&\leq C\left(W_2(n(0,\cdot),\Gamma_{2\sigma^2}(\cdot-Z_n(0)))+C\alpha t\right)e^{-\frac 1{8}t}+C\sqrt\alpha\\
&\leq C W_2(n(0,\cdot),\Gamma_{2\sigma^2}(\cdot-Z_n(0))) e^{-\frac 1{8}t}+C\sqrt\alpha,
\end{align*}
since $t\mapsto \alpha t e^{-\frac 18 t}$ is bounded. Thanks to Lemma~\ref{lem:m2}, $W_2(n(0,\cdot),\Gamma_{2\sigma^2}(\cdot-Z_n(0)))\leq C$, so that we have the following estimate:
\begin{equation}\label{est:W2Z}
\forall t\geq -4\ln \alpha,\quad 
W_2\left(n(t,\cdot),\Gamma_{2\sigma^2}(Z_n(t)-\cdot)\right)\leq C\sqrt \alpha.
\end{equation}

\medskip

\noindent\textbf{Step 3: We show that $Z_n(t)\sim \bar Z$ for $t\geq \frac{-C\ln\alpha}\alpha$, provided $\alpha>0$ is small enough.}

Thanks to the Kantorovich-Rubinstein formula, \eqref{eq:estZ} implies 
\begin{align*}
Z_n'(t)&=\alpha I_n(t)(1+\alpha I_n(t))\\
&\qquad \left[\int_{\mathbb R} y\frac{a(y)}{I_n(t)}\Gamma_{2\sigma^2}(Z_n(t)-y)\,dy-Z_n(t)+\frac{\mathcal{O}(1)}{I_n(t)}W_1(n(t,\cdot),\Gamma_{2\sigma^2}(Z_n(t)-\cdot))\right]\\
&=\alpha (1+\alpha I_n(t))\bigg[\int_{\mathbb R} y a(y)\Gamma_{2\sigma^2}(Z_n(t)-y)\,dy-\left(\int_{\mathbb R} a(x)\Gamma_{2\sigma^2}(x-Z_n(t))\,dx \right)Z_n(t)\\
&\quad +\mathcal{O}(1)W_1(n(t,\cdot),\Gamma_{2\sigma^2}(Z_n(t)-\cdot))Z_n(t)+\mathcal{O}(1)W_1(n(t,\cdot),\Gamma_{2\sigma^2}(Z_n(t)-\cdot))\bigg]\\
&=\alpha (1+\alpha I_n(t))\left[F(Z_n(t))+\mathcal{O}(1)W_1(n(t,\cdot),\Gamma_{2\sigma^2}(Z_n(t)-\cdot))\right].
\end{align*}
Thanks to \eqref{est:W2Z}. Note that the $1-$Wasserstein distance is dominated by the $2-$Wasserstein distance, and then, for $\alpha>0$ small enough and $t\geq -4\ln \alpha$,
\begin{align}
Z_n'(t)&=\alpha (1+\alpha I_n(t))\left[F(Z_n(t))+\mathcal O\left(\sqrt\alpha\right)\right].\label{eq:approxevoZ}
\end{align}
We notice that for $t\in[0,-4\ln\alpha)$, $Z_n(t)$ remains close to $Z^0$:
\begin{equation}\label{est:iniZ}
|Z_n(t)-Z^0|\leq\|Z_n'\|_\infty(-4\ln\alpha)=  \mathcal O(\alpha)(-4\ln\alpha)\xrightarrow[\alpha \to 0]{} 0.
\end{equation}
We recall Assumption~\ref{ass1}, Assumption~\ref{ass2}, and consider three cases:

\medskip

\emph{Case 1: $Z^0$ close to $\bar Z$}\\
Since $F$ is continuously differentiable and $F'(\bar Z)<0$, let $\eta=\frac{-F'(\bar Z)}{2\|F''\|_\infty}$ such that $F'(Z)\leq \frac{F'(\bar Z)}2<0$ for $Z\in[\bar Z-\eta,\bar Z+\eta]$. If $|Z^0-\bar Z|\leq \frac \eta 2$, then \eqref{est:iniZ} implies  $|Z_n(-4\ln\alpha)-\bar Z|< \frac {3\eta}4$ provided $\alpha>0$ is small enough. Let $t_*>-4\ln \alpha$ the first time when $|Z_n(t_*)-\bar Z|=\eta$. Then, for $t\in[-4\ln\alpha,t_*]$,  we have $|Z_n(t)-\bar Z|\leq \eta$, and, thanks to \eqref{eq:approxevoZ},
\begin{align}
\frac d{dt}\left|Z_n-\bar Z\right|(t)&=\alpha (1+\alpha I_n(t))\left[\left(F(Z_n(t))-F(\bar Z)\right)\textrm{sgn }\left(Z_n(t)-\bar Z\right)+\mathcal O\left(\sqrt\alpha\right)\right]\nonumber\\
&=\alpha \left[F'(\theta(t))\left|Z_n(t)-\bar Z\right|+\mathcal O\left(\sqrt\alpha\right)\right],\label{est:linearZ}
\end{align}
for some $\theta(t)\in [\bar Z-\eta,\bar Z+\eta]$. Let $\alpha>0$ small enough, and we prove that $t_*<\infty$ leads to a contradiction : if $t_*<\infty$, the definition of $t_*$ implies $|Z_n(t_*)-\bar Z|= \eta$. If w.l.o.g. we have $Z_n(t_*)=\bar Z+ \eta$, then $Z_n'(t_*)\geq 0$, and we also notice that \eqref{est:linearZ} implies
\[Z_n'(t)=\alpha \left[F'(\theta(t))\eta+\mathcal O\left(\sqrt\alpha\right)\right]<0,\]
provided $\alpha>0$ is smaller than a constant. This is a contradiction, implying $t_*=\infty$, that is 
\[\forall t\geq 0,\quad Z_n(t)\in [\bar Z-\eta,\bar Z+\eta].\]
Let then $y(t):=\left|Z_n(t)-\bar Z\right|-C\sqrt\alpha$ (for some constant $C>0$), which satisfies $y'\leq \alpha F'(\theta(t))y$ when $y\geq0$, and then $y(t)\leq \eta e^{\alpha  \frac{F'(\bar Z)}2t}$. As a consequence, for some constant $C>0$, provided $\alpha>0$ is small enough,
\begin{align}\label{est:ZZ1}
\forall t\geq \frac {C\ln(-\ln\alpha)}\alpha,\quad \left|Z_n(t)-\bar Z\right|\leq \frac C{-\ln\alpha},
\end{align}
and for longer times, we have the improved estimate
\begin{align}\label{est:ZZ2}
\forall t\geq \frac {-C\ln\alpha}\alpha,\quad \left|Z_n(t)-\bar Z\right|\leq C\sqrt\alpha,
\end{align}
which proves \eqref{eq:Ztpetitter}.
\medskip

\emph{Case 2: $Z^0<\bar Z$}\\
Since $F$ is a continuous function, Assumption~\ref{ass2} implies that for some $\varepsilon>0$, 
\[\forall Z\in [Z^0-1/2R\|F'\|_\infty, \bar Z-\eta/2],\quad F(Z)\geq \frac 1{2R},\]
and we recall that $\bar Z-\eta/2=\bar Z-\frac{-F'(\bar Z)}{4\|F''\|_\infty}$. Then, provided $\alpha>0$ is small enough, \eqref{eq:approxevoZ} implies $Z_n'(t)>\frac 1{4R}$ when $Z_n(t)\in (Z^0-1/2R\|F'\|_\infty, \bar Z-\eta/2)$ and $t\geq -4\ln\alpha$. This estimate and \eqref{est:iniZ} imply   $Z_n(t)\geq \bar Z-\eta/2$ for some $t\geq -4\ln\alpha+4R\frac{|\bar Z-Z^0|+1/2R\|F'\|_\infty}{\alpha }$, and the bound \eqref{est:secmomentn0} on $|Z_0|\leq C$. We can then apply the argument for Case 1 for $t\geq -4\ln\alpha+4R\frac{|\bar Z|+C+1/2R\|F'\|_\infty}{\alpha }$ and prove \eqref{est:ZZ1} and \eqref{est:ZZ2}, for $\alpha>0$ small enough, in this case also.

\medskip

\emph{Case 3: $Z^0>\bar Z$}\\
This case is similar to Case 2, and we have thus proven \eqref{est:ZZ1} and \eqref{est:ZZ2} (for $\alpha>0$ small enough) in all three cases.

\medskip

\noindent\textbf{Step 4: We prove that $Z_n(t)$ is close to $Y(\alpha t)$  when $\alpha>0$ is small.}

Thanks to \eqref{est:iniZ} and the Lipschitz continuity of $Y$,
\[\forall t\in[0,-4\ln\alpha],\quad |Z_n(t)-Y(\alpha t)|\leq|Z_n(t)-Z^0|+|Y(\alpha t)-Y(0)|\leq \mathcal O(\alpha)(-\ln \alpha).\]
We can then use \eqref{eq:approxevoZ} to show that for $t\geq -4\ln\alpha$,
\begin{align*}
\frac d{dt}|Z_n(t)-Y(\alpha t)|&\leq|Z_n'(t)-\alpha F(Z_n(t))|+|\alpha F(Z_n(t))-\alpha Y'(\alpha t)|\\
&\leq \mathcal O(\alpha^{3/2})+\mathcal O( \alpha) |Z_n(t)-Y(\alpha t)|\leq C\alpha\left(\sqrt\alpha+|Z_n(t)-Y(\alpha t)|\right),
\end{align*}
where $C>0$ is a constant. Then, thanks to a Gr\"onwall inequality,
\[
|Z_n(t)-Y(\alpha t)|\leq \left(|Z_n(-4\ln\alpha)-Y(-4\alpha \ln\alpha)|+\sqrt\alpha \right)e^{C\alpha (t+4\alpha\ln\alpha)}-\sqrt\alpha.
\]
For $t\in[-4\ln\alpha,C\ln(-\ln\alpha)/\alpha]$, we use \eqref{est:iniZ} to estimate $|Z_n(-4\ln\alpha)-Y(-4\alpha \ln\alpha)|$ and show
\begin{align}
|Z_n(t)-Y(\alpha t)|\leq C\alpha (-\ln\alpha)e^{C\ln(-\ln\alpha)}=C\alpha (-\ln\alpha)(-\ln \alpha)^{C}=C\alpha (-\ln\alpha)^{C+1}.\label{est:Ztmpsinter}
\end{align}

\medskip

Combining \eqref{est:iniZ}, \eqref{est:ZZ1} and \eqref{est:Ztmpsinter}, we obtain \eqref{est:Zlinfty}. Finally, \eqref{est:W2Z} and \eqref{est:ZZ2} imply \eqref{est:Zbasic}, which  concludes the proof of Lemma~\ref{lem:roughmacro}.

\end{proof}

\subsection{Existence of a steady state: Proof of Theorem~\ref{Thm:existence-steady}}
\label{subsec:proofsteady}
%
%
%

Let $\tilde R=\max\{|x|;x\in\textrm{supp }a\}+1$, so that
\begin{equation}\label{eq:suppaexi}
\textrm{supp }a\subset [-\tilde R,\tilde R].
\end{equation}
We recall the definition~\eqref{def:T} of $T$ and, for $n\in\mathcal P_2(\mathbb R)$, let 
\[\mathcal T_\alpha (n):=T\left(\frac{(1+\alpha a(\cdot))n}{1+\alpha I_n}\right),\]
where $I_n$ is defined from $n$ through \eqref{def:In}. We also define $Z_n$ from $n$ thanks to \eqref{def:Z}. 

\medskip

\noindent\textbf{Step 1: Preliminary estimates on $T_\alpha(n)$ for $n\in\mathcal P_2(\R)$}\\
We assume that $W_2(n,\Gamma_{2\sigma^2}(\cdot-\bar Z))\leq 1$, and first relate the first moment of $T_\alpha(n)$, that is $\hat Z:=\int_{\mathbb R} y\mathcal T_\alpha (n)(y)\,dy$ to $Z_n$, the first moment of $n$: thanks to \eqref{eq:centremass},
\begin{align}
\hat Z&=\int_{\mathbb R} y\frac{(1+\alpha a(y))n(y)}{1+\alpha I_n}\,dy=\frac{Z_n(1+\alpha I_n)-\alpha I_nZ_n+\alpha \int_{\mathbb R} ya(y)n(y)\,dy }{1+\alpha I_n}\nonumber\\
&=Z_n+\alpha \frac{F(Z_n)+\left(\int_{\mathbb R} ya(y)n(y)\,dy -\int_{\mathbb R} ya(y)\Gamma_{2\sigma^2}(y-Z_n)\,dy \right)}{1+\alpha I_n}\nonumber\\
&\quad -\frac{\left(\int_{\mathbb R}a(y)n(y)\,dy -\int_{\mathbb R} a(y)\Gamma_{2\sigma^2}(y-Z_n)\,dy \right)Z_n}{1+\alpha I_n}\nonumber\\
&=Z_n+\alpha \frac{F'(\theta)(Z_n-\bar Z)+\mathcal O\left(W_2(n,\Gamma_{2\sigma^2}(\cdot-Z_n))\right)}{1+\alpha I_n},\label{est:hatZ}
\end{align}
with $\theta\in[\bar Z,Z_n]$ (or $\theta\in[Z_n,\bar Z]$ if $Z_n<\bar Z$). Notice that we have used 
\[0=F(\bar Z)=F(Z_n)-F'(\theta)(Z_n-\bar Z)\]
 and the Kantorovich-Rubinstein formula. Next, we estimate $W_2\left(\mathcal T_\alpha (n),\Gamma_{2\sigma^2}\left(\cdot-\hat Z\right)\right)$: thanks to \eqref{eq:Maxwellienne} and the  contraction \eqref{eq:contraction} implied by $T$, we have
\begin{align*}
&W_2^{2}\left(\mathcal T_\alpha (n),\Gamma_{2\sigma^2}\left(\cdot-\hat Z\right)\right)=W_2^{2}\left(T\left(\frac{(1+\alpha a(\cdot))n}{1+\alpha I_n}\right),T\left(\Gamma_{2\sigma^2}\left(\cdot-\hat Z\right)\right)\right)\nonumber\\
&\quad \leq \frac 1{2}W_2^{2}\left(\frac{(1+\alpha a(\cdot))n}{1+\alpha I_n},\Gamma_{2\sigma^2}\left(\cdot-\hat Z\right)\right)\nonumber\\
&\quad\leq \frac 1{2}\left( \frac {1}{1+\alpha I_n}W_2^{2}\left(n,\Gamma_{2\sigma^2}\left(\cdot-\hat Z\right)\right)+\alpha\frac {I_n}{1+\alpha I_n}W_2^{2}\left(\frac{a}{I_n}n,\Gamma_{2\sigma^2}\left(\cdot-\hat Z\right)\right)\right)\nonumber\\
&\quad \leq \frac {1}{2}\left(W_2\left(n,\Gamma_{2\sigma^2}\left(\cdot-Z_n\right)\right)+\left|Z_n-\hat Z\right|\right)^{2}\nonumber\\
&\qquad +\mathcal O(\alpha)\left(W_2\left(\frac{a}{I_n}n,\delta_{0}\right)+W_2\left(\delta_{0},\Gamma_{2\sigma^2}\left(\cdot-\hat Z\right)\right)\right)^{2}\nonumber\\
&\quad \leq \frac {1}{2}W_2^2\left(n,\Gamma_{2\sigma^2}\left(\cdot-Z_n\right)\right)+C\alpha,
\end{align*}
where we have used the fact that the support of $\frac a{I_n}n$ is included in $\textrm{supp }a\subset [-\tilde R,\tilde R]$ (see \eqref{eq:suppaexi}) to show 
\[W_2\left(\frac{a}{I_n}n,\delta_{0}\right)^2\leq \max_{\mu\in\mathcal P([-\tilde R,\tilde R]}\int_{-\tilde R}^{\tilde R}x^2\,d\mu(x)\leq  \tilde R^2\leq C.\]
Then, thanks to a Young inequality,
\begin{align}
W_2\left(\mathcal T_\alpha (n),\Gamma_{2\sigma^2}\left(\cdot-\hat Z\right)\right)&\leq  \frac {1}{\sqrt 2}\left(W_2^2\left(n,\Gamma_{2\sigma^2}\left(\cdot-Z_n\right)\right)+C\alpha\right)^{1/2}\nonumber\\
&\leq  \frac 34W_2\left(n,\Gamma_{2\sigma^2}\left(\cdot-Z_n\right)\right)+C\sqrt\alpha.
\label{est:W2steady}
\end{align}
Then, with $\bar C:=\sqrt 2C/(\sqrt 2-1)$, $C>0$ being the constant appearing in \eqref{est:W2steady}, we have
\begin{equation}\label{est:W2steadybis}
\Big(W_2\left(n,\Gamma_{2\sigma^2}\left(\cdot-Z_n\right)\right)\leq \bar C\sqrt \alpha\Big)\Rightarrow \Big(W_2\left(\mathcal T_\alpha (n),\Gamma_{2\sigma^2}\left(\cdot-\hat Z\right)\right)\leq \bar C\sqrt \alpha\Big),
\end{equation}
and we recall that $Z_n$, $\hat Z$ are respectively the first moments of $n$ and $\mathcal T_\alpha (n)$.
We can use the estimates \eqref{est:hatZ} and \eqref{est:W2steady} to control the following quantity:
\begin{equation*}
L(m):=\left|\int_{\mathbb R} y\,m(y)\,dy-\bar Z\right|+\alpha^{3/4}\,W_2\left(m,\Gamma_{2\sigma^2}\left(\cdot-\int_{\mathbb R} y\,m(y)\,dy\right)\right),
\end{equation*}
for $m\in\mathcal P_2(\mathbb R)$. Indeed, \eqref{est:hatZ} and \eqref{est:W2steady} imply
\begin{align}
L(\mathcal T_\alpha (n))&= \left|\int_{\R}y\mathcal T_\alpha (n)(y)\,dy-\bar Z\right|+\alpha^{3/4}\,W_2\left(\mathcal T_\alpha (n),\Gamma_{2\sigma^2}\left(\cdot-\int_{\R}y\mathcal T_\alpha (n)(y)\,dy\right)\right)\nonumber\\
&= \left|\hat Z-\bar Z\right|+\alpha^{3/4}\,W_2\left(\mathcal T_\alpha (n),\Gamma_{2\sigma^2}\left(\cdot-\hat Z\right)\right)\nonumber\\
&\leq  \left|Z_n+\alpha \frac{F'(\theta)(Z_n-\bar Z)+\mathcal O\left(W_2(n,\Gamma_{2\sigma^2}(\cdot-Z_n))\right)}{1+\alpha I_n}-\bar Z\right|\nonumber\\
&\qquad +\alpha^{3/4}\,\left(\frac 34W_2\left(n,\Gamma_{2\sigma^2}\left(\cdot-Z_n\right)\right)+C\sqrt\alpha\right)\nonumber\\
&\leq \left(1+\frac 34\alpha F'(\theta)\right)\left|Z_n -\bar Z\right|+\frac 34\alpha^{3/4}\, W_2\left(n,\Gamma_{2\sigma^2}(\cdot-Z_n)\right)+ C\alpha^{5/4},\label{est:LL}
\end{align}
where we have used the fact that $\frac 1{1+\alpha I_n}\geq \frac 23$ if $\alpha>0$ is small enough. 

\medskip

\medskip

\noindent\textbf{Step 2: A set $\Omega_\eta$ that is stable under $T_\alpha$}\\
Let
\begin{align*}
\Omega_\eta:=&\Big\{n\in\mathcal P_2(\mathbb R);\quad L(n)\leq \alpha^{1/5},\quad W_2\left(n,\Gamma_{2\sigma^2}\left(\cdot-Z_n\right)\right)\leq \bar C\sqrt \alpha,\\
&\quad  n'(x)>n(x)\quad \textrm{ if }\int_{-\infty}^x n\leq \eta,\quad n'(x)<-n(x)\quad\textrm{ if }\int_x^\infty n\leq \eta\Big\}\cap C^1(\mathbb R)
\end{align*}
and $\eta:=\int_{|\bar Z|+\tilde R+4\sigma^2+1}^\infty \Gamma_{2\sigma^2}$. Note that $\Gamma_{2\sigma^2}(\cdot-\bar Z)\in\Omega_\eta$, which is therefore not empty. 

\medskip

We consider that $n\in \Omega_\eta$, and notice that $n$ satisfies, in particular, $L(n)\leq \alpha^{1/5}$. We recall that in \eqref{est:LL}, the parameter $\theta$ satisfies  $\theta\in[\bar Z,Z_n]$ (or $\theta\in[Z_n,\bar Z]$ if $Z_n<\bar Z$), and $|Z_n-\bar Z|\leq L(n)\leq \alpha^{1/5}$, estimate \eqref{est:LL} therefore implies
\begin{align*}
L(\mathcal T_\alpha (n))&=\left(1+\frac 34\alpha F'(\bar Z)\right)\left|Z_n -\bar Z\right|+\frac 34\alpha^{3/4}\, W_2\left(n,\Gamma_{2\sigma^2}(\cdot-Z_n)\right)+ \mathcal O(\alpha^{5/4})\\
&\leq L(n)+\left(\frac 34\alpha F'(\bar Z)\right)|Z_n-\bar Z|\\
&\quad -\left(1-\frac 34\right)\alpha^{3/4}\, W_2\left(n,\Gamma_{2\sigma^2}(\cdot-Z_n)\right)+ C\alpha^{5/4}\\
&\leq L(n)+\left(\frac 34\alpha F'(\bar Z)\right)\left(|Z_n-\bar Z|+\alpha^{3/4}\, W_2\left(n,\Gamma_{2\sigma^2}(\cdot-Z_n)\right)\right)+ C\alpha^{5/4}\\
&\leq \left(1+\frac 34\alpha F'(\bar Z) \right)L(n)+ C\alpha^{5/4}.
\end{align*}
If $\alpha>0$ is small enough, we have $C\alpha^{5/4}\leq \frac{|F'(\bar Z)|}4\alpha^{6/5}$ and $0\leq 1+\alpha F'(\bar Z)$, which imply
\begin{align}\label{est:LTa}
L(\mathcal T_\alpha (n))&\leq \left|1+\frac 34\alpha F'(\bar Z) \right|\alpha^{1/5}+ \frac{|F'(\bar Z)|}4\alpha^{6/5}=\left|1+\frac \alpha 2 F'(\bar Z) \right|\alpha^{1/5}<\alpha^{1/5}.
\end{align}
Let $\beta>0$ provided by Lemma~\ref{lem:tailpreliminary}. The definition of $\Omega_\eta$ implies that any $n\in\Omega_\eta$ satisfies \eqref{eq:condn}, as well as
\begin{align*}
W_2(n,\Gamma_{2\sigma^2}(\cdot-\bar Z))\leq W_2(n,\Gamma_{2\sigma^2}(\cdot-Z_n))+|Z_n-\bar Z|\leq W_2(n,\Gamma_{2\sigma^2}(\cdot-Z_n))+L(n)\leq \sqrt\alpha+\alpha^{1/5}.
\end{align*}
Then $W_2(n,\Gamma_{2\sigma^2}(\cdot-\bar Z))\leq \beta$ if $\alpha>0$ is small enough, and the condition  \eqref{eq:condW} is satisfied. Lemma~\ref{lem:tailpreliminary} then holds, proving that the tail estimates \eqref{eq:condTn1}, \eqref{eq:condTn2} are true. These two tail estimates as well as \eqref{est:LTa} and \eqref{est:W2steadybis} prove that $\mathcal T_\alpha(\Omega_\eta)\subset\Omega_\eta$ for any $n\in \Omega_\eta$, provided $\alpha>0$ is small enough.

 \medskip
 
\textbf{Step 3: $T_\alpha $ is a compact map over $\Omega_\eta$}\\ 
We first obtain estimates on the tails of $T_\alpha(n)$ for $n\in\Omega_\eta$. We notice that 
\[\sup_{m\in\Omega_\eta}\int_{\mathbb R} x^2T_\alpha(m)(x)<\infty,\]
 and then, if $x_0:=\left(\sup_{m\in\Omega_\eta}\int_{\mathbb R} x^2T_\alpha(m)(x)\right)^{1/2}$, we have, for $n\in\Omega_\eta$,
\begin{align*}
\int_{-\infty}^{-x_0}T_\alpha(n)&\leq \frac 1{x_0^2}\int_{\mathbb R} x^2T_\alpha(n)(x)\,dx\leq \eta,
\end{align*}
and then $T_\alpha(n)\in \Omega_\eta$ implies $\partial_x T_\alpha(n)\geq T_\alpha(n)$ on $(-\infty,-x_0]$. Since $T_\alpha(n)$ is a probability distribution, there is $x_1\in[-x_0-1,x_0]$ such that $T_\alpha(n)(x_1)\leq 1$, and we can then compare $T_\alpha(n)$ to the solution of an ODE to obtain that 
\[\forall x\in(-\infty,-x_0-1],\quad T_\alpha(n)(x)\leq e^{-|x+x_1|}\leq e^{1-|x+x_0|}.\]
A similar argument can be made to show $T_\alpha(n)(x)\leq e^{-|x-x_1|}\leq e^{1-|x-x_0|}$ for $x\in[x_0+1,\infty)$. 

If $(n_k)_k\in (\Omega_\eta)^{\mathbb N}$, then $\|\partial_x T_\alpha(n_k)\|_\infty\leq \|\Gamma_{\sigma^2}'\|_\infty$ and  $(T_\alpha(n_k))_k\in (\Omega_\eta)^{\mathbb N}$ is uniformly Lipschitz continuous. The Arzela-Ascoli Theorem and a diagonal extraction show that for some increasing function $\varphi:\mathbb N\to\mathbb N$ and $\bar n\in W^{2,\infty}(\mathbb R)$, we have $\|T_\alpha(n_{\varphi(k)})-\bar n\|_{L^\infty(-r,r)}\xrightarrow[k\to\infty]{} 0$ for any $r>0$. Since $T_\alpha(n_{\varphi(k)})$ is uniformly bounded in $W^{2,\infty}(\R)$, $\bar n\in W^{2,\infty}(\R)$ (in particular $\bar n\in C^1(\R)$), and the tail estimates satisfied by $T_\alpha(n_{\varphi(k)})\in\Omega_\eta$ (see the definition of $\Omega_\eta$) extend to the limit distribution $\bar n$, and $\bar n\in \mathcal P_2(\R)$. These uniform exponential tail estimates combined to the local uniform convergence of the sequence implies $W_2(T_\alpha(n_{\varphi(k)}),\bar n)\xrightarrow[]{k\to\infty} 0$, which proves that $T_\alpha(\Omega_\eta)$ is relatively compact, and $T_\alpha$ is therefore a compact map over $\Omega_\eta$.

 \medskip

\noindent\textbf{Step 4: Convexity of $\tilde \Omega_\eta$}\\ 
 The pseudo-inverse distributions of probability measures described in Section~\ref{subsec:defnot} (see \cite{carrillo2007contractive} for a more details) provide an isometry between 
 \[(\mathcal P_2(\R),W_2)\textrm{ and }\big(\{v\in L^2([0,1]); v\textrm{ non decreasing}\},\|\cdot\|_{L^2([0,1])}\big),\]
that allow us to restate the properties proven above in a pseudo-inverse setting: we can define $\tilde\Omega_\eta$ the set of pseudo-inverse distributions of elements of $\Omega_\eta$. For $n\in \Omega_\eta$, we denote by $u$ its pseudo-inverse distribution, and by $\tilde T_\alpha(u)$ the pseudo-inverse distribution of $T_\alpha(n)$. The arguments above prove that $\tilde \Omega_\eta$ is stable under the map $\tilde T_\alpha$, and that map is compact over $\tilde\Omega_\eta$ for the topology $\|\cdot\|_{L^2([0,1])}$.
 
\medskip 
 
To show the convexity of $\tilde \Omega_\eta$ (which corresponds to the convexity of $\Omega_\eta$ along geodesics for $W_2$), let $n_0,n_1\in\tilde \Omega_\eta$ and $u_0,\,u_1$ their respective  pseudo-inverse distributions. We define then $u_\theta=(1-\theta) u_0+\theta u_1$, and $n_\theta\in \mathcal P_2(\R)$ with pseudo-inverse distribution $u_\theta$. We show below that $n_\theta\in \Omega_\eta$:
\begin{itemize}
\item Let $v$ the pseudo-inverse distribution of $\Gamma_{2\sigma^2}$.  We notice that the definition of $L(\cdot)$ implies
\begin{align*}
&L(n_\theta)=\left|\int_0^1 (u_\theta(z)-\bar Z)\,dz\right|+\alpha^{3/4}\,\left(\int_0^1 \left|u_\theta(z)-v(z)-\int_0^1 u_\theta(z')\,dz'\right|^2\,dz\right)^{1/2}\\
&\quad =\left|\int_0^1 (1-\theta) \left(u_0(z)-\bar Z\right)+\theta \left(u_1(z)-\bar Z\right)\,dz\right|\\
&\qquad +\alpha^{3/4}\,\left\|(1-\theta) \left(u_0(\cdot)-v(\cdot)-\int_0^1u_0(z')\,dz'\right)+\theta\left(u_1(\cdot)-v(\cdot)-\int_0^1u_1(z')\,dz'\right)\right\|_{L^2([0,1])}\\
&\quad \leq (1-\theta)\left|\int_0^1 \left(u_0(z)-\bar Z\right)\,dz\right|+\theta \left|\int_0^1\left(u_1(z)-\bar Z\right)\,dz\right|\\
&\qquad +\alpha^{3/4}\,\left\|(1-\theta) \left(u_0(\cdot)-v(\cdot)-\int_0^1u_0(z')\,dz'\right)\right\|_{L^2([0,1]}\\
&\qquad +\alpha^{3/4}\left\|\theta\left(u_1(\cdot)-v(\cdot)-\int_0^1u_1(z')\,dz'\right)\right\|_{L^2([0,1])}\\
&\quad\leq (1-\theta)L(n_0)+\theta L(n_1)\leq \alpha^{1/5},
\end{align*}
where we have used the triangular inequality for the $L^2([0,1])$ norm. 
\item Thanks to \eqref{prpoertypseudoinverse}, we have $n(u(z))=\frac 1{u'(z)}$ and $n'(u(z))=-\frac {u''(z)}{(u'(z))^3}$. The conditions relating $n'$ and $n$ in the definition of $\Omega_\eta$ can then be written on the pseudo-inverse distribution $u$ as 
\[\forall z\in[0,\eta],\quad -u''(z)>(u'(z))^2,\quad \forall z\in[1-\eta,1],\quad u''(z)> (u'(z))^2.\]
If these conditions are met by $u_0$ and $u_1$, then for $z\in[0,\eta]$, we use the convexity of the square function to show
\begin{align*}
(u_\theta'(z))^2&=\left((1-\theta)u_0'(z)+\theta u_1'(z)\right)^2\leq (1-\theta)(u_0'(z))^2+\theta (u_1'(z))^2\\
&< -(1-\theta)u_0''(z)-\theta u_1''(z)=-u_\theta''(z),
\end{align*}
and similarly, for $z\in[1-\eta,1]$, $(u_\theta'(z))^2<u_\theta''(z)$.
\end{itemize}
We have then proven $n_\theta\in \Omega_\eta$ (or, equivalently, $u_\theta\in\tilde T_\alpha(\Omega_\eta)$). The set $\tilde \Omega_\eta$ is then convex. 

\medskip

\noindent\textbf{Step 5: Existence and properties of a fixed point of $T_\alpha$}\\
 The set $\tilde \Omega_\eta$ is then a non-empty convex subset of the Banach space $L^2([0,1])$, and $\tilde T_\alpha:\tilde \Omega_\eta\to\tilde \Omega_\eta$ is a compact map thanks to Step 3. We can then apply the Schauder fixed point Theorem (See \cite{granas2003fixed}, Theorem~3.2 p. 119) to show the existence of a fixed point $\bar u\in\Omega_\eta$ for $\tilde T_\alpha$, and $\bar u$ is the pseudo-inverse distribution of $\bar n\in\Omega_\eta$ that satisfies $\bar n=T_\alpha(\bar n)$. 

\medskip

The steady state $\bar n\in\mathcal P_2(\mathbb R)\cap C^1(\mathbb R)$ satisfies the following properties (related to Assumption~\ref{ass2}, see below)
\begin{itemize}
\item \emph{Exponential decay:} We have $\int_x^\infty \bar n(y)\,dy\leq \frac 1{x^2}\int_x^\infty y^2\bar n(y)\,dy\leq \frac C{x^2}$, so that $R_1:=\sqrt{C/\eta}$ implies $\int_{R_1}^\infty \bar n(y)\,dy\leq\eta$, as well as $\int_{-\infty}^{-R_1} \bar n(y)\,dy\leq\eta$. Then $\bar n\in\Omega_\eta$ implies
\[\forall x\in[R_1,\infty),\quad \bar \partial_x n(-x)>\bar n(-x),\quad \partial_x \bar n(x)<-\bar n(x).\]
\item \emph{Lower bound:} For $x_0$ large enough, 
\begin{align*}
\int_{-x_0}^{x_0}\bar n&\geq 1-\int_{[-x_0,x_0]^c}\bar n\geq 1-\frac 1{x_0^2}\int_{\mathbb R} x^2\bar n(x)\,dx\geq \frac 12,
\end{align*}
thanks to $L(\bar n)\leq C\alpha^{1/4}$. Then,
\begin{align}
\bar n(x)&=T\left(\frac{(1+\alpha a(\cdot))\bar n}{1+\alpha I_{\bar n}}\right)\geq \int_{-x_0}^{x_0}\int_{-x_0}^{x_0} \Gamma_{\sigma^2}\left(x-\frac{y_*+y_*'}2\right)\bar n(y_*)\bar n(y_*')\,dy_*\,dy_*'\nonumber\\
&\geq \left(\min_{x_*,x_*'\in[-x_0,x_0]}\Gamma_{\sigma^2}\left(x-\frac{y_*+y_*'}2\right)\right)\int_{-x_0}^{x_0}\int_{-x_0}^{x_0} \Gamma_{\sigma^2}\left(x-\frac{y_*+y_*'}2\right)\bar n(y_*)\bar n(y_*')\,dy_*\,dy_*'\nonumber\\
&\geq \frac 14 \min_{\tilde x\in[-x_0,x_0]}\frac 1{\sigma \sqrt{2\pi}}e^{-\frac{(x-\tilde x)^2}{2\sigma^2}}\geq \frac 14 \frac 1{\sigma \sqrt{2\pi}}e^{\frac{-x^2+x_0^2+2|x|x_0}{2\sigma^2}}\frac 14\frac 1{\sigma \sqrt{2\pi}}e^{\frac{-x^2+x_0^2+|x|^2/2+x_0^2}{2\sigma^2}}\geq \frac 1C\Gamma_{\sigma^2/2}(x),\nonumber
\end{align}
and then, with $R_2=\max(C,2/\sigma^2)$,
\[\forall x\in\mathbb R,\quad \bar n(x)\geq \frac 1{R_2}\Gamma_{1/R_2}(x).\]
\item \emph{$Z^0$ belongs to the basin of attraction of $\bar Z$:} $\bar n\in \Omega_\eta$ implies $\left|\bar Z-Z_0\right|\leq C\alpha^{1/4}$, and $Z_0$ is therefore is in the basin of attraction  of $\bar Z$ provided $\alpha>0$ is small enough. Indeed, $F'(\bar Z)<0$ (see Assumption~\ref{ass1}) and $Z_0$ close to $\bar Z$ implies $F'$ is negative on the interval joining $Z_0$ and $\bar Z$: if $Z_0<\bar Z$, then, for some $R_3>0$,
\[\min_{x\in [Z_0,\bar Z-|F'(\bar Z)|/4\|F''\|_\infty]}F(x)\geq \frac 1{R_3},\]
or, if $Z_0>\bar Z$, $\max_{x\in [Z_0,\bar Z-|F'(\bar Z)|/4\|F''\|_\infty]}F(x)\leq -\frac 1{R_3}$. 
\end{itemize}
The steady state $\bar n$ then satisfies Assumption~\ref{ass2} if $R:=\max(R_1,R_2,R_3)$. Moreover, Since $\bar n=T_\alpha(\bar n)$ and $Z_0=\int_\R y\bar n(y)\,dy=\int_\R yT_\alpha(\bar n)(y)\,dy$, inequality \eqref{est:W2steady} implies
\[ W_2(\bar n,\Gamma_{2\sigma^2}(\cdot-Z_0))\leq C\sqrt\alpha.\]
This concludes the proof of Theorem~\ref{Thm:existence-steady}.

\subsection{Tail estimates}

In this section, derive properties of solutions of \eqref{eq:model} from estimate \eqref{est:Zbasic}. The first property we are interested in is a lower bound of solutions around $x=\bar Z$.
\begin{lem}\label{lem:lowerbound}
Let $R>0$ and $a\in W^{2,\infty}(\mathbb R)$, $\bar Z\in\mathbb R$ satisfying Assumption~\ref{ass1}. There is $\bar \alpha>0$ and  $C>0$ such that for $\alpha\in(0,\bar \alpha)$ and $n^0\in\mathcal P_2(\mathbb R)\cap C^1(\mathbb R)$ satisfying Assumption~\ref{ass2}, the solution $n\in L^\infty(\mathbb R_+,\mathcal P_2(\mathbb R))$ of \eqref{eq:model} with initial data $n^0$ satisfies
\begin{equation}\label{property:bounded-away-from-0}
\forall (t,x)\in[0,\infty)\times\mathbb R,\quad n(t,x)\geq \frac 1C\Gamma_{1/R\wedge \sigma^2/2}(\bar Z-x).
\end{equation}
\end{lem}

\begin{proof}[Proof of Lemma~\ref{lem:lowerbound}]
Thanks to Lemma~\ref{lem:roughmacro} (see \eqref{eq:moment2}), for $x_0>0$,
\[\int_{-x_0}^{x_0} n(t,x)\,dx\geq 1-\int_{(-x_0,x_0)^c} n(t,x)\,dx\geq 1-\frac 1{x_0^2}\int_{\mathbb R} x^2n(t,x)\,dx\geq \frac 12,\]
provided $x_0=\sqrt{2C}$ with $C>0$ the constant appearing in Lemma~\ref{lem:estmacro} (see \eqref{eq:moment2}), which is independent from $\rho\in(0,1)$. For $\alpha>0$ small enough, thanks to \eqref{eq:integralformulation}),
\begin{align*}
&n(t,x)\leq n^0(x)e^{-2t}\\
&\qquad +\int_{ 0}^t\left(\iint_{\mathbb R^2}\Gamma_{\sigma^2}\left(x-\frac{y_*+y_*'}2\right)(1+\alpha a(y_*))n(t,y_*)(1+\alpha a(y_*'))n(t,y_*')\,dy_*\,dy_*'\right)e^{-2(t-s)}\,ds\\
&\quad \geq n^0(x)e^{-2t}+\int_{(t-1)\vee 0}^t\left(\iint_{\mathbb R^2}\Gamma_{\sigma^2}\left(x-\frac{y_*+y_*'}2\right)n(t,y_*)n(t,y_*')\,dy_*\,dy_*'\right)e^{-2(t-s)}\,ds\\
&\quad \geq \frac 1R\Gamma_{1/R}(x)e^{-2t}+\frac 14\min_{y_*,y_*'\in [-x_0,x_0]}\Gamma_{\sigma^2}\left(x-\frac{y_*+y_*'}2\right)\int_{(t-1)\vee 0}^te^{-2(t-s)}\,ds\\
&\quad \geq 1_{t\in[0,1]}\frac 1{Re^2}\Gamma_{1/R}(x)+1_{t\in[1,\infty)}\frac {1}8\Gamma_{\sigma^2}\left(|x-\bar Z|+|x_0|+|\bar Z|\right),
\end{align*}
that is $n(t,x)\geq \frac 1{Re^2}\Gamma_{1/R}(x)$ if $t\in[0,1]$, and $n(t,x)\geq \frac {1}8\Gamma_{\sigma^2}\left(|x-\bar Z|+|x_0|+|\bar Z|\right)$ if $t\geq 1$. We can further estimate the lower bound for $t\geq 1$:
\begin{align*}
\Gamma_{\sigma^2}\left(|x-\bar Z|+|x_0|+|\bar Z|\right)&\geq \frac 1{\sigma \sqrt{2\pi}}e^{\frac{-x^2+(|x_0|+|\bar Z|)^2+2|x|(|x_0|+|\bar Z|)}{2\sigma^2}}\\
&\geq\frac 1{\sigma \sqrt{2\pi}}e^{\frac{-x^2+(|x_0|+|\bar Z|)^2+|x|^2/2+2(|x_0|+|\bar Z|)^2}{2\sigma^2}} \geq C\Gamma_{\sigma^2/2}(x).
\end{align*}
This estimate concludes the proof of \eqref{property:bounded-away-from-0}.
\end{proof}

In the next lemma, we show that Assumption~\ref{ass2} implies  that the solution $n(t,\cdot)$ of \eqref{eq:model} is uniformly exponentially decreasing far from $\bar Z$. Notice that to obtain this estimate, we need to assume that both $\alpha$ and $W_2\left(n^0,\Gamma_{2\sigma^2}\left(\cdot-Z^0\right)\right)$ are small. This second assumption, in particular, necessary to propagate the tail estimate for small times, namely $t\in [0,-(\ln\alpha)/C]$.

\begin{lem}\label{lem:tailprecise}
Let $R>0$ and $a\in W^{2,\infty}(\mathbb R)$, $\bar Z\in\mathbb R$ satisfying Assumption~\ref{ass1}, and let $-\infty<Z_-<Z_+<\infty$. There is $\bar \alpha>0$, $\rho\in(0,1]$ and  $R_0>0$ such that if $\alpha\in(0,\bar \alpha)$ and $n^0\in\mathcal P_2(\mathbb R)\cap C^1(\mathbb R)$ satisfies Assumption~\ref{ass2} as well as $W_2\left(n^0,\Gamma_{2\sigma^2}\left(\cdot-Z^0\right)\right)\leq \rho$, then the solution $n\in L^\infty(\mathbb R_+\times\mathbb R)$ of \eqref{eq:model} with initial data $n^0$ provided by Theorem~\ref{cor:existence}, satisfies
\[\forall (t,x)\in[0,\infty)\times [R_0,+\infty),\quad \partial_x n(t,-x)>n(t,-x),\quad \partial_x n(t,x)<-n(t,x).\]

\end{lem}

To prove this lemma, we introduce a truncated model: for $\varepsilon>0$, we consider the solution $u_\varepsilon\in L^\infty(\mathbb R_+,\mathcal P((-1/\varepsilon,1/\varepsilon)))$ of
\begin{equation}\label{eq:modifiedpb}
\left\{\begin{array}{l}
u_\varepsilon(0,x)=\frac{n^0(x)1_{|x|\leq \frac 1\varepsilon}}{\int_{-1/\varepsilon}^{1/\varepsilon}n^0(y)\,dy},\quad \textrm{ for }x\in(-1/\varepsilon,1/\varepsilon),\\
\partial_tu_\varepsilon(t,x)=\left(1+\alpha\int_{\mathbb R} a(y)u_\varepsilon(t,y)\,dy\right)^2\Big[ T\left(\frac{(1+\alpha a)u_\varepsilon(t,\cdot)}{1+\alpha\int_{\mathbb R} a(y)u_\varepsilon(t,y)\,dy}\right)1_{|x|\leq \frac 1\varepsilon}\\
\phantom{u_\varepsilon(0,x)=}-\left(\int_{-1/\varepsilon}^{1/\varepsilon}T\left(\frac{(1+\alpha a)u_\varepsilon(t,\cdot)}{1+\alpha\int_{\mathbb R} a(y)u_\varepsilon(t,y)\,dy}\right)(z)\,dz\right)u_\varepsilon(t,x)\Big],\quad \textrm{for } t\geq 0,\quad x\in(-1/\varepsilon,1/\varepsilon).
\end{array}\right.
\end{equation}
The existence and uniqueness of $u_\varepsilon$ is proven in Theorem~\ref{thm:existence}, and we notice that for any $t\geq 0$, $u_\varepsilon(t,\cdot)$ can be extended into a measure in $\mathcal P_2(\mathbb R)\cap L^\infty(\mathbb R)$ through $u_\varepsilon(t,x)=0$ for $|x|\geq 1/\varepsilon$. We prove in Lemma~\ref{thm:unifestim} that the solutions $u_\varepsilon(t,\cdot)$ have a second  moment that is uniformly bounded in $\varepsilon>0$ and their distributions are close to Gaussian distributions. We have $u_\varepsilon\to_{\varepsilon\to 0} u$ as well as $\partial_x u_\varepsilon\to_{\varepsilon\to 0} \partial_x u$ pointwise (Lemma~\ref{lem:tailpreciseapproxcv}), where $u$ is the solution of \eqref{eq:model}.

\begin{proof}[Proof of Lemma~\ref{lem:tailprecise}]

Thanks to Lemma~\ref{thm:unifestim}, there is a bound $\bar C$ such that for $\alpha\in(0,\bar \alpha)$ and $\varepsilon\in(0,\bar\varepsilon(\alpha))$ (with $\bar\epsilon(\alpha)>0$), $\|Z_{u_\varepsilon}\|_\infty\leq \bar C$. We use this bound to define $R_0$:
\[R_0:=\max(R,\bar C)
+4\sigma^2,\]
and we define a second quantity $R_1=8(1+\|a\|_\infty)$. 
Let 
\[t_0:=\inf\left\{t\geq 0;\, \exists x_0\in[R_0,1/\varepsilon),\, \partial_x u_\varepsilon(t,-x_0)\leq u_\varepsilon(t,-x_0)\textrm{ or }\partial_x u_\varepsilon(t,x_0)\geq -u_\varepsilon(t,x_0)\right\}\in[0,+\infty].\]
Notice that $u_\varepsilon(t,\cdot)\in C^1((-1/\varepsilon,1/\varepsilon))$ thanks to Theorem~\ref{thm:existence}, so that $t_0$ is well defined. Thanks to assumption~\ref{ass2} and the definition of $u_\varepsilon(0,\cdot)$, 
\begin{equation*}
\partial_x u_\varepsilon(0,-x_0)> u_\varepsilon(0,-x_0)+\eta,\quad \partial_x u_\varepsilon(0,x_0)\leq -u_\varepsilon(0,x_0)-\eta,
\end{equation*}
for some $\eta>0$. Thanks to Theorem~\ref{thm:existence}, $n_\varepsilon(t,x)=u^0(x)g_\varepsilon(t,x)+h_\varepsilon(t,x)$ with $g_\varepsilon,h_\varepsilon\in C^\infty([0,\infty)\times(-1/\varepsilon,1/\varepsilon))$, and since $u^0\in C^1(\R)$, the functions  $u_\varepsilon$ and $\partial_x u_\varepsilon$ on $\mathbb R_+\times(-1/\varepsilon,1/\varepsilon)$ are uniformly continuous. We have then, for $t\geq 0$ small enough,
\begin{equation*}
\partial_x u_\varepsilon(t,-x_0)> u_\varepsilon(t,-x_0),\quad \partial_x u_\varepsilon(t,x_0)\leq -u_\varepsilon(t,x_0),
\end{equation*}
which implies $t_0>0$. Our goal is to show that $t_0=+\infty$, and we proceed by contradiction and assume that $t_0> 0$ is finite. 

\medskip

For $t\in[0,t_0]$, $u_\varepsilon(t,\cdot)$ satisfies $u_\varepsilon(t,x)=0$ for $|x|\geq 1/\varepsilon$ and 
\begin{equation}\label{eq:estAssu1}
\forall (t,x)\in [0,t_0]\times[R_0,1/\varepsilon),\quad \partial_x u_\varepsilon(t,-x)\geq u_\varepsilon(t,-x),\quad \partial_x u_\varepsilon(t,x)\leq-u_\varepsilon(t,x).
\end{equation}
We will show that \eqref{eq:estAssu1} implies, for $(t,x)\in[0,t_0]\times [R_0,+\infty)$,
\begin{equation}\label{eq:estAssu2}
\partial_x \mathcal T(t,-x)>\frac 32\mathcal T(t,-x),\quad \partial_x \mathcal T(t,x)<-\frac 32\mathcal T(t,x),
\end{equation}
where
\[\mathcal T(t,x):=\iint_{\mathbb R^2}\Gamma_{\sigma^2}\left(x-\frac{y_*+y_*'}2\right)(1+\alpha a(y_*))u_\varepsilon(t,y_*)(1+\alpha a(y_*'))u_\varepsilon(t,y_*')\,dy_*\,dy_*'.\]
Thanks to the symmetry of  the problem, we can focus on the second part of \eqref{eq:estAssu2},  where $x\geq R_0$.

We introduce $\varphi$, the piecewise linear function such that $\varphi(x)=0$ on $(-\infty,R_0]$, $\varphi(x)=1$ on $[R_0+R_1,\infty)$ and $\|\partial_x\varphi\|_{L^\infty}\leq \frac 1{R_1}$. We use this function $\varphi$ to decompose the following integral expression:
\begin{align}
&\partial_x\mathcal T(t,x)=\iint_{\mathbb R^2} \partial_x\Gamma_{\sigma^2}\left(x-\frac{y_*+y_*'}2\right)(1+\alpha a(y_*))u_\varepsilon(t,y_*)(1+\alpha a(y_*'))u_\varepsilon(t,y_*')\,dy_*\,dy_*'\nonumber\\
&\quad=\iint_{\mathbb R^2}\left(1-\varphi(y_*)\right) \left(1-\varphi(y_*')\right) \partial_x\Gamma_{\sigma^2}\left(x-\frac{y_*+y_*'}2\right)\nonumber\\
&\phantom{\quad=\iint_{\mathbb R^2}} (1+\alpha a(y_*))u_\varepsilon(t,y_*)(1+\alpha a(y_*'))u_\varepsilon(t,y_*')\,dy_*\,dy_*'\nonumber\\
&\qquad +\iint_{\mathbb R^2} \varphi(y_*) \partial_x\Gamma_{\sigma^2}\left(x-\frac{y_*+y_*'}2\right)(1+\alpha a(y_*))u_\varepsilon(t,y_*)(1+\alpha a(y_*'))u_\varepsilon(t,y_*')\,dy_*\,dy_*'\nonumber\\
&\qquad +\iint_{\mathbb R^2} \varphi(y_*')\left(1-\varphi(y_*)\right) \partial_x\Gamma_{\sigma^2}\left(x-\frac{y_*+y_*'}2\right)\nonumber\\
&\phantom{\quad=\iint_{\mathbb R^2}} (1+\alpha a(y_*))u_\varepsilon(t,y_*)(1+\alpha a(y_*'))u_\varepsilon(t,y_*')\,dy_*\,dy_*'\nonumber\\
&\quad = I_1(x)+I_2(x)+I_3(x).\label{est:Itails}
\end{align}
To estimate $I_2$, we notice that $\varphi(y_*)=0$ if $y_*\leq R_0$, and $u_\varepsilon(t,y_*)=0$ if $y_*> 1/\varepsilon$. We can therefore restrict the domain of the integral over the variable $y_*$ in the formula for $I_2$, leading to
\begin{align*}
I_2&=\int_\R\int_{R_0}^{1/\varepsilon} \varphi(y_*) \partial_x\Gamma_{\sigma^2}\left(x-\frac{y_*+y_*'}2\right)(1+\alpha a(y_*))u_\varepsilon(t,y_*)(1+\alpha a(y_*'))u_\varepsilon(t,y_*')\,dy_*\,dy_*'\\
&=-2\int_{\mathbb R}\int_{R_0}^{1/\varepsilon} \varphi(y_*)\frac{d}{dy_*}\left[\Gamma_{\sigma^2}\left(x-\frac{y_*+y_*'}2\right)\right](1+\alpha a(y_*))u_\varepsilon(t,y_*)(1+\alpha a(y_*'))u_\varepsilon(t,y_*')\,dy_*\,dy_*',
\end{align*}
where the second inequality results from the following calculation:
\[\partial_x\Gamma_{\sigma^2}\left(x-\frac{y_*+y_*'}2\right)=-2\frac{d}{dy_*}\left[\Gamma_{\sigma^2}\left(x-\frac{y_*+y_*'}2\right)\right].\]
We use an integration by parts in $y_*$ to obtain:
\begin{align*}
I_2(x)&= 2\int_{\mathbb R}\int_{R_0}^{1/\varepsilon} \Big[\varphi'(y_*)(1+\alpha a(y_*))u_\varepsilon(t,y_*)+\varphi(y_*)\big((1+\alpha a(y_*))\partial_x u_\varepsilon(t,y_*)+\alpha a'(y_*)u_\varepsilon(t,y_*)\big)\Big]\\
&\phantom{dfgdregrsgrt} \Gamma_{\sigma^2}\left(x-\frac{y_*+y_*'}2\right)(1+\alpha a(y_*'))u_\varepsilon(t,y_*')\,dy_*\,dy_*'\\
&\quad \leq 2\int_{\mathbb R}\int_{R_0}^{1/\varepsilon} \Big[\|\varphi'\|_\infty (1+\alpha \|a\|_\infty ) u_\varepsilon(t,y_*)+\alpha \|\varphi\|_\infty \|a'\|_\infty u_\varepsilon(t,y_*)
\\
&\qquad +\varphi(y_*)(1+\alpha a(y_*))\partial_x u_\varepsilon(t,y_*)\Big]\Gamma_{\sigma^2}\left(x-\frac{y_*+y_*'}2\right)(1+\alpha a(y_*'))u_\varepsilon(t,y_*')\,dy_*'.
\end{align*}
Thanks to \eqref{eq:estAssu1}, for $y_*\in[ R_0,1/\varepsilon)$, we have $(1+\alpha a(y_*))\partial_x u_\varepsilon(t,y_*)\leq -(1+\alpha a(y_*))u_\varepsilon(t,y_*)\leq -u_\varepsilon(t,y_*)$, and thus
\begin{align}
I_2(x)& \leq 2\int_{\mathbb R}\int_{R_0}^{1/\varepsilon} \Big[\|\varphi'\|_\infty (1+\alpha \|a\|_\infty ) u_\varepsilon(t,y_*)+\alpha \|\varphi\|_\infty \|a'\|_\infty u_\varepsilon(t,y_*)\nonumber
\\
&\quad -\varphi(y_*) u_\varepsilon(t,y_*)\Big]\Gamma_{\sigma^2}\left(x-\frac{y_*+y_*'}2\right)(1+\alpha a(y_*'))u_\varepsilon(t,y_*')\,dy_*'\nonumber\\
& \leq \iint_{\mathbb R^2} \left[\mathcal O\left(\frac {1}{R_1}+\alpha\right)-2\varphi(y_*)\right]\Gamma_{\sigma^2}\left(x-\frac{y_*+y_*'}2\right)u_\varepsilon(t,y_*)u_\varepsilon(t,y_*')\,dy_*\,dy_*'.\label{est:I2tails}
\end{align}
where $\left|\mathcal O\left(\frac {1}{R_1}+\alpha\right)\right|\leq \frac {1+\alpha\|a\|_\infty}{R_1}+\alpha\|a'\|_\infty\leq \frac 14$ if $\alpha\leq \frac1{8(\|a'\|_\infty+1)}$ (we recall $R_1=8(1+\|a\|_\infty)$). $I_3(x)$ can be estimated in a similar manner, except the integration by parts is done on variable $y_*'$, to obtain
\begin{align}
I_3(x)& =-2\iint_{\mathbb R^2}\varphi(y_*')\left(1-\varphi(y_*)\right) \frac d{d y_*'}\left[\Gamma_{\sigma^2}\left(x-\frac{y_*+y_*'}2\right)\right]\nonumber\\
&\quad \phantom{=-2\iint_{\mathbb R^2}} (1+\alpha a(y_*))u_\varepsilon(t,y_*)(1+\alpha a(y_*'))u_\varepsilon(t,y_*')\,dy_*\,dy_*'\nonumber\\
& \leq \iint_{\mathbb R^2} (1-\varphi(y_*))\left[\mathcal O\left(\frac {1}{R_1}+\alpha\right)-2\varphi(y_*')\right]\Gamma_{\sigma^2}\left(x-\frac{y_*+y_*'}2\right)u_\varepsilon(t,y_*)u_\varepsilon(t,y_*')\,dy_*\,dy_*'.\label{est:I3tails}
\end{align}

\medskip

To estimate $I_1(x)$, we  notice that $\left(1-\varphi(y_*)\right) \left(1-\varphi(y_*')\right) =0$ as soon as $y_*\geq R_0+R_1$ or $y_*'\geq R_0+R_1$, we can thus restrict the domain of the integrals in the formula for $I_1$, as follows:
\begin{align*}
I_1(x)&=\iint_{(-\infty,R_0+R_1]^2}\left(1-\varphi(y_*)\right) \left(1-\varphi(y_*')\right) \partial_x\Gamma_{\sigma^2}\left(x-\frac{y_*+y_*'}2\right)\\
&\quad \phantom{\iint_{(-\infty,R_0+R_1]^2}}(1+\alpha a(y_*))u_\varepsilon(t,y_*)(1+\alpha a(y_*'))u_\varepsilon(t,y_*')\,dy_*\,dy_*'.
\end{align*}
we notice that if $z\geq 2\sigma^2$, then $\partial_x\Gamma_{\sigma^2}(z)\leq -2\Gamma_{\sigma^2}(z)$. In particular, for $x\geq R_0+R_1+2\sigma^2$ and $y_*,y_*'\leq R_0+R_1$, $\partial_x\Gamma_{\sigma^2}\left(x-\frac{y_*+y_*'}2\right)\leq -2\Gamma_{\sigma^2}\left(x-\frac{y_*+y_*'}2\right)$. Therefore, provided $x\geq R_0+R_1+2\sigma^2$,
\begin{align}
I_1(x)&\leq -2\iint_{(-\infty,R_0+R_1]^2}\left(1-\varphi(y_*)\right) \left(1-\varphi(y_*')\right) \Gamma_{\sigma^2}\left(x-\frac{y_*+y_*'}2\right)\nonumber\\
&\phantom{\leq -2\int_{\mathbb R}\int_{\mathbb R}}\quad (1+\alpha a(y_*))u_\varepsilon(t,y_*)(1+\alpha a(y_*'))u_\varepsilon(t,y_*')\,dy_*\,dy_*'\nonumber\\
&\leq -2\iint_{\R^2}\left(1-\varphi(y_*)\right) \left(1-\varphi(y_*')\right) \Gamma_{\sigma^2}\left(x-\frac{y_*+y_*'}2\right)u_\varepsilon(t,y_*)u_\varepsilon(t,y_*')\,dy_*\,dy_*'.\label{est:I1tails}
\end{align}
If we use estimate \eqref{est:I1tails}, \eqref{est:I2tails}, \eqref{est:I3tails} to estimate $I_1$, $I_2$ and $I_3$ respectively, then \eqref{est:Itails} implies, for $x\geq R_0+R_1+2\sigma^2$:
\begin{align*}
\partial_x\mathcal T(t,x)&\leq \iint_{\mathbb R^2}\left(\mathcal O\left(\frac {1+\alpha}{R_1}+\alpha\right)-2\right)\Gamma_{\sigma^2}\left(x-\frac{y_*+y_*'}2\right)u_\varepsilon(t,y_*)u_\varepsilon(t,y_*')\,dy_*\,dy_*'
\end{align*}
where $\left|\mathcal O\left(\frac {1+\alpha}{R_1}+\alpha\right)\right|\leq \frac 14$. Thus,
\begin{align*}
\partial_x\mathcal T(t,x)&\leq -\frac 74\iint_{\mathbb R^2}\Gamma_{\sigma^2}\left(x-\frac{y_*+y_*'}2\right)u_\varepsilon(t,y_*)u_\varepsilon(t,y_*')\,dy_*\,dy_*'\\
&\leq -\frac 32\iint_{\mathbb R^2}\Gamma_{\sigma^2}\left(x-\frac{y_*+y_*'}2\right)(1+\alpha a(y_*))u_\varepsilon(t,y_*)(1+\alpha a(y_*'))u_\varepsilon(t,y_*')\,dy_*\,dy_*',
\end{align*}
provided $\alpha>0$ is small enough. We have thus proven, for $\alpha>0$ small enough,
\begin{align}\label{eq:Tlargex}
\forall x\in[R_0+R_1+2\sigma^2,\infty),\quad \partial_x\mathcal T(t,x)&\quad < -\frac 32\mathcal T(t,x).
\end{align}
To obtain a similar estimate for $x\in [R_0,R_0+R_1+2\sigma^2]$, we introduce a different estimate, using the Lipschitz continuity of $(y_*,y_*')\mapsto \Gamma_{\sigma^2}\left(x-\frac{y_*+y_*'}2\right)$ and the Kantorovich-Rubinstein formula: 
\begin{align*}
\partial_x\mathcal T(t,x)&=\iint_{\mathbb R^2} \partial_x\Gamma_{\sigma^2}\left(x-\frac{y_*+y_*'}2\right)(1+\alpha a(y_*))u_\varepsilon(t,y_*)(1+\alpha a(y_*'))u_\varepsilon(t,y_*')\,dy_*\,dy_*'\\
&=\iint_{\mathbb R^2} \partial_x\Gamma_{\sigma^2}\left(x-\frac{y_*+y_*'}2\right)u_\varepsilon(t,y_*)u_\varepsilon(t,y_*')\,dy_*\,dy_*'+\mathcal O(\alpha)\\
&=\iint_{\mathbb R^2} \partial_x\Gamma_{\sigma^2}\left(x-\frac{y_*+y_*'}2\right)\Gamma_{2\sigma^2}(y_*-Z_{u_\varepsilon}(t))\Gamma_{2\sigma^2}(y_*'-Z_{u_\varepsilon}(t))\,dy_*\,dy_*'+\mathcal O(\alpha)\\
&\quad +\mathcal O\Big(W_1\big(u_\varepsilon(t,\cdot),\Gamma_{2\sigma^2}(\cdot-Z_{u_\varepsilon}(t))\big)\Big)\\
&=\frac d{dx}\iint_{\mathbb R^2}\Gamma_{\sigma^2}\left(x-\frac{y_*+y_*'}2\right)\Gamma_{2\sigma^2}(y_*-Z_{u_\varepsilon}(t))\Gamma_{2\sigma^2}(y_*'-Z_{u_\varepsilon}(t))\,dy_*\,dy_*'\\
&\quad +\mathcal O(\alpha)+\mathcal O\Big(W_1\big(u_\varepsilon(t,\cdot),\Gamma_{2\sigma^2}(\cdot-Z_{u_\varepsilon}(t))\big)\Big)\\
&=\frac d{dx}\Gamma_{\sigma^2}(x-Z_{u_\varepsilon}(t))+\mathcal O(\alpha)+\mathcal O\Big(W_1\big(u_\varepsilon(t,\cdot),\Gamma_{2\sigma^2}(\cdot-Z_{u_\varepsilon}(t))\big)\Big),
\end{align*}
thanks to formula \eqref{eq:Maxwellienne}. Thanks to Lemma~\ref{thm:unifestim}, if $\alpha>0$ is small enough so that the exponential rate in \eqref{est:W2appendix} satisfies $-({\sqrt 3}/2-{(1+C\alpha)}/{\sqrt 2})t\leq -t/8$, and provided $0<\varepsilon<\bar \varepsilon(\alpha)$ is small enough and $\alpha>0$ is also small enough, we have
\begin{align*}
\partial_x\mathcal T(t,x)
&= \Gamma_{2\sigma^2}'(x-Z_{u_\varepsilon}(t))+ \mathcal O\left(\sqrt\alpha+\sqrt\varepsilon\right)+\mathcal O\left(\rho+\varepsilon\right)e^{-t/8}.
\end{align*}
We use the formula $\Gamma_{2\sigma^2}'(x)=-\frac {x}{2\sigma^2}\Gamma_{2\sigma^2}(x)$ and \eqref{eq:Maxwellienne} once more to prove
\begin{align*}
\partial_x\mathcal T(t,x)&=-\frac {x-Z_{u_\varepsilon}(t)}{2\sigma^2}\Gamma_{2\sigma^2}(x-Z_{u_\varepsilon}(t))+\mathcal O\left(\sqrt\alpha+\sqrt\varepsilon\right)+\mathcal O\left(\rho+\varepsilon\right)e^{-t/8}\\
&=-\frac {x-Z_{u_\varepsilon}(t)}{2\sigma^2}\iint_{\mathbb R^2}\Gamma_{\sigma^2}\left(x-\frac{y_*+y_*'}2\right)\Gamma_{2\sigma^2}(y_*-Z_{u_\varepsilon}(t))\Gamma_{2\sigma^2}(y_*'-Z_{u_\varepsilon}(t))\,dy_*\,dy_*'\\
&\quad+ \mathcal O\left(\sqrt\alpha+\sqrt\varepsilon\right)+\mathcal O\left(\rho+\varepsilon\right)e^{-t/8}\\
&=-\frac {x-Z_{u_\varepsilon}(t)}{2\sigma^2}\mathcal T(t,x)+ \frac {|x-Z_{u_\varepsilon}(t)|}{2\sigma^2}\left(\mathcal O\big(W_1(u_\varepsilon(t,\cdot),\Gamma_{2\sigma^2}(\cdot-Z_{u_\varepsilon}(t)))\big)+\mathcal O(\alpha)\right)\\
&\quad + \mathcal O\left(\sqrt\alpha+\sqrt\varepsilon\right)+\mathcal O\left(\rho+\varepsilon\right)e^{-t/8},
\end{align*}
thanks to the Kantorovich-Rubinstein formula. Since $x\in [R_0,R_0+R_1+2\sigma^2]$, the term $\frac {x-Z_{u_\varepsilon}(t)}{2\sigma^2}$ is bounded, and then
\begin{align*}
\partial_x\mathcal T(t,x)
&=-\frac {x-Z_{u_\varepsilon}(t)}{2\sigma^2}\mathcal T(t,x)+ \mathcal O\left(\sqrt\alpha+\sqrt\varepsilon\right)+\mathcal O\left(\rho+\varepsilon\right)e^{-t/8}\\
&\leq -2\mathcal T(t,x)+ \mathcal O\left(\sqrt\alpha+\sqrt\varepsilon\right)+\mathcal O\left(\rho+\varepsilon\right)e^{-t/8},
\end{align*}
since the definition of $R_0$ implies $x-Z_{u_\varepsilon}(t)\geq 4\sigma^2$. Thanks to Lemma~\ref{thm:unifestim}, $\mathcal T$ is uniformly bounded from below for $(t,x)\in [0,\infty)\times[R_0,R_0+R_1+2\sigma^2]$. Notice that $R_0, R_1,\sigma^2$ are independent from how small $\varepsilon,\rho,\alpha>0$ are. Therefore, provided $\varepsilon>0$, $\rho>0$ and $\alpha>0$ are small enough, we have
\begin{equation}\label{eq:rhosmall}
\forall x\in[R_0,R_0+R_1+2\sigma^2],\quad \mathcal O\left(\sqrt\alpha+\sqrt\varepsilon\right)+\mathcal O\left(\rho+\varepsilon\right)e^{-t/8}\leq \frac 12\mathcal T(t,x),
\end{equation}
and then
\begin{equation}\label{eq:conc}
\forall x\in [R_0,R_0+R_1+2\sigma^2],\quad \partial_x\mathcal T(t,x)\leq -\frac 32\mathcal T(t,x).
\end{equation}
This estimate and \eqref{eq:Tlargex} prove  \eqref{eq:estAssu2} for $(t,x)\in[0,t_0]\times[R_0,\infty)$. 
Since $\mathcal T$ and $\partial_x \mathcal T$ are continuous (uniformly in $x$ and $t$), \eqref{eq:estAssu2} holds for $(t,x)\in[0,t_0+\chi]\times[R_0,1/\varepsilon)$, for some $\chi>0$. Thanks to this estimate and to \eqref{eq:estAssu1} for $t=0$, the equation \eqref{eq:integralformulation} implies, for $(t,x)\in[0,t_0+\chi]\times[R_0,1/\varepsilon)$,
\begin{align*}
\partial_x u_\varepsilon(t,x)&=\partial_x u_\varepsilon(0,x)e^{-\int_0^t\left(1+\alpha I_{u_\varepsilon}(s)\right)^2\left(\int_{-1/\varepsilon}^{1/\varepsilon}T\left(\frac{(1+\alpha a)u_\varepsilon(s,\cdot)}{1+\alpha\int_{\mathbb R} a(y)u_\varepsilon(s,y)\,dy}\right)(z)\,dz\right)\,ds}\nonumber\\
&\quad +\int_0^t\left(1+\alpha I_{u_\varepsilon}(s)\right)^2\partial_xT\left(\frac{(1+\alpha a)u_\varepsilon(s,\cdot)}{1+\alpha\int_{\mathbb R}a(y)u_\varepsilon(s,y)\,dy}\right)(x)1_{|x|\leq \frac 1\varepsilon}\\
&\qquad \phantom{dgfezrsdg}e^{-\int_s^t\left(1+\alpha I_{u_\varepsilon}(\tau)\right)^2\left(\int_{-1/\varepsilon}^{1/\varepsilon}T\left(\frac{(1+\alpha a)u_\varepsilon(\tau,\cdot)}{1+\alpha\int_{\mathbb R}a(y)u_\varepsilon(\tau,y)\,dy}\right)(z)\,dz\right)\,d\tau}\,ds\\
&\leq - u_\varepsilon(0,x)e^{-\int_0^t\left(1+\alpha I_{u_\varepsilon}(s)\right)^2\left(\int_{-1/\varepsilon}^{1/\varepsilon}T\left(\frac{(1+\alpha a)u_\varepsilon(s,\cdot)}{1+\alpha\int_{\mathbb R}a(y)u_\varepsilon(s,y)\,dy}\right)(z)\,dz\right)\,ds}\nonumber\\
&\quad-\frac 32\int_0^t\left(1+\alpha I_{u_\varepsilon}(s)\right)^2T\left(\frac{(1+\alpha a)u_\varepsilon(s,\cdot)}{1+\alpha\int_{\mathbb R}a(y)u_\varepsilon(s,y)\,dy}\right)(x)1_{|x|\leq \frac 1\varepsilon}\\
&\qquad \phantom{dgfezrsdg}e^{-\int_s^t\left(1+\alpha I_{u_\varepsilon}(\tau)\right)^2\left(\int_{-1/\varepsilon}^{1/\varepsilon}T\left(\frac{(1+\alpha a)u_\varepsilon(\tau,\cdot)}{1+\alpha\int_{\mathbb R}a(y)u_\varepsilon(\tau,y)\,dy}\right)(z)\,dz\right)\,d\tau}\,ds \\
&< -u_\varepsilon(t,x),
\end{align*}
since $u_\varepsilon$ is given by \eqref{eq:modifiedpb}. Then \eqref{eq:estAssu1} is satisfied for $t\in[0,t_0+\chi)$ for some $\chi>0$. The hypothesis that $t_0$ is finite therefore leads to a contradiction and $u_\varepsilon$ satisfies \eqref{eq:estAssu1} with $t_0=+\infty$.  If $(t,x)\in [0,+\infty)\times\mathbb R$, $u_\varepsilon(t,x)$ converges to $n(t,x)$ (resp. $\partial_x u_\varepsilon(t,x)$ converges to  $\partial_x n(t,x)$) when $\varepsilon\to 0$, thanks to Lemma~\ref{lem:tailpreciseapproxcv}. We have therefore $\partial_x n(t,x)\leq -n(t,x)$, and $\partial_x n(t,-x)\geq n(t,-x)$ for $(t,x)\in[0,\infty)\times[R_0,\infty)$, which  concludes the proof of Lemma~\ref{lem:tailprecise}.

\end{proof}

\section{Wasserstein estimates}
\label{sec:wassest}

\subsection{Effect of a multiplicative operator acting on the bulk of the distribution}\label{subsec:esta}

\begin{lem}\label{lem:perturbation1}
Let $a\in W^{1,\infty}(\mathbb R,\mathbb R_+)$ compactly supported and $\nu>0$. There exists $\bar \alpha>0$ and $C>0$ such that, if $n,m\in\mathcal P_2(\mathbb R)\cap C^0(\mathbb R)$ satisfy
\begin{equation}\label{property:bounded-away-from-0-mult}
n\geq \nu\Gamma_{\nu}(\bar Z-\cdot),\quad m\geq \nu\Gamma_{\nu}(\bar Z-\cdot),
\end{equation}
and
\begin{equation}\label{eq:assumptiontailspert}
\forall x\in(-\infty,-R],\, n'(x)>n(x),\,m'(x)>m(x),\quad \forall x\in[R,\infty),\, n'(x)<-n(x),\,m'(x)<-m(x),
\end{equation}
then, for $\alpha\in(0,\bar \alpha)$,
\begin{equation}\label{est:bulk}
W_2\left(\left(1-\alpha+\alpha\frac {a}{\int_{\mathbb R}a(x)\,n(x)\,dx}\right)n,\left(1-\alpha+\alpha\frac {a}{\int_{\mathbb R}a(x)\,m(x)\,dx}\right)m\right)\leq (1+C\alpha)W_2(n,m).
\end{equation}
\end{lem}

\begin{proof}[Proof of Lemma~\ref{lem:perturbation1}]

\noindent\textbf{Step 1: Preliminaries}

Let $\tilde n,\tilde m\in\mathcal P_2(\mathbb R)$ defined by
\begin{equation}\label{def:tildenm}
\tilde n(x)=\left(1-\alpha+\alpha\frac {a(x)}{I_n}\right)n(x),\quad \tilde m(x)=\left(1-\alpha+\alpha\frac {a(x)}{I_m}\right)m(x),
\end{equation}
We denote by $u,v:(0,1)\to\mathbb R$ the pseudo-inverse of the cumulative distributions of $n$, $m$ and $\tilde u,\,\tilde v$ the ones corresponding to $\tilde n$, $\tilde m$ (see \eqref{def:pseudo-inverse} for the definition of pseudo-inverse of cumulative distributions). We define $I_n,\,I_m$  as in \eqref{def:In}. Let also $R:=\max\{|x|;\,x\in \supp a\}$, and 
\[r:=\frac \nu 2 \int_{|\bar Z|+R}^\infty \Gamma_{\nu}.\]
Then, thanks to \eqref{property:bounded-away-from-0-mult},
\[r<\int_{-\infty}^{-R} n,\quad r<\int_{-\infty}^{-R} m,\]
which implies, thanks to \eqref{def:tildenm} and the fact that $a(x)=0$ for $|x|\geq R$,
\[(1-\alpha)r<\int_{-\infty}^{-R} (1-\alpha)n=\int_{-\infty}^{-R} \tilde n,\quad (1-\alpha)r<\int_{-\infty}^{-R} (1-\alpha)m=\int_{-\infty}^{-R} \tilde m,\]
and similarly,
\[(1-\alpha)r<\int_{R}^{\infty} (1-\alpha) n=\int_{R}^{\infty} \tilde n,\quad (1-\alpha)r<\int_{R}^{\infty} (1-\alpha)m=\int_{R}^{\infty} \tilde m,\]
which implies
\begin{equation}\label{eq:propuR}
\begin{array}{c}
\max\left(u(r),v(r)\right)<-R,\\
 \max\left(\tilde u((1-\alpha)r),\tilde v((1-\alpha)r)\right)<-R,\quad \min\left(\tilde u(1-(1-\alpha)r),\tilde v(1-(1-\alpha)r)\right)>R.

\end{array}
\end{equation}

\medskip

\noindent\textbf{Step 2: Definition and properties of $\phi_u$, $\phi_v$}

We define $\phi_u$ by 
\begin{equation}\label{eq:EDO}
\phi_u((1-\alpha)r)=r,\quad \phi_u'(s)=\frac {1}{1-\alpha+\alpha\frac {a(u(\phi_u(s)))}{I_n}}.
\end{equation}
Thanks to \eqref{property:bounded-away-from-0-mult}, $u$ is Lipschitz continuous (we recall \eqref{prpoertypseudoinverse}). $a$ is also Lipschitz continuous so that the Cauchy-Lipschitz Theorem provides the local existence and uniqueness of $\phi_u$ around $s=(1-\alpha)r$. If $\bar\alpha>0$ is chosen small enough, namely $\bar \alpha\leq \frac 12\frac{\nu\int_{\mathbb R}a(x)\Gamma_{\nu}(\bar Z-x)\,dx}{\|a\|_\infty}$, then $\frac 23\leq \phi_u'(s)\leq 2$ and $\phi_u:[0,1]\to\mathbb R$ is a well defined increasing function. 

\medskip

For $s\in[0,(1-\alpha)r]$, since both $u$ and $\phi_u$ are non-decreasing, \eqref{eq:propuR} implies $u(\phi_u(s))\leq -R$ and thus $a(u(\phi_u(s)))=0$. This implies, firstly, that for $s\in[0,(1-\alpha)r]$, we have $\phi_u'(s)=\frac 1{1-\alpha}$ and then 
\begin{equation}\label{eq:phiproche0}
\forall s\in[0,(1-\alpha)r],\quad \phi_u(s)=\frac s{1-\alpha}. 
\end{equation}
Secondly,   it provides the following relation for $s\in[0,(1-\alpha)r]$: 
\[s=\int_{-\infty}^{\tilde u(s)}\tilde n(x)\,dx=(1-\alpha)\int_{-\infty}^{\tilde u(s)}n(x)\,dx=(1-\alpha) u^{-1}(\tilde u(s)),\]
since \eqref{eq:propuR} implies $\tilde u(s)<R$. These two relations imply $\phi_u=\frac s{1-\alpha}=u^{-1}(\tilde u(s))$, that is $u\circ \phi_u=\tilde u$ on $[0,(1-\alpha)r]$. We notice next that $u\circ\phi_u$ and $\tilde u$ satisfy
\[\frac d{ds}(u\circ\phi_u)(s)=\phi_u'(s)u'(\phi_u(s))=\frac {1}{1-\alpha+\alpha\frac {a(u(\phi_u(s)))}{I_n}}\frac 1{n(u(\phi_u(s)))},
\] 
\[\frac d{ds}(\tilde u)(s)=\frac 1{\tilde n(\tilde u(s))}=\frac {1}{\left(1-\alpha+\alpha\frac {a(\tilde u(s))}{I_n}\right)n(\tilde u(s))},\]
so that, thanks to the uniqueness given by the Cauchy-Lipschitz Theorem,
\begin{equation}\label{eq:cgvphi}
\forall s\in[0,1],\quad u\circ\phi_u(s)=\tilde u(s),
\end{equation}
The argument we have made can be repeated for $\phi_v$ and we then get
\[\forall s\in [0,1],\quad \tilde u(s)=u\circ \phi_u(s)\quad \textrm{and}\quad  \tilde v(s)=v\circ\phi_v(s),\]
with $\phi_v$ defined by $\phi_v((1-\alpha)r)=r$ and $\phi_v'(s)=\frac {1}{1-\alpha+\alpha\frac {a(v(\phi_v(s)))}{I_m}}$.

\medskip

Thanks to \eqref{eq:phiproche0}, $\phi_u(s)=\frac s{1-\alpha}=\phi_v(s)$ for $s\in[0,(1-\alpha r)]$. Moreover,\eqref{eq:propuR} implies
\[R<\tilde u(1-(1-\alpha)r)=u(\phi_u(1-(1-\alpha)r)),\]
and since both $u$ and $\phi_u$ are non decreasing, the fact that $a(x)=0$ for $x\geq R$ implies that $a(u(\phi_u(s)))=a(\tilde u(s))=0$ for $s\in(1-(1-\alpha)r,1)$. \eqref{eq:EDO} then writes $\phi_u'(s)=\frac 1{1-\alpha}$, and thus $\phi_u(s)=C_0+\frac s{1-\alpha}$. To determine $C_0$, we notice that $\tilde u(s)=u(\phi_u(s))\to_{s\to 1} \infty$. Then,  $\phi_u(s)\xrightarrow[s\to 1]{} 1$, which leads to $\phi_u(s)=\frac {s-\alpha}{1-\alpha}$ for $s\geq 1-(1-\alpha)r$. This argument can be repeated for $\phi_v$ and we get
\begin{equation}\label{eq:phisides}
\forall s\in(0,(1-\alpha)r),\quad \phi_u(s)=\frac s{1-\alpha}=\phi_v(s),\quad \forall s\in(1-(1-\alpha)r,1),\quad \phi_u(s)=\frac {s-\alpha}{1-\alpha}=\phi_v(s).
\end{equation}


\medskip

\noindent\textbf{Step 3: Estimate on $W_2(\tilde n,\tilde m)$}

For $s\in(0,1)$, we have
\begin{align*}
\phi_v(s)-\phi_u(s)&=\phi_v(0)-\phi_u(0)+\int_{0}^s\frac 1{1-\alpha+\alpha\frac {a(u(\phi_u(\tau)))}{I_n}}-\frac 1{1-\alpha+\alpha\frac {a(v(\phi_v(\tau)))}{I_m}} \,d\tau\\
&=\int_{0}^s\mathcal O(\alpha)\frac {|I_n-I_m|+\|a'\|_\infty|u(\phi_u(\tau))-v(\phi_v(\tau))|}{\left(1-\alpha+\alpha\frac {a(u(\phi_u(\tau)))}{I_n}\right)\left(1-\alpha+\alpha\frac {a(v(\phi_v(\tau)))}{I_m}\right)}
\,d\tau,
\end{align*}
and then
\begin{align*}
\left|\phi_v(s)-\phi_u(s)\right|&=\mathcal O(\alpha)|I_n-I_m|+\mathcal O(\alpha)\int_{r}^s |u(\phi_u(\tau))-v(\phi_v(\tau))|\,d\tau\\
&=\mathcal O(\alpha)|I_n-I_m|+\mathcal O(\alpha)\int_{r}^s |\tilde u(\tau)-\tilde v(\tau)|\,d\tau.
\end{align*}
This implies
\begin{equation}\label{est:phiuv}
\left|\phi_v(s)-\phi_u(s)\right|\leq \mathcal O(\alpha)\|a'\|_\infty W_1(n,m)+\mathcal O(\alpha)W_1(\tilde n,\tilde m).
\end{equation}
Thanks to this estimate, we get:
\begin{align}
&\int_0^1|\tilde u(s)-\tilde v(s)|^2\,ds=\int_0^1 \left|u(\phi_u(s))-v(\phi_v(s))\right|^2\,ds\nonumber\\
&\quad =\int_0^1 \left|\left(u(\phi_u(s))-v(\phi_u(s))\right)+\left(v(\phi_u(s))-v(\phi_v(s))\right)\right|^2\,ds\nonumber\\
&\quad \leq (1+\alpha)\int_0^1 \left|u(\phi_u(s))-v(\phi_u(s))\right|^2\,ds+\left(1+\frac 1{\alpha}\right)\int_0^1\left|v(\phi_u(s))-v(\phi_v(s))\right|^2\,ds,\label{eq:esttruc}
\end{align}
thanks to a Young inequality. The definition of $\phi_u$ implies $\phi_u'(s)=1+\mathcal O(\alpha)$, and then
\begin{align}
\int_0^1\left|u(\phi_u(s))-v(\phi_u(s))\right|^2\,ds\,& \leq \int_0^1\left|u(\tilde s)-v(\tilde s)\right|^2\,\left|\frac 1{\phi_u'(\phi_u^{-1}(\tilde s))}\right|\,d\tilde s\nonumber\\
&\leq (1+\mathcal O(\alpha))\int_0^1\left|u(s)-v(s)\right|^2\,ds.\label{eq:estfirstterm}
\end{align}
To estimate the last term of \eqref{eq:esttruc}, we start by estimating the integral over $[r,1-r]$. The definition of $\phi_u$ and \eqref{eq:phiproche0} imply $\phi_u(0)=0$ and $\phi'_u(s)=1+\mathcal O(\alpha)$, we have then
\[\phi_u(s)=\phi_u(0)+\int_{0}^s\phi_u'(\tilde s)\,d\tilde s=0+s(1+\mathcal O(\alpha)).\]
It implies, for $\alpha>0$ small enough, $\phi_u([r,1-r])\subset [r/2,1-r/2]$, and, with a similar argument on $\phi_v$, we get
\begin{equation}\label{eq:phiinterval}
\phi_u([(1-\alpha)r,1-(1-\alpha)r])\subset [r/2,1-r/2], \quad \phi_v([(1-\alpha)r,1-(1-\alpha)r])\subset [r/2,1-r/2],
\end{equation}
provided $\alpha>0$ is small enough.

\medskip

Since $n$ is a probability measure, there is $\bar x\in [R,R+1]$ such that $n(\bar x)\leq 1$, and then \eqref{eq:assumptiontailspert} implies, for $x\geq R+1$, $n(x)\leq e^{-(x-\bar x)}\leq e^{-(x-R-1)}$. As a consequence, $\int_x^{+\infty} n(y)\,dy\leq e^{-(x-R-1)}$, and then $\frac r2=\int_{v(1-r/2)}^{+\infty} n(y)\,dy\leq e^{-(v(1-r/2)-R-1)}$, which implies
\[v(1-r/2)\leq -\ln\frac r2+R+1.\]
This argument can be repeated to show:
\[ v([r/2,1-r/2])\subset [\ln r/2-R-1,-\ln r/2+R+1].\]
Moreover, thanks to \eqref{property:bounded-away-from-0-mult}, we have a lower bound 
\[\forall x\in [\ln r/2-R-1,-\ln r/2+R+1],\quad m(x)\geq \min_{[\ln r/2-R-1,-\ln r/2+R+1]}\nu \Gamma_{\nu}(\bar Z-\cdot)>\frac 1C>0,\]
which implies (see \eqref{prpoertypseudoinverse}) an upper bound 
\[\forall s\in [r/2,1-r/2],\quad |v'(s)|=\left|\frac 1{m(v(s))}\right|\leq C.\]
 Using this estimate and \eqref{eq:phiinterval}, we get
\begin{align*}
\int_r^{1-r}\left|v(\phi_u(s))-v(\phi_v(s))\right|^2\,ds&\leq \|v'\|_{L^\infty([r/2,1-r/2])}^{2}\int_r^{1-r}\left|\phi_u(s)-\phi_v(s)\right|^2\,ds\\
&\leq C\int_r^{1-r}\left|\phi_u(s)-\phi_v(s)\right|^2\,ds,
\end{align*}
and then, thanks to \eqref{est:phiuv},
\begin{align*}
\int_r^{1-r}\left|v(\phi_u(s))-v(\phi_v(s))\right|^2\,ds&\leq \mathcal O(\alpha^{2})\left(W_1(n,m)+W_1(\tilde n,\tilde m)\right)^2\\
&\leq \mathcal O(\alpha^{2})\left(W_2^2(n,m)+W_2^2(\tilde n,\tilde m)\right).
\end{align*}
For $s\in(0,(1-\alpha)r)\cup(1-(1-\alpha)r,1]$, $\phi_u(s)=\phi_v(s)$ thanks to  \eqref{eq:phisides}, and then 
\[\int_0^1\left|v(\phi_u(s))-v(\phi_v(s))\right|^2\,ds\leq \mathcal O(\alpha^{2})\left(W_2^2(n,m)+W_2^2(\tilde n,\tilde m)\right).\]
This estimates and \eqref{eq:estfirstterm} can be used to estimate \eqref{eq:esttruc} further and obtain
\begin{align*}
\int_0^1|\tilde u(s)-\tilde v(s)|^2\,ds&\leq  (1+\alpha)(1+\mathcal O(\alpha))\int_0^1 \left|u(s)-v(s)\right|^2\,ds\\
&\quad +\left(1+\frac 1{\alpha}\right)\mathcal O(\alpha^{2})\left(W_2^2(n,m)+W_2^2(\tilde n,\tilde m)\right).
\end{align*}
Then,
\begin{align*}
(1+\mathcal O(\alpha))W_2^2(\tilde n,\tilde m)&\leq  (1+\mathcal O(\alpha))W_2^2(n,m),
\end{align*}
and then $W_2^2(\tilde n,\tilde m)\leq  (1+\mathcal O(\alpha))W_2^2(n,m)$, which implies
\[W_2(\tilde n,\tilde m)\leq \sqrt{1+\mathcal O(\alpha)}W_2(n,m)\leq \big(1+\mathcal O(\alpha)\big)W_2(n,m),\]
provided $\alpha>0$ is small enough.

\end{proof}

\subsection{Effect of a multiplicative factor on the tails of a distribution}\label{subsec:alpha}

\begin{lem}\label{lem:perturbation2}
Let $\nu>0$, $R>0$, $\bar Z$ and $M>0$. There exists $C>0$ such that if $n\in \mathcal P_2(\mathbb R)$ satisfies $\int_{\mathbb R}x^2n(x)\,dx\leq M$,
\begin{equation}\label{eq:assumptiontailsmult}
\forall x\in(-\infty,-R],\quad n'(x)>n(x),\quad \forall x\in[R,\infty),\quad n'(x)<-n(x),
\end{equation}
\begin{equation}\label{property:bounded-away-from-0-multi}
\forall x\in\mathbb R,\quad n\geq \nu\Gamma_{\nu}(\bar Z-\cdot),
\end{equation}
and if $p\in \mathcal P_2(\mathbb R)$ satisfies $\supp p\subset[-R,R]$, then, for any $\alpha,\tilde \alpha\in [0,1/4]$,
\[W_2\big((1-\alpha)n+\alpha p,(1-\tilde \alpha)n+\tilde \alpha p\big)\leq C|\alpha-\tilde \alpha|.\]
\end{lem}

\begin{rem}
Notice that the convexity of $W_2$ provides a related estimate that holds for any $n,p\in \mathcal P_2(\mathbb R)$:
\[W_2\big((1-\alpha)n+\alpha p,(1-\alpha')n+\alpha' p\big)\leq C\sqrt{|\alpha-\alpha'|},\]
but we need a stronger estimate, linear in $\alpha-\alpha'$. This improved estimate derives from the assumptions made on the tails of the distribution $n$.
\end{rem}

\begin{proof}[Proof of Lemma~\ref{lem:perturbation2}]
W.l.o.g., $\tilde \alpha\leq \alpha$. Let $u$ the pseudo-inverse of the cumulative distribution of $n$, and $u_\alpha$ the one of $(1-\alpha)n+\alpha p$. We define $r:=\frac \nu 2\int_{|Z|+R+1}^\infty \Gamma_{\nu}$, which implies in particular that $u_\alpha(s)< -R-1$ for $s\in(0,r)$. Let $s\in(0,r]$ and $h\in(0,1)$. Since $p$ is supported in $[-R,R]$, we have
\begin{align*}
\int_{-\infty}^{u_{\tilde \alpha}(s)+h}(1-\alpha) n(x)+\alpha p(x)\,dx&=\int_{-\infty}^{u_{\tilde \alpha}(s)+h}(1-\alpha) n(x)\,dx\\
&=\int_{-\infty}^{u_{\tilde \alpha}(s)}(1-\alpha) n(x)\,dx+\int_{u_{\tilde \alpha}(s)}^{u_{\tilde \alpha}(s)+h}(1-\alpha) n(x)\,dx,
\end{align*}
and we notice that for $x\in[u_{\tilde \alpha}(s),u_{\tilde \alpha}(s)+h]$, we have $x\leq -R$, so that \eqref{eq:assumptiontailsmult} implies $n(y)\leq n(x)e^{-(y-x)}$ for $y\leq x$, that is
\[\int_{-\infty}^xn(y)\,dy\leq n(x)\int_{-\infty}^xe^{-(y-x)}\,dy=n(x).\]
Then,
\begin{align*}
\int_{-\infty}^{u_{\tilde \alpha}(s)+h}(1-\alpha) n(x)+\alpha p(x)\,dx&\geq \int _{-\infty}^{u_{\tilde \alpha}(s)}(1-\alpha) n(x)\,dx+\int_{u_{\tilde \alpha}(s)}^{u_{\tilde \alpha}(s)+h}(1-\alpha)\left(\int_{-\infty}^x n(y)\,dy\right)\,dx\\
&\geq \int _{-\infty}^{u_{\tilde \alpha}(s)}(1-\alpha) n(x)\,dx+h(1-\alpha)\left(\int_{-\infty}^{u_{\tilde \alpha}(s)} n(y)\,dy\right).
\end{align*}
Therefore,
\begin{align*}
&\int_{-\infty}^{u_{\tilde \alpha}(s)+h}(1-\alpha) n(x)+\alpha p(x)\,dx\geq (1+h)\int _{-\infty}^{u_{\tilde \alpha}(s)}(1-\alpha) n(x)\,dx\\
&\quad \geq (1+h)\frac{1-\alpha}{1-\tilde\alpha}\int _{-\infty}^{u_{\tilde \alpha}(s)}(1-\tilde \alpha)\, n(x)\,dx= (1+h)\,\frac{1-\alpha}{1-\tilde\alpha} \,s,
\end{align*}
and if we choose $h:= \frac{\alpha-\tilde \alpha}{(1-\alpha)}$, we get
\[\int_{-\infty}^{u_{\tilde \alpha}(s)+\frac{\alpha-\tilde \alpha}{(1-\alpha)}}(1-\alpha) n(x)+\alpha p(x)\,dx\geq s,\]
and then $u_\alpha(s)\leq u_{\tilde \alpha}(s)+\frac{\alpha-\tilde \alpha}{1-\alpha}$.

The opposite relationship between $u_\alpha(s)$ and $u_{\tilde\alpha}(s)$ (we still consider $s\in(0,r)$ here) is easier to obtain: since $\tilde\alpha\leq  \alpha$ and $p(x)=0$ for $x\in(-\infty,u_{\tilde \alpha}(s)]$,  
\[\int_{-\infty}^{u_{\tilde\alpha}(s)}(1-\alpha) n(x)+\alpha p(x)\,dx\leq \int_{-\infty}^{u_{\tilde\alpha}(s)}(1-\tilde \alpha) n(x)+\tilde\alpha p(x)\,dx=s,\]
and then $u_{\tilde \alpha}(s)\leq u_\alpha(s)$. This argument for $s\in(0,r)$ can be repeated for $s\in(1-r,1)$ and we get
\begin{equation}\label{eq:estbords}
\forall s\in (0,r)\cup(1-r,1),\quad \left|u_{\tilde \alpha}(s)- u_\alpha(s)\right|\leq \frac{\alpha-\tilde \alpha}{1-\alpha}.
\end{equation}
For $s\in(r,1-r)$, we notice that 
\[\int_{-\infty}^{u_{\alpha}(s)}(1-\alpha) n(x)+\alpha p(x)\,dx=s=\int_{-\infty}^{u_{\tilde \alpha}(s)}(1-\tilde \alpha) n(x)+\tilde \alpha p(x)\,dx
\]
and thus
\begin{equation}\label{eq:difu}
\left|\int_{-\infty}^{u_\alpha(s)} (\tilde \alpha-\alpha)n(x)+(\alpha-\tilde \alpha)p(x)\,dx\right|\geq (1-\tilde \alpha)|u_\alpha(s)-u_{\tilde \alpha}(s)|\inf_{[-\chi,\chi]}n,
\end{equation}
where $\chi=\max\left(|u_\alpha(r)|,|u_{\tilde\alpha}(r)|,|u_\alpha(1-r)|,u_{\tilde\alpha}(1-r)|\right)$. Since $u_\alpha(r)<-R$, we have
\begin{align}
r=\int_{-\infty}^{u_\alpha(r)}(1-\alpha)n+\alpha p=\int_{-\infty}^{u_\alpha(r)}(1-\alpha)n\leq \frac {1-\alpha}{u_\alpha(r)^2}\int_{-\infty}^{u_\alpha(r)}x^2n(x)\,dx\leq \frac M{u_\alpha(r)^2},\label{eq:cheby}
\end{align}
and then $|u_\alpha(r)|\leq\sqrt{\frac Mr}$. We may repeat the argument to show that $\chi\leq \sqrt{\frac Mr}$, and then $\inf_{[-\chi,\chi]}n\geq \nu\Gamma_{\nu}\left(|\bar Z|+\sqrt{M/r}\right)$, so that \eqref{eq:difu} implies
\begin{equation}\label{eq:bornelipualpha}
|u_\alpha(s)-u_{\alpha'}(s)|\leq \frac 3{\nu\Gamma_{\sigma^2/2}\left(|\bar Z|+\sqrt{M/r}\right)}|\alpha-\tilde\alpha|.
\end{equation}
This estimate for $s\in[r,1-r]$ and \eqref{eq:estbords} for $(0,r)\cup(1-r,1)$ can be used to conclude the proof:
\[W_2^2\Big((1-\alpha)n+\alpha p,(1-\tilde \alpha)n+\tilde \alpha p\Big)=\int_0^{1}|u_{\alpha'}(s)-u_\alpha(s)|^2\,ds\leq C|\alpha-\tilde\alpha|^2.\]
\end{proof}

\subsection{Effect of a translation}\label{subsec:wasserstein-translation}

\begin{lem}\label{lem:perturbation3}
Let $\nu>0$, $\bar Z$, $M>0$, $a\in W^{1,\infty}(\mathbb R,\mathbb R_+)$ with compact support, and $\hat Z>0$. There exists $\bar\alpha>0$ and $C>0$ such that if $n\in \mathcal P_2(\mathbb R)$ satisfies $\int_{\mathbb R}x^2n(x)\,dx\leq M$ and
\begin{equation}\label{property:bounded-away-from-0-translation}
\forall x\in\mathbb R,\quad n\geq \nu\Gamma_{\nu}(\bar Z-\cdot),
\end{equation}
then, for any $Z\in[-\hat Z,\hat Z]$ and $\alpha\in(0,\bar \alpha)$,
\begin{equation}\label{eq:estpert3}
W_2\left(\left(1-\alpha+\alpha\frac {a}{\int_{\mathbb R}a(y)n(y)\,dy}\right)n,\left(1-\alpha+\alpha\frac {a(\cdot-Z)}{\int_{\mathbb R}a(y-Z)n(y)\,dy}\right)n\right)\leq C\alpha|Z|.
\end{equation}
\end{lem}

\begin{proof}[Proof of Lemma~\ref{lem:perturbation3}]
Let $n_Z:=\left(1-\alpha+\alpha\frac {a(\cdot-Z)}{\int_{\mathbb R}a(y-Z)n(y)\,dy}\right)n$ and $u_Z$ the corresponding pseudo-inverse. Let also $R:=\max\{|x|;\,x\in\supp a\}$ and $r:=\frac\nu 2\int_{|\bar Z|+R+\hat Z}^\infty\Gamma_{\nu}$. Then $u_Z(s)=u_0(s)$ for $s\in(0,r]\cup[1-r,1)$, while for $s\in(r,1-r)$, we will use an argument similar to what was used in the proof of Lemma~\ref{lem:perturbation2}: we notice that $\int_{-\infty}^{u_0(s)}n_0=s=\int_{-\infty}^{u_Z(s)}n_Z$, and thus
\begin{align*}
\alpha\int_{-\infty}^{u_0(s)}\left(\frac {a(x)}{\int_{\mathbb R}a(y)n(y)\,dy}-\frac {a(x-Z)}{\int_{\mathbb R}a(y-Z)n(y)\,dy}\right)n(x)\,dx=\int_{u_0(s)}^{u_Z(s)}n_Z,
\end{align*}
which implies
\begin{align*}
C\alpha\|a'\|_\infty |Z|\geq (1-\alpha)|u_0(s)-u_Z(s)|\min_{[-\chi,\chi]}n,
\end{align*}
where $\chi=\max\left(|u_0(r)|,|u_Z(r)|,|u_0(1-r)|,|u_Z(1-r)|\right)$ and $s\in(0,r]\cup[1-r,1)$. We can use the bound $M$ on the second moment of $n$ to control $\chi$, just as it was done in \eqref{eq:cheby} and \eqref{eq:bornelipualpha}. This argument and assumption \eqref{property:bounded-away-from-0-translation} show that for some constant $C>0$, $|u_0(s)-u_Z(s)|\leq C\alpha |Z|$, which is enough to conclude the proof:
\[W_2^2\left(n_0,n_Z\right)=\int_0^{1}|u_0(s)-u_Z(s)|^2\,ds=\int_r^{1-r}|u_0(s)-u_Z(s)|^2\,ds\leq C\alpha^2 |Z|^2.\]

\end{proof}

\subsection{Remarks on the Wasserstein estimates of Section~\ref{sec:wassest}}
\label{subsec:lowerboundremarks}

In this section, we discuss the necessity of a lower bound assumption for the estimates developed in Lemma~\ref{lem:perturbation3} and Lemma~\ref{lem:perturbation1}. Note that solutions of \eqref{eq:model} do satisfy a lower bound assumption thanks to  Lemma~\ref{lem:lowerbound}.

\medskip

\noindent\textbf{Necessity of assumption \eqref{property:bounded-away-from-0-translation} for Lemma~\ref{lem:perturbation3}}

 Let $n:=\frac 12\delta_{-1}+\frac 12\delta_1$, and $a(x)=1+\frac x 2$ on $[-1,1]$. Then, for $Z\in\mathbb R$,
\begin{align*}
&\left(1-\alpha+\alpha\frac {a(x-Z)}{\int_{\mathbb R}a(y-Z)n(y)\,dy}\right)n(x)
=\frac 12\left(1-\frac\alpha 2\frac {1}{1-Z/2}\right)\delta_{-1}+\frac 12\left(1+\frac \alpha 2\frac {1}{1-Z/2}\right)\delta_1,
\end{align*}
and thus
\begin{align*}
&W_2\left(\left(1-\alpha+\alpha\frac {a}{\int_{\mathbb R}a(y)n(y)\,dy}\right)n,\left(1-\alpha+\alpha\frac {a(\cdot-Z)}{\int_{\mathbb R}a(y-Z)n(y)\,dy}\right)n\right)\\
&\quad =\frac {\sqrt \alpha}{ 2}\sqrt{\left|1-\frac {1}{1-Z/2}\right|}\sim \frac {\sqrt \alpha}{ 2\sqrt 2}\sqrt{\left|Z\right|},
\end{align*}
so that $n$ does not satisfy \eqref{eq:estpert3}. This shows the lower bound assumption is necessary for the estimate provided by Lemma~\ref{lem:perturbation3} to hold. Note that this example is closely related to Example~2.4 provided in \cite{calvez2019asymptotic}.

\medskip

\noindent\textbf{Necessity of assumption \eqref{property:bounded-away-from-0-mult} for Lemma~\ref{lem:perturbation1}}

We still consider $a(x)=1+\frac x 2$ for $x\in[-1,1]$, and define, for $\rho\in[0,1]$, 
\[n_\rho:=\frac 12\delta_{-1}+\frac\rho 4\delta_0+\frac 141_{[0,1]}+\frac{1-\rho}4\delta_1.\]
Then the pseudo-inverse distribution of $n_\rho$ (see \eqref{def:pseudo-inverse}), that we denote $u_\rho$, is 
\[u_\rho(z)=\left\{\begin{array}{cl}
-1&\textrm{ if }z\in[0,1/2),\\
0&\textrm{ if }z\in[1/2,1/2+\rho/4),\\
4(z-(1/2+\rho/4))&\textrm{ if }z\in[1/2+\rho/4,3/4+\rho/4),\\
1&\textrm{ if }z\in[3/4+\rho/4,1],
\end{array}\right.\]
so that 
\begin{equation}\label{eq:W2tildeuex1}
W_2(n_\rho,n_{\rho'})=\left(\int_0^1 |u_ \rho(z)-u_{\rho'}(z)|^2\,dz\right)^{1/2}\leq \left(\int_{1/2}^1 |\rho-\rho'|^2\,dz\right)^{1/2}\leq \frac{|\rho-\rho'|}{\sqrt 2}.
\end{equation} 
We compute $\int_{\mathbb R}a(x)n_\rho(x)\,dx=\frac {15}{16}-\frac \rho 8$, and 
notice that 
\[\left(1-\alpha+\alpha\frac {a(x)}{\int_{\mathbb R}a(y)n_\rho(y)\,dy}\right)n_{\rho}(x)=\left(1-\alpha+\frac{8\alpha}{15-\rho}\right)\frac{\delta_{-1}}2,\textrm{ for }x\in(-\infty,0).\]
The pseudo-inverse $\tilde u_\rho$ of this distribution then satisfies 
$\tilde u_\rho(z)=-1$ for $z< \frac 12\left(1-\alpha+\frac{8\alpha}{15-2\rho}\right)$, while $\tilde u_\rho(z)\geq 0$ for $z> \frac 12\left(1-\alpha+\frac{8\alpha}{15-2\rho}\right)$. It follows that
\begin{align}
&W_2\left(\left(1-\alpha+\alpha\frac {a(x)}{\int_{\mathbb R}a(y)n_\rho(y)\,dy}\right)n_{\rho},\left(1-\alpha+\alpha\frac {a(x)}{\int_{\mathbb R}a(y)n_{\rho'}(y)\,dy}\right)n_{\rho'}\right)\nonumber\\
&\quad \geq \sqrt{\left|\frac 12\left(1-\alpha+\frac{8\alpha}{15-2\rho}\right)-\frac 12\left(1-\alpha+\frac{8\alpha}{15-2\rho'}\right)\right|}=2\sqrt{\frac{2\alpha}{(15-2\rho)(15-2\rho')}}\sqrt{|\rho-\rho'|}\\
&\quad \geq \frac {C\sqrt\alpha}{\sqrt{|\rho-\rho'|}}W_2(n_\rho,n_{\rho'}),\label{eq:W2tildeuex2}
\end{align}
for $\rho,\rho'\sim \frac 12$, thanks to \eqref{eq:W2tildeuex1}. It is then impossible to have a constant $C>0$ such that \eqref{est:bulk} holds for any $\rho,\rho'\sim \frac 12$, even if $\alpha>0$ is small. We have shown that a lower bound assumption on $n$ and $m$ is necessary for the conclusion of Lemma~\ref{lem:perturbation1} to hold.

\medskip

\noindent\textbf{Discussion on the effect of lower bounds on distributions}

Both examples above involve sums of Dirac masses. These Dirac masses allowed us to compute explicitly distances between distributions, and they do not satisfy the assumptions \eqref{property:bounded-away-from-0-translation}, \eqref{property:bounded-away-from-0-mult} discussed because their support is disconnected. Note that similar counter-examples could be built with functions, rather than measures singular with respect to the Lebesgue measure. To see the effect of a connected support, we can consider the following families of measures:
\[n_\rho=\rho\delta_{-1}+(1-\rho)\delta_1,\quad m_\rho=\frac{9\rho}{10}\delta_{-1}+\frac{9(1-\rho)}{10}\delta_1+\frac 1{20}1_{[-1,1]},\]
for $\rho\in(0,1)$. Their pseudo-inverse distributions are, respectively,
\[u_\rho(z)=\left\{\begin{array}{cl}
-1&\textrm{ if }z\in[0,\rho),\\
1&\textrm{ if }z\in[\rho,1],
\end{array}\right.,
\quad
v_\rho(z)=\left\{\begin{array}{cl}
-1&\textrm{ if }z\in[0,9\rho/10),\\
20(z-1/2)&\textrm{ if }z\in[9\rho/10,9(1-\rho)/10]\\
1&\textrm{ if }z\in(9(1-\rho)/10,1],
\end{array}\right.\]
Then, for $\rho\geq \frac 12$ close to $\frac 12$, we have 
\begin{align*}
W_2(n_\rho,n_{1/2})&=\left(\int_0^1|u_\rho(z)-u_{1/2}(z)|^2\,dz\right)^2=\left(\int_{1/2}^\rho|2|^2\,dz\right)^{\frac 12}=4\sqrt\rho.
\end{align*}
and, for $\rho\in (1/2,1/10)$,
\begin{align*}
W_2(m_\rho,m_{1/2})&=\left(\int_0^1|v_\rho(z)-v_{1/2}(z)|^2\,dz\right)^2\\
&=\left(\int_{1/2}^{\rho+1/10}\left|\min\left(20\left(z-\frac 12\right),20\left(\rho-\frac 12\right),20\left(\frac 12+\frac 1{10}-z\right)\right)\right|^2\,dz\right)^{\frac 12}\\
&\leq \left(\int_{1/2}^{\rho+1/10}\left|\frac \rho{20}\right|^2\,dz\right)^{\frac 12}=\frac \rho{20}\sqrt{\frac 1{10}+\left(\rho-\frac 12\right)},
\end{align*}
and thus $W_2(m_\rho,m_{1/2})\sim \frac \rho{20\sqrt {10}}$ when $\rho\sim (1/2)^+$. We see that the $W_2$ distance between $n_{1/2}$ and $n_{\rho}$ is of order $\sqrt\rho$, while the $W_2$ distance between $m_{1/2}$ and $m_{\rho}$ is of order $\rho$. This suggests a benefic effect of having a lower bound on distributions for the regularity of optimal transfer plans for $|\cdot|^2$, which we rely on in Section~\ref{sec:wassest}, inspired by methods developed in \cite{carrillo2003kinetic,carrillo2006contractions, cattiaux2008probabilistic,bolley2012convergence,bolley2013uniform}.
\section{Proof of Theorem~\ref{Thm:main}}
\label{sec:asymptotic}

This section is devoted to the proof of the main result of this manuscript, Theorem~\ref{Thm:main}. We consider two initial conditions $n^0$ and $m^0$ satisfying Assumption~\ref{ass1}. Let $n$, $m$ the corresponding solutions of \eqref{eq:model}.

\subsection{Contraction estimate for the microscopic distribution}\label{subsec:estmicro}

Let $\varphi_n,\,\varphi_m:\mathbb R_+\to\mathbb R_+$ the solutions of $\varphi_n(0)=\varphi_m(0)=0$, and
\begin{equation}\label{def:mathcalNM}
\frac {d\varphi_n}{dt}(t)=\frac 1 {(1+\alpha I_n(\varphi_n(t)))^2},\quad \frac {d\varphi_m}{dt}(t)=\frac 1 {(1+\alpha I_m(\varphi_m(t)))^2}.
\end{equation}
Notice that $0\leq I_n,\,I_m\leq \|a\|_\infty$, so that there exists $C>0$ such that
\begin{equation}\label{eq:ineqvarphi}
\left|\varphi_n(t)-t\right|\leq C\alpha t,\quad \left|\varphi_m(t)-t\right|\leq C\alpha t.
\end{equation} 
Notice also that $\mathcal N(t,x):=n(\varphi_n(t),x)$ satisfies 
\[\partial_t \mathcal N(t,x)=T\left(\frac {1+\alpha a(y_*)}{1+\alpha I_n(\varphi_n(t))}\mathcal N(t,y_*)\right)-\mathcal N(t,x),\]
and for $0\leq t_0\leq t$,
\begin{equation}\label{eq:mathcalNinteg}
\mathcal N(t,x)=\mathcal N(t_0,x)e^{-(t-t_0)}+\int_{t_0}^t T\left(\frac {1+\alpha a(y_*)}{1+\alpha I_n(\varphi_n(s))}\mathcal N(s,y_*)\right)e^{-(t-s)}\,ds
\end{equation}
and similar properties are satisfied by $\mathcal M(t,x):=m(\varphi_m(t),x)$. We introduce the notation 
\begin{equation}\label{def:beta}
\beta_m(t):=\frac{\alpha I_m\circ\varphi_m(t)}{1+\alpha I_m\circ\varphi_m(t)},\quad \beta_n(t):=\frac{\alpha I_n\circ\varphi_n(t)}{1+\alpha I_n\circ\varphi_n(t)}.
\end{equation}
The notation $\beta_n$ (and similarly for $\beta_m$) is convenient to write the following quantity as a convex combination of two probability measures:
\[\frac{1+\alpha a(x)}{1+\alpha I_n\circ\varphi_n(t)}\mathcal N(t,x)=\left(1-\beta_n(t)\right)\mathcal N(t,x)+\beta_n(t)\frac{a(x)}{I_n\circ\varphi_n(t)}\mathcal N(t,x).\]
We recall the definition \eqref{eq:defw} of $w$. Thanks to \eqref{eq:W2convex}, \eqref{eq:W2convexcontinu}, we have, for  $t\geq 0$:
\begin{align}
&w(\mathcal N(t,\cdot),\mathcal M(t,\cdot))^2=W_2\left[\mathcal N(t,\cdot),\mathcal M(t,\cdot-(Z_m\circ \varphi_m(t)-Z_n\circ\varphi_n(t)))\right]^2\nonumber\\
&\leq W_2\left[\mathcal N(0,\cdot),\mathcal M(0,\cdot-(Z_m\circ \varphi_m(t)-Z_n\circ\varphi_n(t)))\right]^2e^{-t}\nonumber\\
&\quad +\int_{0}^t e^{-(t-s)}W_2\bigg[T\left(\left(1-\beta_n(s) +\beta_n(s) \frac {a(\cdot)}{I_n\circ\varphi_n(s)}\right)\mathcal N(s,\cdot)\right),\nonumber\\
&\qquad T\bigg(\left(1-\beta_m(s) +\beta_m(s) \frac {a(\cdot-(Z_m\circ \varphi_m(s)-Z_n\circ\varphi_n(s)))}{I_m\circ \varphi_m(s)}\right)\nonumber\\
&\phantom{\qquad T\bigg(}\mathcal M(s,\cdot-(Z_m\circ \varphi_m(s)-Z_n\circ\varphi_n(s)))\bigg)\bigg]^2\,ds,\label{est:W2nm}
\end{align}
where we have also used the translation invariance of $T$, see \eqref{eq:translationinvar}. To estimate the last term of \eqref{est:W2nm}, we notice that
\begin{align}
&J_1:=W_2\bigg[T\left(\left(1-\beta_n +\beta_n(s) \frac {a(\cdot)}{I_n\circ\varphi_n(s)}\right)\mathcal N(s,\cdot)\right),\nonumber\\
&\phantom{J_1:=W_2\bigg[} T\bigg(\left(1-\beta_m(s) +\beta_m(s) \frac {a(\cdot-(Z_m\circ \varphi_m(s)-Z_n\circ\varphi_n(s)))}{I_m\circ \varphi_m(s)}\right)\nonumber\\
&\phantom{J_1:=W_2\bigg[}\phantom{T\bigg(}\mathcal M(s,\cdot-(Z_m\circ \varphi_m(s)-Z_n\circ\varphi_n(s)))\bigg)\bigg]^2\nonumber\\
&\leq \Bigg\{\left|(Z_m\circ \varphi_m(s)-Z_n\circ\varphi_n(s))-(\tilde Z_m\circ \varphi_m(s)-\tilde Z_n\circ\varphi_n(s))\right|\nonumber\\
&\quad +W_2\Bigg[T\left(\left(1-\beta_n(s) +\beta_n(s) \frac {a(\cdot)}{I_n\circ\varphi_n(s)}\right)\mathcal N(s,\cdot)\right),\nonumber\\
&\phantom{\quad +W_2\Bigg[} T\bigg(\left(1-\beta_m(s) +\beta_m(s) \frac {a(\cdot-(\tilde Z_m\circ \varphi_m(s)-\tilde Z_n\circ\varphi_n(s)))}{I_m\circ \varphi_m(s)}\right)\nonumber\\
&\phantom{\quad +W_2\Bigg[ T\bigg(}\mathcal M(s,\cdot-(\tilde Z_m\circ \varphi_m(s)-\tilde Z_n\circ\varphi_n(s)))\bigg)\Bigg]\Bigg\}^2,\label{est:beforecontract}
\end{align}
where $\tilde Z_m\circ \varphi_m(s)$ and $\tilde Z_n\circ \varphi_n(s)$, are defined as follows :
\begin{equation}\label{def:ZZtilde}
\begin{array}{c}
\displaystyle\tilde Z_n\circ\varphi_n(s):=\int_{\R}x\left(1-\beta_n(s) +\beta_n(s) \frac {a(x)}{I_n\circ\varphi_n(s)}\right)\mathcal N(s,x)\,dx,\vspace{0.25cm}\\
\displaystyle\tilde Z_m\circ\varphi_m(s):=\int_{\R} x\left(1-\beta_m(s) +\beta_m(s) \frac {a(x)}{I_m\circ \varphi_m(s)}\right)\mathcal M(s,x)\,dx.
\end{array}
\end{equation}
This definition implies that the two distributions on which the operator $T$ acts in \eqref{est:beforecontract} have the same center of mass. We may then apply  the contraction estimate \eqref{eq:contraction} satisfied by the operator $T$:
\begin{align}
&J_1\leq \Bigg\{\left|(Z_m\circ \varphi_m(s)-Z_n\circ\varphi_n(s))-(\tilde Z_m\circ \varphi_m(s)-\tilde Z_n\circ\varphi_n(s))\right|\nonumber\\
&\phantom{J_1\leq \Bigg\{} +\frac 1{\sqrt 2}W_2\Bigg[\left(1-\beta_n(s) +\beta_n(s) \frac {a(\cdot)}{I_n\circ\varphi_n(s)}\right)\mathcal N(s,\cdot),\nonumber\\
&\phantom{J_1\leq \Bigg\{ +\frac 1{\sqrt 2}W_2\Bigg[} \left(1-\beta_m(s) +\beta_m(s) \frac {a(\cdot-(\tilde Z_m\circ \varphi_m(s)-\tilde Z_n\circ\varphi_n(s)))}{I_m\circ \varphi_m(s)}\right)\nonumber\\
&\phantom{J_1\leq \Bigg\{ +\frac 1{\sqrt 2}W_2\Bigg[}\quad \mathcal M(s,\cdot-(\tilde Z_m\circ \varphi_m(s)-\tilde Z_n\circ\varphi_n(s)))\Bigg]\Bigg\}^2\nonumber\\
&\quad \leq C\left|(Z_m\circ \varphi_m(s)-Z_n\circ\varphi_n(s))-(\tilde Z_m\circ \varphi_m(s)-\tilde Z_n\circ\varphi_n(s))\right|^2+\frac {1.01}{ 2} J_2,\label{eqJ1}
\end{align}
where we have used a Young inequality and where 
\begin{align*}
&J_2=W_2^2\Bigg[\left(1-\beta_n(s) +\beta_n(s) \frac {a(\cdot)}{I_n\circ\varphi_n(s)}\right)\mathcal N(s,\cdot),\nonumber\\
&\phantom{J_2=W_2^2\Bigg[} \left(1-\beta_m(s) +\beta_m(s) \frac {a(\cdot-(\tilde Z_m\circ \varphi_m(s)-\tilde Z_n\circ\varphi_n(s)))}{I_m\circ \varphi_m(s)}\right)\nonumber\\
&\phantom{J_2=W_2^2\Bigg[}\quad \mathcal M(s,\cdot-(\tilde Z_m\circ \varphi_m(s)-\tilde Z_n\circ\varphi_n(s)))\Bigg].
\end{align*}
We can decompose $J_2$ as follows
\begin{align}
&J_2\leq \Bigg\{W_2\Bigg[\left(1-\beta_n(s) +\beta_n(s) \frac {a\left(\cdot\right)}{I_n\circ \varphi_n(s)}\right)\mathcal N\left(s,\cdot\right),\nonumber\\
&\phantom{J_2\leq \Bigg\{W_2\Bigg[}\Bigg(1-\beta_n(s) +\beta_n(s) \frac {a(\cdot-(\tilde Z_m\circ \varphi_m(s)-\tilde Z_n\circ \varphi_n(s)))}{\int_{\mathbb R}a(y-(\tilde Z_m\circ \varphi_m(s)-\tilde Z_n\circ \varphi_n(s)))\mathcal N(s,y)\,dy}\Bigg)\mathcal N\left(s,\cdot\right)\Bigg]\nonumber\\
&\phantom{J_2\leq \Bigg\{} +W_2\Bigg[\Bigg(1-\beta_n(s) +\beta_n(s) \frac {a(\cdot-(\tilde Z_m\circ \varphi_m(s)-\tilde Z_n\circ \varphi_n(s)))}{\int_{\mathbb R}a(y-(\tilde Z_m\circ \varphi_m(s)-\tilde Z_n\circ \varphi_n(s)))\mathcal N(s,y)\,dy}\Bigg)\mathcal N\left(s,\cdot\right)\nonumber\\
&\phantom{J_2\leq \Bigg\{+W_2\Bigg[} \left(1-\beta_n(s) +\beta_n(s) \frac {a(\cdot-(\tilde Z_m\circ \varphi_m(s)-\tilde Z_n\circ \varphi_n(s)))}{I_m\circ \varphi_m(s)}\right)\nonumber\\
&\phantom{J_2\leq \Bigg\{+W_2\Bigg[}\quad \mathcal M(s,\cdot-(\tilde Z_m\circ \varphi_m(s)-\tilde Z_n\circ \varphi_n(s)))\Bigg]\nonumber\\
&\phantom{J_2\leq \Bigg\{}+ W_2\Bigg[\left(1-\beta_n(s) +\beta_n(s) \frac {a(\cdot-(\tilde Z_m\circ \varphi_m(s)-\tilde Z_n\circ \varphi_n(s)))}{I_m\circ \varphi_m(s)}\right)\nonumber\\
&\phantom{J_2\leq \Bigg\{W_2\Bigg[}\quad \mathcal M(s,\cdot-(\tilde Z_m\circ \varphi_m(s)-\tilde Z_n\circ \varphi_n(s))),\nonumber\\
&\phantom{J_2\leq \Bigg\{+ W_2\Bigg[} \left(1-\beta_m(s) +\beta_m(s) \frac {a(\cdot-(\tilde Z_m\circ \varphi_m(s)-\tilde Z_n\circ \varphi_n(s)))}{I_m\circ \varphi_m(s)}\right)\nonumber\\
&\phantom{J_2\leq \Bigg\{+ W_2\Bigg[} \quad\mathcal M(s,\cdot-(\tilde Z_m\circ \varphi_m(s)-\tilde Z_n\circ \varphi_n(s)))\Bigg]\Bigg\}^2.\label{eq:decompositionW2}
\end{align}
The definition \eqref{def:ZZtilde} of $\tilde Z_n\circ\varphi_n(s)$ implies, thanks to \eqref{eq:Ztpetitbis} and the fact that $a$ is compactly supported,
\begin{align}
|\tilde Z_n\circ\varphi_n(s)|&\leq (1-\beta_n(s))\left|\int_{\R}x n(\varphi_n(s),x)\,dx\right|+\beta_n(s)\left|\int_{\R}x\frac {a(x)}{I_n\circ\varphi_n(s)}\mathcal N(s,x)\,dx\right|\nonumber\\
&\leq (1-\beta_n(s))\max\{|Z^0|,\max\{|x|;\,x\in\textrm{supp }a\}\}+\beta_n(s)\max\{|x|;\,x\in\textrm{supp }a\}\nonumber\\
&\leq R,\label{est:tildez}
\end{align}
where $R>0$ is the constant appearing in Assumption~\ref{ass2}. Then,
\[|\tilde Z_m\circ \varphi_m(s)-\tilde Z_n\circ \varphi_n(s)|\leq 2R.\]
Moreover, we can use  Lemma~\ref{lem:lowerbound} to show that $n(t,\cdot)$ and $m(t,\cdot)$ satisfy the uniform in time lower bound \eqref{property:bounded-away-from-0}, and since $\mathcal N(s,\cdot)=n(\varphi_n(s),\cdot)$, $\mathcal M(s,\cdot)=m(\varphi_m(s),\cdot)$, we also have  
\begin{equation}\label{est:low}
\forall (t,x)\in[0,+\infty)\times\mathbb R,\quad \mathcal N(t,x) \geq \frac 1C\Gamma_{A/R\wedge\sigma^2/2}(\bar Z-x),\quad \mathcal M(t,x)\geq \frac 1C\Gamma_{A/R\wedge\sigma^2/2}(\bar Z-x),
\end{equation}
Moreover, thanks to  Lemma~\ref{lem:tailprecise}, $n(t,\cdot)$, $m(t,\cdot)$, and therefore $\mathcal N(t,\cdot)$, $\mathcal M(t,\cdot)$ satisfy the following tail estimate, for any $t\in[0,+\infty)$,
\begin{equation}\label{est:upp}
\begin{array}{c}
\forall x\in(-\infty,-R],\quad\partial_x \mathcal N(t,x)>\mathcal N(t,x),\,\partial_x \mathcal M(t,x)>\mathcal M(t,x),\vspace{0.25cm}\\
 \forall x\in[R,\infty),\quad \partial_x \mathcal N(t,x)<-\mathcal N(t,x),\,\partial_x \mathcal M(t,x)<-\mathcal M(t,x).
 \end{array}
\end{equation}

\medskip

Thanks to the lower bound \eqref{est:low}, we may use Lemma~\ref{lem:perturbation3} with $\hat Z:=2R$ (note that $\mathcal N(t,\cdot)$ and $\mathcal M(t,\cdot)$ have a bounded second moment thanks to \eqref{eq:moment2}), to show that the first term on the right hand side of \eqref{eq:decompositionW2} satisfies:
\begin{align*}
&W_2\Bigg[\left(1-\beta_n(s) +\beta_n(s) \frac {a\left(\cdot\right)}{I_n\circ \varphi_n(s)}\right)\mathcal N\left(s,\cdot\right),\nonumber\\
&\qquad\Bigg(1-\beta_n(s) +\beta_n(s) \frac {a(\cdot-(\tilde Z_m\circ \varphi_m(s)-\tilde Z_n\circ \varphi_n(s)))}{\int_{\mathbb R}a(y-(\tilde Z_m\circ \varphi_m(s)-\tilde Z_n\circ \varphi_n(s)))\mathcal N(s,y)\,dy}\Bigg)\mathcal N\left(s,\cdot\right)\Bigg]\\
&\quad \leq C|\beta_n(s)|\, |\tilde Z_m\circ \varphi_m(s)-\tilde Z_n\circ \varphi_n(s)|\leq C\alpha |\tilde Z_m\circ \varphi_m(s)-\tilde Z_n\circ \varphi_n(s)|,
\end{align*}
where the last inequality results from the definition \eqref{def:beta} of $\beta_n$: $|\beta_n(s)|\leq \alpha \|a\|_{\infty}$. Thanks to the lower bound \eqref{est:low} and the tail estimate \eqref{est:upp}, we may use Lemma~\ref{lem:perturbation1} to estimate the second term in \eqref{eq:decompositionW2}, and Lemma~\ref{lem:perturbation2} to control the third term of \eqref{eq:decompositionW2}.  Then, for $t\geq 0$,
\begin{align}
J_2 &\leq \Big\{\mathcal O\Big(\alpha |\tilde Z_m\circ \varphi_m(s)-\tilde Z_n\circ \varphi_n(s)|\Big)\nonumber\\
&+\left(1+\mathcal O(\alpha)\right)W_2\Big(\mathcal N(s,\cdot),\mathcal M(s,\cdot-(\tilde Z_m\circ \varphi_m(s)-\tilde Z_n\circ \varphi_n(s)))\Big)+\mathcal O\Big(|\beta_m(s)-\beta_n(s)|\Big)\Big\}^2.\label{est:J2deux}
\end{align}
Thanks to the lower bound \eqref{est:low}, there is $\kappa>0$ such that, for $s\geq 0$,
\begin{equation}\label{eq:assumptionslemmaapp}
I_n\circ\varphi_n(s)=\int_{\R} a(x)\mathcal N(s,x)\,dx\geq \kappa,\quad I_m\circ\varphi_m(s)=\int_{\R} a(x)\mathcal M(s,x)\,dx\geq \kappa,
\end{equation}
so that $\mathcal N(s,\cdot)=n(\varphi_n(s),\cdot)$, $\mathcal M(s,\cdot)=m(\varphi_m(s),\cdot)$ satisfy the assumption of Lemma~\ref{lem:estmacro}, and \eqref{est:ImoinsI} implies
\begin{equation}\label{eq:ImI}
\left|I_n\circ\varphi_n(s)- I_m\circ\varphi_m(s)\right|\leq C\big(|Z_n\circ\varphi_n(s)-Z_m\circ\varphi_m(s)|+w(\mathcal N(s,\cdot),\mathcal M(s,\cdot))\big).
\end{equation}
This estimate and the definition of $\beta_n$, $\beta_m$ can be used to estimate the last term in \eqref{est:J2deux}:
\begin{align*}
\left|\beta_m(s)-\beta_n(s)\right|&\leq \mathcal O\big(\alpha \left|I_m\circ \varphi_m(s)-I_n\circ \varphi_n(s)\right|)\\
& \leq  \mathcal O(\alpha)\Big( |Z_n\circ \varphi_n(s)-Z_m\circ \varphi_m(s)|+w(\mathcal N(s,\cdot),\mathcal M(s,\cdot))\Big).
\end{align*}
To estimate the second term on the right hand side of \eqref{est:J2deux}, we recall that the distance $w(\cdot,\cdot)$ satisfies  \eqref{eq:sigmaexplicit}, and then
\begin{align*}
&W_2\Big(\mathcal N(s,\cdot),\mathcal M\big(s,\cdot-(\tilde Z_m\circ \varphi_m(s)-\tilde Z_n\circ \varphi_n(s))\big)\Big)\\
&\quad =W_2\Big(\mathcal N(s,\cdot),\mathcal M\big(s,\cdot-(Z_m\circ \varphi_m(s)- Z_n\circ \varphi_n(s))\big)\Big)\\
&\qquad +\left|(Z_m\circ \varphi_m(s)-Z_n\circ \varphi_n(s))-(\tilde Z_m\circ \varphi_m(s)-\tilde Z_n\circ \varphi_n(s))\right|\\
&\quad =w\big(\mathcal N(s,\cdot),\mathcal M(s,\cdot)\big) +\left|(Z_m\circ \varphi_m(s)-Z_n\circ \varphi_n(s))-(\tilde Z_m\circ \varphi_m(s)-\tilde Z_n\circ \varphi_n(s))\right|.
\end{align*}
Thanks to \eqref{eq:assumptionslemmaapp}, Lemma~\ref{lem:estmacro} holds, and, in particular, \eqref{est:tildeZ} implies
\begin{align}
&\left|\left(Z_m\circ \varphi_m(s)-Z_n\circ \varphi_n(s)\right)-(\tilde Z_m\circ \varphi_m(s)-\tilde Z_n\circ \varphi_n(s))\right|\nonumber\\
&\quad \leq C\alpha \left(\left|Z_m\circ \varphi_m(s)-Z_n\circ \varphi_n(s)\right|+w\big(\mathcal N(s,\cdot),\mathcal M(s,\cdot)\big)\right),\label{est:ZZApp}
\end{align}
and then 
\begin{align*}
&W_2\Big(\mathcal N(s,\cdot),\mathcal M\big(s,\cdot-(\tilde Z_m\circ \varphi_m(s)-\tilde Z_n\circ \varphi_n(s))\big)\Big)\\
&\quad \leq (1+C\alpha)w\big(\mathcal N(s,\cdot),\mathcal M(s,\cdot)\big) +C\alpha \left|Z_m\circ \varphi_m(s)-Z_n\circ \varphi_n(s)\right|.
\end{align*} 
These estimates applied to \eqref{est:J2deux} imply
\begin{align*}
J_2 &\leq \big(C\alpha |Z_m\circ \varphi_m(s)-Z_n\circ \varphi_n(s)|+\left(1+C\alpha\right)w(\mathcal N(s,\cdot),\mathcal M(s,\cdot))\big)^2
\end{align*}
This estimate and \eqref{est:ZZApp} can be used to estimate $J_1$: estimate \eqref{eqJ1} becomes
\begin{align*}
J_1 &\leq C\alpha^2 \left(\left|Z_m\circ \varphi_m(s)-Z_n\circ \varphi_n(s)\right|+w\big(\mathcal N(s,\cdot),\mathcal M(s,\cdot)\big)\right)^2\\
&\quad +\frac{1.01}2\big(C\alpha |Z_m\circ \varphi_m(s)-Z_n\circ \varphi_n(s)|+\left(1+C\alpha\right)w(\mathcal N(s,\cdot),\mathcal M(s,\cdot))\big)^2\\
& \leq C\alpha^2 \big|Z_m\circ \varphi_m(s)-Z_n\circ \varphi_n(s)\big|^2+\left(\frac {1.03}{2}+C\alpha\right)w(\mathcal N(s,\cdot),\mathcal M(s,\cdot))^2,
\end{align*}
thanks to a Young inequality. We can now use this estimate in \eqref{est:W2nm} to control $w(\mathcal N(t,\cdot),\mathcal M(t,\cdot))^2$ (we recall the definition \eqref{est:beforecontract} of $J_1$):
 \begin{align}
&w(\mathcal N(t,\cdot),\mathcal M(t,\cdot))^2\leq W_2\left[\mathcal N(0,\cdot),\mathcal M(0,\cdot-(Z_m\circ \varphi_m(t)-Z_n\circ\varphi_n(t)))\right]^2e^{-t} +\int_{0}^t e^{-(t-s)}J_1\,ds\nonumber\\
&\quad \leq Ce^{-t}+\int_0^t e^{-(t-s)}\left(C\alpha^2\big|Z_m\circ \varphi_m(s)-Z_n\circ \varphi_n(s)\big|^2+\left(\frac {1.03}{2}+C\alpha\right)w(\mathcal N(s,\cdot),\mathcal M(s,\cdot))^2\right)\,ds\nonumber\\
&\quad \leq Ce^{-t}+\int_0^t e^{-(t-s)}\left(C\alpha^2\big|Z_m\circ \varphi_m(s)-Z_n\circ \varphi_n(s)\big|^2+\frac 23w(\mathcal N(s,\cdot),\mathcal M(s,\cdot))^2\right)\,ds,\label{est:intt}
\end{align} 
where we have used Assumption~\ref{ass2} on $n^0$, $m^0$ as well as Lemma~\ref{lem:roughmacro} to show
\begin{align*}
&W_2\left[\mathcal N(0,\cdot),\mathcal M(0,\cdot-(Z_m\circ \varphi_m(t)-Z_n\circ\varphi_n(t)))\right]\\
&\quad\leq W_2\left(\mathcal N(0,\cdot),\delta_0\right)+W_2\left(\delta_0,\mathcal M(0,\cdot)\right) +\left|Z_m\circ \varphi_m(t)-Z_n\circ\varphi_n(t)\right|\\
&\quad \leq \int_{\R} x^2n^0(x)\,dx+\int_{\R} x^2m^0(x)\,dx+\left\|Z_m-Z_n\right\|_{\infty}\leq C<\infty.
\end{align*}
Let $\bar C>0$ the constant appearing in \eqref{est:intt}. Let $X$ the solution of
\begin{equation}\label{def:X}
X(0)=\bar C,\quad X'(t)=-\frac 13 X(t)+\bar C\alpha^2\big|Z_m\circ \varphi_m(s)-Z_n\circ \varphi_n(s)\big|^2,
\end{equation}
so that $X'(t)=-X(t)+\left(\frac 23 X(t)+\bar C\alpha^2\big|Z_m\circ \varphi_m(s)-Z_n\circ \varphi_n(s)\big|^2\right)$, and thus
\begin{align*}
X(t)&=\bar C e^{- t} +\int_0^t e^{-(t-s)}
\left(\frac 23 X(s)+\bar C\alpha^2\big|Z_m\circ \varphi_m(s)-Z_n\circ \varphi_n(s)\big|^2\right)\,ds.
\end{align*}
This equality combined to \eqref{est:intt} implies
\begin{align*}
w(\mathcal N(t,\cdot),\mathcal M(t,\cdot))^2-X(t)&\leq\int_0^t e^{-(t-s)}
\frac 23 \left(w(\mathcal N(s,\cdot),\mathcal M(s,\cdot))^2-X(s)\right)\,ds,
\end{align*}
and the Gronwall inequality implies, for $t\geq 0$,
\begin{align}\label{eq:domwX}
w(\mathcal N(t,\cdot),\mathcal M(t,\cdot))^2&\leq X(t).
\end{align}

\subsection{Contraction estimate for the macroscopic quantity $Z(t)$}

We recall the notations \eqref{def:mathcalNM} and notations $\mathcal N$, $\mathcal M$ introduced in Section~\ref{subsec:estmicro}. The goal of this section is to prove a contraction result on the difference $Z_n\circ\varphi_n-Z_m\circ\varphi_m$. If we multiply \eqref{eq:model} by $x$ and integrate it, we obtain (see \eqref{eq:estZ} for a similar calculation): 
\begin{align*}
&(Z_n\circ\varphi_n)'(t)=\varphi_n'(t)Z_n'(\varphi_n(t))\\
&\quad =\frac 1{\left(1+\alpha I_n\circ \varphi_n(t)\right)^2}\alpha I_n\circ \varphi_n(t)\left(1+\alpha I_n\circ\varphi_n(t)\right)\left[\int_{\mathbb R}y\frac{a(y)\mathcal N(t,y)}{I_n\circ\varphi_n(t)}\,dy-Z_n\circ\varphi_n(t)\right]\\
&\quad = \alpha \frac 1{1+\alpha I_n\circ \varphi_n(t)}\left[\int_{\mathbb R}ya(y)\mathcal N(t,y)\,dy-I_n\circ \varphi_n(t)Z_n\circ\varphi_n(t)\right],
\end{align*}
and $(Z_m\circ\varphi_m)'(t)$ satisfies a similar equation. Then,
\begin{align}
&(Z_n\circ\varphi_n-Z_m\circ\varphi_m)'(t)\nonumber\\
&\quad =\alpha \left(\frac 1{1+\alpha I_n\circ\varphi_n(t)}-\frac 1{1+\alpha I_m\circ\varphi_m(t)}\right)\left(\int_{\mathbb R}ya(y)\mathcal N(t,y)\,dy -I_n\circ\varphi_n(t)Z_n\circ\varphi_n(t)\right)\nonumber\\
&\qquad +\frac\alpha{1+\alpha I_m\circ\varphi_m(t)}\left(\int_{\mathbb R}ya(y)\mathcal N(t,y)\,dy -\int_{\mathbb R}ya(y)\mathcal M(t,y)\,dy\right)\nonumber\\
&\qquad + \frac \alpha{1+\alpha I_m\circ\varphi_m(t)}\big(I_m\circ\varphi_m(t)Z_m\circ\varphi_m(t)-I_n\circ\varphi_n(t)Z_n\circ\varphi_n(t)\big).\label{eq:estZZ}
\end{align}
We notice that $\left|\int_{\mathbb R}ya(y)\mathcal N(t,y)\,dy\right|\leq\max\{|x|;x\in\textrm{supp }a\}<\infty$ and $\|Z_n\|_\infty\leq C<\infty$ thanks to Lemma~\ref{lem:roughmacro}. These two inequalities and \eqref{eq:ImI} can be used to estimate the first term on the right hand side of \eqref{eq:estZZ}:
\begin{align*}
&\alpha \left(\frac 1{1+\alpha I_n\circ\varphi_n(t)}-\frac 1{1+\alpha I_m\circ\varphi_m(t)}\right)\left(\int_{\mathbb R}ya(y)\mathcal N(t,y)\,dy -I_n\circ\varphi_n(t)Z_n\circ\varphi_n(t)\right)\\
&\quad \leq \mathcal O(\alpha^2)\left|I_n\circ\varphi_n(t)-I_m\circ\varphi_m(t)\right|\left(\left|\int_{\mathbb R}ya(y)\mathcal N(t,y)\,dy\right|+\|Z_n\|_\infty\right)\\
&\quad =\mathcal O(\alpha^2)|Z_n\circ\varphi_n(t)-Z_m\circ\varphi_m(t)|+\mathcal O(\alpha^2)w(\mathcal N(t,\cdot),\mathcal M(t,\cdot)).
\end{align*}
We estimate the second term on the right hand side of \eqref{eq:estZZ} as follows:
\begin{align}
&\frac\alpha{1+\alpha I_m\circ\varphi_m(t)}\left(\int_{\mathbb R}ya(y)\mathcal N(t,y)\,dy -\int_{\mathbb R}ya(y)\mathcal M(t,y)\,dy\right)\\
&\quad = \mathcal O(\alpha)\left(\int_{\mathbb R}ya(y)\mathcal N(t,y)\,dy -\int_{\mathbb R}ya(y)\mathcal M(t,y)\,dy\right)\nonumber\\
&\quad =\mathcal O(\alpha)\int_{\mathbb R}(y-Z_m\circ\varphi_m(t))a(y-Z_m\circ\varphi_m(t))\nonumber\\
&\phantom{quad =\mathcal O(\alpha)\int_{\mathbb R}}\left(\mathcal N(t,y-Z_n\circ\varphi_n(t))-\mathcal M(t,y-Z_m\circ\varphi_m(t))\right)\,dy\nonumber\\
&\qquad +\mathcal O(\alpha)\int_{\mathbb R}\big((y-Z_n\circ\varphi_n(t))a(y-Z_n\circ\varphi_n(t))-(y-Z_m\circ\varphi_m(t))a(y-Z_m\circ\varphi_m(t))\big) \nonumber\\
&\qquad \hspace{1cm}\left(\mathcal N(t,y-Z_n\circ\varphi_n(t))-\Gamma_{2\sigma^2}(y)\right)\,dy\nonumber\\
&\qquad +\frac\alpha{1+\alpha I_m\circ\varphi_m(t)}\int_{\mathbb R}\big((y-Z_n\circ\varphi_n(t))a(y-Z_n\circ\varphi_n(t))\nonumber\\
&\phantom{\qquad +\frac\alpha{1+\alpha I_m\circ\varphi_m(t)}\int_{\mathbb R}}-(y-Z_m\circ\varphi_m(t))a(y-Z_m\circ\varphi_m(t))\big) \Gamma_{2\sigma^2}(y)\,dy.\label{eq:decompNM}
\end{align}
Since $a$ is compactly supported and $a\in W^{2,\infty}(\mathbb R)$ (see Assumption~\ref{ass1}), we have  
\begin{equation}\label{eq:aaa}
\|(z\mapsto za'(z))'\|_\infty\leq \|a'\|_\infty+\|z\mapsto za''(z)\|_\infty<\infty,
\end{equation}
which we can use, together with the Kantorovich-Rubinstein inequality, to estimate the first term on the right hand side of \eqref{eq:decompNM}:
\begin{align}
&\left|\int_{\mathbb R}(y-Z_m\circ\varphi_m(t))a(y-Z_m\circ\varphi_m(t))\left(\mathcal N(t,y-Z_n\circ\varphi_n(t))-\mathcal M(t,y-Z_m\circ\varphi_m(t))\right)\,dy\right|\nonumber\\
&\quad \leq \|(z\mapsto za'(z))'\|_\infty W_1\left(\mathcal N(t,\cdot-Z_n\circ\varphi_n(t)),\mathcal M(t,\cdot-Z_m\circ\varphi_m(t))\right)\nonumber\\
&\quad \leq \|(z\mapsto za'(z))'\|_\infty W_2\left(\mathcal N(t,\cdot-Z_n\circ\varphi_n(t)),\mathcal M(t,\cdot-Z_m\circ\varphi_m(t))\right)\leq C w\left(\mathcal N(t,\cdot),\mathcal M(t,\cdot)\right),\label{est:zd1}
\end{align}
thanks to the property \eqref{eq:sigmaexplicit} satisfied by $w(\cdot,\cdot)$. We notice next that $g(y):=(y-Z_n\circ\varphi_n(t))a(y-Z_n\circ\varphi_n(t))-(y-Z_m\circ\varphi_m(t))a(y-Z_m\circ\varphi_m(t))$ satisfies
\begin{align*}
g'(y)&=a(y-Z_n\circ\varphi_n(t))+(y-Z_n\circ\varphi_n(t))a'(y-Z_n\circ\varphi_n(t))\\
&\quad -a(y-Z_m\circ\varphi_m(t))-(y-Z_m\circ\varphi_m(t))a'(y-Z_m\circ\varphi_m(t))\\
&=\int_{y-Z_m\circ\varphi_m(t)}^{y-Z_n\circ\varphi_n(t)}a'(z)\,dz+\int_{y-Z_m\circ\varphi_m(t)}^{y-Z_n\circ\varphi_n(t)}\left(a'(z)+za''(z)\right)\,dz\\
&=\mathcal O(|Z_m\circ\varphi_m(t)-Z_n\circ\varphi_n(t)|),
\end{align*}
thanks to \eqref{eq:aaa}. We can then use the Kantorovich-Rubinstein inequality to estimate the second term on the right hand side of \eqref{eq:decompNM}: 
\begin{align}
&\bigg|\int_{\mathbb R}\big((y-Z_n\circ\varphi_n(t))a(y-Z_n\circ\varphi_n(t))-(y-Z_m\circ\varphi_m(t))a(y-Z_m\circ\varphi_m(t))\big) \nonumber\\
&\qquad \hspace{1cm}\left(\mathcal N(t,y-Z_n\circ\varphi_n(t))-\Gamma_{2\sigma^2}(y)\right)\,dy\bigg|=\bigg|\int_{\mathbb R}g(y)\left(\mathcal N(t,y-Z_n\circ\varphi_n(t))-\Gamma_{2\sigma^2}(y)\right)\,dy\bigg|\nonumber \\
&\quad \leq \|g'\|_\infty W_1 \left(\mathcal N(t,\cdot-Z_n\circ\varphi_n(t)),\Gamma_{2\sigma^2}\right)\nonumber\\
&\quad  \leq C|Z_m\circ\varphi_m(t)-Z_n\circ\varphi_n(t)|\, W_2 \left(\mathcal N(t,\cdot-Z_n\circ\varphi_n(t)),\Gamma_{2\sigma^2}\right).\label{est:zd2}
\end{align} 
Using estimates \eqref{est:zd1} and \eqref{est:zd2}, we can estimate \eqref{eq:decompNM} further as follows:
\begin{align}
&\frac\alpha{1+\alpha I_m\circ\varphi_m(t)}\left(\int_{\mathbb R}ya(y)\mathcal N(t,y)\,dy -\int_{\mathbb R}ya(y)\mathcal M(t,y)\,dy\right)=\mathcal O(\alpha)w(\mathcal N(t,\cdot),\mathcal M(t,\cdot))\nonumber\\
&\qquad+\mathcal O(\alpha)\,\big|Z_n\circ\varphi_n(t)-Z_m\circ\varphi_m(t)\big|\,W_2(\mathcal N(t,\cdot-Z_n\circ\varphi_n(t)),
\Gamma_{2\sigma^2})\nonumber\\
&\qquad +\frac\alpha{1+\alpha I_m\circ\varphi_m(t)}\int_{\mathbb R}\big[(y-Z_n\circ\varphi_n(t))a(y-Z_n\circ\varphi_n(t))\nonumber\\
&\phantom{\qquad +\frac\alpha{1+\alpha I_m\circ\varphi_m(t)}\int_{\mathbb R}}-(y-Z_m\circ\varphi_m(t))a(y-Z_m\circ\varphi_m(t))\big] \Gamma_{2\sigma^2}(y)\,dy,\label{esttra}
\end{align}
which concludes the estimate of the second term on the right hand side of \eqref{eq:estZZ}. To estimate the last term of \eqref{eq:estZZ}, we develop similar arguments:
\begin{align}
&\frac\alpha{1+\alpha I_m\circ\varphi_m(t)}\big(I_m\circ\varphi_m(t)Z_m\circ\varphi_m(t)-I_n\circ\varphi_n(t)Z_n\circ\varphi_n(t)\big)\nonumber\\
&\quad =\frac\alpha{1+\alpha I_m\circ\varphi_m(t)}\bigg(\int_{\mathbb R}Z_m\circ\varphi_m(t)a(y-Z_m\circ\varphi_m(t))\mathcal M(t,y-Z_m\circ\varphi_m(t))\,dy\nonumber\\
&\quad \phantom{\quad =\mathcal O(\alpha)\bigg(\int}-\int_{\mathbb R}Z_n\circ\varphi_n(t)a(y-Z_n\circ\varphi_n(t))\mathcal N(t,y-Z_n\circ\varphi_n(t))\,dy\bigg)\nonumber\\
&\quad = \mathcal O(\alpha)\int_{\mathbb R}Z_m\circ\varphi_m(t)a(y-Z_m\circ\varphi_m(t))\big(\mathcal M(t,y-Z_m\circ\varphi_m(t))-\mathcal N(t,y-Z_n\circ\varphi_n(t))\big)\,dy\nonumber\\
&\qquad +\mathcal O(\alpha)\int_{\R}\big(Z_m\circ\varphi_m(t)a(y-Z_m\circ\varphi_m(t))-Z_n\circ\varphi_n(t)a(y-Z_n\circ\varphi_n(t))\big)\nonumber\\
&\qquad \phantom{\quad =\mathcal O(\alpha)\bigg(\int} \big(\mathcal N(t,y-Z_n\circ\varphi_n(t))-\Gamma_{2\sigma^2}(y)\big)\,dy\nonumber\\
&\qquad +\frac\alpha{1+\alpha I_m\circ\varphi_m(t)}\int_{\R}\big(Z_m\circ\varphi_m(t)a(y-Z_m\circ\varphi_m(t))\nonumber\\
&\phantom{\qquad +\frac\alpha{1+\alpha I_m\circ\varphi_m(t)}\int_{\R}}-Z_n\circ\varphi_n(t)a(y-Z_n\circ\varphi_n(t))\big)\Gamma_{2\sigma^2}(y)\,dy.\label{eq:kl1}
\end{align}
We can estimate the first term on the right hand side of \eqref{eq:kl1} using a Kantorovich-Rubinstein inequality:
\begin{align*}
&\mathcal O(\alpha)\int_{\mathbb R}Z_m\circ\varphi_m(t)a(y-Z_m\circ\varphi_m(t))\big(\mathcal M(t,y-Z_m\circ\varphi_m(t))-\mathcal N(t,y-Z_n\circ\varphi_n(t))\big)\,dy\\
&\quad \leq C\alpha \|Z_m\|_\infty \|a'\|_\infty W_1\big(\mathcal M(t,\cdot-Z_m\circ\varphi_m(t))-\mathcal N(t,\cdot -Z_n\circ\varphi_n(t))\big)\\
&\quad \leq C\alpha  W_2\big(\mathcal M(t,\cdot-Z_m\circ\varphi_m(t))-\mathcal N(t,\cdot -Z_n\circ\varphi_n(t))\big)\leq C\alpha w\big(\mathcal N(t,\cdot),\mathcal M(t,\cdot)\big),
\end{align*}
since $w(\cdot,\cdot)$ satisfies \eqref{eq:sigmaexplicit}. To estimate the second term on the right hand side of \eqref{eq:kl1}, we notice that $h(y):=Z_m\circ\varphi_m(t)a(y-Z_m\circ\varphi_m(t))-Z_n\circ\varphi_n(t)a(y-Z_n\circ\varphi_n(t))$ is Lipschitz continuous, and satisfies 
\begin{align*}
|h'(y)|&=\big|Z_m\circ\varphi_m(t)a'(y-Z_m\circ\varphi_m(t))-Z_n\circ\varphi_n(t)a'(y-Z_n\circ\varphi_n(t))\big|\\
&\leq \|a'\|_\infty\big|Z_m\circ\varphi_m(t)-Z_n\circ\varphi_n(t)\big|+\|Z_n\|_\infty \|a''\|_\infty \big|Z_m\circ\varphi_m(t)-Z_n\circ\varphi_n(t)\big|\\
&\leq C \big|Z_m\circ\varphi_m(t)-Z_n\circ\varphi_n(t)\big|,
\end{align*}
thanks to the bound on $\|Z_n\|_\infty$ provided by Lemma~\ref{lem:roughmacro} and the assumption $a\in W^{2,\infty}(\R)$. The Kantorovich-Rubinstein inequality then implies the following estimate on the second term on the right hand side of \eqref{eq:kl1}:
\begin{align*}
&\mathcal O(\alpha)\bigg|\int_{\R}\big(Z_m\circ\varphi_m(t)a(y-Z_m\circ\varphi_m(t))-Z_n\circ\varphi_n(t)a(y-Z_n\circ\varphi_n(t))\big)\nonumber\\
&\qquad \phantom{\quad =\mathcal O(\alpha)\bigg(\int} \big(\mathcal N(t,y-Z_n\circ\varphi_n(t))-\Gamma_{2\sigma^2}(y)\big)\,dy\bigg|\\
&\quad \leq C\alpha \big|Z_m\circ\varphi_m(t)-Z_n\circ\varphi_n(t)\big| W_2\big(\mathcal N(t,\cdot-Z_n\circ\varphi_n(t)),\Gamma_{2\sigma^2}\big).
\end{align*}
Thanks to these estimates, \eqref{eq:kl1} becomes
 \begin{align*}
 &\frac\alpha{1+\alpha I_m\circ\varphi_m(t)}\big(I_m\circ\varphi_m(t)Z_m\circ\varphi_m(t)-I_n\circ\varphi_n(t)Z_n\circ\varphi_n(t)\big)=\mathcal O(\alpha)w(\mathcal N(t,\cdot),\mathcal M(t,\cdot))\\
&\quad +\mathcal O(\alpha)\,|Z_m\circ\varphi_m(t)-Z_n\circ\varphi_n(t)|\,W_2\big(\mathcal N(t,\cdot-Z_n\circ\varphi_n(t)),\Gamma_{2\sigma^2}\big)\\
&\quad +\frac\alpha{1+\alpha I_m\circ\varphi_m(t)} \int\big[Z_m\circ\varphi_m(t)a(y-Z_m\circ\varphi_m(t))-Z_n\circ\varphi_n(t)a(y-Z_n\circ\varphi_n(t))\big]\Gamma_{2\sigma^2}(y)\,dy,
\end{align*}
which concludes the estimate of the last term of \eqref{eq:estZZ}. Thanks to this estimate and \eqref{esttra}, we can estimate \eqref{eq:estZZ} as follows:
\begin{align*}
&(Z_n\circ\varphi_n-Z_m\circ\varphi_m)'(t) \\
&\quad =\mathcal O(\alpha)w(\mathcal N(t,\cdot),\mathcal M(t,\cdot))+\mathcal O(\alpha) \big|Z_m\circ\varphi_m(t)-Z_n\circ\varphi_n(t)\big|W_2\big(\mathcal N(t,\cdot-Z_n\circ\varphi_n(t)),\Gamma_{2\sigma^2}\big)\\
&\qquad +\frac \alpha{1+\alpha I_m\circ\varphi_m(t)}\bigg[\int_{\mathbb R}ya(y)\Gamma_{2\sigma^2}(y+Z_n\circ\varphi_n(t))\,dy\nonumber\\
&\phantom{\qquad +\frac \alpha{1+\alpha I_m\circ\varphi_m(t)}\bigg[}-Z_n\circ\varphi_n(t)\int_{\mathbb R}a(y)\Gamma_{2\sigma^2}(y+Z_n\circ\varphi_n(t))\,dy\bigg]\\
&\qquad -\frac \alpha{1+\alpha I_m\circ\varphi_m(t)}\bigg[\int_{\mathbb R}ya(y)\Gamma_{2\sigma^2}(y+Z_m\circ\varphi_m(t))\,dy\nonumber\\
&\phantom{\qquad +\frac \alpha{1+\alpha I_m\circ\varphi_m(t)}\bigg[}-Z_m\circ\varphi_m(t)\int_{\mathbb R}a(y)\Gamma_{2\sigma^2}(y+Z_m\circ\varphi_m(t))\,dy\bigg]\\
&\quad =\mathcal O(\alpha)w(\mathcal N(t,\cdot),\mathcal M(t,\cdot))+\mathcal O(\alpha) \big|Z_m\circ\varphi_m(t)-Z_n\circ\varphi_n(t)\big|W_2\big(\mathcal N(t,\cdot-Z_n\circ\varphi_n(t)),\Gamma_{2\sigma^2}\big)\\
&\qquad+\frac \alpha {1+\alpha I_m\circ\varphi_m(t)}\left[F(Z_n\circ\varphi_n(t))-F(Z_m\circ\varphi_m(t))\right],
\end{align*}
thanks to the definition \eqref{def:F} of $F$. Thanks to \eqref{eq:ineqvarphi},  $\varphi_n(t)\geq -2C\ln\alpha/\alpha$ implies 
\[t\geq (1-C\alpha)(-2C\ln\alpha/\alpha)\geq -C\ln\alpha/\alpha,\]
for $0<\alpha<1/C$. Then (up to increasing the value of the constant $C>0$),
for $t\geq t_0=-C\ln\alpha/\alpha$, Lemma~\ref{lem:roughmacro} (see \eqref{eq:Ztpetit}) can be used to show 
\begin{equation}\label{eq:est-macro-G}
W_2\big(\mathcal N(t,\cdot-Z_n\circ\varphi_n(t)),\Gamma_{2\sigma^2}\big)\leq C\sqrt \alpha,
\end{equation}
and then
\begin{align*}
&(Z_n\circ\varphi_n-Z_m\circ\varphi_m)'(t) \\
&\quad =\mathcal O(\alpha)w(\mathcal N(t,\cdot),\mathcal M(t,\cdot))+\mathcal O(\alpha^{3/2}) \big|Z_m\circ\varphi_m(t)-Z_n\circ\varphi_n(t)\big|\\
&\qquad+\frac \alpha {1+\alpha I_m\circ\varphi_m(t)}\left[F(Z_n\circ\varphi_n(t))-F(Z_m\circ\varphi_m(t))\right],
\end{align*}
and \eqref{eq:Ztpetitter} implies, for $t\geq -C\ln\alpha/\alpha$,
\[|Z_n\circ\varphi_n(t)-\bar Z|\leq C\sqrt \alpha,\quad |Z_m\circ\varphi_m(t)-\bar Z|\leq C\sqrt \alpha\]
and then 
\begin{align*}
F(Z_n\circ\varphi_n(t))-F(Z_m\circ\varphi_m(t))&=\int_{Z_m\circ\varphi_m(t)}^{Z_n\circ\varphi_n(t)}F'(y)\,dy\\
&=\left(F'(\bar Z)+\mathcal O(\sqrt\alpha)\right)\big(Z_n\circ\varphi_n(t)-Z_m\circ\varphi_m(t)\big),
\end{align*}
since $[Z_m\circ\varphi_m(t),Z_n\circ\varphi_n(t)]\subset[\bar Z-C\alpha,\bar Z+C\alpha]$ for some $C>0$, and $F'$ is Lipschitz continuous. We also notice that $\frac \alpha {1+\alpha I_m\circ\varphi_m(t)}=\alpha+\mathcal O(\alpha)$, since $I_m$ is uniformly bounded. Then, provided $\alpha>0$ is  small enough and thanks to $F'(\bar Z)<0$ (see Assumption~\ref{ass1}), we get
\begin{equation}
\frac d{dt}|Z_n\circ\varphi_n(t)-Z_m\circ\varphi_m(t)|\leq \alpha\left( F'(\bar Z) +\mathcal O(\sqrt \alpha)\right) |Z_n\circ\varphi_n(t)-Z_m\circ\varphi_m(t)|+\mathcal O(\alpha)w(\mathcal N(t,\cdot),\mathcal M(t,\cdot)),
\end{equation}
for $t\geq -C\ln\alpha/\alpha$, that is, if we multiply that expression by $2|Z_n\circ\varphi_n(t)-Z_m\circ\varphi_m(t)|$,
\begin{align}
\frac d{dt}|Z_n\circ\varphi_n(t)-Z_m\circ\varphi_m(t)|^2&=2|Z_n\circ\varphi_n(t)-Z_m\circ\varphi_m(t)|\frac d{dt}|Z_n\circ\varphi_n(t)-Z_m\circ\varphi_m(t)|\\
&\leq 2\alpha\left( F'(\bar Z) +C\sqrt\alpha\right) |Z_n\circ\varphi_n(t)-Z_m\circ\varphi_m(t)|^2\nonumber\\
&\quad +C\alpha|Z_n\circ\varphi_n(t)-Z_m\circ\varphi_m(t)|\sqrt{X(t)},\label{eq:dtZ}
\end{align}
Where we can use the function $X(s)$ introduced in \eqref{def:X} to estimate $w(\mathcal N(s,\cdot),\mathcal M(s,\cdot))$, thanks to \eqref{eq:domwX}.

\subsection{Conclusion of the proof of Theorem~\ref{Thm:main}}\label{subsec:proofmainThm}

Thanks to \eqref{def:X} and \eqref{eq:dtZ}, for $t\geq -C\ln \alpha/\alpha$,
\begin{align*}
&\frac d{dt}\left(X(t)+|Z_n\circ\varphi_n(t)-Z_m\circ\varphi_m(t)|^2\right)\leq -\frac 13 X(t)+C\alpha^2 |Z_n\circ\varphi_n(t)-Z_m\circ\varphi_m(t)|^2\\
&\qquad +2\alpha\left( F'(\bar Z) +C\sqrt\alpha\right) |Z_n\circ\varphi_n(t)-Z_m\circ\varphi_m(t)|^2+\mathcal O(\alpha)|Z_n\circ\varphi_n(t)-Z_m\circ\varphi_m(t)|\sqrt{X(t)},
\end{align*}
and we can use a Young estimate to show 
\[\mathcal O(\alpha)|Z_n\circ\varphi_n(t)-Z_m\circ\varphi_m(t)|\sqrt{X(t)}\leq C\alpha^2|Z_n\circ\varphi_n(t)-Z_m\circ\varphi_m(t)|^2+\frac 14 X(t),\]
leading to 
\begin{align*}
&\frac d{dt}\left(X(t)+|Z_n\circ\varphi_n(t)-Z_m\circ\varphi_m(t)|^2\right)\\
&\quad \leq -\frac 1{12} X(t)+2\alpha\left( F'(\bar Z) +C\sqrt\alpha\right) |Z_n\circ\varphi_n(t)-Z_m\circ\varphi_m(t)|^2\\
&\quad  \leq 2\alpha\left( F'(\bar Z)+C\alpha\right)\left(X(t)+|Z_n\circ\varphi_n(t)-Z_m\circ\varphi_m(t)|^2\right),
\end{align*}
provided $\alpha>0$ is small enough and $t\geq -C\ln \alpha/\alpha$. We notice that 
\[\left(X(t)+|Z_n\circ\varphi_n(t)-Z_m\circ\varphi_m(t)|^2\right)(t=-C\ln\alpha/\alpha)\leq C,\]
thanks to Lemma~\ref{lem:estmacro}, see \eqref{est:Zbasic} and \eqref{eq:Ztpetitbis}, and then
\begin{align*}
&\left(X(t)+|Z_n\circ\varphi_n(t)-Z_m\circ\varphi_m(t)|^2\right)\leq  C e^{2\left( F'(\bar Z)+C\sqrt \alpha\right)\alpha \left(t+C\ln\alpha/\alpha\right)}.
\end{align*}
We can use next the property  \eqref{eq:sigmaexplicit} satisfied by $w$ and the estimate $w(\mathcal N(t,\cdot),\mathcal M(t,\cdot))^2\leq X(t)$ (see \eqref{eq:domwX}), for $t\geq -C\ln\alpha/\alpha$, to obtain:
\begin{align*}
&W_2(\mathcal N(t,\cdot),\mathcal M(t,\cdot))\leq w(\mathcal N(t,\cdot),\mathcal M(t,\cdot))+|Z_n\circ\varphi_n(t)-Z_m\circ\varphi_m(t)|\\
&\quad \leq \sqrt{X(t)}+|Z_n\circ\varphi_n(t)-Z_m\circ\varphi_m(t)|\leq \sqrt{2\left(X(t)+|Z_n\circ\varphi_n(t)-Z_m\circ\varphi_m(t)|^2\right)}\\
&\quad \leq Ce^{\left( F'(\bar Z) +\sqrt\alpha\right)\alpha(t+C\ln\alpha/\alpha)},
\end{align*}
We apply this estimate with $m(t,\cdot)\equiv \bar n$, where $\bar n$ is the steady state provided by Theorem~\ref{Thm:existence-steady}. Note that $m(0,\cdot)=\bar n$ satisfies the conditions of Assumption~\ref{ass2}, thanks to the properties of $\bar n$ stated in Theorem~\ref{Thm:existence-steady}. Moreover, $m(t,\cdot)\equiv \bar n$ solves \eqref{eq:model} in the sense of Theorem~\ref{cor:existence}. Then, thanks to \eqref{eq:ineqvarphi},
\begin{align*}
W_2(n(t,\cdot),\bar n)&\leq \sup_{[(1-C\alpha)t,(1+C\alpha)t]}W_2(n(\varphi_n(t),\cdot),\bar n)=\sup_{[(1-C\alpha)t,(1+C\alpha)t]}W_2(\mathcal N(t,\cdot),\bar n(\varphi_{\bar n}(t),\cdot))\\
&\leq Ce^{\left( F'(\bar Z) +\sqrt\alpha\right)\alpha(t+C\alpha t+C\ln\alpha/\alpha)}\leq Ce^{\left( F'(\bar Z) +C\sqrt\alpha\right)\alpha(t+C\ln\alpha/\alpha)},
\end{align*}
and for $t\geq C\frac{\ln\alpha}{\alpha^2}$, $(F'(\bar Z))\alpha C{\ln\alpha}/\alpha\leq  (C\alpha)\alpha t$, and then,   
\begin{align*}
W_2(n(t,\cdot),\bar n)&\leq  Ce^{\left( F'(\bar Z) +C\sqrt\alpha\right)\alpha t},
\end{align*}
for $t\geq C\frac{\ln\alpha}{\alpha^2}$. This estimate implies \eqref{estW2thm}. The first inequality in  \eqref{est:diffZY} is provided by \eqref{est:Zlinfty}, while the second is implied by \eqref{eq:Ztpetit}. We have then completed the proof of Theorem~\ref{Thm:main}.

\begin{rem}
Solutions converge rapidly towards the manifold composed of Gaussian distributions of variance $2\sigma^2$, which underlies estimate \eqref{eq:domwX} on the quantity $X(t)$ controlling $w(\mathcal N(t,\cdot),\mathcal M(t,\cdot))$ (see  \eqref{def:X}). The dynamics along this \emph{slow manifold} is much slower (when $\alpha>0$ is small), which explains the factor $\alpha$ appearing on the right hand side of \eqref{eq:dtZ}. This dynamics is reminiscent of slow-fast dynamical systems. The final estimate \eqref{estW2thm} shows a contraction at a speed guided by the slow dynamics, hence the factor $\alpha t$ appearing in the exponential term. This is coherent with the translation invariance of \eqref{eq:model} when $\alpha=0$, which prevents any contraction result when $\alpha=0$. Notice that the fact that the dynamics of the mean phenotypic trait of the population is slower than the stabilization of the distribution of the population as a Gaussian around its mean phenotypic trait is at the core of macroscopic limits of such models \cite{mirrahimi2013dynamics,raoul2017macroscopic}. It also played a central role in  \cite{PoyatoCalvez} where a sexually reproducing population is considered, with a discreet time and non-overlapping generations, see (3.17) that describes the  contraction of mean, and (3.17) that quantifies the contraction of the variance of the distribution. 
\end{rem}

\section{Proof of Corollary~\ref{cor:uniqueness}}\label{sec:proofCor}

\noindent\textbf{Step 1: $\tilde n$ satisfies Assumption~2 for a well-chosen $R=R(a,\bar Z,\tilde R,\tilde Z_-,\tilde Z_+)$}

Since $\tilde n\in\mathcal P_2(\mathbb R)\cap C^1(\mathbb R)$ satisfies the tail condition \eqref{eq:tanx} and $\int_{\mathbb R}x\tilde n(x)\,dx\in[\tilde Z_-,\tilde Z_+]$, it satisfies the following properties (related to Assumption~\ref{ass2}, see below):
\begin{itemize}
\item \emph{Exponential decay:} Thanks to \eqref{eq:tanx}, $\tilde n$ satisfies:
\[\forall x\in[\tilde R,\infty),\quad \bar \partial_x n(-x)>\bar n(-x),\quad \partial_x \bar n(x)<-\bar n(x).\]
\item \emph{Lower bound:} Thanks to the exponential decay assumption on $\tilde n$ and Lemma~\ref{lem:m2}, we have $\int_{\mathbb R}x^2\tilde n(x)\,dx\leq C_{\tilde R}$, and then, thanks to a Chebychev inequality, 
\[\int_{-\sqrt{2C_{\tilde R}}}^{\sqrt{2C_{\tilde R}}}\tilde n(x)\,dx\geq 1-\int_{[-\sqrt{2C_{\tilde R}},\sqrt{2C_{\tilde R}}]^c}\tilde n(x)\,dx\geq \frac 1{2C_{\tilde R}}\int_{\mathbb R}x^2\tilde n(x)\,dx\geq \frac 12.\]
Then,
\begin{align*}
\tilde n(x)&\geq \iint_{[-\sqrt{2C_{\tilde R}},\sqrt{2C_{\tilde R}}]^2} \Gamma_{\sigma^2} \left(x-\frac{y_*+y_*'}2\right)\frac{(1+\alpha a(y_*))\tilde n(y_*)}{1+\alpha I_{\tilde n}}\frac{(1+\alpha a(y_*'))\tilde n(y_*')}{1+\alpha I_{\tilde n}}\,dy_*\,dy_*'\nonumber\\
&\geq \Gamma_{\sigma^2} \left(|x|+\sqrt{2C_{\tilde R}}\right)\frac 1{(1+\alpha\|a\|_\infty)^2}\left(\int_{-\sqrt{2C_{\tilde R}}}^{\sqrt{2C_{\tilde R}}}\tilde n(y_*)\,dy_*\right)^2\nonumber\\
&\geq \frac 1{4(1+\alpha\|a\|_\infty)^2}\Gamma_{\sigma^2} \left(|x|+\sqrt{2C_{\tilde R}}\right)\geq \frac {\sqrt 2e^{-\frac{2C_{\tilde R}}{\sigma^2}}}{4(1+\alpha\|a\|_\infty)^2}\Gamma_{\sigma^2} \left(x\right)
\end{align*}
and then, if $R_1:=\max\left(\frac {4(1+\alpha\|a\|_\infty)^2}{\sqrt 2e^{-\frac{2C_{\tilde R}}{\sigma^2}}},\frac 2{\sigma^2}\right)$, we have
\begin{align}
\tilde n(x)&\geq \frac 1 {R_1}\Gamma_{1/R_1} (x).\label{eq:cmin}
\end{align}
\item \emph{$\tilde Z=\int_{\mathbb R} x\tilde n(x)\,dx$ belongs to the basin of attraction of $\bar Z$:} We notice that $K:=[\tilde Z_-,\bar Z-|F'(\bar Z)|/4\|F''\|_\infty]\cup[\bar Z-|F'(\bar Z)|/4\|F''\|_\infty,\tilde Z_+]$ is compact, and $|F(x)|>0$ for $x\in K$, thanks to the assumption of the Corollary. Since $F\in C^1(\mathbb R)$, if $R_2:=1/\min_{x\in K}F(x)<\infty$, we have
\[\forall x\in [\tilde Z_-,\bar Z-|F'(\bar Z)|/4\|F''\|_\infty], F(x)\geq \frac 1{R_2}, \]
\[\forall x\in [\bar Z+|F'(\bar Z)|/4\|F''\|_\infty,\tilde Z_+], F(x)\leq -\frac 1{R_2}.\]
\end{itemize}
Therefore, if we define
\begin{equation}\label{def:Rcor}
R=R(a,\bar Z,\tilde R,\tilde Z_-,\tilde Z_+):=\max(\{|x|;x\in\textrm{supp }a\},\tilde R,R_1,R_2),
\end{equation}
then $\tilde n\in\mathcal P_2(\mathbb R)\cap C^1(\mathbb R)$ satisfies Assumption~\ref{ass2} for $\bar Z$, $a\in W^{2,\infty}(\mathbb R,\mathbb R_+)$ and $R$. 

\medskip

\noindent\textbf{Step 2: $\tilde n$ satisfies $W_2\left(\tilde n,\Gamma_{2\sigma^2}(\cdot-Z_{\tilde n})\right)\leq \bar C\sqrt\alpha$ for $\bar C=\bar C(\|a\|_\infty,\sigma^2,\tilde R)>0$.}

Next, we take advantage of the assumption that $\tilde n$ is a steady state of \eqref{eq:model} to obtain:
\begin{align}
\tilde n(x)&=\iint_{\mathbb R^2} \Gamma_{\sigma^2} \left(x-\frac{y_*+y_*'}2\right)\frac{(1+\alpha a(y_*))\tilde n(y_*)}{1+\alpha I_{\tilde n}}\frac{(1+\alpha a(y_*'))\tilde n(y_*')}{1+\alpha I_{\tilde n}}\,dy_*\,dy_*'\nonumber\\
&=T\left(\frac{(1+\alpha a)\tilde n}{1+\alpha I_{\tilde n}}\right)(x),\label{eq:tildenTn}
\end{align}
and notice that
\[Z_{\tilde n}:=\int_{\mathbb R}x\tilde n(x)\,dx=\int_{\mathbb R}xT\left(\frac{(1+\alpha a)\tilde n}{1+\alpha I_{\tilde n}}\right)(x)\,dx.\]
We can use equality \eqref{eq:tildenTn}, the invariance of $\Gamma_{2\sigma^2}(\cdot-Z_{\tilde n})$ for the operator $T$ (see \eqref{eq:Maxwellienne}) and the contraction \eqref{eq:contraction} induced by  $T$, to show:
\begin{align}
&W_2\left(\tilde n,\Gamma_{2\sigma^2}(\cdot-Z_{\tilde n})\right)=W_2\left(T\left(\frac{(1+\alpha a)\tilde n}{1+\alpha I_{\tilde n}}\right),T(\Gamma_{2\sigma^2}(\cdot-Z_{\tilde n}))\right)\nonumber\\
&\quad \leq \frac 1{\sqrt 2} W_2\left(\frac{(1+\alpha a)\tilde n}{1+\alpha I_{\tilde n}},\Gamma_{2\sigma^2}(\cdot-Z_{\tilde n})\right)\leq \frac 1{\sqrt 2} \left(W_2\left(\frac{1}{1+\alpha I_{\tilde n}}n+\frac{\alpha I_{\tilde n}}{1+\alpha I_{\tilde n}}\frac{a\,\tilde n}{I_n},\Gamma_{2\sigma^2}(\cdot-Z_{\tilde n})\right)^2\right)^{1/2}\nonumber\\
\end{align}
We can then use the convexity of $W_2^2$ (see \eqref{eq:W2convex}) to obtain:
\begin{align}
W_2\left(\tilde n,\Gamma_{2\sigma^2}(\cdot-Z_{\tilde n})\right)
&\leq \frac 1{\sqrt 2} \left(\frac{1}{1+\alpha I_{\tilde n}}W_2\left(\tilde n,\Gamma_{2\sigma^2}(\cdot-Z_{\tilde n})\right)^2+\frac{\alpha I_{\tilde n}}{1+\alpha I_{\tilde n}}W_2\left(\frac{a\,\tilde n}{I_{\tilde n}},\Gamma_{2\sigma^2}(\cdot-Z_{\tilde n})\right)^2\right)^{1/2}\nonumber\\
&\leq \frac{1}{\sqrt{2}}W_2\left(\tilde n,\Gamma_{2\sigma^2}(\cdot-Z_{\tilde n})\right)+C\sqrt{\alpha},\nonumber
\end{align}
where we have used the following estimate:
\begin{align*}
I_nW_2\left(\frac{a\,\tilde n}{I_{\tilde n}},\Gamma_{2\sigma^2}(\cdot-Z_{\tilde n})\right)^2&\leq 2I_{\tilde n} W_2\left(\frac{a\,\tilde n}{I_{\tilde n}},\delta_{0}\right)^2+2I_{\tilde n}W_2\left(\delta_{0},\Gamma_{2\sigma^2}(\cdot-Z_{\tilde n})\right)^2\\
&\leq 2\|a\|_\infty \int_{\mathbb R} x^2\tilde n(x)\,dx+2\|a\|_\infty\left(2\sigma^2+|Z_{\tilde n}|\right)^2\leq C,
\end{align*}
thanks to Lemma~\ref{lem:m2}. As a consequence,
\begin{equation}\label{eq:proxW2}
W_2\left(\tilde n,\Gamma_{2\sigma^2}(\cdot-Z_{\tilde n})\right)\leq \bar C\sqrt\alpha,
\end{equation}
for a constant $\bar C=\bar C(\|a\|_\infty,\sigma^2,\tilde R)>0$.

\medskip

\noindent\textbf{Step 3: Conclusion of the proof}

Given $\bar Z\in\mathbb R$ and $a\in W^{2,\infty}(\mathbb R,\mathbb R_+)$ satisfying assumption~\ref{ass1}, and $R>0$ defined by \eqref{def:Rcor}, we can apply Theorem~\ref{Thm:main}, that provides the values $\bar \alpha>0$, $\rho\in(0,1]$ and $C>0$. We have already proven that $\tilde n$ satisfies Assumption~\ref{ass2} (thanks to the construction of $R>0$, see \eqref{def:Rcor}), and thanks to \eqref{eq:proxW2}, if $\alpha\in\big(0,\min(\bar \alpha,(\rho/\bar C)^2)\big)$, then $W_2\left(\tilde n,\Gamma_{2\sigma^2}(\cdot-Z_{\tilde n})\right)\leq \rho$. Therefore, provided $\alpha\in\big(0,\min(\bar \alpha,(\rho/\bar C)^2)\big)$, the distribution $\tilde n$, as an initial condition of the model \eqref{eq:model}, satisfies the conditions of Theorem~\ref{Thm:main}, and the corresponding solution $n=n(t,x)$ of \eqref{eq:model} and it satisfies \eqref{estW2thm}, that is 
\[W_2(\tilde n,\bar n)=W_2(n(t,\cdot,\bar n)\leq Ce^{(F'(\bar Z)+C\sqrt\alpha)\alpha t}\xrightarrow[t\to\infty]{}0,\]
since the fact that $\tilde n$ is a steady state of \eqref{eq:model} implies that the solution $n$ is constant in time, that is  $n(t,x)=\tilde n(x)$ for $(t,x)\in[0,\infty)\times\mathbb R$. This convergence proves $\tilde n=\bar n$, which concludes the proof of the Corollary.

\appendix

\section{Existence and uniqueness of solutions}\label{App:existence}

Our goal in this section is to provide an existence and uniqueness setting for \eqref{eq:modifiedpb} when $\varepsilon>0$, and to \eqref{eq:model} (that we assimilate to \eqref{eq:modifiedpb} with $\varepsilon=0$). We recall that throughout this manuscript,  $u_\varepsilon(t,x)$ is seen as a function defined for $t\geq 0$ and $x\in[-1/\varepsilon,1/\varepsilon]$, but we assimilate it to its extension to $x\in\mathbb R$ through $u_\varepsilon(t,x)=0$ for $t\geq 0$ and $|x|>1/\varepsilon$. We also recall the operator $T$ defined by \eqref{def:T}, and, as defined in \eqref{eq:modifiedpb},
\begin{equation}\label{def:inidatamod}
u_\varepsilon(0,x)=\frac{n(0,x)1_{|x|\leq 1/\varepsilon}}{\int_{-1/\varepsilon}^{1/\varepsilon}n(0,y)\,dy}.
\end{equation}
Adopting a notation similar to  \eqref{def:In}, we define:
\[I_{u_\varepsilon}^\varepsilon(t):=\int_{\mathbb R} a(x)u_\varepsilon(t,x)\,dx=\int_{\frac 1\varepsilon}^{\frac 1\varepsilon} a(x)u_\varepsilon(t,x)\,dx.\]


\begin{theo}\label{thm:existence}
Let $a\in W^{1,\infty}(\mathbb R)$, $\alpha,\sigma>0$, $\varepsilon\geq 0$ and $T$ defined by \eqref{def:T}. For 
\[u^0\in\mathcal P_2((-1/\varepsilon,1/\varepsilon))\cap L^\infty((-1/\varepsilon,1/\varepsilon))),\] 
there is a unique solution $u_\varepsilon$ of \eqref{eq:modifiedpb}, in the following sense: $u_\varepsilon\in L^\infty\big([0,T]\times (-1/\varepsilon,1/\varepsilon)\big)$, such that $\int_{-1/\varepsilon}^{1/\varepsilon}u_\varepsilon(t,x)\,dx= 1$ for $t\in[0,\infty)$, and 
\begin{align}
&u_\varepsilon(t,x)=u_\varepsilon(0,x)e^{-\int_0^t\left(1+\alpha  I^\varepsilon_{u_\varepsilon}(s)\right)^2\left(\int_{-1/\varepsilon}^{1/\varepsilon}T\left(\frac{(1+\alpha a)u_\varepsilon(s,\cdot)}{1+\alpha  I^\varepsilon_{u_\varepsilon}(s)}\right)(z)\,dz\right)\,ds}\nonumber\\
&\quad +\int_0^t\left(1+\alpha  I^\varepsilon_{u_\varepsilon}(s)\right)^2T\left(\frac{(1+\alpha a)u_\varepsilon(s,\cdot)}{1+\alpha  I^\varepsilon_{u_\varepsilon}(s)}\right)(x)1_{|x|< \frac 1\varepsilon}\nonumber\\
&\phantom{fzerzerer} e^{-\int_s^t\left(1+\alpha  I^\varepsilon_{u_\varepsilon}(\tau)\right)^2\left(\int_{-1/\varepsilon}^{1/\varepsilon}T\left(\frac{(1+\alpha a)u_\varepsilon(\tau,\cdot)}{1+\alpha  I^\varepsilon_{u_\varepsilon}(\tau)}\right)(z)\,dz\right)\,d\tau}\,ds,\label{eq:intformulation}
\end{align}
for $(t,x)\in[0,\infty)\times (-1/\varepsilon,1/\varepsilon)$.  This solution can be written as $u_\varepsilon(t,x)=u_\varepsilon(0,x)g_\varepsilon(t,x)+h_\varepsilon(t,x)$ with $g_\varepsilon,h_\varepsilon\in C^\infty([0,\infty)\times (-1/\varepsilon,1/\varepsilon))$.
\end{theo}
\begin{rem}\label{rem:W1BV}
To prove this existence theorem, we will use a contraction argument, similar to the proof of the Cauchy-Lipschitz theorem. We base the contraction argument on the distance  
\[d(u,v)=\max_{t\in[0,\bar t]}W_1\big(u(t,\cdot),v(t,\cdot)\big),\]
defined here for $u,v\in\mathcal P((-1/\varepsilon,1/\varepsilon))$. This construction could be extended to prove the existence of measure-valued solutions, although we only consider bounded solutions here to simplify notations. We refer to \cite{canizo2013measure,dull2021spaces} for related methods to construct measure-valued solutions to related equations. Note that an existence and uniqueness result based on a contraction in $L^\infty$ has been developed in \cite{guerand2023moment} for a similar model. 
\end{rem}

\begin{rem}\label{rem:eximodel}
If $\varepsilon=0$, Theorem~\ref{thm:existence} proves the existence for solutions to the original model \eqref{eq:model}, which is stated in Theorem~\ref{cor:existence}.
\end{rem}

\begin{proof}[Proof of Theorem~\ref{thm:existence}]
Let $\bar t>0$ and Let $u\in L^\infty\Big([0,\bar t]\times(-1/\varepsilon,1/\varepsilon),\mathbb R_+\Big)$ such that 
\[\int_{-1/\varepsilon}^{1/\varepsilon} u(t,x)\,dx=1,\]
 and $W_1(u(s,\cdot),u(t,\cdot))\xrightarrow[s\to t]{} 0$ for $t\in[0,\bar t]$. For $(t,x)\in[0,\bar t]\times(-1/\varepsilon,1/\varepsilon)$, we define
\begin{align*}
&\mathcal F(u)(t,x)=u_\varepsilon(0,x)e^{-\int_0^t\left(1+\alpha I_{u}^\varepsilon(s)\right)^2\left(\int_{-1/\varepsilon}^{1/\varepsilon}T\left(\frac{(1+\alpha a)u(s,\cdot)}{1+\alpha I_{u}^\varepsilon(s)}\right)(z)\,dz\right)\,ds}\nonumber\\
&\quad +\int_0^t\left(1+\alpha I_{u}^\varepsilon(s)\right)^2T\left(\frac{(1+\alpha a)u(s,\cdot)}{1+\alpha I_{u}^\varepsilon(s)}\right)(x)1_{|x|< \frac 1\varepsilon}e^{-\int_s^t\left(1+\alpha I_{u}^\varepsilon(\tau)\right)^2\left(\int_{-1/\varepsilon}^{1/\varepsilon}T\left(\frac{(1+\alpha a)u(\tau,\cdot)}{1+\alpha I_{u}^\varepsilon(\tau)}\right)(z)\,dz\right)\,d\tau}\,ds,
\end{align*}
where $u_\varepsilon(0,x)$ stands for the initial datum of the truncated model as defined in \eqref{def:inidatamod}.

\medskip

\noindent\textbf{Step 1: A set invariant under $\mathcal F$}

Note that for any probability distribution $v\in \mathcal P((-1/\varepsilon,1/\varepsilon))$ and for $x\in(-1/\varepsilon,1/\varepsilon)$, we have $0\leq T(v)(x)\leq \|\Gamma_{\sigma^2}\|_\infty$ and the function $(t,x)\mapsto \mathcal F(u)(t,x)$ is then a bounded function: 
\begin{equation}\label{eq:LinftybF}
\forall (t,x)\in[0,\bar t]\times(-1/\varepsilon,1/\varepsilon),\quad \mathcal F(u)(t,x)\leq \|u_{\varepsilon}(0,\cdot)\|_\infty+\int_0^{\bar t} \left(1+\alpha \|a\|_\infty\right)^2\|\Gamma_{\sigma^2}\|_\infty\,ds.
\end{equation}
Since all the terms in the definition of $\mathcal F(u)$ are non-negative and $T(v)>0$ for $v\in \mathcal P((-1/\varepsilon,1/\varepsilon))$, we have $\mathcal F(u)(t,x)> 0$ for $(t,x)\in(0,\bar t]\times(-1/\varepsilon,1/\varepsilon)$. To prove that $\int_{-1/\varepsilon}^{1/\varepsilon}\mathcal F(u)(t,x)\,dx\equiv 1$, we notice that for $t\in[0,\bar t]$, 
\begin{align*}
\int_{-1/\varepsilon}^{1/\varepsilon} \mathcal F(u)(t,x)\,dx&=e^{-\int_0^t\left(1+\alpha I_{u}^\varepsilon(s)\right)^2z(s)\,ds} +\int_0^t\left(1+\alpha I_{u}^\varepsilon(s)\right)^2z(s)e^{-\int_s^t\left(1+\alpha I_{u}^\varepsilon(\tau)\right)^2z(\tau)\,d\tau}\,ds,
\end{align*}
where $z(s):=\int_{-1/\varepsilon}^{1/\varepsilon}T\left(\frac{(1+\alpha a)u(s,\cdot)}{1+\alpha I_{u}^\varepsilon(s)}\right)(x)\,dx\in(0,1)$. Then $y(t):=\int_{-1/\varepsilon}^{1/\varepsilon} \mathcal F(u(t,\cdot))(x)\,dx-1$ is differentiable and
\begin{align*}
&\frac {dy}{dt}(t)=-\left(1+\alpha I_{u}^\varepsilon(t)\right)^2z(t)e^{-\int_0^t\left(1+\alpha I_{u}^\varepsilon(s)\right)^2z(s)\,ds} +\left(1+\alpha I_{u}^\varepsilon(t)\right)^2z(t)\\
&\qquad-\left(1+\alpha I_{u}^\varepsilon(t)\right)^2z(t)\int_0^t\left(1+\alpha I_{u}^\varepsilon(s)\right)^2z(s)e^{-\int_s^t\left(1+\alpha I_{u}^\varepsilon(\tau)\right)^2z(\tau)\,d\tau}\,ds\\
&\quad =-\left(1+\alpha I_{u}^\varepsilon(t)\right)^2z(t)y(t).
\end{align*}
this and $y(0)=0$, implies $y\equiv 0$, that is $ \int_{-1/\varepsilon}^{1/\varepsilon} \mathcal F(u(t,\cdot))(x)\,dx=1$ for $t\in\left[0,\bar t\right]$. Thanks to \eqref{eq:Tmoment1}, a direct calculation shows that 
\begin{align}\label{eq:estcontractx2}
&\int_{-1/\varepsilon}^{1/\varepsilon}x^2\mathcal F(u)(t,x)\,dx\leq \int_{-1/\varepsilon}^{1/\varepsilon}x^2 u_\varepsilon(0,x)\,dx+C_\varepsilon t\left(1+\max_{t\in[0,\bar t]}\int_{-1/\varepsilon}^{1/\varepsilon}x^2u(t,x)\,dx\right),
\end{align}
where $C_\varepsilon>0$ is a constant that depends on the parameters of the model: $a$, $\alpha$, $\sigma$, and $\varepsilon\geq 0$. Note that we stress the dependency of these constants in $\varepsilon\geq 0$, since it is specific to this proof. Moreover, $C_\varepsilon>0$ denotes a generic constant whose value may change from line to line. To check the Lipschitz continuity of $\mathcal F(u)$ for the distance $W_1$, we first show that $t\mapsto \mathcal F(u)(t,\cdot)$ is Lipschitz continuous for the total variation distance. More precisely, we notice that for $t_1,t_2\in[0,\bar t]$ and provided $\bar t\in(0,1]$,
\begin{align*}
&\int_{-1/\varepsilon}^{1/\varepsilon}(1+x^2)\left|\mathcal F(u)(t_1,x)-\mathcal F(u)(t_2,x)\,dx\right|\\
&\quad\leq (1+\alpha\|a\|_\infty)^2\left(1+\max_{t\in[0,\bar t]}\int_{-1/\varepsilon}^{1/\varepsilon}x^2u_\varepsilon(0,x)\,dx\right)|t_1-t_2|\\
&\qquad +(1+\alpha\|a\|_\infty)^2\left(1+ \sigma^2+\max_{t\in[0,\bar t]}\int_{-1/\varepsilon}^{1/\varepsilon}x^2u(t,x)\,dx\right)|t_1-t_2|\\
&\qquad +\int_0^{t_2}\left(1+\alpha I_{u}^\varepsilon(s)\right)^2\left(1+\sigma^2+\max_{t\in[0,\bar t]}\int_{-1/\varepsilon}^{1/\varepsilon}x^2u(t,x)\,dx\right)\\
&\phantom{\qquad +\int_0^{t_2}} e^{-\int_s^{t_2}\left(1+\alpha I_{u}^\varepsilon(\tau)\right)^2\left(\int_{-1/\varepsilon}^{1/\varepsilon}T\left(\frac{(1+\alpha a)u(\tau,\cdot)}{1+\alpha I_{u}^\varepsilon(\tau)}\right)(z)\,dz\right)\,d\tau}\\
&\phantom{\qquad +\int_0^{t_2}}  \left|1-e^{\int_{t_1}^{t_2}\left(1+\alpha I_{u}^\varepsilon(\tau)\right)^2\left(\int_{-1/\varepsilon}^{1/\varepsilon}T\left(\frac{(1+\alpha a)u(\tau,\cdot)}{1+\alpha I_{u}^\varepsilon(\tau)}\right)(z)\,dz\right)\,d\tau}\right|\,ds\\
&\quad \leq C_\varepsilon|t_1-t_2|.
\end{align*}
We can then control the $W_1$ distance as follows:
\begin{align}
&W_1\left(\mathcal F(u)(t_1,\cdot),\mathcal F(u)(t_2,\cdot)\right)\leq \max_{\|\varphi'\|_\infty\leq 1,\varphi(0)=0}\int_{-1/\varepsilon}^{1/\varepsilon}\varphi(x)\left(\mathcal F(u)(t_1,x)-\mathcal F(u)(t_2,x)\right)\,dx\nonumber\\
&\qquad\leq \frac 12\int_{-1/\varepsilon}^{1/\varepsilon}(1+x^2)\left|\mathcal F(u)(t_1,x)-\mathcal F(u)(t_2,x)\right|\,dx\leq C_\varepsilon\, |t_1-t_2|.\label{eq:Lipcst}
\end{align}
This estimate proves that $t\mapsto \mathcal F(u)(t,\cdot)$ is Lipscitz continuous for the distance $W_1$. Thanks to \eqref{eq:estcontractx2}, provided $\bar t\leq \frac 1{2C_\varepsilon}$, we have
\begin{align*}
&\int_{-1/\varepsilon}^{1/\varepsilon}x^2\mathcal F(u)(t,x)\,dx\leq \int_{-1/\varepsilon}^{1/\varepsilon}x^2 u_\varepsilon(0,x)\,dx+\frac 12\left(1+\max_{t\in[0,\bar t]}\int_{-1/\varepsilon}^{1/\varepsilon}x^2u(t,x)\,dx\right),
\end{align*}
and if 
\begin{equation}\label{eq:um2}
\max_{t\in[0,\bar t]}\int_{-1/\varepsilon}^{1/\varepsilon}x^2u(t,x)\,dx\leq \bar V:=2\int_{-1/\varepsilon}^{1/\varepsilon}x^2 u_\varepsilon(0,x)\,dx+1,
\end{equation}
then $\mathcal F(u)$ satisfies the same inequality: $\max_{t\in[0,\bar t]}\int_{-1/\varepsilon}^{1/\varepsilon}x^2\mathcal F(u)(t,x)\,dx\leq \bar V$.  


\medskip

We have shown that $F$ leaves the following set invariant:
\begin{align}
\Omega&:=\bigg\{u:[0,\bar t]\times(-1/\varepsilon,1/\varepsilon)\to\mathbb R_+;\, \max_{t\in[0,\bar t]}\int_{-1/\varepsilon}^{1/\varepsilon}x^2u(t,x)\,dx\leq \bar V,\,\int_{-1/\varepsilon}^{1/\varepsilon}u(t,x)\,dx=1 ,\nonumber\\
&\qquad \,\max_{s,t\in[0,\bar t]}W_1(u(s,\cdot),u(t,\cdot))\leq \bar C_\varepsilon|t-s|\bigg\},\label{def:setF}
\end{align}
where $\bar C_\varepsilon$ designates the constant appearing in \eqref{eq:Lipcst}. The set of Borel measures over $(-1/\varepsilon,1/\varepsilon)$  with a second moment lower than $\bar V$, and  endowed with the distance $W_1$  is compact thanks to Prokhorov's theorem (see \cite{ambrosio2005gradient}, Theorem 5.1.3, or \cite{villani2008optimal}, p.65). Since that set is compact and metric, it is also complete. The set of bounded functions valued in this complete metric set is then itself a complete metric set for the distance
\[d(u,v)
=\max_{t\in[0,\bar t]}W_1\big(u(t,\cdot),v(t,\cdot)\big).\]
We notice that $\Omega$ is a subset of this complete set, precisely the subset of functions that are Lipschitz continuous with a Lipschitz constant non greater than $\bar C_\varepsilon$. That subset is closed for the distance $d(\cdot,\cdot)$, which proves that $(\Omega, d(\cdot,\cdot))$ is a complete metric set.

\medskip

\noindent\textbf{Step 2: $\mathcal F$ is a contraction, provided $\bar t\in(0,1]$ is small enough}


Let $u,v\in \Omega$ (see \eqref{def:setF} for the definition of that set). We introduce the notation
\begin{align*}
J(s)&:=\max\bigg(\left(1+\alpha I_{u}^\varepsilon(s)\right)^2
\int_{-1/\varepsilon}^{1/\varepsilon}T\left(\frac{(1+\alpha a)u(s,\cdot)}{1+\alpha I_{u}^\varepsilon(s)}\right)(z)\,dz,\\
&\phantom{:=\max\bigg(} \quad \left(1+\alpha I_{v}^\varepsilon(s)\right)^2\int_{-1/\varepsilon}^{1/\varepsilon}T\left(\frac{(1+\alpha a)v(s,\cdot)}{1+\alpha I_{v}^\varepsilon(s)}\right)(z)\,dz\bigg).
\end{align*}
To estimate $W_1(\mathcal F(u)(t,\cdot),\mathcal F(v)(t,\cdot))$, we decompose $\mathcal F(u)(t,\cdot)$ (resp. $\mathcal F(v)(t,\cdot)$) into a sum of non-negative terms as follows :
\begin{align*}
&\mathcal F(u)(t,x)=u_\varepsilon(0,x) e^{-\int_0^tJ(s)\,ds}+u_\varepsilon(0,x)\left[e^{-\int_0^t\left(1+\alpha I_{u}^\varepsilon(s)\right)^2\left(\int_{-1/\varepsilon}^{1/\varepsilon}T\left(\frac{(1+\alpha a)u(s,\cdot)}{1+\alpha I_{u}^\varepsilon(s)}\right)(z)\,dz\right)\,ds}-e^{-\int_0^tJ(s)\,ds}\right]\nonumber\\
&\quad +\int_0^t\min\left(
\left(1+\alpha I_{u}^\varepsilon(s)\right)^2,\left(1+\alpha I_{v}^\varepsilon(s)\right)^2\right)
\min\left(1,\frac{\int_{-1/\varepsilon}^{1/\varepsilon}T\left(\frac{(1+\alpha a)v(s,\cdot)}{1+\alpha I_{v}^\varepsilon(s)}\right)(y)\,dy}{\int_{-1/\varepsilon}^{1/\varepsilon}T\left(\frac{(1+\alpha a)u(s,\cdot)}{1+\alpha I_{u}^\varepsilon(s)}\right)(y)\,dy}\right)\\
&\phantom{ertzerterte}T\left(\frac{(1+\alpha a)u(s,\cdot)}{1+\alpha I_{u}^\varepsilon(s)}\right)(x)1_{|x|< \frac 1\varepsilon}e^{-\int_s^t J(\tau)\,d\tau}\,ds\nonumber\\
&\quad +\int_0^t\min\left(
\left(1+\alpha I_{u}^\varepsilon(s)\right)^2,\left(1+\alpha I_{v}^\varepsilon(s)\right)^2\right)
\left[1-\frac{\int_{-1/\varepsilon}^{1/\varepsilon}T\left(\frac{(1+\alpha a)v(s,\cdot)}{1+\alpha I_{v}^\varepsilon(s)}\right)(y)\,dy}{\int_{-1/\varepsilon}^{1/\varepsilon}T\left(\frac{(1+\alpha a)u(s,\cdot)}{1+\alpha I_{u}^\varepsilon(s)}\right)(y)\,dy}\right]_+\\
&\phantom{ertzerterte}T\left(\frac{(1+\alpha a)u(s,\cdot)}{1+\alpha I_{u}^\varepsilon(s)}\right)(x)1_{|x|< \frac 1\varepsilon}e^{-\int_s^t J(\tau)\,d\tau}\,ds\nonumber\\
&\quad +\int_0^t\left[\left(1+\alpha I_{u}^\varepsilon(s)\right)^2-\left(1+\alpha I_{v}^\varepsilon(s)\right)^2\right]_+T\left(\frac{(1+\alpha a)u(s,\cdot)}{1+\alpha I_{u}^\varepsilon(s)}\right)(x)1_{|x|< \frac 1\varepsilon}e^{-\int_s^t J(\tau)\,d\tau}\,ds\nonumber\\
&\quad +\int_0^t\left(1+\alpha I_{u}^\varepsilon(s)\right)^2T\left(\frac{(1+\alpha a)u(s,\cdot)}{1+\alpha I_{u}^\varepsilon(s)}\right)(x)1_{|x|< \frac 1\varepsilon}\\
&\phantom{sfqsdsqdf}\bigg[e^{-\int_s^t\left(1+\alpha I_{u}^\varepsilon(\tau)\right)^2\left(\int_{-1/\varepsilon}^{1/\varepsilon}T\left(\frac{(1+\alpha a)u(\tau,\cdot)}{1+\alpha I_{u}^\varepsilon(\tau)}\right)(z)\,dz\right)\,d\tau}-e^{-\int_s^tJ(\tau)\,d\tau}\bigg]\,ds\\
&\quad=:A_u^1(t,x)+A_u^2(t,x)+A_u^3(t,x)+A_u^4(t,x)+A_u^5(t,x)+A_u^6(t,x),
\end{align*}
where we have introduced the notation $(A_u^i)_{i\in\{1,\dots,6\}}$. A similar decomposition of $\mathcal F(v)$ can be made, with $(A_v^i)_{i\in\{1,\dots,6\}}$. Since each of these terms is non-negative, we have
\begin{align*}
&W_1(\mathcal F(u)(t,\cdot),\mathcal F(v)(t,\cdot))=\sup_{\|\varphi'\|_\infty\leq 1}\int_{-1/\varepsilon}^{1/\varepsilon}\varphi(x)\,d\left(\mathcal F(u)(t,\cdot),\mathcal F(v)(t,\cdot)\right)(x)\\
&\quad =\sup_{\|\varphi'\|_\infty\leq 1,\varphi(0)=0}\int_{-1/\varepsilon}^{1/\varepsilon}\varphi(x)\,d\left(\mathcal F(u)(t,\cdot),\mathcal F(v)(t,\cdot)\right)(x)\\
&\quad\leq \sup_{\|\varphi'\|_\infty\leq 1,\varphi(0)=0}\int_{-1/\varepsilon}^{1/\varepsilon}\varphi(x)\,d(A_u^1(t,\cdot)-A_v^1(t,\cdot))(x)\\
&\qquad +\sup_{\|\varphi'\|_\infty\leq 1,\varphi(0)=0}\int_{-1/\varepsilon}^{1/\varepsilon}\varphi(x)\,d(A_u^1(t,\cdot)-A_v^1(t,\cdot))(x)\\
&\qquad +\int_{-1/\varepsilon}^{1/\varepsilon}|x|A_u^2(t,x)\,dx+\int_{-1/\varepsilon}^{1/\varepsilon}|x|A_v^2(t,x)\,dx+\int_{-1/\varepsilon}^{1/\varepsilon}|x|A_u^4(t,x)\,dx+\int_{-1/\varepsilon}^{1/\varepsilon}|x|A_v^4(t,x)\,dx\\
&\qquad +\int_{-1/\varepsilon}^{1/\varepsilon}|x|A_u^5(t,x)\,dx+\int_{-1/\varepsilon}^{1/\varepsilon}|x|A_v^5(t,x)\,dx+\int_{-1/\varepsilon}^{1/\varepsilon}|x|A_u^6(t,x)\,dx+\int_{-1/\varepsilon}^{1/\varepsilon}|x|A_v^6(t,x)\,dx,
\end{align*}
We can estimate
\begin{align}
\int_{-1/\varepsilon}^{1/\varepsilon}|x|A_u^2(t,x)\,dx&\leq \left|e^{-\int_0^t\left(1+\alpha I_{u}^\varepsilon(s)\right)^2\left(\int_{-1/\varepsilon}^{1/\varepsilon}T\left(\frac{(1+\alpha a)u(s,\cdot)}{1+\alpha I_{u}^\varepsilon(s)}\right)(z)\,dz\right)\,ds}-e^{-\int_0^tJ(s)\,ds}\right|\,\int_{-1/\varepsilon}^{1/\varepsilon}|x|u_\varepsilon(0,x)\,dx\nonumber\\
&\leq \sqrt{\bar V}\int_0^t\left|\left(1+\alpha I_{u}^\varepsilon(s)\right)^2\left(\int_{-1/\varepsilon}^{1/\varepsilon}T\left(\frac{(1+\alpha a)u(s,\cdot)}{1+\alpha I_{u}^\varepsilon(s)}\right)(z)\,dz\right)-J(s)\right|\,ds\nonumber\\
&\leq C_\varepsilon\sqrt{\bar V}\int_0^tW_1\left(u(s,\cdot),v(s,\cdot)\right)\,ds.\label{est:Au2a}
\end{align}
Notice that the constant $C_\varepsilon>0$ is independent from $u_\varepsilon(0,\cdot)\in \mathcal P((-1/\varepsilon,1/\varepsilon))$. To obtain that estimate, we have used the fact that 
\begin{align*}
&\left|\int_{-1/\varepsilon}^{1/\varepsilon}T\left(\frac{(1+\alpha a)u(s,\cdot)}{1+\alpha I_{u}^\varepsilon(s)}\right)(x)-T\left(\frac{(1+\alpha a)v(s,\cdot)}{1+\alpha I_{v}^\varepsilon(s)}\right)(x)\,dx\right|\\
&\quad \leq \bigg|\int_{-1/\varepsilon}^{1/\varepsilon}\frac{(1+\alpha a)u(s,x_*)}{1+\alpha I_{u}^\varepsilon(s)}\Bigg(\int\left(\int_{-1/\varepsilon}^{1/\varepsilon}\Gamma_{\sigma^2}\left(x-\frac {x_*+x_*'}2\right)\,dx\right)\\
&\qquad \phantom{qdfqefersetqetadfgsdfa}\left(\frac{(1+\alpha a)u(s,x_*')}{1+\alpha I_{u}^\varepsilon(s)}-\frac{(1+\alpha a)v(s,x_*')}{1+\alpha I_{v}^\varepsilon(s)}\right)\,dx_*'\Bigg)\,dx_*\bigg|\\
&\qquad + \bigg|\int_{-1/\varepsilon}^{1/\varepsilon}\frac{(1+\alpha a)v(s,x_*')}{1+\alpha I_{v}^\varepsilon(s)}\Bigg(\int\left(\int_{-1/\varepsilon}^{1/\varepsilon}\Gamma_{\sigma^2}\left(x-\frac {x_*+x_*'}2\right)\,dx\right)\\
&\qquad \phantom{qdfqefersetqetadfgsdfa}\left(\frac{(1+\alpha a)u(s,x_*)}{1+\alpha I_{u}^\varepsilon(s)}-\frac{(1+\alpha a)v(s,x_*)}{1+\alpha I_{v}^\varepsilon(s)}\right)\,dx_*\Bigg)\,dx_*'\bigg|\\
&\quad \leq \left(\int_{-1/\varepsilon}^{1/\varepsilon}\frac{(1+\alpha a)v(s,x_*')}{1+\alpha I_{v}^\varepsilon(s)}\,dx_*'+\int_{-1/\varepsilon}^{1/\varepsilon}\frac{(1+\alpha a)u(s,x_*)}{1+\alpha I_{u}^\varepsilon(s)}\,dx_*\right)\\
&\qquad \hspace{1cm}\left(\max_{x_*,x_*'\in[-1/\varepsilon,1/\varepsilon]}\left|\int_{-1/\varepsilon}^{1/\varepsilon}\frac 12\Gamma_{\sigma^2}'\left(x-\frac {x_*+x_*'}2\right)\,dx\right|\right)\\
&\qquad \hspace{1cm}W_1\left(\frac{(1+\alpha a)u(s,\cdot)}{1+\alpha I_{u}^\varepsilon(s)},\frac{(1+\alpha a)v(s,\cdot)}{1+\alpha I_{v}^\varepsilon(s)}\right)\\
&\quad \leq 2(1+\alpha\|a\|_\infty)\|\Gamma_{\sigma^2}'\|_{L^1(\mathbb R)} W_1\left(\frac{(1+\alpha a)u(s,\cdot)}{1+\alpha I_{u}^\varepsilon(s)},\frac{(1+\alpha a)v(s,\cdot)}{1+\alpha I_{v}^\varepsilon(s)}\right)\\
&\quad \leq C_\varepsilon W_1\left(u(s,\cdot),v(s,\cdot)\right)+ C_\varepsilon\alpha|I_u^\varepsilon(s)-I_v^\varepsilon(s)|\leq C_\varepsilon W_1\left(u(s,\cdot),v(s,\cdot)\right),
\end{align*}
where we have used the estimate
\begin{align}\label{est:IumIv}
\left|I_u^\varepsilon(s)-I_v^\varepsilon(s)\right|= \left|\int_{-1/\varepsilon}^{1/\varepsilon} a(x) \left(u(s,x)-v(s,x)\right)\,ds\right|\leq \|a'\|_\infty W_1(u(s,\cdot),v(s,\cdot)).
\end{align}
The term $\int_{-1/\varepsilon}^{1/\varepsilon}|x|A_v^4(t,x)\,dx$ can be estimated with a similar argument:
\begin{align*}
&\int_{-1/\varepsilon}^{1/\varepsilon}|x|A_v^4(t,x)\,dx\\
&\quad \leq C_\varepsilon\sqrt{\bar V}\frac{\int_0^t\left|\int_{-1/\varepsilon}^{1/\varepsilon} T\left(\frac{(1+\alpha a)u(s,\cdot)}{1+\alpha I_{u}^\varepsilon(s)}\right)(x)\,dx-\int_{-1/\varepsilon}^{1/\varepsilon} T\left(\frac{(1+\alpha a)v(s,\cdot)}{1+\alpha I_{v}^\varepsilon(s)}\right)(x)\,dx\right|\,ds}{\int_{-1/\varepsilon}^{1/\varepsilon} T\left(\frac{(1+\alpha a)u(s,\cdot)}{1+\alpha I_{u}^\varepsilon(s)}\right)(x)\,dx}\\
&\quad \leq C_\varepsilon\int_0^tW_1\left(u(s,\cdot),v(s,\cdot)\right)\,ds,
\end{align*}
where we have used the lower bound
\begin{align*}
\int_{-\min(\sqrt{\bar V},1/\varepsilon)}^{\min(\sqrt{\bar V},1/\varepsilon)}u(s,x)\,dx&=1-\int_{(-1/\varepsilon,1/\varepsilon)\setminus[-\sqrt{2\bar V},\sqrt{2\bar V}]}u(s,x)\,dx\\
&\geq 1-\frac 1{2\bar V} \int_{-1/\varepsilon}^{1/\varepsilon}x^2u_\varepsilon(s,x)\,dx\geq \frac 12,
\end{align*}
to show
\begin{align*}
&\int_{-1/\varepsilon}^{1/\varepsilon}T\left(\frac{(1+\alpha a)v(s,\cdot)}{1+\alpha I_{v}^\varepsilon(s)}\right)(x)\,dx\\
&\quad \geq \frac{\left(\inf_{x_*,x_*'\in [-\sqrt{2\bar V}\sqrt{2\bar V}]}\int_{-1/\varepsilon}^{1/\varepsilon}\Gamma_{\sigma^2}\left(x-\frac{x_*+x_*'}2\right)\,dx\right)}{(1+\alpha\|a\|_\infty)^2}
\left(\int_{-\min(\sqrt{\bar V},1/\varepsilon)}^{\min(\sqrt{\bar V},1/\varepsilon)}u(s,x)\,dx\right)^2\geq \frac 1{C_\varepsilon}>0.
\end{align*}
We recall that the estimates in this proof are not uniform in $\varepsilon\in[0,\infty)$, and in this estimate in particular, the constant $C_\varepsilon>0$ here is large if $\varepsilon\geq 0$ is large.
Estimates on $\int_{-1/\varepsilon}^{1/\varepsilon}|x|A_u^5(t,x)\,dx$ and $\int_{-1/\varepsilon}^{1/\varepsilon}|x|A_u^6(t,x)\,dx$ can also be obtained from similar arguments:
\begin{align*}
\int_{-1/\varepsilon}^{1/\varepsilon}|x|A_u^5(t,x)\,dx&\leq C_\varepsilon\alpha\sqrt{\bar V} \int_0^t |I_u^\varepsilon(s)-I_v^\varepsilon(s)|\,ds\leq C_\varepsilon\int_0^tW_1\left(u(s,\cdot),v(s,\cdot)\right)\,ds,
\end{align*}
\begin{align*}
&\int_{-1/\varepsilon}^{1/\varepsilon}|x|A_u^6(t,x)\,dx\\
&\quad \leq C_\varepsilon\sqrt{\bar V} \sup_{s\in[0,t]} \int_s^t \left|\left(1+\alpha I_{u}^\varepsilon(\tau)\right)^2\left(\int_{-1/\varepsilon}^{1/\varepsilon}T\left(\frac{(1+\alpha a)u(\tau,\cdot)}{1+\alpha I_{u}^\varepsilon(\tau)}\right)(z)\,dz\right)-J(\tau)\right|\,d\tau\\
&\quad\leq C_\varepsilon\,\sup_{s\in[0,t]} \int_s^t\Bigg|\left(1+\alpha I_{u}^\varepsilon(\tau)\right)^2\left(\int_{-1/\varepsilon}^{1/\varepsilon}T\left(\frac{(1+\alpha a)u(\tau,\cdot)}{1+\alpha I_{u}^\varepsilon(\tau)}\right)(z)\,dz\right)\\
&\qquad  -\left(1+\alpha I_{v}^{\varepsilon}(\tau)\right)^2\left(\int_{-1/\varepsilon}^{1/\varepsilon}T\left(\frac{(1+\alpha a)v(\tau,\cdot)}{1+\alpha I_{v}^{\varepsilon}(\tau)}\right)(z)\,dz\right)\Bigg|\,d\tau\\
&\quad \leq C_\varepsilon\int_0^tW_1\left(u(s,\cdot),v(s,\cdot)\right)\,ds.
\end{align*}
We notice that $u_\varepsilon(0,\cdot)$ is defined by \eqref{def:inidatamod}, and is therefore independent from $u,v\in\Omega$. It implies $A_u^1=A_v^1$, and thus $\sup_{\|\varphi'\|_\infty\leq 1,\varphi(0)=0}\int_{-1/\varepsilon}^{1/\varepsilon}\varphi(x)\,d(A_u^1(t,\cdot)-A_v^1(t,\cdot))(x)=0$. We are then left with the estimation of $\sup_{\|\varphi'\|_\infty\leq 1,\varphi(0)=0}\int_{-1/\varepsilon}^{1/\varepsilon}\varphi(x)\,d(A_u^3(t,\cdot)-A_v^3(t,\cdot))(x)$:
\begin{align*}
&
\int_{-1/\varepsilon}^{1/\varepsilon}\varphi(x)\,d(A_u^3(t,\cdot)-A_v^3(t,\cdot))(x)\leq (1+\alpha \|a\|_\infty)^2\\
&\qquad \int_0^t \left|\int_{-1/\varepsilon}^{1/\varepsilon}\varphi(x)\left(T\left(\frac{(1+\alpha a)u(s,\cdot)}{1+\alpha I_{u}^\varepsilon(s)}\right)(x)-T\left(\frac{(1+\alpha a)v(s,\cdot)}{1+\alpha I_{v}^\varepsilon(s)}\right)(x)\right)\,dx\right|\,dt.
\end{align*}
We notice that, thanks to the definition \eqref{def:T} of the operator $T$,
\begin{align*}
&\int_{-1/\varepsilon}^{1/\varepsilon}\varphi(x)\left(T\left(\frac{(1+\alpha a)u(s,\cdot)}{1+\alpha I_{u}^\varepsilon(s)}\right)(x)-T\left(\frac{(1+\alpha a)v(s,\cdot)}{1+\alpha I_{v}^\varepsilon(s)}\right)(x)\right)\,dx\\
&\quad = \int_{[-1/\varepsilon,1/\varepsilon]^5}\varphi(x)\left(\Gamma_{\sigma^2}\left(x-\frac {x_*+x_*'}2\right)-\Gamma_{\sigma^2}\left(x-\frac {X_*+X_*'}2\right)\right)\,d\pi(x_*,X_*)\,d\pi(x_*',X_*')\,dx,
\end{align*}
where $\pi$ is any transport plan between $\frac{(1+\alpha a)u(s,x_*)}{1+\alpha I_{u}^\varepsilon(s)}$ and $\frac{(1+\alpha a)v(s,x_*')}{1+\alpha I_{v}^\varepsilon(s)}$. If $\varepsilon>0$, we estimate this quantity as follows:
\begin{align}
&\left|\int_{-1/\varepsilon}^{1/\varepsilon}\varphi(x)\left(T\left(\frac{(1+\alpha a)u(s,\cdot)}{1+\alpha I_{u}^\varepsilon(s)}\right)(x)-T\left(\frac{(1+\alpha a)v(s,\cdot)}{1+\alpha I_{v}^\varepsilon(s)}\right)(x)\right)\,dx\right|\nonumber\\
&\quad \leq \int_{[-1/\varepsilon,1/\varepsilon]^4}\left(\int_{-1/\varepsilon}^{1/\varepsilon}|\varphi(x)|\,dx\right)\|\Gamma_{\sigma^2}'\|_\infty \left|\frac {X_*+X_*'}2-\frac {x_*+x_*'}2\right|\,d\pi(x_*,X_*)\,d\pi(x_*',X_*')\,dx\nonumber\\
&\quad \leq \left(\int_{-1/\varepsilon}^{1/\varepsilon}|x|\,dx\right)\frac{\|\Gamma_{\sigma^2}'\|_\infty}2\int_{[-1/\varepsilon,1/\varepsilon]^4} \left(|X_*-x_*|+|X_*'-x_*'|\right)\,d\pi(x_*,X_*)\,d\pi(x_*',X_*')\,dx\nonumber\\
&\quad \leq C_\varepsilon \int_{[-1/\varepsilon,1/\varepsilon]^2}\left|x_*-X_*\right|\,d\pi(x_*,X_*).\label{est:W1eps1}
\end{align}
If $\varepsilon=0$, $[-1/\varepsilon,1/\varepsilon]=\R$, and then, thanks to changes of variables,
\begin{align}
&\left|\int_{\R}\varphi(x)\left(T\left(\frac{(1+\alpha a)u(s,\cdot)}{1+\alpha I_{u}^\varepsilon(s)}\right)(x)-T\left(\frac{(1+\alpha a)v(s,\cdot)}{1+\alpha I_{v}^\varepsilon(s)}\right)(x)\right)\,dx\right|\nonumber\\
&\quad =\left|\int_{\R^5}\left(\varphi\left(\tilde x+\frac {x_*+x_*'}2\right)-\varphi\left(\tilde x+\frac {X_*+X_*'}2\right)\right)\Gamma_{\sigma^2}\left(\tilde x\right)\,d\tilde x\,d\pi(x_*,X_*)\,d\pi(x_*',X_*')\right|\nonumber\\
&\quad \leq \left(\int_{\R} \Gamma_{\sigma^2}(\tilde x)\,d\tilde x\right)\int_{\R^4}\left|\frac {x_*+x_*'}2-\frac {X_*+X_*'}2\right|\,d\pi(x_*,X_*)\,d\pi(x_*',X_*')\nonumber\\
&\quad \leq \int_{\R^2}\left|x_*-X_*\right|\,d\pi(x_*,X_*),\label{est:W1eps2}
\end{align}
and then, taking the infimum over the transfer plans $\pi$ in either \eqref{est:W1eps1} or \eqref{est:W1eps2}, we obtain, for any $\varepsilon\geq 0$,
\begin{align*}
&\left|\int_{-1/\varepsilon}^{1/\varepsilon}\varphi(x)\left(T\left(\frac{(1+\alpha a)u(s,\cdot)}{1+\alpha I_{u}^\varepsilon(s)}\right)(x)-T\left(\frac{(1+\alpha a)v(s,\cdot)}{1+\alpha I_{v}^\varepsilon(s)}\right)(x)\right)\,dx\right|\\
&\quad \leq C_\varepsilon W_1\left(\frac{(1+\alpha a)u(s,\cdot)}{1+\alpha I_{u}^\varepsilon(s)},\frac{(1+\alpha a)v(s,\cdot)}{1+\alpha I_{v}^\varepsilon(s)}\right)\leq C_\varepsilon W_1(u(s,\cdot),v(s,\cdot)).
\end{align*}

The estimates we have obtained can be combined to obtain, for $t\in[0,\bar t]$ (where $\bar t\in(0,1]$):
\begin{align*}
&W_1(\mathcal F(u)(t,\cdot),\mathcal F(v)(t,\cdot))\leq C_\varepsilon t\max_{s\in [0,t]} W_1\left(u(s,\cdot),v(s,\cdot)\right),
\end{align*}
where the constant $C_\varepsilon>0$ is independent from $u_\varepsilon(0,\cdot)\in \mathcal P((-1/\varepsilon,1/\varepsilon))$. We can choose $\bar t:=\min\left(\frac 1{2C_\varepsilon},1\right)>0$ to have 
\begin{align}
\max_{t\in[0,\bar t]} W_1(\mathcal F(u)(t,\cdot),\mathcal F(v)(t,\cdot))&\leq \frac 12 \max_{s\in [0,\bar t]} W_1\left(u(s,\cdot),v(s,\cdot)\right).\label{eq:con}
\end{align}

\medskip

\noindent\textbf{Step 3: Existence, uniqueness and regularity of solutions to the truncated model}

$\mathcal F$ is then a strict contraction for the distance $d(u,v)=\sup_{t\in[0,\bar t]}W_1(u(t,\cdot),v(t,\cdot))$ on $\Omega$ (see \eqref{def:setF}), and the set $\Omega$ is invariant under $\mathcal F$ thanks to Step~1. The set $\Omega$ equipped with the distance $d(u,v)=\max_{t\in[0,\bar t]}W_1(u(t,\cdot),v(t,\cdot))$ is a complete set, we can thus apply the Banach fixed-point Theorem to show that $\mathcal F$ has a unique fixed point $u_\varepsilon\in \Omega$ and such that
\[u_\varepsilon(t,x)=\mathcal F(u_\varepsilon)(t,x),\quad u_\varepsilon(t,\cdot)\in\mathcal P((-1/\varepsilon,1/\varepsilon)),\quad \int_{-1/\varepsilon}^{1/\varepsilon}x^2u(t,x)\,dx\leq \bar V,\]
for $t\in[0,\bar t]$ and $x\in(-1/\varepsilon,1/\varepsilon)$. $u_\varepsilon$ then satisfies \eqref{eq:intformulation} for $(t,x)\in [0,\bar t]\times(-1/\varepsilon,1/\varepsilon)$. We iterate this construction with $\tilde u_\varepsilon(0,\cdot):=u_{\varepsilon}(k\bar t,\cdot)$, to extend the solution to $t\in[k\bar t,(k+1)\bar t]$, since $\bar t>0$ is independent from the initial distribution $u_\varepsilon(0,\cdot)$. We can then construct a solution $u\in L^\infty\big([0,\infty)\times(-1/\varepsilon,1/\varepsilon)\big)$ such that $u(t,\cdot)\in\mathcal P((-1/\varepsilon,1/\varepsilon))$ for $t\in[0,\infty)$.

\medskip

We notice that $T\left(\frac{(1+\alpha a)u_\varepsilon(s,\cdot)}{1+\alpha I_{u_\varepsilon}^\varepsilon(s)}\right)(x)$ is smooth in the variable $x\in(-1/\varepsilon,1/\varepsilon)$, and $u_\varepsilon$ can then be written as $u_\varepsilon(t,x)=u_\varepsilon(0,x)g_\varepsilon(t)+h_\varepsilon(t,x)$ with $h_\varepsilon$ smooth in the variable $x$:
\[g_\varepsilon(t):=e^{-\int_0^t\left(1+\alpha I_{u_\varepsilon}^\varepsilon(s)\right)^2\left(\int_{-1/\varepsilon}^{1/\varepsilon}T\left(\frac{(1+\alpha a)u_\varepsilon(s,\cdot)}{1+\alpha I_{u_\varepsilon}^\varepsilon(s)}\right)(z)\,dz\right)\,ds},\]
\begin{align*}
h_\varepsilon(t,x)&:=\int_0^t\left(1+\alpha I_{u_\varepsilon}^\varepsilon(s)\right)^2T\left(\frac{(1+\alpha a)u_\varepsilon(s,\cdot)}{1+\alpha I_{u_\varepsilon}^\varepsilon(s)}\right)(x)1_{|x|<\frac 1\varepsilon}\\
&\phantom{:=\int_0^t}\quad e^{-\int_s^t\left(1+\alpha I_{u_\varepsilon}^\varepsilon(\tau)\right)^2\left(\int_{-1/\varepsilon}^{1/\varepsilon}T\left(\frac{(1+\alpha a)u_\varepsilon(\tau,\cdot)}{1+\alpha I_{u_\varepsilon}^\varepsilon(\tau)}\right)(z)\,dz\right)\,d\tau}\,ds.
\end{align*}
Then, 
 \begin{align}
 \partial_t g_\varepsilon(t)&= -\left(1+\alpha I_{u_\varepsilon}^\varepsilon(t)\right)^2\left(\int_{-1/\varepsilon}^{1/\varepsilon}T\left(\frac{(1+\alpha a)u_\varepsilon(t,\cdot)}{1+\alpha I_{u_\varepsilon}^\varepsilon(t)}\right)(z)\,dz\right)\nonumber\\
 &\qquad e^{-\int_0^t\left(1+\alpha I_{u_\varepsilon}^\varepsilon(s)\right)^2\left(\int_{-1/\varepsilon}^{1/\varepsilon}T\left(\frac{(1+\alpha a)u_\varepsilon(s,\cdot)}{1+\alpha I_{u_\varepsilon}^\varepsilon(s)}\right)(z)\,dz\right)\,ds},\label{eq:dtg}
 \end{align}
so that $g_\varepsilon$ is continuously differentiable in $t$, with $\|\partial_t g_\varepsilon\|_\infty\leq (1+\alpha\|a\|_\infty)^2$. For $h_\varepsilon$, we have
\begin{align*}
\partial_t h_\varepsilon(t,x)&=\left(1+\alpha I_{u_\varepsilon}^\varepsilon(t)\right)^2T\left(\frac{(1+\alpha a)u_\varepsilon(t,\cdot)}{1+\alpha I_{u_\varepsilon}^\varepsilon(t)}\right)(x)1_{|x|< \frac 1\varepsilon}\\
&\quad -\left(1+\alpha I_{u_\varepsilon}^\varepsilon(t)\right)^2\left(\int_{-1/\varepsilon}^{1/\varepsilon}T\left(\frac{(1+\alpha a)u_\varepsilon(t,\cdot)}{1+\alpha I_{u_\varepsilon}^\varepsilon(t)}\right)(z)\,dz\right) \int_0^t\left(1+\alpha I_{u_\varepsilon}^\varepsilon(s)\right)^2\\
&\qquad T\left(\frac{(1+\alpha a)u_\varepsilon(s,\cdot)}{1+\alpha I_{u_\varepsilon}^\varepsilon(s)}\right)(x)1_{|x|< \frac 1\varepsilon} e^{-\int_s^t\left(1+\alpha I_{u_\varepsilon}^\varepsilon(\tau)\right)^2\left(\int_{-1/\varepsilon}^{1/\varepsilon}T\left(\frac{(1+\alpha a)u_\varepsilon(\tau,\cdot)}{1+\alpha I_{u_\varepsilon}^\varepsilon(\tau)}\right)(z)\,dz\right)\,d\tau}\,ds,
\end{align*}
and thus $h_\varepsilon$ is continuously differentiable in $t$, and it is also continuously differentiable in $x\in(-1/\varepsilon,1/\varepsilon)$. This argument can be iterated to show that $h_\varepsilon$ and $g_\varepsilon$ are smooth. Then, 
\begin{align*}
\partial_t u_\varepsilon(t,x)&=u_\varepsilon(0,x) g_\varepsilon'(t)+\partial_t h_\varepsilon(t,x)\\
&=-\left(1+\alpha I_{u_\varepsilon}^\varepsilon(t)\right)^2\left(\int_{-1/\varepsilon}^{1/\varepsilon}T\left(\frac{(1+\alpha a)u_\varepsilon(t,\cdot)}{1+\alpha I_{u_\varepsilon}^\varepsilon(t)}\right)(z)\,dz\right)\big(u_\varepsilon(0,x)g_\varepsilon(t)+h_\varepsilon(t,x)\big)\\
&\quad +\left(1+\alpha I_{u_\varepsilon}(t)\right)^2T\left(\frac{(1+\alpha a)u_\varepsilon(t,\cdot)}{1+\alpha I_{u_\varepsilon}^\varepsilon(t)}\right)(x)1_{|x|< \frac 1\varepsilon},
\end{align*}
which proves that $u_\varepsilon$ satisfies \eqref{eq:modifiedpb}. Conversely, an integration in $t$ shows that any solution of \eqref{eq:modifiedpb} that is continuously differentiable in $t$ and $x$ satisfy \eqref{eq:intformulation}, and the uniqueness of solutions of \eqref{eq:modifiedpb} follows from the local contraction argument above.
\end{proof}

\section{Uniform estimates on solutions of the truncated model}\label{App:unifestim}

In this section we derive estimates on solutions $u_\varepsilon$ of \eqref{eq:modifiedpb} that are uniform in $\varepsilon>0$. In particular we derive a bound $\int_{\mathbb R}x^2 u_\varepsilon(t,x)\,dx\leq C$, which is an extension of estimate \eqref{eq:moment2} obtained on the original model. This extension require a rather long an technical proof because restricting the variable $x$ to $(-1/\varepsilon,1/\varepsilon)$ breaks some properties of the sexual reproduction operator $T$, and in particular it no longer conserves the center of mass. This uniform estimate is necessary to pass to the limit $\varepsilon\to 0$ on $u_\varepsilon$, which we use to estimate the tails of solutions $n$ of the original problem (see Lemma~\ref{lem:tailprecise}).  In the statement below, we recall the notation $F$ defined by \eqref{def:F}. 

\begin{lem}\label{thm:unifestim}
Let $a\in W^{1,\infty}(\mathbb R,\mathbb R_+)$ with a compact support, $\alpha,\sigma>0$, $\varepsilon\geq 0$. For $u^0\in\mathcal P_2((-1/\varepsilon,1/\varepsilon))\cap L^\infty((-1/\varepsilon,1/\varepsilon)))$, let $u_\varepsilon\in L^\infty([0,\infty)\times\mathbb R)$ the solution of \eqref{eq:modifiedpb} provided by Theorem~\ref{thm:existence}. 

There are $\bar \alpha>0$, $\bar \varepsilon>0$ and $\bar C>0$ such that if $\alpha\in(0,\bar \alpha)$ and $\varepsilon\in(0,\bar\varepsilon)$, then $Z_{u_\varepsilon}$ is uniformly bounded, $\|Z_{u_\varepsilon}\|_\infty\leq \bar C$. Moreover, for $t\geq 0$ and $x\in(-1/\varepsilon,1/\varepsilon)$,
\begin{equation}\label{est:boundsApp}
\int_{-1/\varepsilon}^{1/\varepsilon}y^2u_\varepsilon(t,y)\,dy\leq \bar C,\quad u_\varepsilon(t,x)\geq \min\left(n^0(x),\frac 1{\bar C} \Gamma_{\sigma^2}\left(|x|+1/\bar C\right)\right),
\end{equation}
\begin{equation}\label{est:W2appendix}
 W_2\left(u_\varepsilon(t,\cdot),\Gamma_{2\sigma^2}(\cdot-Z_{u_\varepsilon}(t)\right)\leq (W_2(n^0,\Gamma_{2\sigma^2}(\cdot-Z^0))+\bar C\varepsilon) e^{-\left(\frac{\sqrt 3}2-\frac{1+\bar C\alpha}{\sqrt 2}\right)t}+\bar C\sqrt \alpha+\bar C\sqrt\varepsilon.
\end{equation}
\end{lem}

To prove Lemma~\ref{thm:unifestim}, we will use an integral formulation of the problem~\eqref{eq:modifiedpb}:
\begin{align}
&u_\varepsilon(t,x)=u_\varepsilon(0,x)e^{-\int_0^t\left(1+\alpha I_{u_\varepsilon}^\varepsilon(s)\right)^2\left(\int_{-1/\varepsilon}^{1/\varepsilon}T\left(\frac{(1+\alpha a)u_\varepsilon(s,\cdot)}{1+\alpha I_{u_\varepsilon}(s)}\right)(z)\,dz\right)\,ds}\nonumber\\
&\quad +\int_0^t\left(1+\alpha I_{u_\varepsilon}^\varepsilon(s)\right)^2T\left(\frac{(1+\alpha a)u_\varepsilon(s,\cdot)}{1+\alpha I_{u_\varepsilon}^\varepsilon(s)}\right)(x)1_{|x|< \frac 1\varepsilon}\nonumber\\
&\phantom{\quad +\int_0^t}e^{-\int_s^t\left(1+\alpha I_{u_\varepsilon}^\varepsilon(\tau)\right)^2\left(\int_{-1/\varepsilon}^{1/\varepsilon}T\left(\frac{(1+\alpha a)u_\varepsilon(\tau,\cdot)}{1+\alpha I_{u_\varepsilon}^\varepsilon(\tau)}\right)(z)\,dz\right)\,d\tau}\,ds.\label{eq:integralformulation2}
\end{align}

\begin{proof}[Proof of Lemma~\ref{thm:unifestim}]

\noindent\textbf{Step 1: Estimate on $\int_{-1/\varepsilon}^{1/\varepsilon}(x-Z_{u_\varepsilon}(t))^2u_\varepsilon(t,x)\,dx$}

We notice that $Z_{u_\varepsilon}(t,x):=\int_{-1/\varepsilon}^{1/\varepsilon}x \left(u^0(x)g_\varepsilon(t,x)+h_\varepsilon(t,x)\right)\,dx$ is continuously differentiable, and $Z_{u_\varepsilon}(t)$ is continuously differentiable. We can therefore define, for $t\geq 0$,
\begin{align}
&\frac d{dt}\int_{-1/\varepsilon}^{1/\varepsilon}(x-Z_{u_\varepsilon}(t))^2u_\varepsilon(t,x)\,dx\nonumber\\
&=-2\partial_t Z_{u_\varepsilon}(t)\int_{-1/\varepsilon}^{1/\varepsilon}(x-Z_{u_\varepsilon}(t))u_\varepsilon(t,x)\,dx +\int_{-1/\varepsilon}^{1/\varepsilon}(x-Z_{u_\varepsilon}(t))^2\partial_t u_\varepsilon(t,x)\,dx\nonumber\\
&\quad= 0+\left(1+\alpha I_{u_\varepsilon}^\varepsilon(t)\right)^2 \Bigg[\int_{-1/\varepsilon}^{1/\varepsilon} (x-Z_{u_\varepsilon}(t))^2T\left(\frac{(1+\alpha a)u_\varepsilon(t,\cdot)}{1+\alpha I_{u_\varepsilon}^\varepsilon(t)}\right)(x)\,dx\nonumber\\
&\qquad -\left(\int_{-1/\varepsilon}^{1/\varepsilon}T\left(\frac{(1+\alpha a)u_\varepsilon(t,\cdot)}{1+\alpha I_{u_\varepsilon}^\varepsilon(t)}\right)(x)\,dx\right)\int_{-1/\varepsilon}^{1/\varepsilon}(x-Z_{u_\varepsilon}(t))^2u_\varepsilon(t,x)\,dx\Bigg].\label{est:M2-1}
\end{align}
If we introduce the notation 
\begin{equation}\label{def:hatZ}
\hat Z(t):=\int_{-1/\varepsilon}^{1/\varepsilon}x \frac{(1+\alpha a)u_\varepsilon(t,x)}{1+\alpha I_{u_\varepsilon}^\varepsilon(t)}\,dx,
\end{equation}
we can estimate the first term in the brackets above as follows, thanks to \eqref{eq:Tmoment2}:
\begin{align}
&\int_{-1/\varepsilon}^{1/\varepsilon} \left(x-Z_{u_\varepsilon}(t)\right)^2T\left(\frac{(1+\alpha a)u_\varepsilon(t,\cdot)}{1+\alpha I_{u_\varepsilon}^\varepsilon(t)}\right)(x)\,dx\leq\int_{\mathbb R} (x-Z_{u_\varepsilon}(t))^2T\left(\frac{(1+\alpha a)u_\varepsilon(t,\cdot)}{1+\alpha I_{u_\varepsilon}^\varepsilon(t)}\right)(x)\,dx\nonumber\\
&\quad = \left(\hat Z(t)-Z_{u_\varepsilon}(t)\right)^2+ \int_{\mathbb R} \left(x-\hat Z(t)\right)^2T\left(\frac{(1+\alpha a)u_\varepsilon(t,\cdot)}{1+\alpha I_{u_\varepsilon}^\varepsilon(t)}\right)(x)\,dx\nonumber\\
&\quad = \left(\hat Z(t)-Z_{u_\varepsilon}(t)\right)^2+ \sigma^2+\frac 12\int_{-1/\varepsilon}^{1/\varepsilon} \left(x-\hat Z(t)\right)^2\frac{(1+\alpha a(x))u_\varepsilon(t,x)}{1+\alpha I_{u_\varepsilon}^\varepsilon(t)}\,dx\nonumber\\
&\quad \leq \left(\hat Z(t)-Z_{u_\varepsilon}(t)\right)^2+ \sigma^2+\frac{1+C\alpha}2\int_{-1/\varepsilon}^{1/\varepsilon} \left(x-\hat Z(t)\right)^2u_\varepsilon(x)\,dx\nonumber\\
&\quad \leq \frac{3+C\alpha}2\left|\hat Z(t)-Z_{u_\varepsilon}(t)\right|^2+ \sigma^2+\frac{1+C\alpha}2\int_{-1/\varepsilon}^{1/\varepsilon} \left(x-Z_{u_\varepsilon}(t)\right)^2u_\varepsilon(x)\,dx.
\label{est:truc1}
\end{align}
Next, we estimate the second term in the brackets of \eqref{est:M2-1}:
\begin{align*}
&\left(\int_{-1/\varepsilon}^{1/\varepsilon}T\left(\frac{(1+\alpha a)u_\varepsilon(t,\cdot)}{1+\alpha I_{u_\varepsilon}^\varepsilon(t)}\right)(x)\,dx\right)\int_{-1/\varepsilon}^{1/\varepsilon}(x-Z_{u_\varepsilon}(t))^2u_\varepsilon(t,x)\,dx\nonumber\\
&\quad = \left(1-\int_{[-1/\varepsilon,1/\varepsilon]^c}T\left(\frac{(1+\alpha a)u_\varepsilon(t,\cdot)}{1+\alpha I_{u_\varepsilon}^\varepsilon(t)}\right)(x)\,dx\right)\int_{-1/\varepsilon}^{1/\varepsilon}(x-Z_{u_\varepsilon}(t))^2u_\varepsilon(t,x)\,dx\nonumber\\
&\quad \geq \left(1-\varepsilon^2 \int_{\mathbb R} x^2T\left(\frac{(1+\alpha a)u_\varepsilon(t,\cdot)}{1+\alpha I_{u_\varepsilon}^\varepsilon(t)}\right)(x)\,dx\right)\int_{-1/\varepsilon}^{1/\varepsilon}(x-Z_{u_\varepsilon}(t))^2u_\varepsilon(t,x)\,dx\nonumber\\
&\quad \geq \left(1-\varepsilon^2 \left|\hat Z(t)\right|^2-\varepsilon^2 \int_{\mathbb R} \left(x-\hat Z(t)\right)^2T\left(\frac{(1+\alpha a)u_\varepsilon(t,\cdot)}{1+\alpha I_{u_\varepsilon}^\varepsilon(t)}\right)(x)\,dx\right)\\
&\qquad \int_{-1/\varepsilon}^{1/\varepsilon}(x-Z_{u_\varepsilon}(t))^2u_\varepsilon(t,x)\,dx. \nonumber\\
\end{align*}
We repeat the argument made in \eqref{est:truc1} to estimate $\int_{\mathbb R} \left(x-\hat Z(t)\right)^2T\left(\frac{(1+\alpha a)u_\varepsilon(t,\cdot)}{1+\alpha I_{u_\varepsilon}^\varepsilon(t)}\right)(x)\,dx$, which leads to
\begin{align}
&\left(\int_{-1/\varepsilon}^{1/\varepsilon}T\left(\frac{(1+\alpha a)u_\varepsilon(t,\cdot)}{1+\alpha I_{u_\varepsilon}^\varepsilon(t)}\right)(x)\,dx\right)\int_{-1/\varepsilon}^{1/\varepsilon}(x-Z_{u_\varepsilon}(t))^2u_\varepsilon(t,x)\,dx\nonumber\\
&\quad \geq \left(\int_{-1/\varepsilon}^{1/\varepsilon}(x-Z_{u_\varepsilon}(t))^2u_\varepsilon(t,x)\,dx\right)\bigg[1-\varepsilon^2 \left(\left|Z_{u_\varepsilon}(t)\right|+\left|\hat Z(t)-Z_{u_\varepsilon}(t)\right|\right)^2\nonumber\\
&\qquad -\varepsilon^2 \left(\frac{3+C\alpha}2\left|\hat Z(t)-Z_{u_\varepsilon}(t)\right|^2+ \sigma^2+\frac{1+C\alpha}2\int_{-1/\varepsilon}^{1/\varepsilon} \left(x-Z_{u_\varepsilon}(t)\right)^2u_\varepsilon(x)\,dx\right)\bigg]\nonumber\\
&\quad \geq \bigg[1-\varepsilon^2\sigma^2-2\varepsilon^2 \left|Z_{u_\varepsilon}(t)\right|^2-C\varepsilon^2 \left|\hat Z(t)-Z_{u_\varepsilon}(t)\right|^2 \bigg]\int_{-1/\varepsilon}^{1/\varepsilon}(x-Z_{u_\varepsilon}(t))^2u_\varepsilon(t,x)\,dx\nonumber\\
&\qquad -C\varepsilon^2\left(\int_{-1/\varepsilon}^{1/\varepsilon} \left(x-Z_{u_\varepsilon}(t)\right)^2u_\varepsilon(x)\,dx\right)^2.\label{est:truc2}
\end{align}
Combining \eqref{est:M2-1}, \eqref{est:truc1} and \eqref{est:truc2}, we obtain, for $\alpha>0$ small enough,
\begin{align*}
&\frac d{dt}\int_{-1/\varepsilon}^{1/\varepsilon}(x-Z_{u_\varepsilon}(t))^2u_\varepsilon(t,x)\,dx\nonumber\\
&\quad \leq \left(1+\alpha \int_{-1/\varepsilon}^{1/\varepsilon}a(y)u_{\varepsilon}(t,y)\,dy\right)^2\bigg[\frac{3+C\alpha}2\left|\hat Z(t)-Z_{u_\varepsilon}(t)\right|^2\\
&\qquad +\frac{1+C\alpha}2\int_{-1/\varepsilon}^{1/\varepsilon} \left(x-Z_{u_\varepsilon}(t)\right)^2u_\varepsilon(x)\,dx\\
&\qquad + \sigma^2-\left(1-\varepsilon^2\sigma^2-2\varepsilon^2 \left|Z_{u_\varepsilon}(t)\right|^2-C\varepsilon^2 \left|\hat Z(t)-Z_{u_\varepsilon}(t)\right|^2 \right)\int_{-1/\varepsilon}^{1/\varepsilon}(x-Z_{u_\varepsilon}(t))^2u_\varepsilon(t,x)\,dx\bigg]\\
&\qquad +C\varepsilon^2\left(\int_{-1/\varepsilon}^{1/\varepsilon} \left(x-Z_{u_\varepsilon}(t)\right)^2u_\varepsilon(x)\,dx\right)^2\\
&\quad \leq \left(1+\alpha \int_{-1/\varepsilon}^{1/\varepsilon}a(y)u_{\varepsilon}(t,y)\,dy\right)^2\bigg[ \frac{3+C\alpha}2\left|\hat Z(t)-Z_{u_\varepsilon}(t)\right|^2+ \sigma^2\\
&\qquad -\left(\frac 12-C\alpha-\varepsilon^2\sigma^2-2\varepsilon^2 \left|Z_{u_\varepsilon}(t)\right|^2-C\varepsilon^2 \left|\hat Z(t)-Z_{u_\varepsilon}(t)\right|^2\right)\int_{-1/\varepsilon}^{1/\varepsilon}(x-Z_{u_\varepsilon}(t))^2u_\varepsilon(t,x)\,dx\bigg]\\
&\qquad +C\varepsilon^2\left(\int_{-1/\varepsilon}^{1/\varepsilon} \left(x-Z_{u_\varepsilon}(t)\right)^2u_\varepsilon(x)\,dx\right)^2
\end{align*}
We notice that
\begin{align*}
\left|\hat Z(t)-Z_{u_\varepsilon}(t)\right|&=\left|\int_{-1/\varepsilon}^{1/\varepsilon}\left(x-Z_{u_\varepsilon}(t)\right) \frac{(1+\alpha a)u_\varepsilon(t,x)}{1+\alpha I_{u_\varepsilon}^\varepsilon(t)}\,dx\right|\nonumber\\
&=\frac\alpha{1+\alpha I_{u_\varepsilon}^\varepsilon(t)}\left|\int_{-1/\varepsilon}^{1/\varepsilon}\left(x-Z_{u_\varepsilon}(t)\right) a(x)u_\varepsilon(t,x)\right|\,dx\nonumber\\
&\leq C\alpha\left(1+\int_{-1/\varepsilon}^{1/\varepsilon}(x-Z_{u_\varepsilon}(t))^2u_\varepsilon(t,x)\,dx\right).
\end{align*}
Therefore,
\begin{align}
&\frac d{dt}\int_{-1/\varepsilon}^{1/\varepsilon}(x-Z_{u_\varepsilon}(t))^2u_\varepsilon(t,x)\,dx\nonumber\\
&\quad \leq  \sigma^2+C\alpha -\left(\frac 12-C\alpha-\varepsilon^2\sigma^2-2\varepsilon^2 \left|Z_{u_\varepsilon}(t)\right|^2\right)\int_{-1/\varepsilon}^{1/\varepsilon}(x-Z_{u_\varepsilon}(t))^2u_\varepsilon(t,x)\,dx\nonumber\\
&\qquad +C(\alpha+\varepsilon^2)\left(\int_{-1/\varepsilon}^{1/\varepsilon} \left(x-Z_{u_\varepsilon}(t)\right)^2u_\varepsilon(x)\,dx\right)^2\label{est:M2-4}
\end{align}

\medskip

\noindent\textbf{Step 2: Assumption on $\|Z\|_{L^\infty([0,\bar t])}$ and consequences}

Let 
\begin{equation}\label{eq:deftildeR}
\tilde R:=\max\left(\max\{|x|;x\in\textrm{supp }a\}, \left|Z^0\right|\right)+1,
\end{equation}
where $Z^0:=\int_{\mathbb R}xn^0(x)\,dx$. We define, for $\varepsilon>0$,
\begin{equation}\label{eq:deftbar}
\bar t_\varepsilon=\sup\{t\in[0,\infty);\,\forall s\in[0,t],\, Z_{u_\varepsilon}(t)\in [-\tilde R,\tilde R]\}.
\end{equation} 
 the maximum time such that $Z_{u_\varepsilon}(t)\in [-\tilde R,\tilde R]$ for $t\in[0,\bar t]$, and \eqref{eq:Z0eps} implies $\bar t>0$ if $\varepsilon\leq 1/4(V^0+1)$. We will derive properties of $u_\varepsilon(t,\cdot)$ for $t\in[0,\bar t_\varepsilon)$, which will allow us to show eventually conclude (in Step~4) that $\bar t=\infty$ provided $\varepsilon>0$ is small enough. This analysis will also provide uniform estimates on moments of $u_\varepsilon(t,\cdot)$.

\medskip

Let $V^0:=\int_{\mathbb R} x^2n^0(x)\,dx$, we notice that 
\begin{align*}
Z_{u_\varepsilon}(0)&=\int_{-1/\varepsilon}^{1/\varepsilon} xu_\varepsilon(0,x)\,dx=\int_{-1/\varepsilon}^{1/\varepsilon} x\frac{n^0(x)1_{|x|\leq 1/\varepsilon}}{\int_{-1/\varepsilon}^{1/\varepsilon} n^0(y)\,dy}\,dx\\
&=\frac{Z^0}{\int_{-1/\varepsilon}^{1/\varepsilon} n^0(x)\,dx}-\int_{[-1/\varepsilon,1/\varepsilon]^c} x\frac{n^0(x)}{\int_{-1/\varepsilon}^{1/\varepsilon} n^0(x)\,dx}\,dx,
\end{align*}
and thanks to Chebychev inequalities, we obtain
\begin{align}
\left|\int_{-1/\varepsilon}^{1/\varepsilon} xu_\varepsilon(0,x)\,dx-Z^0\right|&\leq \frac{\varepsilon^2 V^0}{1-\varepsilon^2V^0}+\frac{\varepsilon V^0}{1-\varepsilon^2V^0}\leq \frac 12,\label{eq:Z0eps}
\end{align}
if $\varepsilon\in(0,1/4(V^0+1))$. In particular $|Z_{u_\varepsilon}(0)|<R$, and the continuity of $Z_{u_\varepsilon}$ implies $\bar t_\varepsilon>0$.

\medskip

Thanks to \eqref{est:M2-4}, provided $\alpha>0$ and $\varepsilon>0$ are small enough and $t\in [0,\bar t_\varepsilon)$,
\begin{align*}
\frac d{dt}\int_{-1/\varepsilon}^{1/\varepsilon}(x-Z_{u_\varepsilon}(t))^2u_\varepsilon(t,x)\,dx&\leq 
(\sigma^2+1)-\frac 49\int_{-1/\varepsilon}^{1/\varepsilon}(x-Z_{u_\varepsilon}(t))^2u_\varepsilon(t,x)\,dx\\
&\quad +C(\alpha+\varepsilon^2)\left(\int_{-1/\varepsilon}^{1/\varepsilon}\left(x-Z_{u_\varepsilon}(t)\right)^2u_\varepsilon(x)\,dx\right)^2\\
&\leq  (\sigma^2+1)-\frac 13\int_{-1/\varepsilon}^{1/\varepsilon}(x-Z_{u_\varepsilon}(t))^2u_\varepsilon(t,x)\,dx,
\end{align*}
since $\left|\left(\int_{-1/\varepsilon}^{1/\varepsilon}\left(x-Z_{u_\varepsilon}(t)\right)^2u_\varepsilon(x)\,dx\right)^2\right|\leq C(1+V^0)^2$. This estimate implies
\begin{align*}
&\int_{-1/\varepsilon}^{1/\varepsilon}(x-Z_{u_\varepsilon}(t))^2u_\varepsilon(t,x)\,dx\leq 3(\sigma^2+1)+\int_{-1/\varepsilon}^{1/\varepsilon}(x-Z_{u_\varepsilon}(0))^2u_\varepsilon(0,x)\,dx\\
&\quad \leq 3(\sigma^2+1)+2|Z^0|^2+2\int_{-1/\varepsilon}^{1/\varepsilon}x^2u_\varepsilon(0,x)\,dx\leq 3(\sigma^2+1)+2|Z^0|^2+3V^0.
\end{align*}
Since $\int_{-1/\varepsilon}^{1/\varepsilon}x^2u_\varepsilon(t,x)\,dx\leq 2\int_{-1/\varepsilon}^{1/\varepsilon}(x-Z_{u_\varepsilon}(t))^2u_\varepsilon(t,x)\,dx+2|Z_{u_\varepsilon}(t)|^2$, this estimate implies, for $t\in[0,\bar t_\varepsilon)$,
\begin{equation*}
\int_{-1/\varepsilon}^{1/\varepsilon}x^2u_\varepsilon(t,x)\,dx\leq 6(\sigma^2+1)+4|Z^0|^2+6V^0+2\tilde R^2,
\end{equation*}
as well as, thanks to \eqref{eq:Tmoment1},
\begin{align}
&\int_{-1/\varepsilon}^{1/\varepsilon} z^2T\left(\frac{(1+\alpha a)u_\varepsilon(t,\cdot)}{1+\alpha I_{u_\varepsilon}^\varepsilon(t)}\right)(z)\,dz\leq \sigma^2+\int_{-1/\varepsilon}^{1/\varepsilon} z^2\frac{(1+\alpha a)u_\varepsilon(t,z)}{1+\alpha I_{u_\varepsilon}(t)}\nonumber\\
&\quad \leq 8(\sigma^2+1)+5|Z^0|^2+7V^0+3\tilde R^2.\label{eq:estm2}
\end{align}
This inequality and Chebyshev inequalities imply, still for $t\in(0,\bar t_{\varepsilon})$ and for $\varepsilon,\alpha>0$ small enough,
\begin{align}
Z_{u_\varepsilon}'(t)&=\int_{-1/\varepsilon}^{1/\varepsilon}x\partial_t u_\varepsilon(t,x)\,dx=\left(1+\alpha I_{u_\varepsilon}(t)\right)^2\bigg[ \int_{-1/\varepsilon}^{1/\varepsilon}xT\left(\frac{(1+\alpha a)u_\varepsilon(t,\cdot)}{1+\alpha I_{u_\varepsilon}(t)}\right)(x)\,dx\nonumber\\
&\quad-\left(\int_{-1/\varepsilon}^{1/\varepsilon}T\left(\frac{(1+\alpha a)u_\varepsilon(t,\cdot)}{1+\alpha I_{u_\varepsilon}(t)}\right)(z)\,dz\right)Z_{u_\varepsilon}(t)\bigg]\nonumber\\
&= \left(1+\alpha I_{u_\varepsilon}(t)\right)^2\bigg[ \left(\hat Z(t)+\mathcal O(\varepsilon)\right)-Z_{u_\varepsilon}(t)\left(1+\mathcal O(\varepsilon^2)\right)\bigg]=\mathcal O(\alpha)+\mathcal O(\varepsilon),\label{est:M2-5}
\end{align}
since $\hat Z(t)-Z_{u_\varepsilon}(t)=\mathcal O(\alpha)$.


\medskip

\noindent\textbf{Step 3: Estimate on $W_2(u_\varepsilon(t,x),\Gamma_{2\sigma^2}(\cdot-Z_{u_\varepsilon}(t))$}

We recall the integral formulation \eqref{eq:integralformulation2} of $u_\varepsilon$. The Gaussian distribution $\Gamma_{2\sigma^2}(\cdot-Z_{u_\varepsilon}(t))$ is a steady state of the operator $T$, and then
\begin{align*}
&\Gamma_{2\sigma^2}(x-Z_{u_\varepsilon}(t))=\Gamma_{2\sigma^2}(x-Z_{u_\varepsilon}(t))e^{-\int_0^t\left(1+\alpha I_{u_\varepsilon}^\varepsilon(s)\right)^2\left(\int_{-1/\varepsilon}^{1/\varepsilon}T\left(\frac{(1+\alpha a)u_\varepsilon(s,\cdot)}{1+\alpha I_{u_\varepsilon}^\varepsilon(s)}\right)(z)\,dz\right)\,ds}\nonumber\\
&\quad +\int_0^t\left(1+\alpha I_{u_\varepsilon}^\varepsilon(s)\right)^2T\left(\Gamma_{2\sigma^2}(x-Z_{u_\varepsilon}(t))\right)(x)e^{-\int_s^t\left(1+\alpha I_{u_\varepsilon}^\varepsilon(\tau)\right)^2\left(\int_{-1/\varepsilon}^{1/\varepsilon}T\left(\frac{(1+\alpha a)u_\varepsilon(\tau,\cdot)}{1+\alpha I_{u_\varepsilon}^\varepsilon(\tau)}\right)(z)\,dz\right)\,d\tau}\,ds.
\end{align*}
Therefore, thanks to \eqref{eq:W2convex}, \eqref{eq:W2convexcontinu},
\begin{align}
&W_2(u_\varepsilon(t,\cdot),\Gamma_{2\sigma^2}(\cdot-Z_{u_\varepsilon}(t)))^2\nonumber\\
&\quad \leq e^{-\int_{0}^t\left(1+\alpha I_{u_\varepsilon}^\varepsilon(s)\right)^2\left(\int_{-1/\varepsilon}^{1/\varepsilon}T\left(\frac{(1+\alpha a)u_\varepsilon(s,\cdot)}{1+\alpha I_{u_\varepsilon}^\varepsilon(s)}\right)(z)\,dz\right)\,ds}W_2\left(u_\varepsilon(0,\cdot),\Gamma_{2\sigma^2}(\cdot-Z_{u_\varepsilon}(t))\right)^2\nonumber\\
&\qquad+\int_{0}^t \left(1+\alpha I_{u_\varepsilon}^\varepsilon(s)\right)^2\left(\int_{-1/\varepsilon}^{1/\varepsilon}T\left(\frac{(1+\alpha a)u_\varepsilon(s,\cdot)}{1+\alpha I_{u_\varepsilon}^\varepsilon(s)}\right)(z)\,dz\right)\nonumber\\
&\phantom{sfdcqsdfq} e^{-\int_s^t(1+\alpha I_{u_\varepsilon}^\varepsilon(\tau))^2\left(\int_{-1/\varepsilon}^{1/\varepsilon}T\left(\frac{(1+\alpha a)u_\varepsilon(\tau,\cdot)}{1+\alpha I_{u_\varepsilon}^\varepsilon(\tau)}\right)(z)\,dz\right)\,d\tau}\nonumber\\
&\phantom{sfdcqsdfq} W_2\left(\frac {T\left(\frac{(1+\alpha a)u_\varepsilon(s,\cdot)}{1+\alpha I_{u_\varepsilon}^\varepsilon(s)}\right)1_{|\cdot|< \frac 1\varepsilon}}{\int_{-1/\varepsilon}^{1/\varepsilon}T\left(\frac{(1+\alpha a)u_\varepsilon(s,\cdot)}{1+\alpha I_{u_\varepsilon}^\varepsilon(s)}\right)(z)\,dz}, \Gamma_{2\sigma^2}(\cdot-Z_{u_\varepsilon}(t))\right)^2\,ds.\label{est:W2appendix1}
\end{align}
Let $s\in[0,t]$ and $\bar\pi=\bar\pi(x,y)\in\mathcal P_2(\mathbb R^2)$ the optimal transport  from $T\left(\frac{(1+\alpha a)u_\varepsilon(s,\cdot)}{1+\alpha I_{u_\varepsilon}^\varepsilon(s)}\right)$ to $\Gamma_{2\sigma^2}(\cdot-Z_{u_\varepsilon}(t))$ for the quadratic distance, which satisfies:
\begin{equation}\label{eq:barpi}
\int_{\mathbb R}\bar \pi(x,y)\,dy=T\left(\frac{(1+\alpha a)u_\varepsilon(s,\cdot)}{1+\alpha I_{u_\varepsilon}^\varepsilon(s)}\right)(x),\quad \int_{\mathbb R}\bar \pi(x,y)\,dx=\Gamma_{2\sigma^2}(y-Z_{u_\varepsilon}(t)).
\end{equation}
We use this coupling function $\bar\pi$ to build the probability measure $\tilde \pi$ on $\mathbb R^2$:
\begin{align*}
\tilde \pi(x,y)&:=\bar \pi(x,y)1_{|x|< 1/\varepsilon}\\
&\quad+\frac{T\left(\frac{(1+\alpha a)u_\varepsilon(s,\cdot)}{1+\alpha I_{u_\varepsilon}^\varepsilon(s)}\right)(x)1_{|x|< 1/\varepsilon}}{\int_{-1/\varepsilon}^{1/\varepsilon}T\left(\frac{(1+\alpha a)u_\varepsilon(s,\cdot)}{1+\alpha I_{u_\varepsilon}^\varepsilon(s)}\right)(z)\,dz}\left(\Gamma_{2\sigma^2}(y-Z_{u_\varepsilon}(t))-\int_{-1/\varepsilon}^{1/\varepsilon} \bar \pi(z,y)\,dz\right).
\end{align*}
Note that $\tilde \pi(x,y)\geq 0$, since \eqref{eq:barpi} implies 
\[\int_{-1/\varepsilon}^{1/\varepsilon} \bar \pi(z,y)\,dz\leq \int_{\mathbb R}\bar \pi(z,y)\,dz=\Gamma_{2\sigma^2}(y-Z_{u_\varepsilon}(t)).\]
 The probability measure $\tilde \pi$ is a transport plan between the two probability measures present in the Wasserstein distance that appear in equation \eqref{est:W2appendix1}, as demonstrated by the calculation of the marginals of $\tilde \pi$:
\begin{align*}
\int_{\mathbb R}\tilde \pi(x,y)\,dy&=T\left(\frac{(1+\alpha a)u_\varepsilon(s,\cdot)}{1+\alpha I_{u_\varepsilon}^\varepsilon(s)}\right)(x)1_{|x|< 1/\varepsilon}\\
&\quad +\frac{T\left(\frac{(1+\alpha a)u_\varepsilon(s,\cdot)}{1+\alpha I_{u_\varepsilon}^\varepsilon(s)}\right)(x)1_{|x|< 1/\varepsilon}}{\int_{-1/\varepsilon}^{1/\varepsilon}T\left(\frac{(1+\alpha a)u_\varepsilon(s,\cdot)}{1+\alpha I_{u_\varepsilon}^\varepsilon(s)}\right)(z)\,dz}\left(1-\int_{-1/\varepsilon}^{1/\varepsilon} T\left(\frac{(1+\alpha a)u_\varepsilon(s,\cdot)}{1+\alpha I_{u_\varepsilon}^\varepsilon(s)}\right)(z)\,dz\right)\\
&=\frac{T\left(\frac{(1+\alpha a)u_\varepsilon(s,\cdot)}{1+\alpha I_{u_\varepsilon}^\varepsilon(s)}\right)(x)1_{|x|< 1/\varepsilon}}{\int_{-1/\varepsilon}^{1/\varepsilon}T\left(\frac{(1+\alpha a)u_\varepsilon(s,\cdot)}{1+\alpha I_{u_\varepsilon}^\varepsilon(s)}\right)(z)\,dz},
\end{align*}
\begin{align*}
\int\tilde \pi(x,y)\,dx&=\int_{-1/\varepsilon}^{1/\varepsilon}\bar \pi(x,y)\,dx+\left(\Gamma_{2\sigma^2}(y-Z_{u_\varepsilon}(t))-\int_{-1/\varepsilon}^{1/\varepsilon} \bar \pi(z,y)\,dz\right)=\Gamma_{2\sigma^2}(y-Z_{u_\varepsilon}(t)).
\end{align*}
Then, for $t\in[0,\bar t]$, we have \eqref{eq:estm2}, that implies 
\[\int_{-1/\varepsilon}^{1/\varepsilon}T\left(\frac{(1+\alpha a)u_\varepsilon(s,\cdot)}{1+\alpha I_{u_\varepsilon}^\varepsilon(s)}\right)(z)\,dz\geq 1-\int_{[-1/\varepsilon,1/\varepsilon]^c}T\left(\frac{(1+\alpha a)u_\varepsilon(s,\cdot)}{1+\alpha I_{u_\varepsilon}^\varepsilon(s)}\right)(z)\,dz\geq 1-C\varepsilon^2\geq \frac 23,\]
for $\varepsilon>0$ small enough, and then,
\begin{align}
&W_2\left(\frac {T\left(\frac{(1+\alpha a)u_\varepsilon(s,\cdot)}{1+\alpha I_{u_\varepsilon}^\varepsilon(s)}\right)1_{|\cdot|< \frac 1\varepsilon}}{\int_{-1/\varepsilon}^{1/\varepsilon}T\left(\frac{(1+\alpha a)u_\varepsilon(s,\cdot)}{1+\alpha I_{u_\varepsilon}^\varepsilon(s)}\right)(z)\,dz}, \Gamma_{2\sigma^2}(\cdot-Z_{u_\varepsilon}(t))\right)^2\leq \int_{\mathbb R}\int_{\mathbb R} |x-y|^2\,d\tilde \pi(x,y)\nonumber \\
&\quad \leq \int_{\mathbb R}\int_{\mathbb R} |x-y|^2\,d\bar \pi(x,y)\nonumber\\
&\qquad +2\int_{\mathbb R}\int_{\mathbb R} (x^2+y^2) \frac{T\left(\frac{(1+\alpha a)u_\varepsilon(s,\cdot)}{1+\alpha I_{u_\varepsilon}^\varepsilon(s)}\right)(x)1_{|x|<1/\varepsilon}}{\int_{-1/\varepsilon}^{1/\varepsilon}T\left(\frac{(1+\alpha a)u_\varepsilon(s,\cdot)}{1+\alpha I_{u_\varepsilon}^\varepsilon(s)}\right)(z)\,dz}\left(\Gamma_{2\sigma^2}(y-Z_{u_\varepsilon}(t))-\int_{-1/\varepsilon}^{1/\varepsilon} \bar \pi(z,y)\,dz\right)\,dx\,dy\nonumber\\
&\quad\leq W_2\left(T\left(\frac{(1+\alpha a)u_\varepsilon(s,\cdot)}{1+\alpha I_{u_\varepsilon}^\varepsilon(s)}\right), \Gamma_{2\sigma^2}(\cdot-Z_{u_\varepsilon}(t))\right)^2\nonumber\\
&\qquad +3\left(\int_{\mathbb R} x^2T\left(\frac{(1+\alpha a)u_\varepsilon(s,\cdot)}{1+\alpha I_{u_\varepsilon}^\varepsilon(s)}\right)\,dx\right)\int_{\mathbb R} \bigg(\Gamma_{2\sigma^2}(y-Z_{u_\varepsilon}(t))-\int_{-1/\varepsilon}^{1/\varepsilon} \bar \pi(z,y)\,dz\bigg)\,dy\nonumber\\
&\qquad + 2\int_{\mathbb R} y^2\left(\Gamma_{2\sigma^2}(y-Z_{u_\varepsilon}(t))-\int_{-1/\varepsilon}^{1/\varepsilon} \bar \pi(z,y)\,dz\right)\,dy,\label{est:W2appendix2}
\end{align}
and \eqref{eq:estm2} can be used again to estimate $\int_{\mathbb R} x^2T\left(\frac{(1+\alpha a)u_\varepsilon(s,\cdot)}{1+\alpha I_{u_\varepsilon}^\varepsilon(s)}\right)\,dx\leq C$. The two last terms of \eqref{est:W2appendix2} can be further estimated as follows:
\begin{align}
&\left|\int_{\mathbb R} (1+y^2)\left(\Gamma_{2\sigma^2}(y-Z_{u_\varepsilon}(t))-\int_{-1/\varepsilon}^{1/\varepsilon} \bar \pi(z,y)\,dz\right)\,dy\right|\nonumber\\
&\quad \leq (1+1/\varepsilon)\left|\int_{-1/\sqrt\varepsilon}^{1/\sqrt \varepsilon}\left(\Gamma_{2\sigma^2}(y-Z_{u_\varepsilon}(t))-\int_{-1/\varepsilon}^{1/\varepsilon} \bar \pi(z,y)\,dz\right)\,dy\right|\nonumber\\
&\qquad +\varepsilon\left|\int_{\mathbb R} y^2(1+y^2)\left(\Gamma_{2\sigma^2}(y-Z_{u_\varepsilon}(t))-\int_{-1/\varepsilon}^{1/\varepsilon} \bar \pi(z,y)\,dz\right)\,dy\right|.\label{eq:preli}
\end{align}
We notice that
\begin{align*}
&\left|\int_{-1/\sqrt\varepsilon}^{1/\sqrt \varepsilon}\left(\Gamma_{2\sigma^2}(y-Z_{u_\varepsilon}(t))-\int_{-1/\varepsilon}^{1/\varepsilon} \bar \pi(z,y)\,dz\right)\,dy\right|=\int_{-1/\sqrt\varepsilon}^{1/\sqrt \varepsilon}\left(\int_{[-1/\varepsilon,1/\varepsilon]^c} \bar \pi(z,y)\,dz\right)\,dy\\
&\quad \leq \varepsilon^2\int\left(\int_{\mathbb R^2} z^2\bar \pi(z,y)\,dz\right)\,dy=\varepsilon^2\int_{\mathbb R^2} z^2T\left(\frac{(1+\alpha a)u_\varepsilon(s,\cdot)}{1+\alpha I_{u_\varepsilon}^\varepsilon(s)}\right)(z)\,dz\leq C\varepsilon^2,
\end{align*}
thanks to \eqref{eq:estm2}. Moreover $\left|\int_{-1/\varepsilon}^{1/\varepsilon} \bar \pi(z,y)\,dz\right|\leq \int_{\mathbb R} \bar \pi(z,y)\,dz=\Gamma_{2\sigma^2}(y-Z_{u_\varepsilon}(t))$ implies
\begin{align*}
&\left|\int_{\mathbb R} y^2(1+y^2)\left(\Gamma_{2\sigma^2}(y-Z_{u_\varepsilon}(t))-\int_{-1/\varepsilon}^{1/\varepsilon} \bar \pi(z,y)\,dz\right)\,dy\right|\\
&\quad \leq \int_{\mathbb R}y^2\left(1+y^2\right)\Gamma_{2\sigma^2}(y-Z_{u_\varepsilon}(t))\,dy\leq C,
\end{align*}
for $t\in[0,\bar t_\varepsilon]$, where $\bar t_\varepsilon$ is defined in \eqref{eq:deftbar}. The term $Z_{u_\varepsilon}$ is then uniformly bounded ($\|Z_{u_\varepsilon}\|_{L^\infty([0,\bar t_\varepsilon]}\leq \tilde R$). We can use these estimates to complete the estimate \eqref{eq:preli}:
\begin{align*}
&\left|\int_{\mathbb R} (1+y^2)\left(\Gamma_{2\sigma^2}(y-Z_{u_\varepsilon}(t))-\int_{-1/\varepsilon}^{1/\varepsilon} \bar \pi(z,y)\,dz\right)\,dy\right|\leq C\varepsilon,
\end{align*}
and estimate \eqref{est:W2appendix2} then becomes:
\begin{align*}
&W_2\left(\frac {T\left(\frac{(1+\alpha a)u_\varepsilon(s,\cdot)}{1+\alpha I_{u_\varepsilon}^\varepsilon(s)}\right)1_{|\cdot|< \frac 1\varepsilon}}{\int_{-1/\varepsilon}^{1/\varepsilon}T\left(\frac{(1+\alpha a)u_\varepsilon(s,\cdot)}{1+\alpha I_{u_\varepsilon}^\varepsilon(s)}\right)(z)\,dz}, \Gamma_{2\sigma^2}(\cdot-Z_{u_\varepsilon}(t))\right)^2\\
&\quad \leq W_2\left(T\left(\frac{(1+\alpha a)u_\varepsilon(s,\cdot)}{1+\alpha I_{u_\varepsilon}^\varepsilon(s)}\right), \Gamma_{2\sigma^2}(\cdot-Z_{u_\varepsilon}(t))\right)^2 +C\varepsilon\\
&\quad \leq \bigg[W_2\left(T\left(\frac{(1+\alpha a)u_\varepsilon(s,\cdot)}{1+\alpha I_{u_\varepsilon}^\varepsilon(s)}\right), T\left(\Gamma_{2\sigma^2}(\cdot-\hat Z(s))\right)\right)\\
&\qquad +W_2\left( \Gamma_{2\sigma^2}(\cdot-\hat Z(s)),\Gamma_{2\sigma^2}(\cdot-Z_{u_\varepsilon}(t))\right)\bigg]^2+C\varepsilon,
\end{align*}
where we have used that $\Gamma_{2\sigma^2}(\cdot-Z_{u_\varepsilon}(t))$ is invariant under $T$ (see \eqref{eq:Maxwellienne}). The definition \eqref{def:hatZ} of $\hat Z(s)$ implies $\hat Z(s)=Z_{u_\varepsilon}(s)+\mathcal O(\alpha)$, and  since we have proven (see \eqref{est:M2-5}) that $\|Z_{u_\varepsilon}'\|_\infty=\mathcal O(\alpha)+\mathcal O(\varepsilon)$,
\[\hat Z(s)=Z_{u_\varepsilon}(s)+\mathcal O(\alpha)=Z_{u_\varepsilon}(t)+\mathcal O(\alpha)(1+|t-s|)+\mathcal O(\varepsilon)|t-s|.\]
Then, using the $W_2-$contraction implied by $T$ on probability measures with the same center of mass,
\begin{align*}
&W_2\left(\frac {T\left(\frac{(1+\alpha a)u_\varepsilon(s,\cdot)}{1+\alpha I_{u_\varepsilon}^\varepsilon(s)}\right)1_{|\cdot|< \frac 1\varepsilon}}{\int_{-1/\varepsilon}^{1/\varepsilon}T\left(\frac{(1+\alpha a)u_\varepsilon(s,\cdot)}{1+\alpha I_{u_\varepsilon}^\varepsilon(s)}\right)(z)\,dz}, \Gamma_{2\sigma^2}(\cdot-Z_{u_\varepsilon}(t))\right)^2\\
&\quad \leq \bigg(\frac 1{\sqrt 2}W_2\left(\frac{(1+\alpha a)u_\varepsilon(s,\cdot)}{1+\alpha I_{u_\varepsilon}^\varepsilon(s)}, \Gamma_{2\sigma^2}(\cdot-\hat Z(s))\right)+C\alpha(1+|t-s|)+C\varepsilon|t-s|\bigg)^2+C\varepsilon\\
&\quad \leq \frac 58W_2\left(\frac{(1+\alpha a)u_\varepsilon(s,\cdot)}{1+\alpha I_{u_\varepsilon}^\varepsilon(s)}, \Gamma_{2\sigma^2}(\cdot-\hat Z(s))\right)^2+C\alpha^2(1+|t-s|^2)+C\varepsilon^2|t-s|^2+C\varepsilon.
\end{align*}
We use the convex decomposition \eqref{eq:convdecomp}, and the calculations from \eqref{eq:convdecomp} to \eqref{est:W2conde} to show:
\begin{align*}
&W_2\left(\frac{(1+\alpha a)u_\varepsilon(s,\cdot)}{1+\alpha I_{u_\varepsilon}(s)},\Gamma_{2\sigma^2}\left(\cdot-\hat Z(s)\right)\right)^2\leq \frac 65 \,W_2\left(u_\varepsilon(s,\cdot),\Gamma_{2\sigma^2}\left(\cdot-Z_{u_\varepsilon}(s)\right)\right)^2+C\alpha.
\end{align*}
Then, provided $t\in[0,\bar t_\varepsilon]$, \eqref{est:W2appendix1} becomes
\begin{align*}
&W_2(u_\varepsilon(t,\cdot),\Gamma_{2\sigma^2}(\cdot-Z_{u_\varepsilon}(t)))^2\leq W_2(u_\varepsilon(0,\cdot),\Gamma_{2\sigma^2}(\cdot-Z_{u_\varepsilon}(t)))^2e^{-\frac 78t}+\left(1+\alpha\|a\|_\infty\right)^2\nonumber\\
&\qquad \int_0^t e^{-\frac 78(t-s)}\bigg[\frac 58W_2\left(u_\varepsilon(s,\cdot),\Gamma_{2\sigma^2}\left(\cdot-Z_{u_\varepsilon}(s)\right)\right)^2+C(\alpha+\varepsilon)\left(1+|t-s|\right)\bigg]\,ds\\
&\quad \leq W_2(u_\varepsilon(0,\cdot),\Gamma_{2\sigma^2}(\cdot-Z_{u_\varepsilon}(t)))^2e^{-\frac 78t}\nonumber\\
&\qquad +\frac 65\frac 58\int_0^t e^{-\frac 78(t-s)}W_2\left(u_\varepsilon(s,\cdot),\Gamma_{2\sigma^2}\left(\cdot-Z_{u_\varepsilon}(s)\right)\right)^2\,ds+C \alpha+C\varepsilon,
\end{align*}
We can reproduce the argument from Lemma~\ref{lem:roughmacro}, from \eqref{eq:introphi} to \eqref{est:W2Z00} (with minor modifications), to show 
\begin{align}
 W_2\left(u_\varepsilon(t,\cdot),\Gamma_{2\sigma^2}(\cdot-Z_{u_\varepsilon}(t)\right)\leq CW_2\left(u_\varepsilon(0,\cdot),\Gamma_{2\sigma^2}(\cdot-Z_{u_\varepsilon}(0)\right)e^{-\frac 12\left(\frac78-\frac68\right)t}+ C\sqrt\alpha+C\sqrt\varepsilon,\label{eq:estuf}
\end{align}
provided $\alpha,\,\varepsilon>0$ are small enough and $t\in[0,\bar t_\varepsilon]$. We notice that 
\begin{equation}\label{eq:Zu0}
|Z_n(0)-Z_{u_\varepsilon}(0)|\leq \int_{[-1/\varepsilon,1/\varepsilon]^c}xn(0,x)\,dx+\left(\int_{[-1/\varepsilon,1/\varepsilon]^c}n(0,x)\,dx\right)\int_{-1/\varepsilon}^{1/\varepsilon}xn(0,x)\,dx\leq C\varepsilon,
\end{equation}
and then
\begin{align}
W_2(u_\varepsilon(0,\cdot),\Gamma_{2\sigma^2}(\cdot-Z_{u_\varepsilon}(0)))&\leq W_2(u_\varepsilon(0,\cdot),\Gamma_{2\sigma^2}(\cdot-Z_n(0)))+|Z_{u_\varepsilon}(0)-Z_n(0)|\nonumber\\
&\leq W_2(u_\varepsilon(0,\cdot),n(0,\cdot))+W_2(n(0,\cdot),\Gamma_{2\sigma^2}(\cdot-Z_n(0)))+C\varepsilon\nonumber\\
&\leq  \rho+C\varepsilon,\label{eq:t0ueps}
\end{align}
since
\begin{align*}
W_2(u_\varepsilon(0,\cdot),n(0,\cdot))&\leq \left(\int_{[-1/\varepsilon,1/\varepsilon]^c}n(0,x)\,dx\right)\bigg[W_2\bigg(\frac{n(0,\cdot)1_{|\cdot|< 1/\varepsilon}}{\int_{-1/\varepsilon}^{1/\varepsilon}n(0,x)\,dx},\delta_0\bigg)\\
&\qquad +W_2\bigg(\frac{n(0,\cdot)1_{|\cdot|\geq 1/\varepsilon}}{\int_{[-1/\varepsilon,1/\varepsilon]^c}n(0,x)\,dx},\delta_0\bigg)\bigg]\\
&\leq C\int_{[-1/\varepsilon,1/\varepsilon]^c}n(0,x)\,dx+C\int_{[-1/\varepsilon,1/\varepsilon]^c}y^2n(0,y)\,dy\leq C\varepsilon^2,
\end{align*}
where we have used the property $n(0,x)\leq Ce^{-|x|}$ provided by Assumption~\ref{ass2} to estimate:
\begin{align*}
&\left(\int_{[-1/\varepsilon,1/\varepsilon]^c}n(0,x)\,dx\right)W_2\bigg(\frac{n(0,\cdot)1_{|\cdot|\geq 1/\varepsilon}}{\int_{[-1/\varepsilon,1/\varepsilon]^c}n(0,x)\,dx},\delta_0\bigg)\\
&\quad =\left(\int_{[-1/\varepsilon,1/\varepsilon]^c}n(0,x)\,dx\right)\int_{[-1/\varepsilon,1/\varepsilon]^c}y^2\frac{n(0,y)}{\int_{[-1/\varepsilon,1/\varepsilon]^c}n(0,x)\,dx}\,dy\\
&\quad = \int_{[-1/\varepsilon,1/\varepsilon]^c}y^2n(0,y)\,dy\leq \int_{[-1/\varepsilon,1/\varepsilon]^c}y^2Ce^{-|y|}\,dy=2Ce^{-1/\varepsilon}\left(\frac 1{\varepsilon^2}+\frac 2\varepsilon+2\right)\leq C\varepsilon^2
\end{align*}
Estimates \eqref{eq:t0ueps} and \eqref{eq:estuf} prove \eqref{est:W2appendix} for $t\in[0,\bar t_\varepsilon]$.

\medskip

\noindent\textbf{Step 4: Conclusion}

Thanks to \eqref{est:M2-5} and the Kantorovich-Rubinstein formula, for $t\in[0,\bar t_\varepsilon]$,
\begin{align*}
&Z_{u_\varepsilon}'(t)=\left(1+\alpha I_{u_\varepsilon}(t)\right)^2\left(\int_{\mathbb R}x\left(\frac{(1+\alpha a(x)u_\varepsilon(t,x)}{1+\alpha I_{u_\varepsilon}^\varepsilon(t)}-u_\varepsilon(t,x)\right)\,dx+\mathcal O(\varepsilon)\right)\\
&=\alpha\left(1+\alpha I_{u_\varepsilon}(t)\right)I_{u_\varepsilon}(t)\left(\int_{\mathbb R}x \frac{a(x)u_\varepsilon(t,x)}{I_{u_\varepsilon}^\varepsilon(t)}\,dx-Z_{u_\varepsilon}(t)+\mathcal O(\varepsilon)\right)\\
&=\alpha\left(1+\mathcal O(\alpha)\right)\bigg[\int_{\mathbb R}a(x)\Gamma_{2\sigma^2}(Z_{u_\varepsilon}(t)-x)\,dx+\mathcal O(1)W_1(u_\varepsilon(t,\cdot),\Gamma_{2\sigma^2}(Z_{u_\varepsilon}(t)-\cdot))\\
&\qquad -Z_{u_\varepsilon}(t)\left(\int_{\mathbb R}a(x)\Gamma_{2\sigma^2}(Z_{u_\varepsilon}(t)-x)\,dx\right)+\left|Z_{u_\varepsilon}(t)\right|\mathcal O(1)W_1(u_\varepsilon(t,\cdot),\Gamma_{2\sigma^2}(Z_{u_\varepsilon}(t)-\cdot))\bigg]\\
&=\alpha\left(1+\mathcal O(\alpha)\right)F(Z_{u_\varepsilon}(t))+ \mathcal O(\alpha)\left((\rho+C\varepsilon) e^{-t/16}+C\sqrt \alpha+C\sqrt\varepsilon\right),
\end{align*}
where we have used the fact that $|Z_{u_\varepsilon}(t)|\leq \tilde R$ for $t\in[0,\bar t_\varepsilon]$. Thanks to the definition of $\tilde R$ (see \eqref{eq:deftbar}) and \eqref{eq:Zu0}, provided $\varepsilon>0$ is small enough, $|Z_{u_\varepsilon}(0)|\leq \tilde R-2/3$, and since $|Z'_{u_\varepsilon}|\leq C\alpha$, provided $\alpha>0$ is small enough,
\[\min_{t\in[0,1/C\alpha]\cap[0,\bar t_\varepsilon}|Z_{u_\varepsilon}(t)|\leq \tilde R-1/3,\]
which implies $\bar t_\varepsilon>1/C\alpha$. For $t\in[1/C\alpha,\bar t]$, we have 
\[Z_{u_\varepsilon}'(t)=\alpha\left[F(Z_{u_\varepsilon}(t))+ g_{\alpha,\varepsilon}(t)\right],\]
where 
\begin{align*}
|g_{\alpha,\varepsilon}(t)|\leq (\rho+C\varepsilon) e^{-\frac 1{16C\alpha}}+C(\sqrt\alpha+\sqrt\varepsilon)\leq C(\sqrt\alpha+\sqrt\varepsilon),
\end{align*}
 for $t\in[1/C\alpha,\bar t_\varepsilon]$. Then, provided $\alpha>0$ and $\varepsilon>0$ are small enough, we have $|g_{\alpha,\varepsilon}(t)|<\frac 12\min(F(-\tilde R+1/3),-F(\tilde R-1/3))$, thanks to
\begin{align*}
F(\tilde R-1/3)&=\int_{\mathbb R}x a(x)\Gamma_{2\sigma^2}(x-\tilde R+1/3)\,dx-x_+\int_{\mathbb R} a(x)\Gamma_{2\sigma^2}(x-\tilde R+1/3)\,dx\\
&=\int_{\mathbb R}(x-\tilde R+1/3) a(x)\Gamma_{2\sigma^2}(x-\tilde R+1/3)\,dx<0,
\end{align*}
since $\textrm{supp }a\subset [-\tilde R+1,\tilde R-1]$. Similarly, $F(-\tilde R+1/3)>0$.

\medskip

Next, we use a contradiction argument: If $\bar t_\varepsilon>0$ is finite, then $|Z_{u_\varepsilon}(\tilde t)|=\tilde R$, and there is $\tilde t\in (1/C\alpha,\bar t_\varepsilon)$ such that
\[|Z_{u_\varepsilon}(\tilde t)|=\tilde R-1/3=\min_{t\in[0,\tilde t]}|Z_{u_\varepsilon}(t)|,\]
which is a contradiction, since
\begin{align*}
\frac d{dt}|Z_{u_\varepsilon}(\tilde t)|&=\frac{Z_{u_\varepsilon}(\tilde t)}{|Z_{u_\varepsilon}(\tilde t)|}\alpha\left[F(Z_{u_\varepsilon}(\tilde t))+ g_{\alpha,\varepsilon}(\tilde t)\right]\\
&\leq \alpha\left[\max(F(\tilde R-1/3),-F(-\tilde R+1/3))+ \frac 12\min(F(-\tilde R+1/3),-F(\tilde R-1/3))\right]<0.
\end{align*}
We have then proven that $-\tilde R<Z_{u_\varepsilon}(t)<\tilde R$ for $t\geq 0$, provided $\varepsilon,\,\alpha>0$ are small enough. Note that this property implies $\bar t_\varepsilon=\infty$, which completes the proof of \eqref{est:W2appendix}, and  implies in particular a uniform bound  $\int_{\mathbb R}x^2u_\varepsilon(t,x)\,dx\leq C$.

We have shown that there is $\bar V>0$ such that $\int_{\mathbb R}x^2u_\varepsilon(t,x)\,dx<\bar V$ for $\alpha,\,\varepsilon>0$ small enough and $(t,x)\in \mathbb R_+\times(-1/\varepsilon,1/\varepsilon)$. Then,
\begin{align}
T\left(\frac{(1+\alpha a)u_\varepsilon(t,\cdot)}{1+\alpha I_{u_\varepsilon}^\varepsilon(s)}\right)(x)&\geq \int_{-\sqrt{2\bar V}}^{\sqrt{2\bar V}}\int_{-\sqrt{2\bar V}}^{\sqrt{2\bar V}} \Gamma_{\sigma^2}\left(x-\frac{y_*+y_*'}2\right)\frac{u_\varepsilon(t,y_*)}{1+\alpha \|a\|_\infty}\frac{u_\varepsilon(t,y_*')}{1+\alpha \|a\|_\infty}\,dy_*\,dy_*'\nonumber\\
& \geq \frac 1{(1+\alpha \|a\|_\infty)^2} \Gamma_{\sigma^2}\left(|x|+\sqrt{2\bar V}\right)\left(\int_{-\sqrt{2\bar V}}^{\sqrt{2\bar V}}u_\varepsilon(t,y)\,dy\right)^2\nonumber\\
&\geq \frac 1{(1+\alpha \|a\|_\infty)^2} \Gamma_{\sigma^2}\left(|x|+\sqrt{2\bar V}\right)\left(1-\frac 1{2\bar V}\int_{\mathbb R}y^2u_\varepsilon(t,y)\,dy\right)^2\nonumber\\
&\geq \frac 1C \Gamma_{\sigma^2}\left(|x|+\sqrt{2\bar V}\right).\label{est:lowertailprecisebis}
\end{align}
We notice that \eqref{eq:integralformulation2} is a convex combination of $u_\varepsilon(0,x)$ and $T\left(\frac{(1+\alpha a)u_\varepsilon(s,\cdot)}{1+\alpha I_{u_\varepsilon}^\varepsilon(s)}\right)(x)$ for $s\in[0,t]$. Then, \eqref{est:lowertailprecisebis} and \eqref{eq:integralformulation2} implies
\[u_\varepsilon(t,x)\geq \min\left(u_\varepsilon(0,x),\frac 1C \Gamma_{\sigma^2}\left(|x|+\sqrt{2\bar V}\right)\right),\]
which provides the lower bound on $u_\varepsilon$ necessary to conclude this proof.

\end{proof}

\section{Convergence of the truncated model to the original model}\label{App:cv}

In this section, we prove that the solutions of the truncated model \eqref{eq:modifiedpb} with $\varepsilon>0$, converge to the solution of the original problem \eqref{eq:model} when $\varepsilon\to 0$. This convergence is used in the proof of Lemma~\ref{lem:tailprecise}. This result also proves the uniqueness of the solution of the original model.

\begin{lem}\label{lem:tailpreciseapproxcv}
Let $R>0$ and $a\in W^{2,\infty}(\mathbb R)$, $\bar Z\in\mathbb R$ satisfying Assumption~\ref{ass1}, and let $-\infty<Z_-<Z_+<\infty$. Let $\alpha>0$ small enough, $\rho\in(0,1]$ and $n^0\in\mathcal P_2(\mathbb R)\cap C^1(\mathbb R)$ satisfying Assumption~\ref{ass2}. For $\varepsilon>0$ small enough, let $u_\varepsilon$ the solution of \eqref{eq:modifiedpb} in $L^\infty(\mathbb R_+\times(-1/\varepsilon,1/\varepsilon))$ provided by Theorem~\ref{thm:existence} and $n\in L^\infty(\mathbb R_+\times\mathbb R)$ the solution of \eqref{eq:model} provided by Theorem~\ref{cor:existence}. Then, for $(t,x)\in[0,\infty)\times \mathbb R$,
\begin{equation*}
 \lim_{\varepsilon\to 0}u_\varepsilon(t,x)=n(t,x), \quad  \lim_{\varepsilon\to 0}\partial_x u_\varepsilon(t,x)=\partial_x n(t,x).
\end{equation*}
\end{lem}
We recall that the problem \eqref{eq:modifiedpb} with $\varepsilon=0$ is equivalent to the problem \eqref{eq:model}, so that $n$ is actually the solution of \eqref{eq:model}.

\begin{proof}[Proof of Lemma~\ref{lem:tailpreciseapproxcv}]
Let $\alpha\in(0,\bar \alpha)$ and $\varepsilon\in(0,\bar\varepsilon)$ small enough for Lemma~\ref{lem:roughmacro}, Lemma~\ref{lem:lowerbound} and Lemma~\ref{thm:unifestim} to apply. Thanks to Lemma~\ref{lem:roughmacro} (see \eqref{eq:moment2}), there is an upper bound on $\int_{\mathbb  R} x^2u_\varepsilon(t,x)\,dx$, while Lemma~\ref{thm:unifestim} (see \eqref{est:boundsApp}) implies an upper bound on $\int_{\mathbb  R} x^2u_\varepsilon(t,x)\,dx$, that is uniform in $\alpha\in(0,\bar \alpha)$ and $\varepsilon\in(0,\bar\varepsilon)$. Lemma~\ref{lem:lowerbound} provides a uniform lower bound on $u_\varepsilon(t,x))$, see \eqref{est:boundsApp}. These uniform bounds imply the following estimates, thanks to \eqref{eq:Tmoment1}:
\begin{align}
\int_{\mathbb R}\left(1+x^2\right)T\left(\frac{(1+\alpha a)u_\varepsilon(s,\cdot)}{1+\alpha I_{u_\varepsilon}^\varepsilon(s)}\right)(x)\,dx&=1+\sigma^2+\int_{\mathbb R}x^2\frac{(1+\alpha a(x))u_\varepsilon(s,x)}{1+\alpha I_{u_\varepsilon}^\varepsilon(s)}\,dx\nonumber\\
&\leq \sigma^2+ \left(1+\alpha\|a\|_\infty\right)\int_{\mathbb R}x^2u_\varepsilon(s,x)\,dx\leq C,\label{eq:estTueps2}
\end{align}
as well as, using \eqref{eq:Tmoment1} again,
\begin{align}
\int_{[-1/\varepsilon,1/\varepsilon]^c}(1+|z|)T\left(\frac{(1+\alpha a)n(s,\cdot)}{1+\alpha I_n(s)}\right)(z)\,dz&\leq \varepsilon\int_{\mathbb R}x^2 T\left(\frac{(1+\alpha a)n(s,\cdot)}{1+\alpha I_n(s)}\right)(z)\,dz\nonumber\\
&\leq C\varepsilon \left(1+\int_{\mathbb R}x^2 n(s,x)\,dx\right)\leq C\varepsilon,\label{est:Rm2}
\end{align}
and
\begin{align}
&\int_{-1/\varepsilon}^{1/\varepsilon}T\left(\frac{(1+\alpha a)u_\varepsilon(s,\cdot)}{1+\alpha I_{u_\varepsilon}^\varepsilon(s)}\right)(x)\,dx\nonumber\\
&\quad \geq \int_{-1}^{1}\int_{-1/\varepsilon}^{1/\varepsilon}\left(\int_{-1}^1\Gamma_{\sigma^2}\left(x-\frac{y_*+y_*'}2\right)\,dx\right)\frac{u_\varepsilon(s,y_*)}{1+\alpha \|a\|_\infty}\frac{u_\varepsilon(s,y_*')}{1+\alpha \|a\|_\infty}\,dy_*\,dy_*'\nonumber\\
&\quad \geq  \int_{-1}^{1}\int_{-1/\varepsilon}^{1/\varepsilon}2\Gamma_{\sigma^2}\left(2\right)\frac{\min\left(n^0(y_*),\Gamma_{\sigma^2}(|y_*|+1/\bar C)\right)}{1+\alpha \|a\|_\infty}\frac{\min\left(n^0(y_*'),\Gamma_{\sigma^2}(|y_*'|+1/\bar C)\right)}{1+\alpha \|a\|_\infty}\,dy_*\,dy_*'\nonumber\\
&\quad\geq \frac 1C>0,\label{eq:lb}
\end{align}
for a constant $C>0$ independent from $\alpha\in(0,\bar \alpha)$ and $\varepsilon\in(0,\bar\varepsilon)$. Let now $\varphi$ a Lipschitz continuous function with $\|\varphi'\|_\infty\leq 1$ and $\varphi(0)=0$. We have:
\begin{align}
&\int_{\mathbb R}\varphi(x)\left(\frac{T\left(\frac{(1+\alpha a)u_\varepsilon(s,\cdot)}{1+\alpha I_{u_\varepsilon}^\varepsilon(s)}\right)(x)1_{|x|\leq \frac 1\varepsilon}}{\int_{-1/\varepsilon}^{1/\varepsilon}T\left(\frac{(1+\alpha a)u_\varepsilon(s,\cdot)}{1+\alpha I_{u_\varepsilon}^\varepsilon(s)}\right)(y)\,dy}-T\left(\frac{(1+\alpha a(x))n(s,x)}{1+\alpha I_n(s)}\right)\right)\,dx\nonumber\\
&\quad \leq W_1\left(T\left(\frac{(1+\alpha a)u_\varepsilon(s,\cdot)}{1+\alpha I_{u_\varepsilon}^\varepsilon(s)}\right),T\left(\frac{(1+\alpha a)n(s,\cdot)}{1+\alpha I_{n}(s)}\right)\right)\nonumber\\
&\qquad +\int_{\mathbb R}\varphi(x)\left(\frac{T\left(\frac{(1+\alpha a)u_\varepsilon(s,\cdot)}{1+\alpha I_{u_\varepsilon}^\varepsilon(s)}\right)(x)1_{|x|\leq \frac 1\varepsilon}}{\int_{-1/\varepsilon}^{1/\varepsilon}T\left(\frac{(1+\alpha a)u_\varepsilon(s,\cdot)}{1+\alpha I_{u_\varepsilon}^\varepsilon(s)}\right)(y)\,dy}-T\left(\frac{(1+\alpha a)u_\varepsilon(s,\cdot)}{1+\alpha I_{u_\varepsilon}^\varepsilon(s)}\right)(x)\right)\,dx.\label{est:W1uepsilon}
\end{align}
Let us denote by $I$ the last term of this estimate, and we estimate it as follows:
\begin{align*}
I&=\int_{\mathbb R}\varphi(x)\Bigg[\frac{\int_{[-1/\varepsilon,1/\varepsilon]^c}T\left(\frac{(1+\alpha a)u_\varepsilon(s,\cdot)}{1+\alpha I_{u_\varepsilon}^\varepsilon(s)}\right)(y)\,dy}{\int_{-1/\varepsilon}^{1/\varepsilon}T\left(\frac{(1+\alpha a)u_\varepsilon(s,\cdot)}{1+\alpha I_{u_\varepsilon}^\varepsilon(s)}\right)(y)\,dy}T\left(\frac{(1+\alpha a)u_\varepsilon(s,\cdot)}{1+\alpha I_{u_\varepsilon}^\varepsilon(s)}\right)(x)1_{|x|\leq \frac 1\varepsilon}\\
&\qquad\phantom{=\int_{\mathbb R}\varphi(x)\Bigg[} -T\left(\frac{(1+\alpha a)u_\varepsilon(s,\cdot)}{1+\alpha I_{u_\varepsilon}^\varepsilon(s)}\right)(x)1_{|x|> \frac 1\varepsilon}\Bigg]\,dx\\
&\quad\leq C\left(\int_{[-1/\varepsilon,1/\varepsilon]^c}T\left(\frac{(1+\alpha a)u_\varepsilon(s,\cdot)}{1+\alpha I_{u_\varepsilon}^\varepsilon(s)}\right)(y)\,dy\right)\frac{\int_{[-1/\varepsilon,1/\varepsilon]}|y|T\left(\frac{(1+\alpha a)u_\varepsilon(s,\cdot)}{1+\alpha I_{u_\varepsilon}^\varepsilon(s)}\right)(y)\,dy}{\int_{[-1/\varepsilon,1/\varepsilon]}T\left(\frac{(1+\alpha a)u_\varepsilon(s,\cdot)}{1+\alpha I_{u_\varepsilon}^\varepsilon(s)}\right)(y)\,dy}\\
&\qquad+\int_{[-1/\varepsilon,1/\varepsilon]^c}|x|T\left(\frac{(1+\alpha a)u_\varepsilon(s,\cdot)}{1+\alpha I_{u_\varepsilon}^\varepsilon(s)}\right)(y)\,dy\\
&\quad\leq C\left(\varepsilon\int_{[-1/\varepsilon,1/\varepsilon]^c}|y|^2T\left(\frac{(1+\alpha a)u_\varepsilon(s,\cdot)}{1+\alpha I_{u_\varepsilon}^\varepsilon(s)}\right)(y)\,dy\right)\frac{\int_{[-1/\varepsilon,1/\varepsilon]}\frac{1+|y|^2}2T\left(\frac{(1+\alpha a)u_\varepsilon(s,\cdot)}{1+\alpha I_{u_\varepsilon}^\varepsilon(s)}\right)(y)\,dy}{\int_{[-1/\varepsilon,1/\varepsilon]}T\left(\frac{(1+\alpha a)u_\varepsilon(s,\cdot)}{1+\alpha I_{u_\varepsilon}^\varepsilon(s)}\right)(y)\,dy}\\
&\qquad+\varepsilon\int_{[-1/\varepsilon,1/\varepsilon]^c}|x|^2T\left(\frac{(1+\alpha a)u_\varepsilon(s,\cdot)}{1+\alpha I_{u_\varepsilon}^\varepsilon(s)}\right)(y)\,dy\leq C\varepsilon,
\end{align*}
where we have used the fact that $\left|\varphi(x)\right|\leq |x|$, Chebyshev's inequalities as well as \eqref{eq:estTueps2} and \eqref{eq:lb}. Taking the suppremum over test functions $\varphi$ such that $\|\varphi'\|_\infty\leq 1$ and $\varphi(0)=0$, the estimate \eqref{est:W1uepsilon} becomes, thanks to the estimate above on $I$,
\begin{align}
&W_1\left(\frac{T\left(\frac{(1+\alpha a)u_\varepsilon(s,\cdot)}{1+\alpha I_{u_\varepsilon}^\varepsilon(s)}\right)(\cdot)1_{|\cdot|\leq \frac 1\varepsilon}}{\int_{-1/\varepsilon}^{1/\varepsilon}T\left(\frac{(1+\alpha a)u_\varepsilon(s,\cdot)}{1+\alpha I_{u_\varepsilon}^\varepsilon(s)}\right)(x)\,dx},T\left(\frac{(1+\alpha a)n(s,\cdot)}{1+\alpha I_{n}(s)}\right)\right)\nonumber\\
&\quad \leq  W_1\left(T\left(\frac{(1+\alpha a)u_\varepsilon(s,\cdot)}{1+\alpha I_{u_\varepsilon}^\varepsilon(s)}\right),T\left(\frac{(1+\alpha a)n(s,\cdot)}{1+\alpha I_{n}(s)}\right)\right)+C\varepsilon\nonumber\\
&\quad \leq  W_1\left(\frac{(1+\alpha a)u_\varepsilon(s,\cdot)}{1+\alpha I_{u_\varepsilon}^\varepsilon(s)},\frac{(1+\alpha a)n(s,\cdot)}{1+\alpha I_{n}(s)}\right)+C\varepsilon,\label{eq:W1Ce}
\end{align}
where the last inequality is a consequence of \eqref{est:W1}. Let $\varphi$ a Lipschitz continuous function with $\|\varphi'\|_\infty\leq 1$ and $\varphi(0)=0$. Then,
\begin{align*}
&\int_{\mathbb R}\varphi(x)\left(\frac{(1+\alpha a(x))u_\varepsilon(s,x)}{1+\alpha I_{u_\varepsilon}^\varepsilon(s)}-\frac{(1+\alpha a(x))n(s,x)}{1+\alpha I_{n}(s)}\right)\,dx\\
&\quad \leq \frac{1}{1+\alpha I_{u_\varepsilon}^\varepsilon(s)}\int_{\mathbb R}\varphi(x)(1+\alpha a(x))\left(u_\varepsilon(s,x)-n(s,x)\right)\,dx\\
&\qquad +\frac{\alpha |I_{n}(s)-I_{u_\varepsilon}^\varepsilon(s)|}{(1+\alpha I_{u_\varepsilon}^\varepsilon(s))(1+\alpha I_{n}(s))}\int_{\mathbb R}|x|n(s,x)\,dx\\
&\quad \leq \left((1+\alpha \|a\|_\infty)\|\varphi'\|_\infty+ \alpha\|a'\|\,\|\varphi\|_{L^\infty(\textrm{supp }a)}\right)\int_{\mathbb R}\frac{\varphi(x)(1+\alpha a(x))}{\left\|\frac d{dx}\left(\varphi(1+\alpha a)\right)(x)\right\|_\infty}\left(u_\varepsilon(s,x)-n(s,x)\right)\,dx\\
&\qquad + C\alpha \|a'\|_{L^\infty} W_1(u_\varepsilon(s,\cdot),n(s,\cdot))\int_{\mathbb R}\frac{1+|x|^2}2n(s,x)\,dx\leq (1+C\alpha)W_1(u_\varepsilon(s,\cdot),n(s,\cdot)),
\end{align*}
Finally, taking the suppremum over such test functions $\varphi$, this inequality combined to \eqref{eq:W1Ce} implies
\begin{align}
&W_1\left(\frac{T\left(\frac{(1+\alpha a)u_\varepsilon(s,\cdot)}{1+\alpha I_{u_\varepsilon}^\varepsilon(s)}\right)(\cdot)1_{|\cdot|\leq \frac 1\varepsilon}}{\int_{-1/\varepsilon}^{1/\varepsilon}T\left(\frac{(1+\alpha a)u_\varepsilon(s,\cdot)}{1+\alpha I_{u_\varepsilon}^\varepsilon(s)}\right)(x)\,dx},T\left(\frac{(1+\alpha a)n(s,\cdot)}{1+\alpha I_{n}(s)}\right)\right)\nonumber\\
&\quad \leq  (1+C\alpha)W_1(u_\varepsilon(s,\cdot),n(s,\cdot))+C\varepsilon.\label{eq:W1nu}
\end{align}
Equipped with these inequalities, we can estimate the distance between solutions of \eqref{eq:model} and \eqref{eq:modifiedpb} for $t\in[0,\bar t]$ (where $\bar t>0$ is fixed), when $\varepsilon>0$ is small. Thanks to the integral formulation \eqref{eq:integralformulation}, a similar integral formulation for the approximate model (see \eqref{eq:integralformulation2}) and the convexity of the optimal cost $(x-y)\mapsto |x-y|$ (see Theorem~4.8 in \cite{villani2008optimal}), we have, for $t\in[0,\bar t]$ and any Lipschitz continuous function $\varphi$ satisfying $\|\varphi'\|_\infty\leq 1$ and $\varphi(0)=0$,
\begin{align}
&\int_{\mathbb R}\varphi(x)\left(n(t,x)-u_\varepsilon(t,x)\right)\,dx=\left(\int_{\mathbb R}\varphi(x)\left(n(0,x)-u_\varepsilon(0,x)\right)\,dx\right)e^{-\int_0^t\left(1+\alpha I_n(s)\right)^2\,ds}\nonumber\\
&\qquad + \left(\int_{\mathbb R}\varphi(x)u_\varepsilon(0,x)\,dx\right)\left\{e^{-\int_0^t\left(1+\alpha I_n(s)\right)^2\,ds}-e^{-\int_0^t\left(1+\alpha I_{u_\varepsilon}^\varepsilon(s)\right)^2\left(\int_{-1/\varepsilon}^{1/\varepsilon}T\left(\frac{(1+\alpha a)u_\varepsilon(s,\cdot)}{1+\alpha I_{u_\varepsilon}(s)}\right)(z)\,dz\right)\,ds}\right\}\nonumber\\
&\quad +\int_0^t\bigg\{\int_{\mathbb R} \varphi(x)\Big[\left(1+\alpha I_n(s)\right)^2\left(T\left(\frac{(1+\alpha a)n(s,\cdot)}{1+\alpha I_n(s)}\right)(x)-T\left(\frac{(1+\alpha a)u_\varepsilon(s,\cdot)}{1+\alpha I_{u_\varepsilon}^\varepsilon(s)}\right)(x)\right)\nonumber\\
&\qquad \phantom{+\int_0^t\bigg\{\int_{\mathbb R} \varphi(x)\Big[} +\left(\left(1+\alpha I_{u_\varepsilon}^\varepsilon(s)\right)^2-\left(1+\alpha I_n(s)\right)^2\right)T\left(\frac{(1+\alpha a)u_\varepsilon(s,\cdot)}{1+\alpha I_{u_\varepsilon}^\varepsilon(s)}\right)(x)\nonumber\\
&\qquad \phantom{+\int_0^t\bigg(\int_{\mathbb R} \varphi(x)}+\left(1+\alpha I_{u_\varepsilon}^\varepsilon(s)\right)^2T\left(\frac{(1+\alpha a)u_\varepsilon(s,\cdot)}{1+\alpha I_{u_\varepsilon}^\varepsilon(s)}\right)(x)1_{|x|\geq  \frac 1\varepsilon}\Big]\,dx\bigg\}e^{-\int_s^t\left(1+\alpha I_n(\tau)\right)^2\,d\tau}\,ds\nonumber\\
&\quad +\int_0^t\left(\int_{\mathbb R}\varphi(x)\left(1+\alpha I_{u_\varepsilon}^\varepsilon(s)\right)^2T\left(\frac{(1+\alpha a)u_\varepsilon(s,\cdot)}{1+\alpha I_{u_\varepsilon}^\varepsilon(s)}\right)(x)1_{|x|< \frac 1\varepsilon}\,dx\right)\nonumber\\
&\qquad \phantom{+\int_0^t\bigg(}\left\{e^{-\int_s^t\left(1+\alpha I_n(\tau)\right)^2\,d\tau}-e^{-\int_s^t\left(1+\alpha I_{u_\varepsilon}^\varepsilon(\tau)\right)^2\left(\int_{-1/\varepsilon}^{1/\varepsilon}T\left(\frac{(1+\alpha a)u_\varepsilon(\tau,\cdot)}{1+\alpha I_{u_\varepsilon}^\varepsilon(\tau)}\right)(z)\,dz\right)\,d\tau}\right\}\,ds,\nonumber
\end{align}
and then, thanks to the Kantorovich-Rubinstein inequality,
\begin{align}
&\int_{\mathbb R}\varphi(x)\left(n(t,x)-u_\varepsilon(t,x)\right)\,dx\nonumber\\
&\quad \leq W_1\left(n(0,\cdot)-u_\varepsilon(0,\cdot)\right)e^{-t}\nonumber+ \left(\int_{\mathbb R}\frac {1+x^2}2n(0,x)\,dx\right)\\
&\qquad \int_0^t\bigg|\left(1+\alpha I_n(s)\right)^2-\left(1+\alpha I_{u_\varepsilon}^\varepsilon(s)\right)^2\bigg(\int_{-1/\varepsilon}^{1/\varepsilon}T\left(\frac{(1+\alpha a)u_\varepsilon(s,\cdot)}{1+\alpha I_{u_\varepsilon}(s)}\right)(z)\,dz\bigg)\bigg|\,ds\nonumber\\
&\quad +\int_0^t\bigg[\left(1+\alpha \|a\|_\infty\right)^2W_1\left(T\left(\frac{(1+\alpha a)n(s,\cdot)}{1+\alpha I_n(s)}\right),T\left(\frac{(1+\alpha a)u_\varepsilon(s,\cdot)}{1+\alpha I_{u_\varepsilon}^\varepsilon(s)}\right)\right)\nonumber\\
&\qquad\phantom{+\int_0^t\bigg[} +C\alpha (2+2\alpha\|a\|_\infty)W_1(u_\varepsilon(s,\cdot),n(s,\cdot))\int_{\mathbb R}\frac {1+x^2}2T\left(\frac{(1+\alpha a)n(s,\cdot)}{1+\alpha I_n(s)}\right)(x)\,dx\nonumber\\
&\qquad \phantom{+\int_0^t\bigg[}+(1+\alpha\|a\|_\infty)^2\int_{[-1/\varepsilon,1/\varepsilon]^c}|x|T\left(\frac{(1+\alpha a)n(s,\cdot)}{1+\alpha I_n(s)}\right)(x)\,dx\bigg\}e^{-(t-s)}\,ds\nonumber\\
&\quad +\int_0^t\left(\left(1+\alpha \|a\|_\infty\right)^2\int_{\mathbb R}\frac {1+x^2}2T\left(\frac{(1+\alpha a)n(s,\cdot)}{1+\alpha I_n(s)}\right)(x)\,dx\right)\nonumber\\
&\qquad \phantom{+\int_0^t\bigg(}\left\{\int_s^t\left|\left(1+\alpha I_n(\tau)\right)^2-\left(1+\alpha I_{u_\varepsilon}^\varepsilon(\tau)\right)^2\left(\int_{-1/\varepsilon}^{1/\varepsilon}T\left(\frac{(1+\alpha a)u_\varepsilon(\tau,\cdot)}{1+\alpha I_{u_\varepsilon}^\varepsilon(\tau)}\right)(z)\,dz\right)\right|\,d\tau\right\}\,ds.\label{eq:estimap}
\end{align}
We notice that for $t\geq 0$ and $s\in[0,t]$,
\begin{align*}
&\int_s^t\left|\left(1+\alpha I_n(\tau)\right)^2-\left(1+\alpha I_{u_\varepsilon}^\varepsilon(\tau)\right)^2\left(\int_{-1/\varepsilon}^{1/\varepsilon}T\left(\frac{(1+\alpha a)u_\varepsilon(\tau,\cdot)}{1+\alpha I_{u_\varepsilon}^\varepsilon(\tau)}\right)(z)\,dz\right)\right|\,d\tau\\
&\quad \leq \alpha (2+2\|a\|_\infty)\int_s^t  \left|I_n(\tau)- I_{u_\varepsilon}^\varepsilon(\tau)\right|\,d\tau\\
&\qquad +(1+\alpha \|a\|_\infty)^2\int_s^t\int_{[-1/\varepsilon,1/\varepsilon]^c}T\left(\frac{(1+\alpha a)n(\tau,\cdot)}{1+\alpha I_n(\tau)}\right)(x)\,dx\,d\tau\\
&\leq C\alpha \|a'\|_\infty \int_s^t  W_1(u_\varepsilon(\tau,\cdot),n(\tau,\cdot))\,d\tau+C\varepsilon (t-s),
\end{align*}
thanks to the Kantorovich-Rubinstein inequality and \eqref{est:Rm2}. Thanks to this inequality, \eqref{eq:W1nu}, \eqref{eq:estTueps2} and \eqref{est:Rm2}, estimate \eqref{eq:estimap} implies
\begin{align}
&\int_{\mathbb R}\varphi(x)\left(n(t,x)-u_\varepsilon(t,x)\right)\,dx\nonumber\\
&\quad \leq W_1(n(0,\cdot),u_\varepsilon(0,\cdot))+\int_0^t\left((1+C\alpha)W_1(n(s,\cdot),u_\varepsilon(s,\cdot))+C\varepsilon\right)\,ds,\nonumber
\end{align}
and taking the maximum over Lipschitz continuous test functions $\varphi$ satisfying $\|\varphi'\|_\infty\leq 1$ and $\varphi(0)=0$,
\begin{align*}
&W_1\left(n(t,\cdot),u_\varepsilon(t,\cdot)\right) \leq W_1\left(u_\varepsilon(0,\cdot),n(0,\cdot)\right)+C\int_0^t\left(W_1(u_\varepsilon(s,\cdot),n(s,\cdot))+\varepsilon\right)\,ds.
\end{align*}
We can apply a Gronwall inequality to show the convergence of $u_\varepsilon(t,\cdot)$ to $n(t,\cdot)$ for the distance $W_1$ (this convergence is locally iniform in $t\in[0,\infty)$), when $\varepsilon\to 0$:
\begin{align*}
W_1\left(n(t,\cdot),u_\varepsilon(t,\cdot)\right)& \leq W_1\left(u_\varepsilon(0,\cdot),n(0,\cdot)\right)+C\varepsilon t+C\int_0^t\left(W_1\left(u_\varepsilon(0,\cdot),n(0,\cdot)\right)+C\varepsilon s\right)e^{C(t-s)}\,ds\\
&\leq C\varepsilon e^{Ct}.
\end{align*}
To obtain a pointwise convergence of $u_\varepsilon(t,x)$ and  $\partial u_\varepsilon(t,x)$ (for $t\geq 0$ $x\in\mathbb R$), we notice that for $m,\tilde m\in\mathcal P_2(\mathbb R)$,
\begin{align*}
\left\|T(m)-T(\tilde m)\right\|_\infty&=\sup_{x\in\mathbb R}\left|\int_{\mathbb R}\int_{\mathbb R} \Gamma_{\sigma^2}\left(x-\frac {y_*+y_*'}2\right)m(y_*)m(y_*')-\tilde m(y_*)\tilde m(y_*')\,dy_*\,dy_*'\right|\\
&\leq \sup_{x\in\mathbb R}\left|\int_{\mathbb R} \left(\int_{\mathbb R}\Gamma_{\sigma^2}\left(x-\frac {y_*+y_*'}2\right)m(y_*)\,dy_*\right)\left(m(y_*')-\tilde m(y_*')\right)\,dy_*'\right|\\
&\quad + \sup_{x\in\mathbb R}\left|\int_{\mathbb R} \left(\int_{\mathbb R}\Gamma_{\sigma^2}\left(x-\frac {y_*+y_*'}2\right)\tilde m(y_*')\,dy_*'\right)\left(m(y_*)-\tilde m(y_*)\right)\,dy_*\right|\\
&\leq \|\Gamma_{\sigma^2}'\|_\infty W_1(m,\tilde m).
\end{align*}
Similarly, $\left|\partial_xT(m)-\partial_x T(\tilde m)\right|\leq \|\Gamma_{\sigma^2}''\|_\infty W_1(m,\tilde m)$. Let $t>0$ a given time. Thanks to estimates above, there is a constant $C_t>0$ (that is not independent from $t\geq 0$), such that for $s\in(0,t)$,
\[\left\|T(n(s,\cdot))-T(u_\varepsilon(s,\cdot))\right\|_\infty\leq C_t\varepsilon,\quad \left|\partial_x T(n(s,\cdot))-\partial_x T(u_\varepsilon(s,\cdot))\right|\leq C_t\varepsilon,\]
\[\left|I_n(s)-I_{u_\varepsilon}^\varepsilon(s)\right|\leq \|a'\|_\infty W_1\left(n(s,\cdot),u_\varepsilon(s,\cdot)\right)\leq C_t\varepsilon,\]
and thanks to the definition \eqref{def:inidatamod} of $u_\varepsilon(0,\cdot)$ and Assumption~\ref{ass2}, 
\begin{align*}
\left|u_\varepsilon(0,\cdot)-n(0,\cdot)\right|&\leq \|n(0,\cdot)\|_\infty \frac 1{1-\int_{[-1/\varepsilon,1/\varepsilon]^c}n(0,y)\,dy}\leq   \frac {\|n(0,\cdot)\|_\infty}{1-\varepsilon^2\int_{\mathbb R}x^2n^0(x)\,dx}\\
&\leq \frac {\|n(0,\cdot)\|_\infty}{1-\varepsilon^2\left(W_2(n^0,\Gamma_{2\sigma^2}(\cdot-Z^0))+W_2(\Gamma_{2\sigma^2}(\cdot-Z^0),\delta_0)\right)}\leq C\varepsilon^2,
\end{align*}
provided $\varepsilon>0$ is small enough. Let $(t,x)\in[0,\infty)\times \mathbb R$. For $\varepsilon>0$ is small enough,the integral formulations of $n(t,x)$ and $u_\varepsilon$ (see \eqref{eq:integralformulation}, \eqref{eq:integralformulation2}),

\begin{align*}
&\left| u_\varepsilon(t,x)-n(t,x)\right|\\
&\quad \leq \left| u_\varepsilon(0,x)e^{-\int_0^t\left(1+\alpha I_{u_\varepsilon}^\varepsilon(s)\right)^2\left(\int_{-1/\varepsilon}^{1/\varepsilon}T\left(\frac{(1+\alpha a)u_\varepsilon(s,\cdot)}{1+\alpha I_{u_\varepsilon}^\varepsilon(s)}\right)(z)\,dz\right)\,ds}-n(0,x)e^{-\int_0^t\left(1+\alpha I_n(s)\right)^2\,ds}\right|\nonumber\\
&\quad +\int_0^t\bigg|\left(1+\alpha I_{u_\varepsilon}^\varepsilon(s)\right)^2T\left(\frac{(1+\alpha a)u_\varepsilon(s,\cdot)}{1+\alpha I_{u_\varepsilon}^\varepsilon(s)}\right)(x)\\
&\qquad \phantom{dgfezrsdg}e^{-\int_s^t\left(1+\alpha I_{u_\varepsilon}^\varepsilon(\tau)\right)^2\left(\int_{-1/\varepsilon}^{1/\varepsilon}T\left(\frac{(1+\alpha a)u_\varepsilon(\tau,\cdot)}{1+\alpha I_{u_\varepsilon}^\varepsilon(\tau)}\right)(z)\,dz\right)\,d\tau}\\
&\qquad -\left(1+\alpha I_n(s)\right)^2T\left(\frac{(1+\alpha a)n(s,\cdot)}{1+\alpha\int_{\mathbb R}a(y)n(s,y)\,dy}\right)(x)e^{-\int_s^t\left(1+\alpha I_n(\tau)\right)^2\,d\tau}\bigg|\,ds\leq C_t\varepsilon\to_{\varepsilon\to 0} 0.
\end{align*}
Similarly, $\left|\partial_x u_\varepsilon(t,x)-\partial_x n(t,x)\right|\to 0$ when $\varepsilon\to 0$, and these estimates conclude the proof of Lemma~\ref{lem:tailpreciseapproxcv}.

\end{proof}

\section{$W_1$ estimate on the operator $T$}\label{Appendix:batTrough}
The operator $T$ is not strictly contractant for $W_1$, but the following holds: 
For $n,m\in\mathcal P_2(\mathbb R)$,
\begin{equation}\label{est:W1}
W_1(T(n),T(m))\leq W_1(n,m).
\end{equation}
Indeed, using the fact that a convolution is a parallel transport, we have
\begin{align*}
W_1(T(n),T(m))&=W_1\left(\Gamma_{\sigma^2}\ast \iint_{\mathbb R^2} \delta_{\cdot=\frac{y_*+y_*'}2}\,dn(y_*)\,dn(y_*'),\Gamma_{\sigma^2}\ast \iint_{\mathbb R^2} \delta_{\cdot=\frac{y_*+y_*'}2}\,dm(y_*)\,dm(y_*')\right)\\
&\leq W_1\left(\iint_{\mathbb R^2} \delta_{\cdot=\frac{y_*+y_*'}2}\,dn(y_*)\,dn(y_*'),\iint_{\mathbb R^2} \delta_{\cdot=\frac{y_*+y_*'}2}\,dm(y_*)\,dm(y_*')\right),
\end{align*}
and if $\pi\in\mathcal P_2(\mathbb R^2)$ is the optimal transport from $n$ to $m$, we introduce $\bar\pi\in\mathcal P_2(\mathbb R^2)$ defined by
\[\iint_{\mathbb R^2}\varphi(x,y)\,d\bar \pi(x,y)=\iint_{\mathbb R^2}\varphi(x,y)\iint_{\mathbb R^2} \delta_{x=\frac{x_*+y_*}2}\delta_{y=\frac{x_*'+y_*'}2}\,d\pi(x_*,x_*')\,d\pi(y_*,y_*'),\]
for any continuous test function $\varphi$. $\bar \pi$ is a transport map from $\iint_{\mathbb R^2} \delta_{\cdot=\frac{y_*+y_*'}2}\,dn(y_*)\,dn(y_*')$ to $\iint_{\mathbb R^2} \delta_{\cdot=\frac{y_*+y_*'}2}\,dm(y_*)\,dm(y_*')$. This can be noticed with test functions of the form $\varphi(x,y)=\psi(x)$ and  $\varphi(x,y)=\psi(y)$. Then,
\begin{align*}
W_1(T(n),T(m))&\leq \iint_{\mathbb R^2}|x-y|\,d\bar\pi(x,y)=\iint_{\mathbb R^2}\left|\frac{x_*+y_*}2-\frac{x_*'+y_*'}2\right|\,d\pi(x_*,x_*')\,d\pi(y_*,y_*')\nonumber\\
&\leq \frac 12\iint_{\mathbb R^2}\left(\left|x_*-x_*'\right|+\left|y_*-y_*'\right|\right)\,d\pi(x_*,x_*')\,d\pi(y_*,y_*')\leq W_1(n,m).
\end{align*}

\section{Technical estimates on tails of steady states}\label{App:tails}

In this section, we provide tail estimates on the distribution $T\left(\frac{(1+\alpha a(\cdot))n}{1+\alpha I_n}\right)'(x)$, provided the tails of $n$ satisfy similar properties. The proof of this result is closely related to an argument used in the proof of Lemma~\ref{lem:tailprecise}, where solutions of the time-dependent model are considered.
\begin{lem}\label{lem:tailpreliminary}
Let $R>0$ and $a\in W^{2,\infty}(\mathbb R)$, $\bar Z\in\mathbb R$ satisfying Assumption~\ref{ass1}, and 
\[\eta:=\int_{|\bar Z|+R+4\sigma^2+1}^\infty \Gamma_{2\sigma^2}(x)\,dx.\]
There is $\beta>0$ such that if $n\in \mathcal P_2(\mathbb R)$ satisfies
\begin{align}\label{eq:condn}
 n'(x)<-n(x)\textrm{ if }\int_x^\infty n\leq \eta, \quad  n'(-x)>n(-x) \textrm{ if }\int_{-\infty}^x n\leq \eta,
\end{align}
\begin{equation}\label{eq:condW}
W_2(n, \Gamma_{2\sigma^2}(\cdot-\bar Z))\leq \beta,
\end{equation}
then, for $\alpha\in(0,\beta)$,
\begin{align}
 &T\left(\frac{(1+\alpha a(\cdot))n}{1+\alpha I_n}\right)'(x)<-T\left(\frac{(1+\alpha a(\cdot))n}{1+\alpha I_n}\right)(x)\textrm{ if }\int_x^\infty T\left(\frac{(1+\alpha a(\cdot))n}{1+\alpha I_n}\right)\leq \eta,\label{eq:condTn1}\\
 & T\left(\frac{(1+\alpha a(\cdot))n}{1+\alpha I_n}\right)'(-x)>T\left(\frac{(1+\alpha a(\cdot))n}{1+\alpha I_n}\right)(-x) \textrm{ if }\int_{-\infty}^x T\left(\frac{(1+\alpha a(\cdot))n}{1+\alpha I_n}\right)\leq \eta.\label{eq:condTn2}
\end{align}
\end{lem}

\begin{proof}[Proof of Lemma~\ref{lem:tailpreliminary}]
Let $R_0:=|\bar Z|+R+4\sigma^2+1$ and $\chi:\mathbb R\to[0,1]$ the piecewise linear and Lipschitz continuous function defined by
\[\chi(x)\equiv 0\textrm{ for }x\in(-\infty,0],\quad  \chi(x)=x\textrm{ for }\forall x\in[0,1],\quad \chi(x)\equiv 1\textrm{ for }x\in[1,\infty).\]
Thanks to the Kantorovich-Rubinstein inequality,
\begin{align*}
\int_{R_0+1}^\infty n&\leq \int_{\mathbb R}\chi(x-R_0)n(x)\,dx=\int_{\mathbb R}\chi(x-R_0)\Gamma_{2\sigma^2}(x)\,dx+\int_{\mathbb R}\chi(x-R_0)\left(n(x)-\Gamma_{2\sigma^2}(x)\right)\,dx\\
&\leq\int_{\mathbb R}\chi(x-R_0)\Gamma_{2\sigma^2}(x)\,dx
+W_1(n, \Gamma_{2\sigma^2}(\cdot-\bar Z))\\
&\leq\int_{\mathbb R}\chi(x-R_0)\Gamma_{2\sigma^2}(x)\,dx
+W_2(n, \Gamma_{2\sigma^2}(\cdot-\bar Z)).
\end{align*}
Since $\eta=\int_{R_0}^\infty \Gamma_{2\sigma^2}(x)\,dx$, we have  
\[\eta-\int_{\mathbb R}\chi(x-R_0)\Gamma_{2\sigma^2}(x)\,dx> \eta-\int_{R_0}^\infty\Gamma_{2\sigma^2}(x)\,dx=0,\]
and we can define $\beta:=\eta-\int_{\mathbb R}\chi(x-R_0)\Gamma_{2\sigma^2}(x)\,dx>0$, to obtain 
\begin{align*}
\int_{R_0+1}^\infty n&\leq\int_{\mathbb R}\chi(x-R_0)\Gamma_{2\sigma^2}(x)\,dx
+W_2(n, \Gamma_{2\sigma^2}(\cdot-\bar Z))=(\eta-\beta)+\beta\leq \eta.
\end{align*}
The assumption \eqref{eq:condn} made on $n\in\mathcal P_2(\mathbb R)$ then implies 
\begin{align}\label{eq:condn2}
 n'(x)<-n(x)\textrm{ if }x\in[R_0+1,\infty).
\end{align}
We introduce now $R_1:=8(1+\|a\|_\infty)$ and  $\varphi(x)=\chi\left(\frac{x-(R_0+1)}{R_1}\right)$, the piecewise linear function such that $\varphi(x)=0$ on $(-\infty,R_0+1]$, $\varphi(x)=1$ on $[R_0+1+R_1,\infty)$ and $\|\partial_x\varphi\|_{L^\infty}\leq \frac 1{R_1}$. We use this function $\varphi$ to decompose the following integral expression:
\begin{align*}
&T\left(\frac{(1+\alpha a(\cdot))n}{1+\alpha I_n}\right)'(x)=\iint_{\mathbb R^2} \partial_x\Gamma_{\sigma^2}\left(x-\frac{y_*+y_*'}2\right)(1+\alpha a(y_*))n(y_*)(1+\alpha a(y_*'))n(y_*')\,dy_*\,dy_*'\\
&\quad=\iint_{\mathbb R^2}\left(1-\varphi(y_*)\right) \left(1-\varphi(y_*')\right) \partial_x\Gamma_{\sigma^2}\left(x-\frac{y_*+y_*'}2\right)(1+\alpha a(y_*))n(y_*)(1+\alpha a(y_*'))n(y_*')\,dy_*\,dy_*'\\
&\qquad +\iint_{\mathbb R^2} \varphi(y_*) \partial_x\Gamma_{\sigma^2}\left(x-\frac{y_*+y_*'}2\right)(1+\alpha a(y_*))n(y_*)(1+\alpha a(y_*'))n(y_*')\,dy_*\,dy_*'\\
&\qquad +\iint_{\mathbb R^2} \varphi(y_*')\left(1-\varphi(y_*)\right) \partial_x\Gamma_{\sigma^2}\left(x-\frac{y_*+y_*'}2\right)(1+\alpha a(y_*))n(y_*)(1+\alpha a(y_*'))n(y_*')\,dy_*\,dy_*'\\
&\quad = I_1(x)+I_2(x)+I_3(x).
\end{align*}
We use an integration by parts in $y_*$ and \eqref{eq:condn2} to estimate $I_2$:
\begin{align*}
I_2(x)&=-2\iint_{\mathbb R^2} \varphi(y_*)\frac{d}{dy_*}\left[\Gamma_{\sigma^2}\left(x-\frac{y_*+y_*'}2\right)\right](1+\alpha a(y_*))n(y_*)(1+\alpha a(y_*'))n(y_*')\,dy_*\,dy_*'\\
&\quad = 2\iint_{\mathbb R^2} \Big[\varphi'(y_*)(1+\alpha a(y_*))n(t,y_*)+\varphi(y_*)\big((1+\alpha a(y_*))\partial_x n(y_*)+\alpha a'(y_*)n(y_*)\big)\Big]\\
&\phantom{dfgdregrsgrt} \Gamma_{\sigma^2}\left(x-\frac{y_*+y_*'}2\right)(1+\alpha a(y_*'))n(y_*')\,dy_*\,dy_*'\\
&\quad \leq\iint_{\mathbb R^2} \left[C\left(\frac {1}{R_1}+\alpha\right)-2\varphi(y_*)\right]\Gamma_{\sigma^2}\left(x-\frac{y_*+y_*'}2\right)n(y_*)n(y_*')\,dy_*\,dy_*'.
\end{align*}
where $\left|C\left(\frac {1}{R_1}+\alpha\right)\right|\leq \frac 18$ if $1/R_1$ and $\alpha>0$ are small enough. $I_3(x)$ can be estimated in a similar manner (through an integration by parts in $y_*'$), to obtain:
\begin{align*}
I_3&=\iint_{\mathbb R^2} \left(1-\varphi(y_*)\right) \left[C\left(\frac 1{R_1}+\alpha\right)-2\varphi(y_*')\right]\Gamma_{\sigma^2}\left(x-\frac{y_*+y_*'}2\right)n(y_*)n(y_*')\,dy_*\,dy_*'
\end{align*}

\medskip

To estimate $I_1(x)$, we  notice that $(1-\varphi(y_*))(1-\varphi(y_*'))=0$ as soon as $y_*\geq R_0+R_1$, or $y_*'\geq R_0+R_1$. Moreover, for $x\geq R_0+1+R_1+2\sigma^2$ and $y_*,y_*'\leq R_0+1+R_1$, 
\[\partial_x\Gamma_{\sigma^2}\left(x-\frac{y_*+y_*'}2\right)\leq -2\Gamma_{\sigma^2}\left(x-\frac{y_*+y_*'}2\right),\]
and then,
\begin{align*}
I_1(x)&\leq -2\iint_{\mathbb R^2}\left(1-\varphi(y_*)\right) \left(1-\varphi(y_*')\right) \Gamma_{\sigma^2}\left(x-\frac{y_*+y_*'}2\right)n(y_*)n(y_*')\,dy_*\,dy_*'.
\end{align*}
Our original estimate then becomes, for $x\geq R_0+1+R_1+2\sigma^2$:
\begin{align*}
&T\left(\frac{(1+\alpha a(\cdot))n}{1+\alpha I_n}\right)'(x)\\
&\quad \leq \iint_{\mathbb R^2}\left(2C\left(\frac {1+\alpha}{R_1}+\alpha\right)-2\right)\Gamma_{\sigma^2}\left(x-\frac{y_*+y_*'}2\right)n(y_*)n(y_*')\,dy_*\,dy_*'\\
&\quad \leq -\frac 74\iint_{\mathbb R^2}\Gamma_{\sigma^2}\left(x-\frac{y_*+y_*'}2\right)n(y_*)n(y_*')\,dy_*\,dy_*',
\end{align*}
provided $\frac 1{R_1}>0$ and $\alpha>0$ are small enough for $\left|C\left(\frac {1+\alpha}{R_1}+\alpha\right)\right|\leq \frac 18$ to hold. Then, if $\frac 1{R_1}>0$ and $\alpha>0$ are small enough, for $x\in[R_0+1+R_1+2\sigma^2,\infty)$,
\begin{align}
 T\left(\frac{(1+\alpha a(\cdot))n}{1+\alpha I_n}\right)'(x) &\leq -\frac 32\iint_{\mathbb R^2}\Gamma_{\sigma^2}\left(x-\frac{y_*+y_*'}2\right)(1+\alpha a(y_*))n(y_*)(1+\alpha a(y_*'))n(y_*')\,dy_*\,dy_*'\\
 &\leq  -\frac 32T\left(\frac{(1+\alpha a(\cdot))n}{1+\alpha I_n}\right)(x).\label{eq:TlargexApp}
\end{align}
To obtain a similar estimate for $x\in [R_0-1,R_0+1+R_1+2\sigma^2]$, we introduce a different estimate, using the Lipschitz continuity of $(y_*,y_*')\mapsto \Gamma_{\sigma^2}\left(x-\frac{y_*+y_*'}2\right)$ and the Kantorovich-Rubinstein formula: 
\begin{align}
T\left(\frac{(1+\alpha a(\cdot))n}{1+\alpha I_n}\right)'(x)&=\iint_{\mathbb R^2} \partial_x\Gamma_{\sigma^2}\left(x-\frac{y_*+y_*'}2\right)n(y_*)n(y_*')\,dy_*\,dy_*'\nonumber\\
&=\frac d{dx}\iint_{\mathbb R^2} \Gamma_{\sigma^2}\left(x-\frac{y_*+y_*'}2\right)\Gamma_{2\sigma^2}(y_*-\bar Z)\Gamma_{2\sigma^2}(y_*'-\bar Z)\,dy_*\,dy_*'\nonumber\\
&\quad +\mathcal O\Big(W_1\big(n,\Gamma_{2\sigma^2}(\cdot-\bar Z)\big)\Big)= \Gamma_{2\sigma^2}'(x-\bar Z)+ \mathcal O(\beta),\nonumber
\end{align}
and then, for $x\in [R_0-1,R_0+1+R_1+2\sigma^2]$,
\begin{align}
&T\left(\frac{(1+\alpha a(\cdot))n}{1+\alpha I_n}\right)'(x)=-\frac {x-\bar Z}{2\sigma^2}\Gamma_{2\sigma^2}(x-\bar Z)+ \mathcal O(\beta)\nonumber\\
&\quad =-\frac {x-\bar Z}{2\sigma^2}\iint_{\mathbb R^2} \Gamma_{\sigma^2}\left(x-\frac{y_*+y_*'}2\right)\Gamma_{2\sigma^2}(y_*-\bar Z)\Gamma_{2\sigma^2}(y_*'-\bar Z)\,dy_*\,dy_*'+ \mathcal O(\beta)\nonumber\\
&\quad =-\frac {x-\bar Z}{2\sigma^2}\left(T\left(\frac{(1+\alpha a(\cdot))n}{1+\alpha I_n}\right)(x)+\mathcal O(\alpha)+ \mathcal O\big(W_1(n,\Gamma_{2\sigma^2}(\cdot-\bar Z))\big)\right)+ \mathcal O(\beta)\nonumber\\
&\quad \leq -2T\left(\frac{(1+\alpha a(\cdot))n}{1+\alpha I_n}\right)(x)+ C\alpha+C\beta.\label{eq:esttrucApp}
\end{align}
to conclude, we have to show that $C\alpha+C\beta\leq T\left(\frac{(1+\alpha a(\cdot))n}{1+\alpha I_n}\right)(x)$ for $x\in [R_0-1,R_0+1+R_1+2\sigma^2]$, when $\alpha,\beta>0$ are small. We notice that 
\begin{align*}
\int_{\mathbb R}(x-\bar Z)^2n(x)\,dx&=W_2(\delta_{\bar Z},n)^2\leq \left(W_2(\delta_{\bar Z},\Gamma_{2\sigma^2}(\cdot-\bar Z))+W_2(n,\Gamma_{2\sigma^2}(\cdot-\bar Z))\right)^2\leq (\sqrt 2\sigma+\beta)^2.
\end{align*}
Then, thanks to a Chebyshev's inequality, $\int_{\bar Z-\sqrt{2\left(\sqrt 2\sigma+\beta\right)}}^{\bar Z+\sqrt{2\left(\sqrt 2\sigma+\beta\right)}}n(x)\,dx\geq \frac 12$. As a consequence, for $x\in[R_0-1,R_0+1+R_1+2\sigma^2]$,
\begin{align*}
&T\left(\frac{(1+\alpha a(\cdot))n}{1+\alpha I_n}\right)(x)\geq \iint_{[\bar Z-2\l(\sqrt 2\sigma+\beta),\bar Z+2(\sqrt 2\sigma+\beta)]^2} \Gamma_{\sigma^2}\left(x-\frac{y_*+y_*'}2\right)n(y_*)n(y_*')\,dy_*\,dy_*'\\
&\quad\geq \Gamma_{2\sigma^2}\left(R_0+1+R_1+2\sigma^2+|\bar Z|+\sqrt{2(\sqrt 2\sigma+\beta)}\right)\left(\int_{\bar Z-\sqrt{2\left(\sqrt 2\sigma+\beta\right)}}^{\bar Z+\sqrt{2\left(\sqrt 2\sigma+\beta\right)}}n(x)\,dx\right)^2\\
&\quad \geq \frac 14\Gamma_{2\sigma^2}\left(R_0+1+R_1+2\sigma^2+|\bar Z|+\sqrt{2\left(\sqrt 2\sigma+\beta\right)}\right).
\end{align*}
$T\left(\frac{(1+\alpha a(\cdot))n}{1+\alpha I_n}\right)$ is then uniformly bounded from below for $x\in [R_0, R_0+1+R_1+2\sigma^2]$. If $\beta>0$ is small enough, \eqref{eq:esttrucApp} then implies 
\begin{equation*}
\forall x\in [R_0-1,R_0+1+R_1+2\sigma^2],\quad T\left(\frac{(1+\alpha a(\cdot))n}{1+\alpha I_n}\right)'(x)\leq -\frac 32T\left(\frac{(1+\alpha a(\cdot))n}{1+\alpha I_n}\right)(x),
\end{equation*}
which, together with \eqref{eq:TlargexApp}, implies
\begin{equation}\label{est:IR0App}
\forall x\in [R_0-1,\infty),\quad T\left(\frac{(1+\alpha a(\cdot))n}{1+\alpha I_n}\right)'(x)\leq -\frac 32T\left(\frac{(1+\alpha a(\cdot))n}{1+\alpha I_n}\right)(x),
\end{equation}
provided $\alpha,\beta>0$ are small enough. 

\medskip

Finally, we will show that $\int_{R_0-1}^\infty T\left(\frac{(1+\alpha a(\cdot))n}{1+\alpha I_n}\right)(x)\,dx\geq \eta$, which will prove \eqref{eq:condTn2}, thanks to \eqref{est:IR0App}. We recall notice now that $T(\Gamma_{2\sigma^2}(\cdot-Z_n))=T(\Gamma_{2\sigma^2}(\cdot-Z_n))$ (see \eqref{eq:Maxwellienne}), as well as the contraction implied by $T$ (see \eqref{eq:contraction}) on distributions with the same center of mass. Then,  for $\hat Z_n:=\int_{\mathbb R} x\frac{(1+\alpha a(x))n(x)}{1+\alpha I_n}\,dx$,
\begin{align}
&W_2\left(T\left(\frac{(1+\alpha a(\cdot))n}{1+\alpha I_n}\right), \Gamma_{2\sigma^2}(\cdot-Z_n)\right)=W_2\left(T\left(\frac{(1+\alpha a(\cdot))n}{1+\alpha I_n}\right), T(\Gamma_{2\sigma^2}(\cdot-Z_n))\right)\nonumber\\
&\quad \leq W_2\left(T\left(\frac{(1+\alpha a(\cdot))n}{1+\alpha I_n}\right), T(\Gamma_{2\sigma^2}(\cdot-\hat Z_n))\right)+\left|Z_n-\hat Z_n\right|\nonumber\\
&\quad \leq \frac 1{\sqrt 2} W_2\left(\frac{(1+\alpha a(\cdot))n}{1+\alpha I_n}, \Gamma_{2\sigma^2}(\cdot-Z_n)\right)+C\left|Z_n-\hat Z_n\right|.\label{eq:est:W2TG}
\end{align}
We can use the following estimate:
\begin{align*}
\left|Z_n-\hat Z_n\right|&= \left|\int_{\mathbb R} x\left(\frac{(1+\alpha a(x))n(x)}{1+\alpha I_n}-n(x)\right)\,dx\right|\\
&\leq \left|\frac 1{1+\alpha I_n}-1\right|\int_{\mathbb R} |x|(1+\alpha \|a\|_\infty)n(x)\,dx+\left|\alpha \int_{\mathbb R} x\left(a(x)-1\right)n(x)\,dx\right|\\
&\leq \alpha \|a\|_\infty(1+\alpha \|a\|_\infty)\int_{\mathbb R} \frac{1+x^2}2n(x)\,dx+\alpha(1+\|a\|_\infty)\int_{\mathbb R} \frac{1+x^2}2n(x)\,dx\leq C\alpha,
\end{align*}
while the following  decomposition of $\frac{(1+\alpha a(\cdot))n}{1+\alpha I_n}$:
\begin{align*}
\frac{(1+\alpha a)n}{1+\alpha I_n}&=\frac{1}{1+\alpha I_n}n+\frac{\alpha I_n}{1+\alpha I_n}\frac{a\,n}{I_n},
\end{align*}
provides the following estimate:
\begin{align*}
&W_2\left(\frac{(1+\alpha a)n}{1+\alpha I_n},\Gamma_{2\sigma^2}\left(\cdot-\hat Z_n\right)\right)^2\nonumber\\
&\quad\leq \frac 1{1+\alpha I_n}W_2\left(n,\Gamma_{2\sigma^2}\left(\cdot-\hat Z_n\right)\right)^2+\frac{\alpha I_n}{1+\alpha I_n}W_2\left(\frac{a n}{I_n},\Gamma_{2\sigma^2}\left(\cdot-\hat Z_n\right)\right)^2\nonumber\\
&\quad\leq \frac 1{1+\alpha I_n}\left(W_2\left(n,\Gamma_{2\sigma^2}\left(\cdot-Z_n\right)\right)+\left|Z_n-\hat Z_n\right|\right)^2+C\alpha\nonumber\\
&\quad\leq \frac 1{1+\alpha I_n}\big(W_2\left(n,\Gamma_{2\sigma^2}\left(\cdot-Z_n\right)\right)+C\alpha\big)^2+C\alpha,\label{eq:deconn}
\end{align*}
and then $W_2\left(\frac{(1+\alpha a)n}{1+\alpha I_n},\Gamma_{2\sigma^2}\left(\cdot-\hat Z_n\right)\right)\leq W_2\left(n,\Gamma_{2\sigma^2}\left(\cdot-Z_n\right)\right)+C\sqrt\alpha\leq C\beta+C\sqrt\alpha$. Then, \eqref{eq:est:W2TG} becomes
\begin{align*}
W_2\left(T\left(\frac{(1+\alpha a(\cdot))n}{1+\alpha I_n}\right), \Gamma_{2\sigma^2}(\cdot-\bar Z)\right)&\leq C\beta+C\sqrt\alpha.
\end{align*}
Then
\begin{align*}
&\int_{R_0-1}^\infty T\left(\frac{(1+\alpha a(\cdot))n}{1+\alpha I_n}\right)(x)\,dx\geq \int_{\mathbb R}\chi(x-R_0+1)T\left(\frac{(1+\alpha a(\cdot))n}{1+\alpha I_n}\right)(x)\,dx\\
&\quad \geq \int_{\mathbb R}\chi(x-R_0+1)\Gamma_{2\sigma^2}(x-\bar Z)\,dx+W_2\left(T\left(\frac{(1+\alpha a(\cdot))n}{1+\alpha I_n}\right), \Gamma_{2\sigma^2}(\cdot-\bar Z)\right)\\
&\quad>\int _{R_0}^\infty\Gamma_{2\sigma^2}(x-\bar Z)\,dx=\eta,
\end{align*}
provided $\beta>0$ is small enough. 
This estimate and \eqref{est:IR0App} show \eqref{eq:condTn2}. This argument can be reproduced to prove \eqref{eq:condTn1} and thus concludes the proof.
\end{proof}

\section{Estimates on macroscopic quantities}\label{App:macro}
We derive here technical estimates on $I_n-I_m$ and the center of mass of $Z_n$ and $Z_m$. We recall the definition \eqref{eq:defw} of the distance $w(n,m)$ that appears in these estimates.

\begin{lem}\label{lem:estmacro}

Let $a\in W^{1,\infty}(\mathbb R)$ compactly supported, and $\kappa>0$. There exists $C>0$ such that if $m,n\in\mathcal P_2(\mathbb R)$ satisfy $I_m,\,I_n\geq \kappa>0$ (where $I_m$, $I_n$ are defined from $m$, $n$ thanks to \eqref{def:In}), then for $\alpha\in(0,1)$,
\begin{equation}\label{est:ImoinsI}
|I_n-I_m|\leq C\left(|Z_n-Z_m|+w(n,m)\right),
\end{equation}
\begin{equation}\label{est:tildeZ}
\big|(\tilde Z_n-\tilde Z_m)-(Z_n-Z_m)\big|\leq C\alpha \left(|Z_n-Z_m|+w(n,m)\right),
\end{equation}
with $Z_n,\,Z_m$ defined from $m$, $n$ by \eqref{def:Z} and $\tilde Z_n,\,\tilde Z_m$ are
\begin{equation}\label{def:Ztilde}
\tilde Z_n:=\int_{\mathbb R}y\left(1-\alpha +\alpha \frac {a(y)}{I_n}\right)n(y)\,dy,\quad \tilde Z_m:=\int_{\mathbb R}y\left(1-\alpha +\alpha \frac {a(y)}{I_m}\right)m(y)\,dy.
\end{equation}
\end{lem}

\begin{proof}[Proof of Lemma~\ref{lem:estmacro}]
We estimate:
\begin{align*}
|I_n-I_m|&=\left|\int_{\mathbb R}a(y)n(y)\,dy-\int_{\mathbb R}a(y)m(y)\,dy\right|\nonumber\\
&=\bigg|\int_{\mathbb R}a(y-Z_n)n(y-Z_n)\,dy-\int_{\mathbb R}a(y-Z_m)n(y-Z_n)\,dy\nonumber\\
&\quad +\int_{\mathbb R}a(y-Z_m)\left(n(y-Z_n)-m(y-Z_m)\right)\,dy\bigg|\nonumber\\
&=\int_{\mathbb R}\left|\int_{y-Z_m}^{y-Z_n}a'(z)\,dz\right|n(y-Z_n)\,dy+\|a'\|_\infty W_1(n(\cdot-Z_n),m(\cdot-Z_m)),\nonumber
\end{align*}
thanks to a Kantorovich-Rubinstein inequality. Then,
\begin{align*}
I_n-I_m&\leq\|a'\|_\infty  |Z_n-Z_m|+\|a'\|_\infty W_2(n(\cdot-Z_n),m(\cdot-Z_m))\nonumber\\
&\leq C|Z_n-Z_m|+Cw(n,m),
\end{align*}
where we have used the property \eqref{eq:sigmaexplicit} satisfied by $w(\cdot,\cdot)$. To obtain the second inequality, we proceed as follows:
\begin{align*}
&\big|(\tilde Z_n-\tilde Z_m)-(Z_n-Z_m)\big|=\left|\int_{\mathbb R}y\left(\alpha -\alpha \frac {a(y)}{I_m}\right)m(y)\,dy-\int_{\mathbb R}y\left(\alpha -\alpha \frac {a(y)}{I_n(t)}\right)n(y)\,dy\right|\nonumber\\
&\quad \leq \alpha\big|Z_n-Z_m\big|+\alpha\left|\int_{\mathbb R}a(y)\left(\frac{n(y)}{I_n}-\frac{m(y)}{I_m}\right)\,dy\right|\nonumber\\
&\quad \leq \alpha\big|Z_n-Z_m\big|+\alpha\left|\int_{\mathbb R}a(y-Z_m)\left(\frac{n(y-Z_n)}{I_n}-\frac{m(y-Z_m)}{I_m}\right)\,dy\right|\nonumber\\
&\qquad +\alpha\left|\int_{\mathbb R}a(y)\left(\frac{n(y)}{I_n}-\frac{n(y-Z_n+Z_m)}{I_n}\right)\,dy\right|\nonumber\\
&\quad \leq \alpha\big|Z_n-Z_m\big|+\frac{C\alpha}{\kappa} w(n,m)+C\alpha\left|\frac 1{I_n}-\frac 1{I_m}\right|+\frac{\alpha}{\kappa}\left|\int_{\mathbb R}\Big(a(y)-a\big(y+Z_n-Z_m\big)\Big)n(y)\,dy\right|\nonumber\\
&\quad\leq C\alpha\Big(\big|Z_n-Z_m\big|+w(n,m)\Big),
\end{align*}
which concludes the proof of the lemma.
\end{proof}

\textbf{Acknowledgements.} This work was supported by the ANR project DEEV ANR-20- CE40- 0011-01. It was also partially supported by the ANR-16-CE35-0012 STEEP and the Chaire “Modélisation Mathématique et Biodiversité” of Veolia Environnement-École Polytechnique- Muséum national d’Histoire naturelle-Fondation X. It was also funded by the European Union (ERC-Adg SINGER, 101054787). The author thanks the referees for their exceptional work.

\bibliographystyle{abbrv}

\bibliography{biblioWass}

\end{document}